\documentclass[10pt,a4paper]{amsart}

\usepackage[left=2.8cm, right=2.8cm, top=2.2cm, bottom=2.2cm]{geometry}

\usepackage{amsmath,amssymb,amsthm,mathrsfs,stmaryrd,latexsym}
\usepackage{amsfonts}
\usepackage{accents}

\usepackage{tcolorbox}
\tcbuselibrary{skins, breakable, theorems}

\usepackage{enumerate} 
\usepackage{tikz}
\usepackage{etoolbox}

\def\Luoma#1{\uppercase\expandafter{\romannumeral#1}}
\def\luoma#1{\romannumeral#1}

\usepackage{url}

\newtheorem{mythm}{Theorem}[section]
\newtheorem{mylem}[mythm]{Lemma}
\newtheorem{myprop}[mythm]{Proposition}
\newtheorem{mycor}[mythm]{Corollary}

\newtheorem{myconj}[mythm]{Conjecture}

\theoremstyle{definition}
\newtheorem{mydefn}[mythm]{Definition}

\theoremstyle{remark}
\newtheorem{myrem}[mythm]{Remark}
\newtheorem{mypara}[mythm]{}

\usepackage[all]{xy}
\usepackage{graphicx}
\usepackage{subfigure}

\newcommand{\bb}{\mathbb}
\newcommand{\ca}{\mathcal}
\newcommand{\ak}{\mathfrak}
\newcommand{\scr}{\mathscr}

\newcommand{\mbf}{\mathbf}
\newcommand{\mrm}{\mathrm}

\newcommand{\trm}{\textrm}

\def\op#1{\mathop{\mathrm{#1}}}

\newcommand{\mo}{\mrm{Mor}}
\newcommand{\ho}{\mrm{Hom}}

\newcommand{\ke}{\mrm{Ker}}
\newcommand{\cok}{\mrm{Coker}}
\newcommand{\im}{\mrm{Im}}
\newcommand{\df}{\mrm{d}}
\newcommand{\id}{\mrm{id}}

\newcommand{\spec}{\op{Spec}}
\newcommand{\colim}{\op{colim}}

\newcommand{\rr}{\mrm{R}}
\newcommand{\dl}{\mrm{L}}

\newcommand{\iso}{\stackrel{\sim}{\longrightarrow}}

\newcommand{\plim}{\varprojlim}

\newcommand{\spa}{\op{Spa}}

\newcommand{\al}{\mrm{al}}

\newcommand{\triv}{\mrm{tr}}
\newcommand{\sen}{\mrm{Sen}}

\newcommand{\geo}{\mrm{geo}}

\newcommand{\lie}{\mrm{Lie}}
\newcommand{\gal}{\mrm{Gal}}

\newcommand{\sch}{\mbf{Sch}}
\newcommand{\schqcqs}{\mbf{Sch}^\mrm{coh}}

\newcommand{\repn}{\mbf{Rep}_{\mrm{cont}}}
\newcommand{\repnpr}{\mbf{Rep}^{\mrm{proj}}_{\mrm{cont}}}
\newcommand{\repnfr}{\mbf{Rep}^{\mrm{free}}_{\mrm{cont}}}

\def\repnan#1{\mbf{Rep}^{\mrm{proj}}_{\mrm{cont},{#1}\trm{-}\mrm{an}}}

\newcommand{\pro}{\mbf{Pro}}
\newcommand{\nbd}{\mbf{Nbd}}

\newcommand{\et}{\mrm{\acute{e}t}}
\newcommand{\fet}{\mrm{f\acute{e}t}}
\newcommand{\ket}{\mrm{k\acute{e}t}}

\newcommand{\aetale}{\mrm{a\acute{e}t}}

\newcommand{\profet}{\mrm{prof\acute{e}t}}
\newcommand{\proket}{\mrm{prok\acute{e}t}}

\newcommand{\fal}{\mbf{E}}

\newcommand{\falb}{\overline{\scr{B}}}

\newcommand{\rz}{\mrm{RZ}}

\def\ff#1{\scr{F}^{\mrm{fini}}_{\overline{#1}/#1}}

\title[Perfectoidness via Sen Theory]{Perfectoidness via Sen Theory and Applications to Shimura Varieties}%
\author{Tongmu He}
\date{\today}
\address{Tongmu He, Institute for Advanced Study, 1 Einstein Drive, 08540 New Jersey, the United States}
\email{hetm@ias.edu}

\usepackage{hyperref}
\hypersetup{colorlinks,
	citecolor=black,
	filecolor=black,
	linkcolor=black,
	urlcolor=black,
	pdftitle={Perfectoidness via Sen Theory and Applications to Shimura Varieties},
	pdfauthor={Tongmu He},
	pdftex}

\numberwithin{equation}{mythm}

\begin{document}
	\maketitle
	
\begin{abstract}
	Sen's theorem on the ramification of a $p$-adic analytic Galois extension of $p$-adic local fields shows that its perfectoidness is equivalent to the non-vanishing of its arithmetic Sen operator. By developing $p$-adic Hodge theory for general valuation rings, we establish a geometric analogue of Sen's criterion for any $p$-adic analytic Galois extension of $p$-adic varieties: its (Riemann-Zariski) stalkwise perfectoidness is necessary for the non-vanishing of the geometric Sen operators. As the latter is verified for general Shimura varieties by Pan and Rodr\'iguez Camargo, we obtain the perfectoidness of every completed stalk of general Shimura varieties at infinite level at $p$. As an application, we prove that the integral completed cohomology groups vanish in higher degrees, verifying a conjecture of Calegari-Emerton for general Shimura varieties.
\end{abstract}
\footnotetext{First published in \textit{Journal of the American Mathematical Society}, DOI: \url{https://doi.org/10.1090/jams/1060}.}
	
\tableofcontents

\section{Introduction}
\begin{mypara}
	Let $G$ be a reductive group over $\bb{Q}$. The $p$-adic Langlands program aims to establish a correspondence between algebraic automorphic representations of $G(\bb{A})$ (where $\bb{A}$ is the ring of ad\`eles of $\bb{Q}$) and $p$-adic Galois representations $\gal(\overline{\bb{Q}}/\bb{Q})\to {}^{L}G(\overline{\bb{Q}}_p)$ (where ${}^{L}G$ is the Langlands dual of $G$). To connect the automorphic and Galois representations, Emerton \cite{emerton2006interpolation} introduced the notion of \emph{completed cohomology} for $G$, which serves as a suitable space of $p$-adic automorphic forms. 
	
	Assume that $G$ admits a Shimura datum $(G,X)$. Let $(\mrm{Sh}_{K,\bb{C}})_K$ be the associated inverse system of Shimura varieties over $\bb{C}$, where $K\subseteq G(\bb{A}_f)$ runs through all the neat compact open subgroups. We note that each $\mrm{Sh}_{K,\bb{C}}$ is a quasi-projective smooth $\bb{C}$-scheme whose $\bb{C}$-points are canonically identified with
	\begin{align}
		\mrm{Sh}_K(\bb{C})=G(\bb{Q})\backslash (X\times G(\bb{A}_f))/K.
	\end{align}
	For any compact open subgroup $K^p\subseteq G(\bb{A}_f^p)$ (of level prime to $p$), Emerton's completed cohomology group for $(G,X)$ is defined as  
	\begin{align}
		\widetilde{H}^q(K^p,\bb{Z}_p)=\lim_{n\in\bb{N}}\colim_{K_p\subseteq G(\bb{Q}_p)}H^q_{\et}(\mrm{Sh}_{K^pK_p,\bb{C}},\bb{Z}/p^n\bb{Z}),
	\end{align}
	where $K_p\subseteq G(\bb{Q}_p)$ runs through all the neat compact open subgroups. Motivated by the $p$-adic Langlands program, Calegari-Emerton \cite{calegariemerton2012completed} predict the vanishing of completed cohomology in higher degrees:
\end{mypara}

\begin{myconj}[{Calegari-Emerton \cite[1.5]{calegariemerton2012completed}, cf. \cite[1.3]{hansenjohansson2023perfectoid}}]\label{conj:calegari-emerton}
	Let $d$ be the common dimension of $\mrm{Sh}_{K,\bb{C}}$. Then, for any integer $q>d$, we have
	\begin{align}
		\widetilde{H}^q(K^p,\bb{Z}_p)=0.
	\end{align}
\end{myconj}

\begin{mypara}
	Scholze \cite{scholze2015torsion} made the first fundamental progress on this conjecture. Indeed, for Shimura varieties of Hodge type, he proved that they are perfectoid as $p$-adic analytic spaces at infinite level at $p$. This established a profound connection between the \'etale cohomology of Shimura varieties with the analytic cohomology of certain coherent sheaves, leading to resolutions of numerous conjectures including the higher vanishing of the compactly supported completed cohomology for $(G,X)$ of Hodge type (a variant of Conjecture \ref{conj:calegari-emerton}).
	
	Afterwards, Scholze's perfectoidness and vanishing results were extended to proper Shimura varieties of abelian type by Shen \cite{shen2017perfectoid} and then to (non-proper) Shimura varieties of pre-abelian type by Hansen-Johansson \cite{hansenjohansson2023perfectoid}. Furthermore, by a careful analysis of the boundary cohomology, Hansen-Johansson also proved the vanishing for (non-compactly supported) completed cohomology, i.e., Conjecture \ref{conj:calegari-emerton} for $(G,X)$ of pre-abelian type. However, the perfectoidness for general Shimura varieties remains an open question, which forms a primary obstacle in fully proving this conjecture.
	
	More recently, motivated by Pan's work \cite{pan2021locally} on the locally analytic vectors in completed cohomology for modular curves, Rodr\'iguez Camargo \cite{camargo2022completed} proved a rational version of Conjecture \ref{conj:calegari-emerton} for general Shimura varieties, i.e., $\widetilde{H}^q(K^p,\bb{Z}_p)[1/p]=0$ for any integer $q>d$, using geometric Sen theory without proving the perfectoidness. 
\end{mypara}

\begin{mypara}\label{para:intro-claim}
	In this article, we prove Conjecture \ref{conj:calegari-emerton} completely. Indeed, for any Shimura variety at infinite level at $p$, we prove that its completed stalks are perfectoid and that the ramification at every boundary point has infinite exponent at $p$ (see \ref{intro-para:vanish-shimura}). This is sufficient to relate the \'etale cohomology with the analytic cohomology of certain coherent sheaves (see \ref{intro-para:vanish}) and thus sufficient to verify Calegari-Emerton's conjecture for general Shimura varieties (see \ref{intro-thm:vanish-shimura}). We remark that the implication from stalkwise perfectoidness to perfectoidness is discussed in \cite{he2024purity} and that Calegari-Emerton also predict the vanishing for locally symmetric spaces in \cite[1.5]{calegariemerton2012completed}. 
\end{mypara}

\begin{mypara}
	In fact, we establish a criterion for stalkwise perfectoidness via Sen theory. The very starting point is Sen's theorem on the ramification of a $p$-adic analytic Galois extension of $p$-adic local fields (contained in $\overline{\bb{Q}}_p$), which shows that the perfectoidness of the Galois extension is equivalent to the non-vanishing of its arithmetic Sen operator (see \ref{rem:geom-sen-nonzero}). Recently, the author \cite{he2022sen} gave a canonical construction of (arithmetic and geometric) Sen operators over $p$-adic varieties. It enables us to establish a geometric analogue of Sen's criterion for any $p$-adic analytic Galois extension of $p$-adic varieties over $\overline{\bb{Q}}_p$: its stalkwise perfectoidness (and $p$-infinite ramification at boundary points) is necessary for the non-vanishing of the geometric Sen operators (see \ref{intro-thm:sen-val-perfd}). As the latter condition is verified for general Shimura varieties by Pan \cite{pan2021locally} (for the curve case) and Rodr\'iguez Camargo \cite{camargo2022completed} (in general), we thus obtain the results claimed in \ref{para:intro-claim}. This approach to the perfectoidness of Shimura varieties is completely different from Scholze's and avoids any specialized analysis of their boundaries.
\end{mypara}

\begin{mypara}
	Although the statement of our perfectoidness criterion involves only the geometric Sen operators, the arithmetic ones are indispensable for its proof. In fact, arithmetic and geometric Sen operators are unified in a canonical way as an action by a canonical Lie algebra in \cite{he2022sen}. This canonical Lie algebra is defined as the (twisted) dual of the Faltings extension of the $p$-adic variety, where the latter is a central object in $p$-adic Hodge theory connecting differentials with Galois cohomology. On the one hand, the non-vanishing of the geometric Sen operators implies that the Faltings extension has enough many Galois invariants. On the other hand, the latter is related to the computation of Galois cohomology by differentials, and thus to the ramification of valuation rings. These relations are realized through Gabber-Ramero's computation of differentials, our construction of Galois equivariant Faltings extensions, and Tate's normalized trace maps for general (non-discrete) valuation rings (see \ref{intro-thm:perfd-val}), which enable us to prove the perfectoidness criterion and, furthermore, open up $p$-adic Hodge theory for general valuation rings.
\end{mypara}

\begin{mypara}
	We start by our study on the ramification of general valuation ring extensions. Let $K$ be a complete discrete valuation field extension of $\bb{Q}_p$ with perfect residue field, $\ca{F}$ a Henselian valuation field of height $1$ extension of $K$ with finite transcendental degree $\mrm{trdeg}_K(\ca{F})<\infty$, $\overline{\ca{F}}$ an algebraic closure of $\ca{F}$ whose valuation ring $\ca{O}_{\overline{\ca{F}}}$ is the integral closure of the valuation ring $\ca{O}_{\ca{F}}$ of $\ca{F}$ in $\overline{\ca{F}}$, $\widehat{\overline{\ca{F}}}$ the completion of $\overline{\ca{F}}$ with respect to the $p$-adic topology of $\ca{O}_{\overline{\ca{F}}}$. Then, there is a canonical exact sequence of finite free $\widehat{\overline{\ca{F}}}$-representations of $\gal(\overline{\ca{F}}/\ca{F})$ (see \ref{thm:fal-ext-val}),
	\begin{align}\label{eq:intro-thm:intro-fal-ext-val-3}
		\xymatrix{
			0\ar[r]&\widehat{\overline{\ca{F}}}(1)\ar[r]^-{\iota}&\scr{E}_{\ca{O}_{\ca{F}}}\ar[r]^-{\jmath}&\widehat{\overline{\ca{F}}}\otimes_{\ca{F}}\Omega^1_{\ca{F}/K}\ar[r]&0,
		}
	\end{align}
	called the \emph{Faltings extension of $\ca{O}_{\ca{F}}$ over $\ca{O}_K$ at $\ca{O}_{\overline{\ca{F}}}$}. Its construction is essentially the same as the construction for complete discrete valuation rings in \cite{he2021faltingsext}, but here we make full use of Gabber-Ramero's computation of the cotangent complexes of valuation rings \cite[\textsection6]{gabber2003almost}. Our perfectoidness criterion via Sen theory originates from the following criterion for geometric valuation rings via Faltings extension:
\end{mypara}

\begin{mythm}[{see \ref{thm:perfd-val}}]\label{intro-thm:perfd-val}
	Assume that $\ca{F}$ contains an algebraic closure $\overline{K}$ of $K$ and that the subspace $\scr{E}_{\ca{O}_{\ca{F}}}^{\gal(\overline{\ca{F}}/\ca{F})}$ of $\gal(\overline{\ca{F}}/\ca{F})$-invariant elements of $\scr{E}_{\ca{O}_{\ca{F}}}$ has dimension $1+\mrm{trdeg}_K(\ca{F})$ over $\widehat{\ca{F}}$. Then, $\Omega^1_{\ca{O}_{\ca{F}}/\ca{O}_{\overline{K}}}$ is $p$-divisible. In particular, $\ca{F}$ is a pre-perfectoid field {\rm(\ref{para:notation-perfd})}.
\end{mythm}

\begin{mypara}
	We expected such a perfectoidness criterion for the following philosophical reason revealed in \cite{he2022sen}: the (twisted) dual of Faltings extension should be regarded as the ``$p$-adic Lie algebra" of the arithmetic fundamental group of a $p$-adic variety. Previously, concrete examples of this principle only existed for complete discrete valuation rings by the work of Sen \cite{sen1972ramification} (see \ref{rem:geom-sen-nonzero}) and Ohkubo \cite{ohkubo2010bdr} (see \ref{rem:thm:perfd-val-3}). The proof of \ref{intro-thm:perfd-val} consists of the following main steps:
\begin{enumerate}
	\setcounter{enumi}{-1}
	\renewcommand{\labelenumi}{{\rm(\theenumi)}}
	\item Suppose that $\df t\in \Omega^1_{\ca{O}_{\ca{F}}/\ca{O}_{\overline{K}}}$ is not $p^\infty$-divisible for some $t\in \ca{O}_{\ca{F}}$. Then, the Kummer tower $(\ca{F}_n=\ca{F}(t^{1/p^n}))_{n\in\bb{N}}$ is non-trivial.
	\item As $\Omega^1_{\ca{O}_{\ca{F}}/\ca{O}_{\overline{K}}}$ is torsion-free (due to Gabber-Ramero), we can bound the valuation of the different ideals $\scr{D}_{\ca{O}_{\ca{F}_n}/\ca{O}_{\ca{F}}}$ by Gabber-Ramero's generalization of the relation between differentials and different ideals to general valuation rings (see \ref{prop:different-range}). This reflects the highly ramified nature of the tower $(\ca{F}_n)_{n\in\bb{N}}$.
	\item It allows us to construct Tate's normalized trace map $\widehat{\ca{F}_\infty}\to \widehat{\ca{F}}$, which is a Galois equivariant continuous $\widehat{\ca{F}}$-linear retraction of the inclusion $\widehat{\ca{F}}\to \widehat{\ca{F}_\infty}$ (see \ref{prop:tate-trace}).
	\item The existence of Tate's normalized trace map implies that the coboundary map $\delta:\widehat{\ca{F}}\otimes_{\ca{F}}\Omega^1_{\ca{F}/K}\to H^1(\gal(\overline{\ca{F}}/\ca{F}),\widehat{\overline{\ca{F}}}(1))$ induced by the Faltings extension \eqref{eq:intro-thm:intro-fal-ext-val-3} does not vanish on the element $\df t$ (see \ref{cor:ax-sen-tate}). This contradicts the assumption in \ref{intro-thm:perfd-val}, which implies that $\jmath:\scr{E}_{\ca{O}_{\ca{F}}}^{\gal(\overline{\ca{F}}/\ca{F})}\to \widehat{\ca{F}}\otimes_{\ca{F}}\Omega^1_{\ca{F}/K}$ is surjective (and thus $\delta=0$, see \ref{lem:perfd-val}).
\end{enumerate}
	Our proof is a continuation of the ideas of Tate \cite{tate1967p}, Fontaine \cite{fontaine1982formes} and Faltings \cite{faltings1988p}. Moreover, it also allows us to compute the Galois cohomology group $H^q(\gal(\overline{\ca{F}}/\ca{F}),\widehat{\overline{\ca{F}}})$ in terms of differentials via $\delta$ (see \ref{rem:thm:perfd-val-1}). This computation will be presented in the subsequent paper \cite{he2025galoiscoh} on the $p$-adic Galois cohomology of general valuation fields.
\end{mypara}

\begin{mypara}\label{intro-para:notation-Y}
	Then, we move to a geometric situation. Let $Y$ be an irreducible smooth $K$-scheme of finite type with a normal crossings divisor $D$, $\overline{\ca{K}}$ an algebraic closure of the fraction field $\ca{K}$ of $Y$, $\ca{L}$ a Galois extension of $\ca{K}$ contained in $\overline{\ca{K}}$ such that the integral closure $Y^{\ca{L}}$ of $Y$ in $\ca{L}$ is pro-finite \'etale over $Y^{\triv}=Y\setminus D$ and that $\ca{G}=\gal(\ca{L}/\ca{K})$ is a compact $p$-adic analytic group (i.e., a closed subgroup of $\mrm{GL}_n(\bb{Z}_p)$, see \cite[3.9]{he2022sen}). We would like to define Sen operators associated to the $p$-adic analytic Galois extension $Y^{\ca{L}}/Y$. To avoid introducing complicated topoi, we focus only on the construction at each \emph{geometric valuative point} of $Y$, that is, a point $\overline{y}=\spec(\overline{\ca{F}})$ of the integral closure $Y^{\overline{\ca{K}}}$ of $Y$ in $\overline{\ca{K}}$ endowed with a valuation ring $\ca{O}_{\overline{\ca{F}}}$ of height $1$ extension of $\ca{O}_K$ with fraction field $\overline{\ca{F}}$. This is sufficient for our applications.
\end{mypara}

\begin{mythm}[{see \ref{thm:fal-ext-val-pt} and \ref{thm:sen-val-pt}}]\label{intro-thm:sen-val-pt}
	There is a canonical exact sequence of finite free $\widehat{\overline{\ca{F}}}$-representations of $\gal(\overline{\ca{K}}/\ca{K}^{\mrm{vh}})$ (the decomposition group of $Y$ over $\ca{O}_K$ at the geometric valuative point $\overline{y}$, see {\rm\ref{para:notation-neighborhood-2}}),
	\begin{align}\label{eq:intro-thm:fal-ext-val-pt-3}
		0\longrightarrow \widehat{\overline{\ca{F}}}(1)\stackrel{\iota}{\longrightarrow}\scr{E}_{Y^{\triv}\to Y,\overline{y}}\stackrel{\jmath}{\longrightarrow} \widehat{\overline{\ca{F}}}\otimes_{\ca{O}_Y}\Omega^1_{(Y,\scr{M}_Y)/K}\longrightarrow 0,
	\end{align}
	where $\scr{M}_Y$ is the compactifying log structure associated to the open immersion $Y^{\triv}\to Y$, and a canonical $\gal(\overline{\ca{K}}/\ca{K}^{\mrm{vh}})$-equivariant homomorphism of $\widehat{\overline{\ca{F}}}$-linear Lie algebras,
	\begin{align}\label{eq:intro-thm:sen-val-pt-1}
		\varphi_{\sen}|_{\ca{G},\overline{y}}: \scr{E}^*_{Y^{\triv}\to Y,\overline{y}}(1)\longrightarrow \widehat{\overline{\ca{F}}}\otimes_{\bb{Q}_p}\lie(\ca{G})
	\end{align}
	where $\scr{E}^*_{Y^{\triv}\to Y,\overline{y}}(1)=\ho_{\widehat{\overline{\ca{F}}}}(\scr{E}_{Y^{\triv}\to Y,\overline{y}}(-1),\widehat{\overline{\ca{F}}})$ is endowed with the canonical Lie algebra structure associated to the linear form $\iota^*:\scr{E}^*_{Y^{\triv}\to Y,\overline{y}}(1)\to \widehat{\overline{\ca{F}}}$ (see {\rm\ref{para:val-fal-ext-dual}}).
\end{mythm}

\begin{mydefn}\label{intro-defn:sen-val-pt}
	We call the exact sequence \eqref{eq:intro-thm:fal-ext-val-pt-3}  the \emph{Faltings extension of the open immersion $Y^{\triv}\to Y$ over $\ca{O}_K$ at the geometric valuative point $\overline{y}=\spec(\overline{\ca{F}})$}, and we call the homomorphism \eqref{eq:intro-thm:sen-val-pt-1} the \emph{universal Sen action of the $p$-adic analytic Galois extension $\ca{L}$ of $\ca{K}$ at the geometric valuative point $\overline{y}=\spec(\overline{\ca{F}})$ of the open immersion $Y^{\triv}\to Y$ over $\ca{O}_K$}.
\end{mydefn}

\begin{mypara}
	Recall that in \cite{he2022sen} we constructed Faltings extensions and the universal Sen actions over certain affine $\ca{O}_K$-schemes given by a class of $\ca{O}_K$-algebras called \emph{adequate $\ca{O}_K$-algebras} (see Section \ref{sec:adequate}). These are local models of log smooth schemes by a theorem of Abbes-Gros and Tsuji. In order to construct Faltings extension and the universal Sen actions at each geometric valuative point of a general smooth $K$-scheme, it suffices to show that adequate $\ca{O}_K$-algebras (or log smooth schemes) form a neighborhood basis. This follows directly from the following variant of Temkin's admissibly \'etale local uniformization theorem (see \ref{prop:aet-neighborhood}):
\end{mypara}

\begin{mythm}[{see \ref{thm:uniformization}}]\label{intro-thm:uniformization}
	We put $\eta=\spec(K)$, $S=\spec(\ca{O}_K)$. Let $X$ be an $S$-scheme of finite presentation with smooth generic fibre $X_\eta$, $D_\eta$ a strict normal crossings divisor of $X_\eta$, $X^{\triv}=X_\eta\setminus D_\eta$. Then, there exists an admissibly \'etale covering (see {\rm\ref{defn:adm-etale}}),
	\begin{align}
		X'\longrightarrow S'\times_SX,
	\end{align}
	where $S'=\spec(\ca{O}_{K'})$, $K'$ is a finite field extension of $K$, and $X'$ is a flat $S'$-scheme of finite presentation such that the log scheme $(X',\scr{M}_{X'})$ endowed with the compactifying log structure associated to the open immersion $X'^{\triv}=X^{\triv}\times_XX'\to X'$ is a regular fs log scheme, smooth and saturated over the regular fs log scheme $(S',\scr{M}_{S'})$ endowed with the compactifying log structure defined by the closed point of $S'$.
\end{mythm}

\begin{mypara}
	Temkin \cite[2.5.2]{temkin2017etale} actually proved the case where $D_\eta=\emptyset$ (but over general valuation rings with stronger requirement that $X'$ is strictly semi-stable over $S'$). We show that his arguments still work in our setting. 
\end{mypara}

\begin{mypara}
	Combining \ref{intro-thm:sen-val-pt} with \ref{intro-thm:perfd-val}, we obtain the following perfectoidness criterion via the non-vanishing of the universal geometric Sen action:
\end{mypara}
\begin{mythm}[{see \ref{thm:sen-val-perfd}}]\label{intro-thm:sen-val-perfd}
	With the notation in {\rm\ref{intro-para:notation-Y}}, let $y\in Y$ and $y_{\ca{L}}\in Y^{\ca{L}}$ be the images of $\overline{y}\in Y^{\overline{\ca{K}}}$, and let $\{t_1,\dots, t_s\}$ be a regular system of parameters of the strict Henselization $\ca{O}_{Y,y}^{\mrm{sh}}$ of $Y$ at $\overline{y}$ such that $t_1\cdots t_r=0$ defines the normal crossings divisor $D$, where $0\leq r\leq s\leq\dim(Y)$. Assume that $\ca{L}$ contains a compatible system of primitive $p$-power roots of unity and that the restriction of the universal Sen action \eqref{eq:intro-thm:sen-val-pt-1} via the inclusion $\jmath^*$ (called the \emph{universal geometric Sen action}),
	\begin{align}\label{eq:intro-thm:sen-val-perfd-1}
		\varphi^{\mrm{geo}}_{\sen}|_{\ca{G},\overline{y}}: \ho_{\ca{O}_Y}(\Omega^1_{(Y,\scr{M}_Y)/K}(-1),\widehat{\overline{\ca{F}}})\longrightarrow \widehat{\overline{\ca{F}}}\otimes_{\bb{Q}_p}\lie(\ca{G}),
	\end{align}
	is injective. Then, we have the following properties:
	\begin{enumerate}
		\renewcommand{\labelenumi}{{\rm(\theenumi)}}
		\item (Stalkwise perfectoidness) The residue field $\ca{F}_{\ca{L}}$ of $y_{\ca{L}}$ is a pre-perfectoid field with respect to the valuation ring $\ca{O}_{\ca{F}_{\ca{L}}}=\ca{F}_{\ca{L}}\cap\ca{O}_{\overline{\ca{F}}}$.
		\item ($p$-infinite ramification on boundary) The elements $t_1,\dots,t_r$ admit compatible systems of $p$-power roots in the strict Henselization $\ca{O}_{Y^{\ca{L}},y_{\ca{L}}}^{\mrm{sh}}$ of $Y^{\ca{L}}$ at $\overline{y}$.
	\end{enumerate}
\end{mythm}

\begin{mypara}
	Indeed, the injectivity of $\varphi^{\mrm{geo}}_{\sen}|_{\ca{G},\overline{y}}$ implies the injectivity of $\varphi_{\sen}|_{\ca{G},\overline{y}}$ as $\ca{L}$ contains a cyclotomic extension of $K$ (see \ref{prop:val-geom-sen-nonzero}). Taking (twisted) dual, we see that $\ho_{\bb{Q}_p}(\lie(\ca{G}),\widehat{\overline{\ca{F}}}(1))\to \scr{E}_{Y^{\triv}\to Y,\overline{y}}$ is surjective, which implies that there are enough many invariants of $\scr{E}_{Y^{\triv}\to Y,\overline{y}}$ under the action of the decomposition group $\gal(\overline{\ca{K}}/\ca{L}^{\mrm{vh}})$ of $Y^{\ca{L}}$ over $\ca{O}_K$ at the geometric valuative point $\overline{y}$ (\ref{para:notation-neighborhood-2}). Thus, the stalkwise perfectoidness essentially follows from \ref{intro-thm:perfd-val} (see \ref{thm:perfd-val-log-1}). On the other hand, analyzing the action of the inertia subgroup $\gal(\overline{\ca{K}}/\ca{L}^{\mrm{sh}})$ of $Y^{\ca{L}}$ over $\ca{O}_K$ at the geometric valuative point $\overline{y}$ (\ref{para:notation-neighborhood-2}) on the Faltings extension, we obtain the $p$-infinite ramification at boundary points by a similar argument as in \ref{intro-thm:perfd-val} (see \ref{thm:perfd-val-log-2}).
	
	The stalkwise perfectoidness and $p$-infinite ramification at boundary points are sufficient for the vanishing of \'etale cohomology in higher degrees:
\end{mypara}

\begin{mythm}[{see \ref{thm:vanish}}]\label{intro-thm:vanish}
	Let $Y$ be a proper smooth $\overline{K}$-scheme, $D$ a normal crossings divisor on $Y$, $\widetilde{Y}^{\triv}$ a coherent scheme pro-finite \'etale over $Y^{\triv}=Y\setminus D$, $\widetilde{Y}$ the integral closure of $Y$ in $\widetilde{Y}^{\triv}$. Assume that the following conditions hold for any point $\widetilde{y}\in \widetilde{Y}$:
	\begin{enumerate}
		\renewcommand{\labelenumi}{{\rm(\theenumi)}}
		\item (Stalkwise perfectoidness) Its residue field $\kappa(\widetilde{y})$ is a pre-perfectoid field with respect to any valuation ring $W$ of height $1$ extension of $\ca{O}_{\overline{K}}$ with fraction field $W[1/p]=\kappa(\widetilde{y})$.\label{item:intro-thm:vanish-1}
		\item ($p$-infinite ramification on boundary) There exists a regular system of parameters $\{t_1,\dots,t_s\}$ of the strict Henselization $\ca{O}_{Y,y}^{\mrm{sh}}$ of $\ca{O}_{Y,y}$ (where $y\in Y$ is the image of $\widetilde{y}\in\widetilde{Y}$) such that $D$ is defined by $t_1\cdots t_r=0$ over $\ca{O}_{Y,y}^{\mrm{sh}}$ for some integer $0\leq r\leq s$ and that $t_1,\dots,t_r$ admit compatible systems of $p$-power roots in the strict Henselization $\ca{O}_{\widetilde{Y},\widetilde{y}}^{\mrm{sh}}$ of $\ca{O}_{\widetilde{Y},\widetilde{y}}$.\label{item:intro-thm:vanish-2}
	\end{enumerate}
	Then, for any integer $q>\dim(Y)$ and $n\in\bb{N}$, we have
	\begin{align}\label{eq:intro-thm:vanish-1}
		H^q_\et(\widetilde{Y}^{\triv},\bb{Z}/p^n\bb{Z})=0.
	\end{align}
\end{mythm}
\begin{mypara}\label{intro-para:vanish}
	Indeed, the $p$-infinite ramification on boundary implies that
	\begin{align}\label{eq:intro-thm:vanish-2}
		H^q_\et(\widetilde{Y}^{\triv},\bb{Z}/p^n\bb{Z})=H^q_\et(\widetilde{Y},\bb{Z}/p^n\bb{Z})
	\end{align}
	by Abhyankar's lemma and results on $K(\pi,1)$-schemes (see \ref{prop:open-et-coh}). Let $\widetilde{X}$ be the inverse limit of all flat proper integral models of finite subextensions of $\widetilde{Y}/Y$ as locally ringed spaces (see \ref{lem:vanish}). Then, the stalkwise perfectoidness implies that there is a canonical isomorphism of almost $\ca{O}_{\overline{K}}$-modules
	\begin{align}\label{eq:intro-thm:vanish-3}
		H^q_\et(\widetilde{Y},\bb{Z}/p^n\bb{Z})\iso H^q(\widetilde{X},\ca{O}_{\widetilde{X}}/p^n\ca{O}_{\widetilde{X}})
	\end{align}
	 by Faltings' main $p$-adic comparison theorem for (non-smooth) proper $\ca{O}_{\overline{K}}$-schemes \cite[5.17]{he2024falmain} and the cohomological descent for Faltings ringed topos \cite[8.24]{he2024coh} (see \ref{cor:fal-comp-perfd}). Then, the vanishing follows from Grothendieck's vanishing of sheaf cohomology on limits of Noetherian spectral spaces.
\end{mypara}

\begin{mypara}
	As an application, we consider a Shimura datum $(G,X)$ (\cite[2.1.1]{deligne1979shimura}, see also \cite[5.5]{milne2005shimura}) and let $E\subseteq \bb{C}$ be its reflex field. We denote by $\bb{A}_f$ (resp.  $\bb{A}_f^p$) the ring of (resp. prime-to-$p$) finite ad\`eles of $\bb{Q}$. For any neat compact open subgroup $K\subseteq G(\bb{A}_f)$, we denote by $\mrm{Sh}_K$ the canonical model of the Shimura variety associated to $(G,X)$ of level $K$. It is a quasi-projective smooth $E$-scheme, whose $\bb{C}$-points are canonically identified with
	\begin{align}
		\mrm{Sh}_K(\bb{C})=G(\bb{Q})\backslash (X\times G(\bb{A}_f))/K.
	\end{align}
	Moreover, these canonical models form a directed inverse system of $E$-schemes $(\mrm{Sh}_K)_{K\subseteq G(\bb{A}_f)}$ with finite \'etale transition morphisms. We denote by $d$ the common dimension of $\mrm{Sh}_K$.
\end{mypara}

\begin{mythm}[{see \ref{thm:vanish-shimura}}]\label{intro-thm:vanish-shimura}
	Let $K^p\subseteq G(\bb{A}_f^p)$ be a compact open subgroup. Then, for any integers $q>d$ and $n\in\bb{N}$, we have
	\begin{align}
		H^q_\et(\mrm{Sh}_{K^p,\bb{C}},\bb{Z}/p^n\bb{Z})=\colim_{K_p\subseteq G(\bb{Q}_p)}H^q_\et(\mrm{Sh}_{K^pK_p,\bb{C}},\bb{Z}/p^n\bb{Z})=0,
	\end{align}
	where $K_p\subseteq G(\bb{Q}_p)$ runs through all the neat compact open subgroups and $\mrm{Sh}_{K^p,\bb{C}}=\spec(\bb{C})\times_{\spec(E)}\mrm{Sh}_{K^p}=\lim_{K_p\subseteq G(\bb{Q}_p)} \mrm{Sh}_{K^pK_p,\bb{C}}$ is the Shimura variety of $(G,X)$ over $\bb{C}$ at the infinite level $K^p$.
\end{mythm}
\begin{mypara}\label{intro-para:vanish-shimura}
	Indeed, by Pan and Rodr\'iguez Camargo's computation of geometric Sen operators associated to the $p$-adic analytic Galois extension of toroidal compactifications of $\mrm{Sh}_{K^p,\overline{\bb{Q}}_p}$ over $\mrm{Sh}_{K^pK_p,\overline{\bb{Q}}_p}$, the universal geometric Sen action is injective at any geometric valuative point (see \ref{thm:camargo} and \ref{prop:camargo-sen-comp}). Then, our criterion \ref{intro-thm:sen-val-perfd} implies the stalkwise perfectoidness and $p$-infinite ramification on boundary of $\mrm{Sh}_{K^p,\overline{\bb{Q}}_p}$ (see \ref{rem:vanish-shimura-1}). Therefore, the vanishing of higher cohomology of Shimura varieties follows directly from \ref{intro-thm:vanish}. 
\end{mypara}

\begin{mypara}
	The article is structured as follows. In Section \ref{sec:fal-ext-val}, we construct Faltings extensions for general valuation rings. Then, we deduce a perfectoidness criterion for geometric valuation rings via Faltings extensions in Section \ref{sec:trace} by constructing Tate's normalized trace map. In Section \ref{sec:aet}, we prove a variant of Temkin's admissibly \'etale local uniformization theorem. We review the construction of Faltings extension and Sen operators over adequate algebras in Section \ref{sec:adequate} and discuss their functoriality. Then, we extend this construction to geometric valuative points of a smooth $p$-adic variety in Section \ref{sec:fal-ext} and \ref{sec:sen}, where we also extend our perfectoidness criterion to this geometric situation. In Section \ref{sec:vanish}, we prove that the stalkwise perfectoidness and $p$-infinite ramification on boundary are sufficient for the vanishing of \'etale cohomology in higher degrees. By comparison with Rodr\'iguez Camargo's construction of geometric Sen operators over log smooth adic spaces in Section \ref{sec:camargo}, we finally prove the Calegari-Emerton's conjecture on the vanishing of higher completed cohomology groups of Shimura varieties in Section \ref{sec:shimura}.
\end{mypara}

\subsection*{Acknowledgements}
I am grateful to Lue Pan and Juan Esteban Rodr\'{i}guez Camargo for their questions on perfectoidness after my Luminy lecture in June 2022, which was the starting point of this research. Many ingredients of this article originate from several projects during my doctoral studies under the guidance of Ahmed Abbes. I also thank him for reviewing and providing suggestions on the initial draft of this manuscript. Additionally, I would like to extend my gratitude to Kiran S. Kedlaya, Peter Scholze, and Takeshi Tsuji for their invaluable feedback and comments.

This material is based upon work supported by the National Science Foundation under Grant No. DMS-1926686 during my membership at Institute for Advanced Study in the ``$p$-adic Arithmetic Geometry" special year.

\section{Notation and Conventions}\label{sec:notation}
\begin{mypara}\label{para:product}
	Let $d$ be a natural number, i.e., $d\in\bb{N}$. We endow the set $(\bb{N}\cup\{\infty\})^d$ with the partial order defined by $\underline{m}\leq \underline{m'}$ if ether $m'_i=\infty$ or $m_i\leq m'_i<\infty$ for any $1\leq i\leq d$, where $\underline{m}=(m_1,\dots,m_d)$ and $\underline{m'}=(m_1',\dots,m_d')$. For any $r\in \bb{N}\cup\{\infty\}$, we set $\underline{r}=(r,\dots,r)\in (\bb{N}\cup\{\infty\})^d$.
	
	On the other hand, we endow the set $(\bb{N}_{>0}\cup\{\infty\})^d$ with the partial order defined by $\underline{N}| \underline{N'}$ if either $N'_i=\infty$ or $N_i$ divides $N'_i<\infty$ for any $1\leq i\leq d$, where $\underline{N}=(N_1,\dots,N_d)$ and $\underline{N'}=(N_1',\dots,N_d')$.
\end{mypara}

\begin{mypara}\label{chap1-para:notation-intclos}
	Following \cite[\Luoma{6}.1.22]{sga4-2}, a \emph{coherent} scheme (resp. morphism of schemes) stands for a quasi-compact and quasi-separated scheme (resp. morphism of schemes). For a coherent morphism $Y \to X$ of schemes, we denote by $X^Y$ the integral closure of $X$ in $Y$ (\cite[\href{https://stacks.math.columbia.edu/tag/0BAK}{0BAK}]{stacks-project}). When $Y=\spec(B)$ is affine, we also put $X^B=X^Y$.
\end{mypara}

\begin{mypara}\label{para:notation-Tate-mod}
	All monoids and rings considered in this article are unitary and commutative, and we fix a prime number $p$ throughout this article. For an abelian group $M$, we set
	\begin{align}
		T_p(M)=&\plim_{x\mapsto px}M[p^n]=\ho_{\bb{Z}}(\bb{Z}[1/p]/\bb{Z},M),\label{eq:para:notation-Tate-mod-1}\\
		V_p(M)=&\plim_{x\mapsto px}M=\ho_{\bb{Z}}(\bb{Z}[1/p],M).\label{eq:para:notation-Tate-mod-2}
	\end{align}
	We remark that $T_p(M)$ is a $p$-adically complete $\bb{Z}_p$-module (\cite[4.4]{jannsen1988cont}), and that if $M=M[p^\infty]$ (i.e., $M$ is $p$-primary torsion) then $V_p(M)=T_p(M)\otimes_{\bb{Z}_p}\bb{Q}_p$. We fix an algebraic closure $\overline{\bb{Q}}_p$ of $\bb{Q}_p$. We set $\bb{Z}_p(1)=T_p(\overline{\bb{Q}}_p^\times)$, which is a free $\bb{Z}_p$-module of rank $1$ and any compatible system of primitive $p$-power roots of unity $\zeta=(\zeta_{p^n})_{n\in \bb{N}}$ in $\overline{\bb{Q}}_p$ (i.e., $\zeta_{p^{n+1}}^p=\zeta_{p^n}$, $\zeta_1=1$, $\zeta_p\neq 1$) forms a basis of it. We endow $\bb{Z}_p(1)$ with the natural continuous action of the Galois group $\gal(\overline{\bb{Q}}_p/\bb{Q}_p)$. For any $\bb{Z}_p$-module $M$ and $r\in \bb{Z}$, we set $M(r)=M\otimes_{\bb{Z}_p}\bb{Z}_p(1)^{\otimes r}$, the $r$-th Tate twist of $M$.
\end{mypara}

\begin{mypara}\label{defn:repn}
	Let $G$ be a topological group, $A$ a topological ring endowed with a continuous action by $G$. An \emph{$A$-representation $(W,\rho_W)$ of $G$} is a topological $A$-module $W$ endowed with a continuous semi-linear action $\rho_W:G\times W\to W$ of $G$. We usually denote $(W,\rho_W)$ simply by $W$. A morphism $W\to W'$ of $A$-representations of $G$ is a continuous $A$-linear homomorphism compatible with the action of $G$. We denote by $\repn(G,A)$ the category of $A$-representations of $G$. 
	
	We say that an $A$-representation $W$ of $G$ is \emph{finite projective} (resp. \emph{finite free}) if $W$ is a finite projective (resp. finite free) $A$-module endowed with the canonical topology (\cite[page 820]{tsuji2018localsimpson}, see also \cite[2.3, 2.4]{he2022sen}). We denote by $\repnpr(G,A)$ (resp. $\repnfr(G,A)$) the full subcategory of  $\repn(G,A)$ consisting of finite projective (resp. finite free) $A$-representations of $G$. 
	
	Assume that the topology on $A$ is linear. Let $A'$ be a linearly topologized ring endowed with a continuous action of a topological group $G'$, $G'\to G$ a continuous group homomorphism, $A\to A'$ a continuous ring homomorphism compatible with the actions of $G$ and $G'$. Then, the tensor product defines a natural functor (see \cite[2.4]{he2022sen})
	\begin{align}
		\repnpr(G,A)\longrightarrow \repnpr(G',A'),\ W\mapsto A'\otimes_A W.
	\end{align}
\end{mypara}

\begin{mypara}\label{para:notation-val}
	A \emph{valuation field} is a pair $(K,\ca{O}_K)$ where $\ca{O}_K$ is a valuation ring with fraction field $K$ (\cite[\Luoma{6}.\textsection1.2, D\'efinition 2]{bourbaki2006commalg5-7}). We denote by $\ak{m}_K$ the maximal ideal of $\ca{O}_K$ and call the number of the nonzero prime ideals of $\ca{O}_K$ the \emph{height} (or \emph{rank}) of $K$ (\cite[\Luoma{6}.\textsection4.4, Proposition 5]{bourbaki2006commalg5-7}). We also refer to \cite[\href{https://stacks.math.columbia.edu/tag/00I8}{00I8}]{stacks-project} for basic properties on valuation rings. For a non-discrete valuation field $K$ of height $1$ (so that $\ak{m}_K=\ak{m}_K^2$), when referring to almost modules over $\ca{O}_K$ we always take $(\ca{O}_K,\ak{m}_K)$ as the basic setup (\cite[2.1.1]{gabber2003almost}). We also refer to \cite[\textsection5]{he2024coh} for a brief review on basic notions in almost ring theory.
\end{mypara}

\begin{mypara}\label{para:notation-perfd}
	Following \cite[\textsection5]{he2024coh}, a \emph{pre-perfectoid field} is a valuation field $K$ of height $1$ with non-discrete valuation of residue characteristic $p$ such that the Frobenius map on $\ca{O}_K/p\ca{O}_K$ is surjective (\cite[5.1]{he2024coh}). We note that for any nonzero element $\pi\in\ak{m}_K$ the fraction field $\widehat{K}$ of the $\pi$-adic completion of $\ca{O}_K$ is a perfectoid field in the sense of \cite[3.1]{scholze2012perfectoid} (see \cite[5.2]{he2024coh}). Given a pre-perfectoid field $K$, we say that an $\ca{O}_K$-algebra $R$ is (resp. \emph{almost}) \emph{pre-perfectoid} if there is a nonzero element $\pi\in\ak{m}_K$ such that $p\in \pi^p\ca{O}_K$, that the $\pi$-adic completion $\widehat{R}$ is (resp. almost) flat over $\ca{O}_{\widehat{K}}$ and that the Frobenius induces an (resp. almost) isomorphism $R/\pi R\to R/\pi^p R$ (see \cite[5.19]{he2024coh}). This definition does not depend on the choice of $\pi$ by \cite[5.23]{he2024coh} and is equivalent to the fact that the $\ca{O}_{\widehat{K}}$-algebra $\widehat{R}$ (resp. almost $\ca{O}_{\widehat{K}}$-algebra $\widehat{R}^{\al}$ associated to $\widehat{R}$) is perfectoid in the sense of \cite[3.10.(\luoma{2})]{bhattmorrowscholze2018integral} (resp. \cite[5.1.(\luoma{2})]{scholze2012perfectoid}, see \cite[5.18]{he2024coh}).
\end{mypara}

\section{Faltings Extension of General Valuation Rings}\label{sec:fal-ext-val}
We extend the construction of Faltings extension of complete discrete valuation rings \cite[4.4]{he2021faltingsext} to more general valuation rings of height $1$, using Gabber-Ramero's computation \cite[\textsection6]{gabber2003almost} of cotangent complexes of general valuation rings (see \ref{thm:fal-ext-val}).

\begin{mypara}\label{para:notation-fal-ext}
	In this section, we fix a complete discrete valuation field $K$ extension of $\bb{Q}_p$ with perfect residue field, and an algebraically closed valuation field $\ca{F}$ of height $1$ extension of $K$ with finite transcendental degree $d=\mrm{trdeg}_K(\ca{F})<\infty$. We identify $\overline{\bb{Q}}_p$ (fixed in \ref{para:notation-Tate-mod}) with the algebraic closure of $\bb{Q}_p$ in $\ca{F}$.
\end{mypara}

\begin{mypara}\label{para:notation-fal-ext-2}
	Let $A$ be an $\ca{O}_K$-subalgebra of $\ca{O}_{\ca{F}}$ such that $\ca{F}$ is algebraic over the fraction field $\ca{K}$ of $A$. Note that the valuation on $\ca{F}$ induces a valuation on $\ca{K}$ (i.e., $\ca{O}_{\ca{K}}=\ca{K}\cap \ca{O}_{\ca{F}}$, see \cite[3.1]{he2024purity}). Moreover, $\ca{O}_{\ca{F}}$ is the localization of the integral closure of $\ca{O}_{\ca{K}}$ in $\ca{F}$ at a maximal ideal (\cite[\Luoma{6}.\textsection8.6, Proposition 6]{bourbaki2006commalg5-7}). Let $\ca{O}_{\ca{K}^{\mrm{h}}}$ be the Henselization of $\ca{O}_{\ca{K}}$, which is still a valuation ring with fraction field $\ca{K}^{\mrm{h}}$ algebraic over $\ca{K}$ (\cite[6.1.12.(\luoma{6})]{gabber2003almost}). Then, there is a canonical extension $\ca{O}_{\ca{K}^{\mrm{h}}}\to \ca{O}_{\ca{F}}$ of valuation rings (\cite[\href{https://stacks.math.columbia.edu/tag/04GS}{04GS}]{stacks-project}).
	\begin{align}\label{eq:para:notation-fal-ext-2-1}
		\ca{O}_{\ca{F}}\longleftarrow \ca{O}_{\ca{K}^{\mrm{h}}}\longleftarrow \ca{O}_{\ca{K}}\longleftarrow A.
	\end{align}
	Notice that the group of $A$-algebra automorphisms of $\ca{O}_{\ca{F}}$ is identified with the decomposition subgroup of the absolute Galois group of $\ca{K}$ at the maximal ideal of $\ca{O}_{\ca{F}}$ and also with the absolute Galois group of $\ca{K}^{\mrm{h}}$ (see \cite[\Luoma{5}.\textsection2.3, Remarque 2]{bourbaki2006commalg5-7} and \cite[\href{https://stacks.math.columbia.edu/tag/0BSD}{0BSD}]{stacks-project}):
	\begin{align}
		\mrm{Aut}_A(\ca{O}_{\ca{F}})=\{\sigma\in \gal(\ca{F}/\ca{K})\ |\ \sigma(\ca{O}_{\ca{F}})=\ca{O}_{\ca{F}}\}=\gal(\ca{F}/\ca{K}^{\mrm{h}}).
	\end{align}
	It naturally acts on the module of differentials $\Omega^1_{\ca{O}_{\ca{F}}/A}$. Note that the Tate module $T_p(\Omega^1_{\ca{O}_{\ca{F}}/A})=\lim_{x\mapsto px}\Omega^1_{\ca{O}_{\ca{F}}/A}[p^n]$ is $p$-adically complete and that $\Omega^1_{\ca{O}_{\ca{F}}/A}[1/p]=\Omega^1_{\ca{F}/\ca{K}}=0$. Thus, 
	\begin{align}
		V_p(\Omega^1_{\ca{O}_{\ca{F}}/A})=\lim_{x\mapsto px}\Omega^1_{\ca{O}_{\ca{F}}/A}=T_p(\Omega^1_{\ca{O}_{\ca{F}}/A})[1/p]
	\end{align}
	is an $\widehat{\ca{F}}$-module (see \ref{para:notation-Tate-mod}) endowed with a natural action of $\gal(\ca{F}/\ca{K}^{\mrm{h}})$, where $\widehat{\ca{F}}$ is the fraction field of the $p$-adic completion of $\ca{O}_{\ca{F}}$ (also equal to the completion of $\ca{F}$ with respect to the topology induced by the $p$-adic topology of $\ca{O}_{\ca{F}}$).
	
	We note that for any $(s_{p^n})_{n\in\bb{N}}\in V_p(\ca{F}^\times)$ and any $k\in p\bb{N}$ such that $p^ks_1, p^ks_1^{-1}\in\ca{O}_{\ca{F}}$ (so that $p^ks_{p^n}^{\pm1}\in\ca{O}_{\ca{F}}$), the element $p^{-2k}(p^ks_{p^n}^{-1}\df(p^ks_{p^n}))_{n\in\bb{N}}\in V_p(\Omega^1_{\ca{O}_{\ca{F}}/A})$ does not depend on the choice of $k$, and we denote it by $(\df\log(s_{p^n}))_{n\in\bb{N}}$ (see \cite[5.6]{he2022sen}). Thus, there is a canonical $\gal(\ca{F}/\ca{K}^{\mrm{h}})$-equivariant group homomorphism
	\begin{align}\label{eq:para:notation-fal-ext-2-4}
		V_p(\ca{F}^\times)\longrightarrow  V_p(\Omega^1_{\ca{O}_{\ca{F}}/A}),\  (s_{p^n})_{n\in\bb{N}}\mapsto (\df\log(s_{p^n}))_{n\in\bb{N}}.
	\end{align}
\end{mypara}

\begin{mylem}\label{lem:val-mod-str}
	Let $V$ be a valuation ring, $M$ a finitely presented $V$-module. Then, there exists an isomorphism of $V$-modules
	\begin{align}
		M\cong V^{\oplus n}\oplus V/a_1V\oplus \cdots \oplus V/a_mV
	\end{align}
	for some $n,m\in\bb{N}$ and $a_1,\dots,a_m\in V$.
\end{mylem}
\begin{proof}
	The case for $M$ torsion is proved in \cite[6.1.14]{gabber2003almost}. In general, let $M_{\trm{tor}}$ be the submodule of torsion elements of $M$. Then, $M/M_{\trm{tor}}$ is a finitely generated torsion-free $V$-module so that it is finite free by \cite[\Luoma{6}.\textsection3.6, Lemme 1]{bourbaki2006commalg5-7}. Thus, there is a decomposition $M\cong M/M_{\trm{tor}}\oplus M_{\trm{tor}}$. In particular, the torsion $V$-module $M_{\trm{tor}}$ is finitely presented as a direct summand. The conclusion follows from the previous special case.
\end{proof}

\begin{myprop}[{cf. \cite[4.4]{he2021faltingsext}, \cite[5.7]{he2022sen}}]\label{prop:fal-ext-val}
	Let $A$ be a finitely generated $\ca{O}_K$-subalgebra of $\ca{O}_{\ca{F}}$ such that $\ca{F}$ is algebraic over the fraction field $\ca{K}$ of $A$. Then, there exists a canonical $\gal(\ca{F}/\ca{K}^{\mrm{h}})$-equivariant exact sequence of $\widehat{\ca{F}}$-modules,
	\begin{align}\label{eq:prop:fal-ext-val-1}
		\xymatrix{
			0\ar[r]&\widehat{\ca{F}}(1)\ar[r]^-{\iota}&V_p(\Omega^1_{\ca{O}_{\ca{F}}/A})\ar[r]^-{\jmath}&\widehat{\ca{F}}\otimes_{\ca{F}}\Omega^1_{\ca{F}/K}\ar[r]&0,
		}
	\end{align}
	where $\widehat{\ca{F}}(1)$ is the first Tate twist of the completion $\widehat{\ca{F}}$ of $\ca{F}$ {\rm (\ref{para:notation-Tate-mod})},
	satisfying the following properties:
	\begin{enumerate}
		\renewcommand{\labelenumi}{{\rm(\theenumi)}}
		\item For any compatible system of $p$-power roots of unity $(\zeta_{p^n})_{n\in\bb{N}}$ contained in $\ca{F}$, we have $\iota((\zeta_{p^n})_{n\in\bb{N}})=(\df\log(\zeta_{p^n}))_{n\in\bb{N}}$.\label{item:prop:fal-ext-val-1}
		\item For any $s\in \ca{F}^\times$ with compatible system of $p$-power roots $(s_{p^n})_{n\in\bb{N}}$ contained in $\ca{F}$, we have $\jmath((\df\log(s_{p^n}))_{n\in\bb{N}})=\df\log(s)$ \eqref{eq:para:notation-fal-ext-2-4}. \label{item:prop:fal-ext-val-2}
		\item For any transcendental basis $t_1,\dots,t_d$ of $\ca{F}$ over $K$ with compatible systems of $p$-power roots $(t_{1,p^n})_{n\in\bb{N}}, \dots,(t_{d,p^n})_{n\in\bb{N}}$ contained in $\ca{F}$, the $\widehat{\ca{F}}$-linear surjection $\jmath$ admits a section sending $\df\log(t_i)$ to $(\df\log(t_{i,p^n}))_{n\in\bb{N}}$.\label{item:prop:fal-ext-val-3}
	\end{enumerate}
	In particular, $V_p(\Omega^1_{\ca{O}_{\ca{F}}/A})$ is a finite free $\widehat{\ca{F}}$-module with basis $\{(\df\log(t_{i,p^n})_{n\in\bb{N}})\}_{0\leq i\leq d}$, where $t_{0,p^n}=\zeta_{p^n}$. Moreover, $\gal(\ca{F}/\ca{K}^{\mrm{h}})$ acts continuously on $V_p(\Omega^1_{\ca{O}_{\ca{F}}/A})$ with respect to the canonical topology {\rm(\cite[2.3]{he2022sen})}, where $\widehat{\ca{F}}$ is endowed with the topology induced by the $p$-adic topology of its valuation ring.
\end{myprop}
\begin{proof}
	Consider the exact sequence of modules of differentials associated to the maps $\ca{O}_K\to A\to \ca{O}_{\ca{F}}$,
	\begin{align}
		\xymatrix{
			\ca{O}_{\ca{F}}\otimes_A\Omega^1_{A/\ca{O}_K}\ar[r]^-{\alpha}&\Omega^1_{\ca{O}_{\ca{F}}/\ca{O}_K}\ar[r]^-{\beta}&\Omega^1_{\ca{O}_{\ca{F}}/A}\ar[r]&0.
		}
	\end{align}
	Firstly, we claim that $\ke(\alpha)$ is killed by a certain power of $p$. Indeed, there is a canonical commutative diagram
	\begin{align}
		\xymatrix{
			\ca{O}_{\ca{F}}\otimes_A\Omega^1_{A/\ca{O}_K}\ar[rr]^-{\alpha}\ar[d]&&\Omega^1_{\ca{O}_{\ca{F}}/\ca{O}_K}\ar[d]\\
			\ca{F}\otimes_A\Omega^1_{A/\ca{O}_K}\ar@{=}[r]&\ca{F}\otimes_{\ca{K}}\Omega^1_{\ca{K}/K}\ar@{=}[r]&\Omega^1_{\ca{F}/K}
		}
	\end{align}
	where the first identity follows from the fact that $\ca{K}$ is a localization of $A$, and the second identity follows from the fact that $\ca{F}$ is separable over $\ca{K}$ by assumption. In particular, $\ke(\alpha)$ is contained in the torsion submodule of $\ca{O}_{\ca{F}}\otimes_A\Omega^1_{A/\ca{O}_K}$. Since $A$ is of finite type over $\ca{O}_K$, $\Omega^1_{A/\ca{O}_K}$ is a finitely presented $A$-module (as $A$ is Noetherian). Thus, $\ca{O}_{\ca{F}}\otimes_A\Omega^1_{A/\ca{O}_K}$ is a finitely presented $\ca{O}_{\ca{F}}$-module, whose torsion submodule is killed by a certain power of $p$ by \ref{lem:val-mod-str}.
	
	Applying the functor $\ho_{\bb{Z}}(\bb{Z}/p^n\bb{Z},-)$ to the exact sequence $0\to \im(\alpha)\stackrel{\alpha}{\longrightarrow} \Omega^1_{\ca{O}_{\ca{F}}/\ca{O}_K}\stackrel{\beta}{\longrightarrow}\Omega^1_{\ca{O}_{\ca{F}}/A}\to 0$, we obtain a long exact sequence of $\ca{O}_{\ca{F}}$-modules,
	\begin{align}\label{eq:prop:fal-ext-val-4}
		\xymatrix{
			0\ar[r]&\im(\alpha)[p^n]\ar[r]&\Omega^1_{\ca{O}_{\ca{F}}/\ca{O}_K}[p^n]\ar[r]&\Omega^1_{\ca{O}_{\ca{F}}/A}[p^n]\ar[r]&\im(\alpha)/p^n\ar[r]&\Omega^1_{\ca{O}_{\ca{F}}/\ca{O}_K}/p^n.
		}
	\end{align}
	We have the following properties:
	\begin{enumerate}
		\renewcommand{\theenumi}{\roman{enumi}}
		\renewcommand{\labelenumi}{{\rm(\theenumi)}}
		\item The inverse system $(\im(\alpha)[p^n])_{n\in\bb{N}}$ is essentially zero. Indeed, since $\ke(\alpha)$ is killed by a certain power of $p$ and the torsion submodule of $\ca{O}_{\ca{F}}\otimes_A\Omega^1_{A/\ca{O}_K}$ is killed by a certain power of $p$ by \ref{lem:val-mod-str}, the torsion submodule of $\im(\alpha)$ is killed by $p^r$ for some $r\in\bb{N}$. Thus, the transition map $p^r:\im(\alpha)[p^{n+r}]\to \im(\alpha)[p^n]$ is zero for any $n\in\bb{N}$, which proves the claim.
		\item There is a canonical isomorphism $\ca{O}_{\widehat{\ca{F}}}(1)\iso \rr\lim_{n\in\bb{N}}(\Omega^1_{\ca{O}_{\ca{F}}/\ca{O}_K}[p^n])$ sending $(\zeta_{p^n})_{n\in\bb{N}}$ to $(\df\log(\zeta_{p^n}))_{n\in\bb{N}}$. Indeed, there is a canonical exact sequence (\cite[6.5.12.(\luoma{2})]{gabber2003almost}),
		\begin{align}
			\xymatrix{
				0\ar[r]&\ca{O}_{\ca{F}}\otimes_{\ca{O}_{\overline{K}}}\Omega^1_{\ca{O}_{\overline{K}}/\ca{O}_K}\ar[r]&\Omega^1_{\ca{O}_{\ca{F}}/\ca{O}_K}\ar[r]&\Omega^1_{\ca{O}_{\ca{F}}/\ca{O}_{\overline{K}}}\ar[r]&0,
			}
		\end{align}
		where $\overline{K}$ is the algebraic closure of $K$ in $\ca{F}$. Moreover, $\Omega^1_{\ca{O}_{\ca{F}}/\ca{O}_{\overline{K}}}$ is torsion-free as $\overline{K}$ is algebraically closed (\cite[6.5.20.(\luoma{1})]{gabber2003almost}). Thus, we have
		\begin{align}
			\Omega^1_{\ca{O}_{\ca{F}}/\ca{O}_K}[p^n]=&\ca{O}_{\ca{F}}\otimes_{\ca{O}_{\overline{K}}}\Omega^1_{\ca{O}_{\overline{K}}/\ca{O}_K}[p^n]=(\ca{O}_{\ca{F}}\otimes_{\ca{O}_{\overline{K}}}(\overline{K}/\ak{a})(1))[p^n]=\ca{O}_{\ca{F}}\otimes_{\ca{O}_{\overline{K}}}(p^{-n}\ak{a}/\ak{a})(1),
		\end{align}
		where the second identity follows from Fontaine's computation $\Omega^1_{\ca{O}_{\overline{K}}/\ca{O}_K}=(\overline{K}/\ak{a})(1)$ identifying $\df\log(\zeta_{p^n})$ with $p^{-n}(\zeta_{p^n})_{n\in\bb{N}}$, and where $\ak{a}$ is a fractional ideal of $\overline{K}$ containing $\ca{O}_{\overline{K}}$ (\cite[Th\'eor\`eme 1']{fontaine1982formes}). Then, we see that the inverse system $(\Omega^1_{\ca{O}_{\ca{F}}/\ca{O}_K}[p^n])_{n\in\bb{N}}$ satisfies the Mittag-Leffler condition, from which the claim follows (\cite[\href{https://stacks.math.columbia.edu/tag/07KW}{07KW}]{stacks-project}).
		\item For any $n\in\bb{N}$, $\Omega^1_{\ca{O}_{\ca{F}}/\ca{O}_K}/p^n=0$. Indeed, since $\ca{F}$ is algebraically closed, any element of $\ca{O}_{\ca{F}}$ admits a $p$-th root so that $\Omega^1_{\ca{O}_{\ca{F}}/\ca{O}_K}$ is $p$-divisible.
	\end{enumerate}
	Therefore, applying $\rr\lim_{n\in\bb{N}}$ to the exact sequence of inverse systems associated to \eqref{eq:prop:fal-ext-val-4}, we obtain an exact sequence
	\begin{align}
		\xymatrix{
			0\ar[r]&\ca{O}_{\widehat{\ca{F}}}(1)\ar[r]^-{\iota}&T_p(\Omega^1_{\ca{O}_{\ca{F}}/A})\ar[r]^-{\jmath}&\widehat{\im(\alpha)}\ar[r]&0.
		}
	\end{align}
	Again, since $\ke(\alpha)$ is killed by a certain power of $p$, the kernel and cokernel of the homomorphism of $p$-adic completions $(\ca{O}_{\ca{F}}\otimes_A\Omega^1_{A/\ca{O}_K})^\wedge\to \widehat{\im(\alpha)}$ are also killed by a certain power of $p$ (\cite[7.3]{he2022sen}). On the other hand, as $\ca{O}_{\ca{F}}\otimes_A\Omega^1_{A/\ca{O}_K}$ is a finitely presented $\ca{O}_{\ca{F}}$-module, we have $(\ca{O}_{\ca{F}}\otimes_A\Omega^1_{A/\ca{O}_K})^\wedge=\ca{O}_{\widehat{\ca{F}}}\otimes_A\Omega^1_{A/\ca{O}_K}$ by \ref{lem:val-mod-str}. In conclusion, after inverting $p$, we obtain a canonical exact sequence of $\widehat{\ca{F}}$-modules
	\begin{align}
		\xymatrix{
			0\ar[r]&\widehat{\ca{F}}(1)\ar[r]^-{\iota}&V_p(\Omega^1_{\ca{O}_{\ca{F}}/A})\ar[r]^-{\jmath}&\widehat{\ca{F}}\otimes_{\ca{F}}\Omega^1_{\ca{F}/K}\ar[r]&0,
		}
	\end{align}
	which is clearly $\gal(\ca{F}/\ca{K}^{\mrm{h}})$-equivariant by construction. Unwinding the construction, one can check easily the properties (\ref{item:prop:fal-ext-val-1}) and (\ref{item:prop:fal-ext-val-2}), so that (\ref{item:prop:fal-ext-val-3}) follows. 
	
	In particular, since $\{\df\log(t_i)\}_{1\leq i\leq d}$ forms a basis of $\Omega^1_{\ca{F}/K}$, we see that $V_p(\Omega^1_{\ca{O}_{\ca{F}}/A})$ is a finite free $\widehat{\ca{F}}$-module with basis $\{(\df\log(t_{i,p^n})_{n\in\bb{N}})\}_{0\leq i\leq d}$ by \eqref{eq:prop:fal-ext-val-1} and properties (\ref{item:prop:fal-ext-val-1}) and (\ref{item:prop:fal-ext-val-3}). The continuity of the $\gal(\ca{F}/\ca{K}^{\mrm{h}})$-action can be checked using the same argument as in \cite[5.7]{he2022sen} sketched as follows: let $\scr{E}^+$ be the finite free $\ca{O}_{\widehat{\ca{F}}}$-submodule of $V_p(\Omega^1_{\ca{O}_{\ca{F}}/A})$ generated by $\{(\df\log(t_{i,p^n})_{n\in\bb{N}})\}_{0\leq i\leq d}$. Then, it is stable under the action of $\gal(\ca{F}/\ca{K}^{\mrm{h}})$. Taking inverse limit on the continuous action $\gal(\ca{F}/\ca{K}^{\mrm{h}})\times V_p(\Omega^1_{\ca{O}_{\ca{F}}/A})/p^r\scr{E}^+\to V_p(\Omega^1_{\ca{O}_{\ca{F}}/A})/p^r\scr{E}^+$ over $r\in\bb{N}$ (where the $\ca{O}_{\widehat{\ca{F}}}$-module $V_p(\Omega^1_{\ca{O}_{\ca{F}}/A})/p^r\scr{E}^+$ is endowed with the discrete topology), we see that the action $\gal(\ca{F}/\ca{K}^{\mrm{h}})\times V_p(\Omega^1_{\ca{O}_{\ca{F}}/A})\to V_p(\Omega^1_{\ca{O}_{\ca{F}}/A})$ is continuous with respect to the limit topology on $V_p(\Omega^1_{\ca{O}_{\ca{F}}/A})$, which indeed coincides with the canonical topology on a finite free $\widehat{\ca{F}}$-module (\cite[2.3]{he2022sen}).
\end{proof}

\begin{mydefn}\label{defn:fal-ext-val-A}
	 We call the exact sequence \eqref{eq:prop:fal-ext-val-1} the \emph{Faltings extension of $A$ over $\ca{O}_K$ at $\ca{O}_{\ca{F}}$}.
\end{mydefn}

\begin{myrem}\label{rem:fal-ext-val}
	We see that the construction of \eqref{eq:prop:fal-ext-val-1} is functorial in $\ca{O}_K\to A\to\ca{O}_{\ca{F}}$. More precisely, let $K'$ be a complete discrete valuation field extension of $\bb{Q}_p$ with perfect residue field, $\ca{F}'$ an algebraically closed valuation field of height $1$ extension of $K'$ with finite transcendental degree, $A'$ a finitely generated $\ca{O}_{K'}$-subalgebra of $\ca{O}_{\ca{F}'}$ such that $\ca{F}'$ is algebraic over the fraction field $\ca{K}'$ of $A'$. Then, any commutative diagram of $\bb{Z}_p$-algebras extending the horizontal structural homomorphisms
	\begin{align}\label{eq:cor:fal-ext-val-1}
		\xymatrix{
			\ca{O}_{\ca{F}'}&A'\ar[l]&\ca{O}_{K'}\ar[l]\\
			\ca{O}_{\ca{F}}\ar[u]&A\ar[l]\ar[u]&\ca{O}_K\ar[u]\ar[l]
		}
	\end{align}
	induces a natural morphism of exact sequences
	\begin{align}\label{eq:cor:fal-ext-val-2}
		\xymatrix{
			0\ar[r]&\widehat{\ca{F}'}(1)\ar[r]^-{\iota}&V_p(\Omega^1_{\ca{O}_{\ca{F}'}/A'})\ar[r]^-{\jmath}&\widehat{\ca{F}'}\otimes_{\ca{F}'}\Omega^1_{\ca{F}'/K'}\ar[r]&0\\
			0\ar[r]&\widehat{\ca{F}}(1)\ar[u]\ar[r]^-{\iota}&V_p(\Omega^1_{\ca{O}_{\ca{F}}/A})\ar[u]\ar[r]^-{\jmath}&\widehat{\ca{F}}\otimes_{\ca{F}}\Omega^1_{\ca{F}/K}\ar[u]\ar[r]&0.
		}
	\end{align}
	In particular, if $\ca{F}'=\ca{F}$ and if $K'$ is finite over $K$, then $\widehat{\ca{F}'}(1)=\widehat{\ca{F}}(1)$ and $\Omega^1_{\ca{F}'/K'}=\Omega^1_{\ca{F}/K}$ so that the vertical homomorphisms in \eqref{eq:cor:fal-ext-val-2} are isomorphisms.
\end{myrem}

\begin{mypara}\label{para:fal-ext-val-func}
	We define a directed partially ordered set $\scr{C}_{\ca{O}_{\ca{F}}/\ca{O}_K}$ as follows: its elements are pairs $(L,B)$, where $L$ is a finite field extension of $K$ contained in $\ca{F}$ and $B$ is a finitely generated $\ca{O}_L$-subalgebra of $\ca{O}_{\ca{F}}$ such that $\ca{F}$ is algebraic over the fraction field of $B$; and $(L,B)\leq (L',B')$ if and only if $L\subseteq L'$ and $B\subseteq B'$. It is indeed directed, since for any two pairs $(L,B)$ and $(L',B')$, let $L''$ be the composite of $L$ and $L'$ in $\ca{F}$ and let $B''$ be the image of $B\otimes_{\ca{O}_{L}}\ca{O}_{L''}\otimes_{\ca{O}_{L'}}B'$ in $\ca{O}_{\ca{F}}$, then we see that $(L,B)\leq (L'',B'')\geq (L',B')$.
	
	The ordered set $\scr{C}_{\ca{O}_{\ca{F}}/\ca{O}_K}$ is endowed with a natural action of the group $\mrm{Aut}_{\ca{O}_K}(\ca{O}_{\ca{F}})$ of $\ca{O}_K$-algebra automorphisms of $\ca{O}_{\ca{F}}$ as follows: for any $\sigma\in \mrm{Aut}_{\ca{O}_K}(\ca{O}_{\ca{F}})$ and $(L,B)\in \scr{C}_{\ca{O}_{\ca{F}}/\ca{O}_K}$, the image $\sigma(L)$ is still a finite field extension of $K$ contained in $\sigma(\ca{F})=\ca{F}$ and $\sigma(B)$ is a finitely generated $\ca{O}_{\sigma(L)}$-subalgebra of $\sigma(\ca{O}_{\ca{F}})=\ca{O}_{\ca{F}}$ such that $\ca{F}$ is algebraic over the fraction field of $\sigma(B)$. Thus, we define $\sigma(L,B)=(\sigma(L),\sigma(B))$. It is clear that this action preserves the order.
	
	For any two elements $(L,B)\leq (L',B')$ in $\scr{C}_{\ca{O}_{\ca{F}}/\ca{O}_K}$, there is a natural isomorphism of exact sequences by \ref{rem:fal-ext-val},
	\begin{align}\label{eq:para:fal-ext-val-func-1}
		\xymatrix{
			0\ar[r]&\widehat{\ca{F}}(1)\ar[r]^-{\iota}&V_p(\Omega^1_{\ca{O}_{\ca{F}}/B'})\ar[r]^-{\jmath}&\widehat{\ca{F}}\otimes_{\ca{F}}\Omega^1_{\ca{F}/L'}\ar[r]&0\\
			0\ar[r]&\widehat{\ca{F}}(1)\ar[u]^-{\wr}\ar[r]^-{\iota}&V_p(\Omega^1_{\ca{O}_{\ca{F}}/B})\ar[u]^-{\wr}\ar[r]^-{\jmath}&\widehat{\ca{F}}\otimes_{\ca{F}}\Omega^1_{\ca{F}/L}\ar[u]^-{\wr}\ar[r]&0.
		}
	\end{align}
	Moreover, for any $\sigma\in \mrm{Aut}_{\ca{O}_K}(\ca{O}_{\ca{F}})$ and $(L,B)\in \scr{C}_{\ca{O}_{\ca{F}}/\ca{O}_K}$, there is a canonical isomorphism 
	\begin{align}\label{eq:para:fal-ext-val-func-2}
		\xymatrix{
			V_p(\Omega^1_{\ca{O}_{\ca{F}}/B})\ar[r]^-{\sim}_-{\sigma}&V_p(\Omega^1_{\sigma(\ca{O}_{\ca{F}})/\sigma(B)})\ar@{=}[r]&V_p(\Omega^1_{\ca{O}_{\ca{F}}/\sigma(B)})
		}
	\end{align}
	sending $(\omega_n)_{n\in\bb{N}}$ to $(\sigma(\omega_n))_{n\in\bb{N}}$. 
\end{mypara}

\begin{mythm}\label{thm:fal-ext-val}
	With the notation in {\rm\ref{para:notation-fal-ext}} and {\rm\ref{para:fal-ext-val-func}}, the $\widehat{\ca{F}}$-module
	\begin{align}\label{eq:thm:fal-ext-val-1}
		\scr{E}_{\ca{O}_{\ca{F}}}=\colim_{(L,B)\in \scr{C}_{\ca{O}_{\ca{F}}/\ca{O}_K}} V_p(\Omega^1_{\ca{O}_{\ca{F}}/B})
	\end{align}
	is endowed with a canonical action of $\mrm{Aut}_{\ca{O}_K}(\ca{O}_{\ca{F}})$. Moreover, there is a canonical $\mrm{Aut}_{\ca{O}_K}(\ca{O}_{\ca{F}})$-equivariant group homomorphism induced by \eqref{eq:para:notation-fal-ext-2-4},
	\begin{align}\label{eq:thm:fal-ext-val-2}
		V_p(\ca{F}^\times)\longrightarrow \scr{E}_{\ca{O}_{\ca{F}}},
	\end{align}
	for which we still denote the image of an element $(s_{p^n})_{n\in\bb{N}}\in V_p(\ca{F}^\times)$ by $(\df\log(s_{p^n}))_{n\in\bb{N}}$, and there is a canonical $\mrm{Aut}_{\ca{O}_K}(\ca{O}_{\ca{F}})$-equivariant exact sequence of $\widehat{\ca{F}}$-modules by taking filtered colimits of \eqref{eq:prop:fal-ext-val-1} over $\scr{C}_{\ca{O}_{\ca{F}}/\ca{O}_K}$,
	\begin{align}\label{eq:thm:fal-ext-val-3}
		\xymatrix{
			0\ar[r]&\widehat{\ca{F}}(1)\ar[r]^-{\iota}&\scr{E}_{\ca{O}_{\ca{F}}}\ar[r]^-{\jmath}&\widehat{\ca{F}}\otimes_{\ca{F}}\Omega^1_{\ca{F}/K}\ar[r]&0,
		}
	\end{align}
	satisfying the following properties:
	\begin{enumerate}
		\renewcommand{\labelenumi}{{\rm(\theenumi)}}
		\item For any compatible system of $p$-power roots of unity $(\zeta_{p^n})_{n\in\bb{N}}$ contained in $\ca{F}$, we have $\iota((\zeta_{p^n})_{n\in\bb{N}})=(\df\log(\zeta_{p^n}))_{n\in\bb{N}}$.\label{item:thm:fal-ext-val-1}
		\item For any $s\in \ca{F}^\times$ with compatible system of $p$-power roots $(s_{p^n})_{n\in\bb{N}}$ contained in $\ca{F}$, we have $\jmath((\df\log(s_{p^n}))_{n\in\bb{N}})=\df\log(s)$. \label{item:thm:fal-ext-val-2}
		\item For any transcendental basis $t_1,\dots,t_d$ of $\ca{F}$ over $K$ with compatible systems of $p$-power roots $(t_{1,p^n})_{n\in\bb{N}}, \dots,(t_{d,p^n})_{n\in\bb{N}}$ contained in $\ca{F}$, the $\widehat{\ca{F}}$-linear surjection $\jmath$ admits a section sending $\df\log(t_i)$ to $(\df\log(t_{i,p^n}))_{n\in\bb{N}}$.\label{item:thm:fal-ext-val-3}
	\end{enumerate}
\end{mythm}
\begin{proof}
	The $\mrm{Aut}_{\ca{O}_K}(\ca{O}_{\ca{F}})$-action on $\scr{E}_{\ca{O}_{\ca{F}}}$ is defined in \ref{para:fal-ext-val-func}. We see that the canonical homomorphism $V_p(\ca{F}^\times)\to  V_p(\Omega^1_{\ca{O}_{\ca{F}}/A})$ \eqref{eq:para:notation-fal-ext-2-4} is compatible with the transition morphisms in the colimit \eqref{eq:thm:fal-ext-val-1} and also compatible with the $\mrm{Aut}_{\ca{O}_K}(\ca{O}_{\ca{F}})$-action \eqref{eq:para:fal-ext-val-func-2}. Thus, we obtain a canonical $\mrm{Aut}_{\ca{O}_K}(\ca{O}_{\ca{F}})$-equivariant homomorphism $V_p(\ca{F}^\times)\to \scr{E}_{\ca{O}_{\ca{F}}}$ and we still denote the image of $(s_{p^n})_{n\in\bb{N}}$ by $(\df\log(s_{p^n}))_{n\in\bb{N}}$. Then, the rest of the properties follow directly from \ref{prop:fal-ext-val}. 
\end{proof}

\begin{mydefn}\label{defn:fal-ext-val}
	We call the exact sequence \eqref{eq:thm:fal-ext-val-3} the \emph{Faltings extension of $\ca{O}_{\ca{F}}$ over $\ca{O}_K$}.
\end{mydefn}

\begin{myrem}\label{rem:defn:fal-ext-val}
	The construction of \eqref{eq:thm:fal-ext-val-2} and \eqref{eq:thm:fal-ext-val-3} is functorial in $\ca{O}_K\to \ca{O}_{\ca{F}}$. More precisely, let $K'$ be a complete discrete valuation field extension of $\bb{Q}_p$ with perfect residue field and let $\ca{F}'$ be an algebraically closed valuation field of height $1$ extension of $K'$ with finite transcendental degree. Given a commutative diagram of $\bb{Z}_p$-algebras extending the horizontal structural homomorphisms
	\begin{align}\label{eq:rem:defn:fal-ext-val-1}
		\xymatrix{
			\ca{O}_{\ca{F}'}&\ca{O}_{K'}\ar[l]\\
			\ca{O}_{\ca{F}}\ar[u]&\ca{O}_K\ar[u]\ar[l]
		}
	\end{align}
	for any $(L,B)\in \scr{C}_{\ca{O}_{\ca{F}}/\ca{O}_K}$ and $(L',B')\in \scr{C}_{\ca{O}_{\ca{F}'}/\ca{O}_{K'}}$, we define $(L,B)\leq (L',B')$ if and only if $L\subseteq L'$ and $B\subseteq B'$ via the inclusion $\ca{F}\subseteq\ca{F}'$. It is clear that the subset $(\scr{C}_{\ca{O}_{\ca{F}'}/\ca{O}_{K'}})_{\geq (L,B)}$ is cofinal in $\scr{C}_{\ca{O}_{\ca{F}'}/\ca{O}_{K'}}$. By \ref{rem:fal-ext-val}, there is a natural morphism of exact sequences 
	\begin{align}
		\xymatrix{
			0\ar[r]&\widehat{\ca{F}'}(1)\ar[r]^-{\iota}&\colim_{(L',B')\in(\scr{C}_{\ca{O}_{\ca{F}'}/\ca{O}_{K'}})_{\geq (L,B)}}V_p(\Omega^1_{\ca{O}_{\ca{F}'}/B'})\ar[r]^-{\jmath}&\widehat{\ca{F}'}\otimes_{\ca{F}'}\Omega^1_{\ca{F}'/K'}\ar[r]&0\\
			0\ar[r]&\widehat{\ca{F}}(1)\ar[u]\ar[r]^-{\iota}&V_p(\Omega^1_{\ca{O}_{\ca{F}}/B})\ar[u]\ar[r]^-{\jmath}&\widehat{\ca{F}}\otimes_{\ca{F}}\Omega^1_{\ca{F}/K}\ar[u]\ar[r]&0.
		}
	\end{align}
	Taking filtered colimit over $(L,B)\in \scr{C}_{\ca{O}_{\ca{F}}/\ca{O}_K}$, we obtain a natural morphism of Faltings extensions
	\begin{align}\label{eq:rem:defn:fal-ext-val-4}
		\xymatrix{
			0\ar[r]&\widehat{\ca{F}'}(1)\ar[r]^-{\iota}&\scr{E}_{\ca{O}_{\ca{F}'}}\ar[r]^-{\jmath}&\widehat{\ca{F}'}\otimes_{\ca{F}'}\Omega^1_{\ca{F}'/K'}\ar[r]&0\\
			0\ar[r]&\widehat{\ca{F}}(1)\ar[u]\ar[r]^-{\iota}&\scr{E}_{\ca{O}_{\ca{F}}}\ar[u]\ar[r]^-{\jmath}&\widehat{\ca{F}}\otimes_{\ca{F}}\Omega^1_{\ca{F}/K}\ar[u]\ar[r]&0.
		}
	\end{align}
	Similarly, we obtain a natural commutative diagram
	\begin{align}\label{eq:rem:defn:fal-ext-val-5}
		\xymatrix{
			V_p(\ca{F}'^\times)\ar[r]&\scr{E}_{\ca{O}_{\ca{F}'}}\\
			V_p(\ca{F}^\times)\ar[u]\ar[r]&\scr{E}_{\ca{O}_{\ca{F}}}.\ar[u]
		}
	\end{align}
	We remark that if $\ca{F}'=\ca{F}$ and if $K'$ is finite over $K$, then the vertical homomorphisms in \eqref{eq:rem:defn:fal-ext-val-4} are isomorphisms.
\end{myrem}

\begin{myrem}\label{rem:fal-val-disc}
	One can still construct a Faltings extension for $\ca{O}_{\ca{F}}$ without the finiteness assumption on transcendental degree. Indeed, we put 
	\begin{align}\label{eq:rem:fal-val-disc-1}
		\scr{E}_{\ca{O}_{\ca{F}}}=\colim_{\ca{F}'\subseteq \ca{F}}\widehat{\ca{F}}\otimes_{\widehat{\ca{F}'}} \scr{E}_{\ca{O}_{\ca{F}'}}
	\end{align}
	where $\ca{F}'$ runs through all algebraically closed subextensions of $\ca{F}$ with finite transcendental degree over $K$ (see \ref{rem:defn:fal-ext-val}). Taking filtered colimit of the Faltings extensions of $\ca{O}_{\ca{F}'}$, we obtain a canonical $\mrm{Aut}_{\ca{O}_K}(\ca{O}_{\ca{F}})$-equivariant exact sequence of $\widehat{\ca{F}}$-modules
	\begin{align}\label{eq:rem:fal-val-disc-2}
		\xymatrix{
			0\ar[r]&\widehat{\ca{F}}(1)\ar[r]^-{\iota}&\scr{E}_{\ca{O}_{\ca{F}}}\ar[r]^-{\jmath}&\widehat{\ca{F}}\otimes_{\ca{F}}\Omega^1_{\ca{F}/K}\ar[r]&0,
		}
	\end{align}
	satisfying the same properties as in \ref{thm:fal-ext-val}. However, we don't know what topology to put on $\scr{E}_{\ca{O}_{\ca{F}}}$ such that the Galois action is continuous. More precisely, for any transcendental basis $\ak{T}=\{t_i\}_{i\in I}$ of $\ca{F}$ over $K$, $\{(\df\log(\zeta_{p^n}))_{n\in\bb{N}}\}\coprod \{(\df\log(t_{i,p^n}))_{n\in\bb{N}}\}_{i\in I}$ forms an $\widehat{\ca{F}}$-basis of $\scr{E}_{\ca{O}_{\ca{F}}}$ and generates a free $\ca{O}_{\widehat{\ca{F}}}$-submodule $\scr{E}^+_{\ak{T}}$. Then, for any Henselian subfield $\ca{F}_0$ of $\ca{F}$ with $\ca{F}/\ca{F}_0$ algebraic, the natural $\gal(\ca{F}/\ca{F}_0)$-action on $\scr{E}_{\ca{O}_{\ca{F}}}=\scr{E}^+_{\ak{T}}[1/p]$ is continuous with respect to the topology induced by the $p$-adic topology of the stable lattice $\scr{E}^+_{\ak{T}}$ (cf. \ref{prop:fal-ext-val}). However, this topology seems to rely heavily on the choice of $\ak{T}$ unless $|I|=\mrm{trdeg}_K(\ca{F})$ is finite. We wish to construct a completed version of \eqref{eq:rem:fal-val-disc-2} with natural topology in the future (as what we did in \cite[4.4]{he2021faltingsext} for any complete discrete valuation ring).
\end{myrem}

\begin{mylem}\label{lem:val-log-diff}
	Let $F$ be a valuation field of height $1$, $A\to\ca{O}_F$ a ring homomorphism. Then, the morphisms of log rings $(1\to A)\to (1\to \ca{O}_F)\to(\ca{O}_F\setminus 0\to \ca{O}_F)$ induce a canonical exact sequence of modules of log differentials
	\begin{align}
		0\longrightarrow \Omega^1_{\ca{O}_F/A}\longrightarrow \Omega^1_{(\ca{O}_F,\ca{O}_F\setminus 0)/A}\longrightarrow \ca{O}_F/\ak{m}_F\otimes_{\bb{Z}} F^\times/\ca{O}_F^\times\longrightarrow 0.
	\end{align}
	In particular, if $p\in\ak{m}_F$, then the canonical morphism
	\begin{align}
		V_p(\Omega^1_{\ca{O}_F/A})\longrightarrow V_p(\Omega^1_{(\ca{O}_F,\ca{O}_F\setminus 0)/A})
	\end{align}
	is an isomorphism.
\end{mylem}
\begin{proof}
	For $A=\bb{Z}$, this is a special case of \cite[6.4.15]{gabber2003almost}. The general case follows immediately from the following morphism of exact sequences
	\begin{align}
		\xymatrix{
			\ca{O}_F\otimes_{A}\Omega^1_{A/\bb{Z}}\ar[r]\ar@{=}[d]&\Omega^1_{\ca{O}_F/\bb{Z}}\ar[r]\ar[d]&\Omega^1_{\ca{O}_F/A}\ar[r]\ar[d]&0\\
			\ca{O}_F\otimes_{A}\Omega^1_{A/\bb{Z}}\ar[r]&\Omega^1_{(\ca{O}_F,\ca{O}_F\setminus 0)/\bb{Z}}\ar[r]&\Omega^1_{(\ca{O}_F,\ca{O}_F\setminus 0)/A}\ar[r]&0.
		}
	\end{align}
	In particular, if $p\in \ak{m}_F$, then the kernel and cokernel of $\Omega^1_{\ca{O}_F/A}\to \Omega^1_{(\ca{O}_F,\ca{O}_F\setminus 0)/A}$ are both killed by $p$. Thus, the kernel and cokernel of $V_p(\Omega^1_{\ca{O}_F/A})\longrightarrow V_p(\Omega^1_{(\ca{O}_F,\ca{O}_F\setminus 0)/A})$ are both killed by $p^2$ (\cite[7.3.(2)]{he2022sen}). Since multiplying by $p$ is invertible on both $V_p(\Omega^1_{\ca{O}_F/A})$ and  $V_p(\Omega^1_{(\ca{O}_F,\ca{O}_F\setminus 0)/A})$ \eqref{eq:para:notation-Tate-mod-2}, we conclude that $V_p(\Omega^1_{\ca{O}_F/A})= V_p(\Omega^1_{(\ca{O}_F,\ca{O}_F\setminus 0)/A})$.
\end{proof}

\section{Tate's Normalized Trace Map of General Valuation Rings}\label{sec:trace}
We establish a perfectoidness criterion for general geometric valuation fields by looking at the Galois action on Faltings extension (see \ref{thm:perfd-val}). The proof relies on Gabber-Ramero's computation \cite{gabber2003almost} of different ideals of general valuation rings and a construction of Tate's normalized trace map for non-discrete valuation rings (see \ref{prop:tate-trace}, cf. Tate's work \cite{tate1967p} on discrete valuation rings).

\begin{mypara}\label{para:notation-trace}
	In this section, we fix a complete discrete valuation field $K$ extension of $\bb{Q}_p$ with perfect residue field, and we fix an algebraically closed valuation field $\overline{\ca{F}}$ of height $1$ extension of $K$ with finite transcendental degree $d=\mrm{trdeg}_K(\overline{\ca{F}})<\infty$. Let $\overline{K}$ be the algebraic closure of $K$ in $\overline{\ca{F}}$ and we identify $\overline{\bb{Q}}_p$ (fixed in \ref{para:notation-Tate-mod}) with the algebraic closure of $\bb{Q}_p$ in $\overline{K}$. We also fix a Henselian valuation subfield $\ca{F}$ of $\overline{\ca{F}}$ extension of $\overline{K}$ such that $\overline{\ca{F}}$ is algebraic over $\ca{F}$ (so that $\overline{\ca{F}}$ is an algebraic closure of $\ca{F}$ and we have $d=\mrm{trdeg}_K(\ca{F})=\mrm{trdeg}_{\overline{K}}(\ca{F})$).
	\begin{align}
		\ca{O}_{\overline{\ca{F}}}\longleftarrow\ca{O}_{\ca{F}}\longleftarrow \ca{O}_{\overline{K}}\longleftarrow \ca{O}_K.
	\end{align}
	We put $G_{\ca{F}}=\gal(\overline{\ca{F}}/\ca{F})$ the absolute Galois group of $\ca{F}$. As $\ca{O}_{\ca{F}}$ is Henselian, the integral closure $\ca{O}_{\ca{F}'}$ of $\ca{O}_{\ca{F}}$ in any subextension $\ca{F}'$ of $\ca{F}$ in $\overline{\ca{F}}$ is a valuation ring of height $1$ (see \cite[\Luoma{6}.\textsection8.6, Proposition 6]{bourbaki2006commalg5-7} and \cite[\href{https://stacks.math.columbia.edu/tag/04GH}{04GH}]{stacks-project}). In particular, $\ca{O}_{\overline{\ca{F}}}$ is the integral closure of $\ca{O}_{\ca{F}}$ in $\overline{\ca{F}}$ and is stable under the action of $G_{\ca{F}}$.
\end{mypara}

\begin{mypara}\label{para:valution-norm}
	Let $v_p:\overline{\ca{F}}\to \bb{R}\cup\{\infty\}$ be a valuation map with $v_p(p)=1$ and $v_p(0)=\infty$ (\cite[\Luoma{6}.\textsection4.5, Proposition 7]{bourbaki2006commalg5-7}) and let 
	\begin{align}\label{eq:para:notation-trace-1}
		|\cdot|=p^{-v_p(\cdot)}: \overline{\ca{F}}\longrightarrow \bb{R}_{\geq 0}
	\end{align}
	be the associated ultrametric absolute value (\cite[\Luoma{6}.\textsection6.2, Proposition 3]{bourbaki2006commalg5-7}). We take $(\ca{O}_{\ca{F}},\ak{m}_{\ca{F}})$ as the basic setup for almost ring theory (\ref{para:notation-val}).
	
	A \emph{fractional ideal} $\ak{a}$ of $\ca{O}_{\ca{F}}$ is an $\ca{O}_{\ca{F}}$-submodule of $\ca{F}$ such that $x\ak{a}\subseteq \ca{O}_{\ca{F}}$ for some $x\in\ca{F}^\times$. As $\ca{F}$ is of height $1$, an $\ca{O}_{\ca{F}}$-submodule $\ak{a}$ of $\ca{F}$ is a fractional ideal if and only if $\ak{a}\neq \ca{F}$. For any fractional ideal $\ak{a}$ of $\ca{O}_{\ca{F}}$, we define the \emph{norm of $\ak{a}$} to be
	\begin{align}\label{eq:para:valution-norm-2}
		|\ak{a}|=\sup_{x\in\ak{a}}|x|\in\bb{R}_{\geq 0}.
	\end{align}
	Moreover, we have $|\ak{m}_{\ca{F}}\ak{a}|=|\ak{a}|$ as the valuation on $\ca{F}$ is non-discrete. We also note that for any fractional ideal $\ak{b}$ of $\ca{O}_{\ca{F}}$, we have $|\ak{a}\ak{b}|=|\ak{a}|\cdot |\ak{b}|$.
\end{mypara}

\begin{mypara}\label{para:different}
	For any finite field extension $\ca{F}'$ of $\ca{F}$ contained in $\overline{\ca{F}}$, recall that $\ca{O}_{\ca{F}'}$ is almost finite projective over $\ca{O}_{\ca{F}}$ (\cite[6.3.8]{gabber2003almost}). Thus, the canonical homomorphism of $\ca{O}_{\ca{F}}$-modules
	\begin{align}
		\omega_{\ca{O}_{\ca{F}'}/\ca{O}_{\ca{F}}}:\ca{O}_{\ca{F}'}\otimes_{\ca{O}_{\ca{F}}}\ca{O}_{\ca{F}'}^*\longrightarrow \mrm{End}_{\ca{O}_{\ca{F}}}(\ca{O}_{\ca{F}'}),\ x\otimes f\mapsto (y\mapsto f(y)x)
	\end{align}
	is an almost isomorphism (\cite[2.4.29.(\luoma{1}.b)]{gabber2003almost}, cf. \cite[\Luoma{5}.4.1]{abbes2016p}), where $\ca{O}_{\ca{F}'}^*=\ho_{\ca{O}_{\ca{F}}}(\ca{O}_{\ca{F}'},\ca{O}_{\ca{F}})$. We define the \emph{trace morphism} $\mrm{tr}_{\ca{O}_{\ca{F}'}/\ca{O}_{\ca{F}}}:\mrm{End}_{\ca{O}_{\ca{F}}}(\ca{O}_{\ca{F}'})\to \ho_{\ca{O}_{\ca{F}}}(\ak{m}_{\ca{F}},\ca{O}_{\ca{F}})$ to be the morphism fitting into the following commutative diagram (\cite[\Luoma{5}.4]{abbes2016p})
	\begin{align}
		\xymatrix{
			\mrm{End}_{\ca{O}_{\ca{F}}}(\ca{O}_{\ca{F}'})\ar[d]\ar[r]^-{\mrm{tr}_{\ca{O}_{\ca{F}'}/\ca{O}_{\ca{F}}}}&\ho_{\ca{O}_{\ca{F}}}(\ak{m}_{\ca{F}},\ca{O}_{\ca{F}})\\
			\ho_{\ca{O}_{\ca{F}}}(\ak{m}_{\ca{F}},\mrm{End}_{\ca{O}_{\ca{F}}}(\ca{O}_{\ca{F}'}))\ar[r]^-{\sim}_-{\omega^{-1}_{\ca{O}_{\ca{F}'}/\ca{O}_{\ca{F}}}}&\ho_{\ca{O}_{\ca{F}}}(\ak{m}_{\ca{F}},\ca{O}_{\ca{F}'}\otimes_{\ca{O}_{\ca{F}}}\ca{O}_{\ca{F}'}^*)\ar[u]
		}
	\end{align}
	where the left vertical arrow sends each element $g$ to the homomorphism $\ak{m}_{\ca{F}}\to \mrm{End}_{\ca{O}_{\ca{F}}}(\ca{O}_{\ca{F}'}),\ \epsilon\mapsto \epsilon g$, and the right vertical arrow is induced by the canonical morphism $\mrm{ev}_{\ca{O}_{\ca{F}'}/\ca{O}_{\ca{F}}}:\ca{O}_{\ca{F}'}\otimes_{\ca{O}_{\ca{F}}}\ca{O}_{\ca{F}'}^*\to \ca{O}_{\ca{F}},\ x\otimes f\mapsto f(x)$. Notice that the canonical homomorphism of $\ca{O}_{\ca{F}}$-modules
	\begin{align}
		\ca{O}_{\ca{F}}\longrightarrow \ho_{\ca{O}_{\ca{F}}}(\ak{m}_{\ca{F}},\ca{O}_{\ca{F}})
	\end{align}
	is an isomorphism as $\ca{O}_{\ca{F}}$ is a non-discrete valuation ring of height $1$ (see the proof of \cite[5.8]{he2024coh}). Then, we put  
	\begin{align}
		\mrm{Tr}_{\ca{O}_{\ca{F}'}/\ca{O}_{\ca{F}}}:\ca{O}_{\ca{F}'}\longrightarrow \ca{O}_{\ca{F}},\ x\mapsto \mrm{tr}_{\ca{O}_{\ca{F}'}/\ca{O}_{\ca{F}}}(\mu_x),
	\end{align}
	where $\mu_x\in \mrm{End}_{\ca{O}_{\ca{F}}}(\ca{O}_{\ca{F}'})$ is the multiplication by $x$. It is clear that the associated morphism of almost modules
	\begin{align}
		\mrm{Tr}_{\ca{O}_{\ca{F}'}/\ca{O}_{\ca{F}}}^{\al}:\ca{O}_{\ca{F}'}^{\al}\to \ca{O}_{\ca{F}}^{\al}
	\end{align}
	is the trace morphism defined in \cite[4.1.7]{gabber2003almost} and that after inverting $p$ we obtain the usual trace morphism for the finite field extension $\ca{F}'/\ca{F}$,
	\begin{align}
		\id_{\ca{F}}\otimes_{\ca{O}_{\ca{F}}}\mrm{Tr}_{\ca{O}_{\ca{F}'}/\ca{O}_{\ca{F}}}=\mrm{Tr}_{\ca{F}'/\ca{F}}:\ca{F}'\longrightarrow \ca{F},
	\end{align}
	which also shows that $\mrm{Tr}_{\ca{O}_{\ca{F}'}/\ca{O}_{\ca{F}}}$ is the restriction of $\mrm{Tr}_{\ca{F}'/\ca{F}}$ on $\ca{O}_{\ca{F}'}$.	Furthermore, there is a canonical homomorphism of $\ca{O}_{\ca{F}'}$-modules
	\begin{align}
		\tau_{\ca{O}_{\ca{F}'}/\ca{O}_{\ca{F}}}:\ca{O}_{\ca{F}'}\longrightarrow  \ca{O}_{\ca{F}'}^*,\ x\mapsto (y\mapsto \mrm{Tr}_{\ca{O}_{\ca{F}'}/\ca{O}_{\ca{F}}}(xy)).
	\end{align}
	We define the \emph{different ideal} of $\ca{O}_{\ca{F}'}$ over $\ca{O}_{\ca{F}}$ to be the annihilator of the $\ca{O}_{\ca{F}'}$-module $\cok(\tau_{\ca{O}_{\ca{F}'}/\ca{O}_{\ca{F}}})$,
	\begin{align}
		\scr{D}_{\ca{O}_{\ca{F}'}/\ca{O}_{\ca{F}}}=\mrm{Ann}_{\ca{O}_{\ca{F}'}}(\cok(\tau_{\ca{O}_{\ca{F}'}/\ca{O}_{\ca{F}}}))\subseteq \ca{O}_{\ca{F}'}.
	\end{align}
	It is clear that its associated almost ideal $\scr{D}_{\ca{O}_{\ca{F}'}/\ca{O}_{\ca{F}}}^\al$ is the different ideal of $\ca{O}_{\ca{F}'}^{\al}$ over $\ca{O}_{\ca{F}}^{\al}$ defined in \cite[4.1.22]{gabber2003almost}.
\end{mypara}

\begin{mylem}\label{lem:frac-ideal}
	Let $\ak{a}$ be a non-zero fractional ideal of $\ca{O}_{\ca{F}}$ and let $\ak{b}=\{x\in \ca{F}\ |\ x\ak{a}\subseteq \ca{O}_{\ca{F}}\}$. Then, we have $\ak{m}_{\ca{F}}\subseteq \ak{a}\ak{b}\subseteq \ca{O}_{\ca{F}}$.
\end{mylem}
\begin{proof}
	By the definition of $\ak{b}$, we have $x\ak{a}\subseteq \ca{O}_{\ca{F}}$ for any $x\in \ak{b}$ so that $\ak{a}\ak{b}\subseteq \ca{O}_{\ca{F}}$. On the other hand, for any $\epsilon\in \ak{m}_{\ca{F}}$, since the valuation ring $\ca{O}_{\ca{F}}$ is of height $1$ and $\ak{a}\neq \ca{F}$, there exists a nonzero element $y\in \ak{a}$ such that $\ak{a}\subseteq \epsilon^{-1}y\ca{O}_{\ca{F}}$. Hence, $y^{-1}\epsilon\in \ak{b}$ which implies that $\epsilon\in \ak{a}\ak{b}$.
\end{proof}

\begin{myprop}[{cf. \cite[\Luoma{3}.\textsection3, Proposition 7]{serre1979local}}]\label{prop:different-trace}
	With the notation in {\rm\ref{para:different}}, let $\ak{a}$ (resp. $\ak{a}'$) be a non-zero fractional ideal of $\ca{O}_{\ca{F}}$ (resp. $\ca{O}_{\ca{F}'}$). Then, $\mrm{Tr}_{\ca{F}'/\ca{F}}(\ak{m}_{\ca{F}}\ak{a}')\subseteq \ak{a}$ if and only if $\ak{m}_{\ca{F}}\ak{a}'\cdot\scr{D}_{\ca{O}_{\ca{F}'}/\ca{O}_{\ca{F}}}\subseteq \ak{a}\ca{O}_{\ca{F}'}$. 
\end{myprop}
\begin{proof}
	Recall that the trace morphism of $\ca{F}'$ over $\ca{F}$ induces an isomorphism $\tau_{\ca{F}'/\ca{F}}:\ca{F}'\iso \ca{F}'^*,\ x\mapsto (y\mapsto \mrm{Tr}_{\ca{F}'/\ca{F}}(xy))$ (\cite[4.1.14]{gabber2003almost}). By construction, there is a canonical commutative diagram
	\begin{align}
		\xymatrix{
			\ca{F}'\ar[r]^-{\sim}_-{\tau_{\ca{F}'/\ca{F}}}&\ca{F}'^*\\
			\ca{O}_{\ca{F}'}\ar[r]^-{\tau_{\ca{O}_{\ca{F}'}/\ca{O}_{\ca{F}}}}\ar[u]&\ca{O}_{\ca{F}'}^*\ar[u]
		}
	\end{align}
	where the vertical arrows are the canonical inclusions (as $\ca{O}_{\ca{F}'}$ and $\ca{O}_{\ca{F}'}^*$ are both torsion-free). Therefore, $\ca{O}_{\ca{F}'}^*$ is identified with the $\ca{O}_{\ca{F}'}$-submodule of $\ca{F}'$ defined by
	\begin{align}\label{eq:lem:different-trace-2}
		\ca{O}_{\ca{F}'}^*=\{x\in \ca{F}'\ |\ \mrm{Tr}_{\ca{F}'/\ca{F}}(xy)\in\ca{O}_{\ca{F}},\ \forall y\in \ca{O}_{\ca{F}'}\}.
	\end{align}
	It is clearly not equal to $\ca{F}'$ (as $\mrm{Tr}_{\ca{F}'/\ca{F}}(\ca{F}')=\ca{F}$) and is thus a fractional ideal containing $\ca{O}_{\ca{F}'}$. Then, we have 
	\begin{align}\label{eq:prop:different-trace-3}
		\scr{D}_{\ca{O}_{\ca{F}'}/\ca{O}_{\ca{F}}}=\mrm{Ann}_{\ca{O}_{\ca{F}'}}(\ca{O}_{\ca{F}'}^*/\ca{O}_{\ca{F}'})=\{x\in \ca{O}_{\ca{F}'}\ |\ x\ca{O}_{\ca{F}'}^*\subseteq \ca{O}_{\ca{F}'}\}.
	\end{align}	
	In particular, we have $\ak{m}_{\ca{F}'}\subseteq \scr{D}_{\ca{O}_{\ca{F}'}/\ca{O}_{\ca{F}}}\cdot\ca{O}_{\ca{F}'}^*\subseteq \ca{O}_{\ca{F}'}$ by \ref{lem:frac-ideal}. We put $\ak{b}=\{x\in \ca{F}\ |\ x\ak{a}\subseteq \ca{O}_{\ca{F}}\}$. The conclusion follows from the following equivalences
	\begin{align}
		\mrm{Tr}_{\ca{F}'/\ca{F}}(\ak{m}_{\ca{F}}\ak{a}')\subseteq \ak{a}\stackrel{(1)}{\Leftrightarrow} \mrm{Tr}_{\ca{F}'/\ca{F}}(\ak{m}_{\ca{F}}\ak{b}\ak{a}')\subseteq \ca{O}_{\ca{F}}\stackrel{(2)}{\Leftrightarrow} \ak{m}_{\ca{F}}\ak{b}\ak{a}'\subseteq \ca{O}_{\ca{F}'}^*\stackrel{(3)}{\Leftrightarrow} \ak{m}_{\ca{F}}\ak{a}'\scr{D}_{\ca{O}_{\ca{F}'}/\ca{O}_{\ca{F}}}\subseteq \ak{a}\ca{O}_{\ca{F}'}
	\end{align}
	proved as follows:
	\begin{enumerate}
		\renewcommand{\labelenumi}{{\rm(\theenumi)}}
		\item If $\mrm{Tr}_{\ca{F}'/\ca{F}}(\ak{m}_{\ca{F}}\ak{a}')\subseteq \ak{a}$, multiplying both sides by $\ak{b}$, we get $\mrm{Tr}_{\ca{F}'/\ca{F}}(\ak{m}_{\ca{F}}\ak{b}\ak{a}')\subseteq \ak{a}\ak{b}\subseteq \ca{O}_{\ca{F}}$ by \ref{lem:frac-ideal}. Conversely, if $\mrm{Tr}_{\ca{F}'/\ca{F}}(\ak{m}_{\ca{F}}\ak{b}\ak{a}')\subseteq \ca{O}_{\ca{F}}$, multiplying both sides by $\ak{a}$, we get $\mrm{Tr}_{\ca{F}'/\ca{F}}(\ak{m}_{\ca{F}}\ak{a}')=\mrm{Tr}_{\ca{F}'/\ca{F}}(\ak{m}_{\ca{F}}^2\ak{a}')\subseteq \mrm{Tr}_{\ca{F}'/\ca{F}}(\ak{m}_{\ca{F}}\ak{a}\ak{b}\ak{a}')\subseteq \ak{a}$ by $\ak{m}_{\ca{F}}=\ak{m}_{\ca{F}}^2$ and \ref{lem:frac-ideal}.
		\item It follows directly from \eqref{eq:lem:different-trace-2}.
		\item If $\ak{m}_{\ca{F}}\ak{b}\ak{a}'\subseteq \ca{O}_{\ca{F}'}^*$, multiplying both sides by $\ak{a}\scr{D}_{\ca{O}_{\ca{F}'}/\ca{O}_{\ca{F}}}$, we get $\ak{m}_{\ca{F}}\ak{a}'\scr{D}_{\ca{O}_{\ca{F}'}/\ca{O}_{\ca{F}}}=\ak{m}_{\ca{F}}^2\ak{a}'\scr{D}_{\ca{O}_{\ca{F}'}/\ca{O}_{\ca{F}}}\subseteq \ak{m}_{\ca{F}}\ak{a}\ak{b}\ak{a}'\scr{D}_{\ca{O}_{\ca{F}'}/\ca{O}_{\ca{F}}}\subseteq \ak{a}\scr{D}_{\ca{O}_{\ca{F}'}/\ca{O}_{\ca{F}}}\ca{O}_{\ca{F}'}^*\subseteq \ak{a}\ca{O}_{\ca{F}'}$ by $\ak{m}_{\ca{F}}=\ak{m}_{\ca{F}}^2$ and \ref{lem:frac-ideal}. Conversely, if $\ak{m}_{\ca{F}}\ak{a}'\scr{D}_{\ca{O}_{\ca{F}'}/\ca{O}_{\ca{F}}}\subseteq \ak{a}\ca{O}_{\ca{F}'}$, multiplying both sides by $\ak{b}\ca{O}_{\ca{F}'}^*$, we get $\ak{m}_{\ca{F}}\ak{b}\ak{a}'=\ak{m}_{\ca{F}}^2\ak{b}\ak{a}'\subseteq \ak{m}_{\ca{F}}\ak{b}\ak{a}'\scr{D}_{\ca{O}_{\ca{F}'}/\ca{O}_{\ca{F}}}\ca{O}_{\ca{F}'}^*\subseteq \ak{a}\ak{b} \ca{O}_{\ca{F}'}^*\subseteq \ca{O}_{\ca{F}'}^*$ by $\ak{m}_{\ca{F}}=\ak{m}_{\ca{F}}^2$ and \ref{lem:frac-ideal}.
	\end{enumerate}
\end{proof}

\begin{mycor}\label{cor:prop:different-trace}
	With the notation in {\rm\ref{para:different}}, let $|\scr{D}_{\ca{O}_{\ca{F}'}/\ca{O}_{\ca{F}}}|=\sup_{y\in \scr{D}_{\ca{O}_{\ca{F}'}/\ca{O}_{\ca{F}}}}|y|$ \eqref{eq:para:valution-norm-2}. Then, for any $x\in \ca{F}'$, we have
	\begin{align}
		|\mrm{Tr}_{\ca{F}'/\ca{F}}(x)|\leq |\scr{D}_{\ca{O}_{\ca{F}'}/\ca{O}_{\ca{F}}}|\cdot |x|.
	\end{align}
	In particular, we have $|[\ca{F}':\ca{F}]|\leq |\scr{D}_{\ca{O}_{\ca{F}'}/\ca{O}_{\ca{F}}}|$.
\end{mycor}
\begin{proof} 	
	As the valuation on $\ca{F}$ is non-discrete, we have
	\begin{align}
		|\mrm{Tr}_{\ca{F}'/\ca{F}}(x)|=\sup_{y\in \ak{m}_{\ca{F}}\cdot x}|\mrm{Tr}_{\ca{F}'/\ca{F}}(y)|\leq\sup_{y\in \ak{m}_{\ca{F}}\cdot x\ca{O}_{\ca{F}'}}|\mrm{Tr}_{\ca{F}'/\ca{F}}(y)|.
	\end{align}
	Applying \ref{prop:different-trace} to the fractional ideal $\ak{a}'=x\ca{O}_{\ca{F}'}$ and varying the fractional ideal $\ak{a}\subseteq \ca{F}$, we see that 
	\begin{align}
		\sup_{y\in \ak{m}_{\ca{F}}\cdot x\ca{O}_{\ca{F}'}}|\mrm{Tr}_{\ca{F}'/\ca{F}}(y)|= \sup_{y\in \ak{m}_{\ca{F}}\cdot x\scr{D}_{\ca{O}_{\ca{F}'}/\ca{O}_{\ca{F}}}}|y|=|\scr{D}_{\ca{O}_{\ca{F}'}/\ca{O}_{\ca{F}}}|\cdot |x|,
	\end{align}
	where we used again the non-discreteness of the valuation on $\ca{F}$.
	
	The ``in particular" part follows immediately by taking $x=1$.
\end{proof}

\begin{mypara}\label{para:tower}
	Let $t$ be an element of $\ca{F}$. We fix a compatible system of $p$-power roots $(t_{p^n})_{n\in\bb{N}}$ of $t$ contained in $\overline{\ca{F}}$ and for any $n\in\bb{N}$, we put $\ca{F}_n=\ca{F}(t_{p^n})$ the extension of $\ca{F}$ generated by $t_{p^n}$, which is a finite Galois extension of $\ca{F}$ independent of the choice of the $p^n$-th root $t_{p^n}$ of $t$. We put $\ca{F}_{\infty}=\bigcup_{n\in\bb{N}}\ca{F}_n$. We fix a compatible system of primitive $p$-power roots of unity $(\zeta_{p^n})_{n\in\bb{N}}$ contained in $\overline{K}\subseteq \ca{F}$. Then, there is a continuous group homomorphism
	\begin{align}\label{eq:para:tower-1}
		\xi_t:G_{\ca{F}}\longrightarrow \bb{Z}_p
	\end{align}
	characterized by $\tau(t_{p^n})=\zeta_{p^n}^{\xi_t(\tau)}t_{p^n}$ for any $\tau\in G_{\ca{F}}$ and $n\in\bb{N}$. It factors through a continuous injection $\gal(\ca{F}_{\infty}/\ca{F})\hookrightarrow \bb{Z}_p$.
\end{mypara}

\begin{mylem}\label{lem:tower}
	With the notation in {\rm\ref{para:tower}}, assume that $\ca{F}_\infty\neq \ca{F}$. Then, there exists $n_0\in\bb{N}$ such that $\ca{F}=\ca{F}_{n_0}$ and that \eqref{eq:para:tower-1} induces an isomorphism
	\begin{align}
		\gal(\ca{F}_\infty/\ca{F}_n)\iso p^n\bb{Z}_p
	\end{align}
	for any integer $n\geq n_0$. In particular, we have $[\ca{F}_n:\ca{F}]=p^{n-n_0}$.
\end{mylem}
\begin{proof}
	Since $\gal(\ca{F}_{\infty}/\ca{F})$ is a profinite group, $\xi_t$ identifies it with a closed subgroup of $\bb{Z}_p$, i.e., an ideal of $\bb{Z}_p$. As $\gal(\ca{F}_{\infty}/\ca{F})$ is nonzero by the assumption, it is identified with the subgroup $p^{n_0}\bb{Z}_p$ for some $n_0\in\bb{N}$. Then, we show by induction that $\gal(\ca{F}_\infty/\ca{F}_n)=p^n\bb{Z}_p$. For $n=n_0$, since $\ca{F}_{n_0}$ is fixed by $\xi_t^{-1}(p^{n_0}\bb{Z}_p)$ (which is equal to $G_{\ca{F}}$), we have $\ca{F}_{n_0}=\ca{F}$. 
	
	Assume that the claim holds for some integer $n\geq n_0$. Notice that $t_{p^{n+1}}$ is not fixed by the action of $\xi_t^{-1}(p^n\bb{Z}_p)$. Thus, $\ca{F}_{n+1}\neq \ca{F}_n$ so that $[\ca{F}_{n+1}:\ca{F}_n]=p$ and thus $\gal(\ca{F}_\infty/\ca{F}_{n+1})$ is identified with the unique subgroup of $p^n\bb{Z}_p$ of index $p$, i.e., $\gal(\ca{F}_\infty/\ca{F}_{n+1})=p^{n+1}\bb{Z}_p$.
\end{proof}

\begin{mypara}\label{para:ramified-tower}
	Following \ref{para:tower}, assume that $t$ is an element of $\ca{O}_{\ca{F}}$ such that the element $\df t$ of $\Omega^1_{\ca{O}_{\ca{F}}/\ca{O}_{\overline{K}}}$ is not \emph{$p^\infty$-divisible}, i.e., not contained in $\bigcap_{n\in\bb{N}}p^n\Omega^1_{\ca{O}_{\ca{F}}/\ca{O}_{\overline{K}}}$. In particular, we see that $\ca{F}$ does not contain a compatible system of $p$-power roots of $t$, i.e., $\ca{F}_\infty\neq \ca{F}$.
\end{mypara}

\begin{mylem}\label{lem:ann-differential}
	Under the assumption in {\rm\ref{para:ramified-tower}} and with the same notation in {\rm\ref{para:tower}}, there exists an integer $n_1>n_0$ {\rm(\ref{lem:tower})} such that for any integer $n\geq n_1$, the annihilator of the element $\df t_{p^n}$ of the $\ca{O}_{\ca{F}_n}$-module $\Omega^1_{\ca{O}_{\ca{F}_n}/\ca{O}_{\ca{F}}}$ satisfies the following relations
	\begin{align}
		p^nt\ca{O}_{\ca{F}_n}\subseteq \mrm{Ann}_{\ca{O}_{\ca{F}_n}}(\df t_{p^n})\subseteq p^{n-n_1}\ca{O}_{\ca{F}_n}.
	\end{align}
\end{mylem}
\begin{proof}
	The assumption implies that there exists $n_1\in\bb{N}$ such that $\df t$ is not divided by $p^{n_1}$ as an element of $\Omega^1_{\ca{O}_{\ca{F}}/\ca{O}_{\overline{K}}}$ (\ref{para:ramified-tower}). In particular, $\ca{F}\neq \ca{F}_{n_1}$ so that $n_1>n_0$. For any integer $n\geq n_1$, consider the canonical exact sequence (\cite[6.3.23]{gabber2003almost})
	\begin{align}
		\xymatrix{
			0\ar[r]&\ca{O}_{\ca{F}_n}\otimes_{\ca{O}_{\ca{F}}}\Omega^1_{\ca{O}_{\ca{F}}/\ca{O}_{\overline{K}}}\ar[r]^-{\alpha_n}&\Omega^1_{\ca{O}_{\ca{F}_n}/\ca{O}_{\overline{K}}}\ar[r]^-{\beta_n}&\Omega^1_{\ca{O}_{\ca{F}_n}/\ca{O}_{\ca{F}}}\ar[r]&0.
		}
	\end{align}
	We need to show that for the element $\df t_{p^n}$ of $\Omega^1_{\ca{O}_{\ca{F}_n}/\ca{O}_{\overline{K}}}$, we have $p^nt\df t_{p^n}\in \im(\alpha_n)$ but $p^{n-n_1}\df t_{p^n}\notin \im(\alpha_n)$. 
	
	For the first assertion, we actually have $p^nt\df t_{p^n}=p^nt_{p^n}^{p^n-1}t_{p^n}\df t_{p^n}=t_{p^n}\df t\in \im(\alpha_n)$.
	 
	For the second assertion, suppose that $p^{n-n_1}\df t_{p^n}\in \im(\alpha_n)$. Thus, there exists $\omega\in \ca{O}_{\ca{F}_n}\otimes_{\ca{O}_{\ca{F}}}\Omega^1_{\ca{O}_{\ca{F}}/\ca{O}_{\overline{K}}}$ such that $\alpha_n(\omega)=p^{n-n_1}t_{p^n}^{p^n-1}\df t_{p^n}$. Thus, $\alpha_n(p^{n_1}\omega)=\alpha_n(\df t)$. The injectivity of $\alpha_n$ implies that $\df t=p^{n_1}\omega$ in $ \ca{O}_{\ca{F}_n}\otimes_{\ca{O}_{\ca{F}}}\Omega^1_{\ca{O}_{\ca{F}}/\ca{O}_{\overline{K}}}$, i.e., $\df t$ is zero in $\ca{O}_{\ca{F}_n}\otimes_{\ca{O}_{\ca{F}}}\Omega^1_{\ca{O}_{\ca{F}}/\ca{O}_{\overline{K}}}/p^{n_1}$. Notice that $\Omega^1_{\ca{O}_{\ca{F}}/\ca{O}_{\overline{K}}}$ is a flat $\ca{O}_{\ca{F}}$-module as $\overline{K}$ is algebraically closed (\cite[6.5.20.(\luoma{1})]{gabber2003almost}). Thus, the canonical homomorphism $\Omega^1_{\ca{O}_{\ca{F}}/\ca{O}_{\overline{K}}}/p^{n_1}\to \ca{O}_{\ca{F}_n}\otimes_{\ca{O}_{\ca{F}}}\Omega^1_{\ca{O}_{\ca{F}}/\ca{O}_{\overline{K}}}/p^{n_1}$ is injective as $\ca{O}_{\ca{F}}/p^{n_1}\to \ca{O}_{\ca{F}_n}/p^{n_1}$ is so. Hence, $\df t$ is also zero in $\Omega^1_{\ca{O}_{\ca{F}}/\ca{O}_{\overline{K}}}/p^{n_1}$, which contradicts our choice of $n_1$. Therefore, $p^{n-n_1}\df t_{p^n}\notin \im(\alpha_n)$.
\end{proof}

\begin{myprop}[{cf. \cite[\textsection3.1, Proposition 5]{tate1967p}}]\label{prop:different-range}
	Under the assumption in {\rm\ref{para:ramified-tower}} and with the same notation in {\rm\ref{para:tower}}, let $n_1\in\bb{N}$ be the integer defined in {\rm\ref{lem:ann-differential}}. Then, for any integer $n\geq n_1$, we have
	\begin{align}
		p^{n-n_0}\ak{m}_{\ca{F}_n}\subseteq \scr{D}_{\ca{O}_{\ca{F}_n}/\ca{O}_{\ca{F}}}\subseteq p^{n-n_1}\ca{O}_{\ca{F}_n}.
	\end{align}
	In particular, we have $|\scr{D}_{\ca{O}_{\ca{F}_n}/\ca{O}_{\ca{F}}}|\leq p^{n_1-n}$.
\end{myprop}
\begin{proof}
	Recall that $\Omega^1_{\ca{O}_{\ca{F}_n}/\ca{O}_{\ca{F}}}$ is a uniformly almost finitely generated $\ca{O}_{\ca{F}_n}$-module (\cite[6.3.8]{gabber2003almost}) and thus one can define the $0$-th Fitting ideal $F_0(\Omega^{1,\al}_{\ca{O}_{\ca{F}_n}/\ca{O}_{\ca{F}}})$ of its associated almost module (\cite[2.3.24]{gabber2003almost}). By \cite[6.3.23]{gabber2003almost}, the associated almost module of the different ideal $\scr{D}_{\ca{O}_{\ca{F}_n}/\ca{O}_{\ca{F}}}^\al$ is isomorphic to $F_0(\Omega^{1,\al}_{\ca{O}_{\ca{F}_n}/\ca{O}_{\ca{F}}})$. In particular, $\ak{m}_{\ca{F}}\scr{D}_{\ca{O}_{\ca{F}_n}/\ca{O}_{\ca{F}}}$ annihilates $\Omega^1_{\ca{O}_{\ca{F}_n}/\ca{O}_{\ca{F}}}$ (\cite[6.3.6.(\luoma{3})]{gabber2003almost}). Hence, we have $\ak{m}_{\ca{F}}\scr{D}_{\ca{O}_{\ca{F}_n}/\ca{O}_{\ca{F}}}\subseteq p^{n-n_1}\ca{O}_{\ca{F}_n}$ by \ref{lem:ann-differential}. Then, it is easy to see that $\scr{D}_{\ca{O}_{\ca{F}_n}/\ca{O}_{\ca{F}}}\subseteq p^{n-n_1}\ca{O}_{\ca{F}_n}$ (cf. \ref{para:valution-norm}). In particular, we have $|\scr{D}_{\ca{O}_{\ca{F}_n}/\ca{O}_{\ca{F}}}|\leq p^{n_1-n}$.
	
	On the other hand, we have $|p^{n-n_0}|=|[\ca{F}_n:\ca{F}]|\leq |\scr{D}_{\ca{O}_{\ca{F}_n}/\ca{O}_{\ca{F}}}|$ by \ref{cor:prop:different-trace} and \ref{lem:tower} (as $n_1>n_0$). Thus, $p^{n-n_0}\ak{m}_{\ca{F}_n}\subseteq \scr{D}_{\ca{O}_{\ca{F}_n}/\ca{O}_{\ca{F}}}$, which completes the proof.
\end{proof}

\begin{mycor}[{cf. \cite[\textsection3.1, Corollary 3]{tate1967p}}]\label{cor:trace-range}
	Under the assumption in {\rm\ref{para:ramified-tower}} and with the same notation in {\rm\ref{para:tower}}, let $n_1\in\bb{N}$ be the integer defined in {\rm\ref{lem:ann-differential}}. Then, for any integer $n\geq n_1$ and any $x\in \ca{F}_n$, we have
	\begin{align}
		|p^{-n}\mrm{Tr}_{\ca{F}_n/\ca{F}}(x)|\leq p^{n_1}|x|.
	\end{align}
\end{mycor}
\begin{proof}
	It follows directly from \ref{cor:prop:different-trace} and \ref{prop:different-range}.
\end{proof}

\begin{mypara}\label{para:tate-trace}
	With the notation in {\rm\ref{para:tower}}, for any $n\in\bb{N}$, the normalized trace map $[\ca{F}_n:\ca{F}]^{-1}\mrm{Tr}_{\ca{F}_n/\ca{F}}:\ca{F}_n\to \ca{F}$ is an $\ca{F}$-linear retraction of the inclusion $\ca{F}\to \ca{F}_n$. It is also $G_{\ca{F}}$-equivariant as $\gal(\ca{F}_\infty/\ca{F})$ is commutative (\ref{lem:tower}). In particular, we obtain a system of maps $([\ca{F}_n:\ca{F}]^{-1}\mrm{Tr}_{\ca{F}_n/\ca{F}}:\ca{F}_n\to \ca{F})_{n\in\bb{N}}$ compatible with the inclusions $\ca{F}_n\to\ca{F}_{n+1}$. Taking filtered union, we obtain a $G_{\ca{F}}$-equivariant $\ca{F}$-linear retraction
	\begin{align}\label{eq:para:tate-trace-1}
		\ca{T}:\ca{F}_\infty\longrightarrow \ca{F}
	\end{align}
	of the inclusion $\ca{F}\to\ca{F}_\infty$. In particular, there is a canonical decomposition of $\ca{F}$-modules for any $n\in\bb{N}\cup\{\infty\}$,
	\begin{align}\label{eq:para:tate-trace-2}
		\ca{F}_n=\ca{F}\oplus \ke(\ca{T}|_{\ca{F}_n}),
	\end{align}
	where $\ca{T}|_{\ca{F}_n}$ denotes the restriction of $\ca{T}$ on $\ca{F}_n$.
\end{mypara}

\begin{mylem}[{cf. \cite[\textsection3.1, Lemma 2]{tate1967p}}]\label{lem:trace-tau}
	With the notation in {\rm\ref{para:tower}}, let $\tau\in \gal(\ca{F}_\infty/\ca{F})$ be a topological generator (which exists by {\rm\ref{lem:tower}}). Then, for any $n\in \bb{N}$ and $x\in \ca{F}_{n+1}$, we have 
	\begin{align}
		|\mrm{Tr}_{\ca{F}_{n+1}/\ca{F}_n}(x)-[\ca{F}_{n+1}:\ca{F}_n]\cdot x|\leq  |\tau(x)-x|.
	\end{align}
\end{mylem}
\begin{proof}
	Note that $\gal(\ca{F}_{n+1}/\ca{F}_n)$ is a cyclic group of order $a=[\ca{F}_{n+1}:\ca{F}_n]$ generated by $\tau^b$ (where $b=[\ca{F}_n:\ca{F}]$) by \ref{lem:tower}. Then, we have
	\begin{align}
		\mrm{Tr}_{\ca{F}_{n+1}/\ca{F}_n}(x)-ax=\sum_{i=0}^{a-1}(\tau^{bi}-1)(x)=\sum_{i=1}^{a-1}(\tau^{bi-1}+\cdots+\tau+1)(\tau-1)(x).
	\end{align}
	Hence, $|\mrm{Tr}_{\ca{F}_{n+1}/\ca{F}_n}(x)-ax|\leq |\tau(x)-x|$.
\end{proof}

\begin{myprop}[{cf. \cite[\textsection3.1, Proposition 6]{tate1967p}}]\label{prop:trace-tau}
	With the notation in {\rm\ref{para:tower}}, let $\tau\in \gal(\ca{F}_\infty/\ca{F})$ be a topological generator (which exists by {\rm\ref{lem:tower}}). Then, the sequence of positive real numbers
	\begin{align}\label{eq:prop:trace-tau-1}
		([\ca{F}_n:\ca{F}]\cdot |\scr{D}_{\ca{O}_{\ca{F}_n}/\ca{O}_{\ca{F}}}|)_{n\in\bb{N}}
	\end{align}
	is non-decreasing. Moreover, for any $n\in\bb{N}$ and $x\in\ca{F}_n$, we have
	\begin{align}\label{eq:prop:trace-tau-2}
		|\ca{T}(x)-x|\leq p\cdot [\ca{F}_n:\ca{F}]\cdot|\scr{D}_{\ca{O}_{\ca{F}_n}/\ca{O}_{\ca{F}}}|\cdot |\tau(x)-x|,
	\end{align}
	where $\ca{T}$ is the normalized trace map \eqref{eq:para:tate-trace-1}.
\end{myprop}
\begin{proof}
	Recall that the associated almost different ideals satisfy the relation (\cite[4.1.25]{gabber2003almost})
	\begin{align}
		\scr{D}_{\ca{O}_{\ca{F}_{n+1}}/\ca{O}_{\ca{F}}}^{\al}=\scr{D}_{\ca{O}_{\ca{F}_{n+1}}/\ca{O}_{\ca{F}_n}}^\al\cdot \scr{D}_{\ca{O}_{\ca{F}_n}/\ca{O}_{\ca{F}}}^\al.
	\end{align}
	Equivalently, we have $\ak{m}_{\ca{F}}\scr{D}_{\ca{O}_{\ca{F}_{n+1}}/\ca{O}_{\ca{F}}}=\ak{m}_{\ca{F}}\scr{D}_{\ca{O}_{\ca{F}_{n+1}}/\ca{O}_{\ca{F}_n}}\cdot \ak{m}_{\ca{F}}\scr{D}_{\ca{O}_{\ca{F}_n}/\ca{O}_{\ca{F}}}$. Taking norms, we obtain that (see \ref{para:valution-norm})
	\begin{align}\label{eq:prop:trace-tau-3}
		|\scr{D}_{\ca{O}_{\ca{F}_{n+1}}/\ca{O}_{\ca{F}}}|=|\scr{D}_{\ca{O}_{\ca{F}_{n+1}}/\ca{O}_{\ca{F}_n}}|\cdot |\scr{D}_{\ca{O}_{\ca{F}_n}/\ca{O}_{\ca{F}}}|.
	\end{align}
	Since $[\ca{F}_{n+1}:\ca{F}_n]$ is equal to $1$ or $p$, its norm is equal to $[\ca{F}_{n+1}:\ca{F}_n]^{-1}$. Thus,
	\begin{align}\label{eq:prop:trace-tau-4}
		[\ca{F}_{n+1}:\ca{F}_n]\cdot|\scr{D}_{\ca{O}_{\ca{F}_{n+1}}/\ca{O}_{\ca{F}_n}}|\geq [\ca{F}_{n+1}:\ca{F}_n]\cdot |[\ca{F}_{n+1}:\ca{F}_n]|=1,
	\end{align}
	where the inequality follows from \ref{cor:prop:different-trace}. This proves that $([\ca{F}_n:\ca{F}]\cdot |\scr{D}_{\ca{O}_{\ca{F}_n}/\ca{O}_{\ca{F}}}|)_{n\in\bb{N}}$ is non-decreasing.
	
	Then, we prove \eqref{eq:prop:trace-tau-2} by induction on $n$. For $n=0$, both sides are zero. Suppose that it holds for $n$. Then, for any $x\in\ca{F}_{n+1}$, applying \eqref{eq:prop:trace-tau-2} to $\mrm{Tr}_{\ca{F}_{n+1}/\ca{F}_n}(x)\in\ca{F}_n$, we obtain that
	\begin{align}
		&|[\ca{F}_n:\ca{F}]^{-1}\mrm{Tr}_{\ca{F}_{n+1}/\ca{F}}(x)-\mrm{Tr}_{\ca{F}_{n+1}/\ca{F}_n}(x)|\\
		\leq\ & p\cdot [\ca{F}_n:\ca{F}]\cdot|\scr{D}_{\ca{O}_{\ca{F}_n}/\ca{O}_{\ca{F}}}|\cdot |\tau(\mrm{Tr}_{\ca{F}_{n+1}/\ca{F}_n}(x))-\mrm{Tr}_{\ca{F}_{n+1}/\ca{F}_n}(x)|\nonumber\\
		=\ & p\cdot [\ca{F}_n:\ca{F}]\cdot|\scr{D}_{\ca{O}_{\ca{F}_n}/\ca{O}_{\ca{F}}}|\cdot |\mrm{Tr}_{\ca{F}_{n+1}/\ca{F}_n}(\tau(x)-x)|\nonumber\\
		\leq\ & p\cdot [\ca{F}_n:\ca{F}]\cdot|\scr{D}_{\ca{O}_{\ca{F}_n}/\ca{O}_{\ca{F}}}|\cdot |\scr{D}_{\ca{O}_{\ca{F}_{n+1}}/\ca{O}_{\ca{F}_n}}|\cdot|\tau(x)-x|\nonumber,
	\end{align}
	where the last inequality follows from \ref{cor:prop:different-trace}. We deduce from \eqref{eq:prop:trace-tau-3} and \eqref{eq:prop:trace-tau-4} that
	\begin{align}
		p\cdot [\ca{F}_n:\ca{F}]\cdot|\scr{D}_{\ca{O}_{\ca{F}_{n+1}}/\ca{O}_{\ca{F}}}|=p\cdot [\ca{F}_n:\ca{F}]\cdot|\scr{D}_{\ca{O}_{\ca{F}_n}/\ca{O}_{\ca{F}}}|\cdot |\scr{D}_{\ca{O}_{\ca{F}_{n+1}}/\ca{O}_{\ca{F}_n}}|\geq 1.
	\end{align}
	Combining this with \ref{lem:trace-tau}, we see that
	\begin{align}
		&|[\ca{F}_n:\ca{F}]^{-1}\mrm{Tr}_{\ca{F}_{n+1}/\ca{F}}(x)-[\ca{F}_{n+1}:\ca{F}_n]\cdot x|\\
		\leq& \max(|[\ca{F}_n:\ca{F}]^{-1}\mrm{Tr}_{\ca{F}_{n+1}/\ca{F}}(x)-\mrm{Tr}_{\ca{F}_{n+1}/\ca{F}_n}(x)|,|\mrm{Tr}_{\ca{F}_{n+1}/\ca{F}_n}(x)-[\ca{F}_{n+1}:\ca{F}_n]\cdot x|)\nonumber\\
		\leq& \max(p\cdot [\ca{F}_n:\ca{F}]\cdot|\scr{D}_{\ca{O}_{\ca{F}_{n+1}}/\ca{O}_{\ca{F}}}|,1)\cdot |\tau(x)-x|\nonumber\\
		=&p\cdot[\ca{F}_n:\ca{F}]\cdot|\scr{D}_{\ca{O}_{\ca{F}_{n+1}}/\ca{O}_{\ca{F}}}|\cdot |\tau(x)-x|,\nonumber
	\end{align}
	which completes the induction process.
\end{proof}

\begin{mycor}\label{cor:trace-tau}
	Under the assumption in {\rm\ref{para:ramified-tower}} and with the same notation in {\rm\ref{para:tower}}, let $n_0\in\bb{N}$ (resp. $n_1\in\bb{N}$) be the integer defined in {\rm\ref{lem:tower}} (resp. {\rm\ref{lem:ann-differential}}). Then, for any $x\in\ca{F}_\infty$, we have
	\begin{align}\label{eq:cor:trace-tau-1}
		|\ca{T}(x)-x|\leq p^{n_1-n_0+1}\cdot |\tau(x)-x|,
	\end{align}
	where $\ca{T}$ is the normalized trace map \eqref{eq:para:tate-trace-1}.
\end{mycor}
\begin{proof}
	It follows directly from \ref{prop:trace-tau}, \ref{lem:tower} and \ref{prop:different-range}.
\end{proof}

\begin{myrem}\label{rem:trace-tau}
	In the case of complete discrete valuation rings (i.e., replacing $\ca{F}$ by $K$), there is no explicit relation as \ref{cor:prop:different-trace} between the norms of trace maps and different ideals ($|\scr{D}_{\ca{O}_{\ca{F}_{n+1}}/\ca{O}_{\ca{F}_n}}|$ may be less than $1/p$). In order to estimate the error term $1/p-|\scr{D}_{\ca{O}_{\ca{F}_{n+1}}/\ca{O}_{\ca{F}_n}}|$, Tate \cite[\textsection3.1, Proposition 5]{tate1967p} gave a much more precise estimate of $|\scr{D}_{\ca{O}_{\ca{F}_n}/\ca{O}_{\ca{F}}}|$ using local class field theory (cf. \ref{prop:different-range}). Then, he was able to bound the norm of trace maps by the (corrected) norm of different ideals in \cite[\textsection3.1, Proposition 6]{tate1967p}.
\end{myrem}

\begin{myprop}[{cf. \cite[\textsection3.1]{tate1967p}}]\label{prop:tate-trace}
	Under the assumption in {\rm\ref{para:ramified-tower}} and with the same notation in {\rm\ref{para:tower}}, the map $\ca{T}:\ca{F}_\infty\to \ca{F}$ \eqref{eq:para:tate-trace-1} is continuous with respect to the topology induced by the corresponding valuation rings (which is also induced by the absolute value \eqref{eq:para:notation-trace-1}). In particular, it induces a continuous $G_{\ca{F}}$-equivariant $\widehat{\ca{F}}$-linear retraction
	\begin{align}\label{eq:prop:tate-trace-1}
		\widehat{\ca{T}}:\widehat{\ca{F}_\infty}\longrightarrow \widehat{\ca{F}}
	\end{align}
	 of the inclusion of the completions $\widehat{\ca{F}}\to\widehat{\ca{F}_\infty}$.
\end{myprop}
\begin{proof}
	By \ref{lem:tower}, there exists $n_0\in\bb{N}$ such that $[\ca{F}_n:\ca{F}]=p^{n-n_0}$ for any integer $n\geq n_0$. Combining with \ref{cor:trace-range}, we see that for any integer $n\geq n_1$ and any $x\in \ca{F}_n$ we have $|\ca{T}(x)|\leq p^{n_1-n_0}|x|$, which verifies the continuity of $\ca{T}$. Thus, it extends continuously to $\widehat{\ca{T}}:\widehat{\ca{F}_\infty}\to\widehat{\ca{F}}$ (\cite[6.1]{he2022sen}). It is clearly a retraction of the inclusion $\widehat{\ca{F}}\to \widehat{\ca{F}_\infty}$. Moreover, it is also $G_{\ca{F}}$-equivariant and $\widehat{\ca{F}}$-linear, since the action by $G_{\ca{F}}$ and the multiplication by $\ca{F}$ on $\ca{F}_\infty$ are continuous (where we used the fact that $G_{\ca{F}}$ stabilizes $\ca{O}_{\overline{\ca{F}}}$ as $\ca{F}$ is Henselian, see \ref{para:notation-trace}).
\end{proof}

\begin{mycor}[{cf. \cite[\textsection3.1, Proposition 7]{tate1967p}}]\label{cor:tate-trace}
	We keep the same notation and assumption as in {\rm\ref{para:tower}} and {\rm\ref{para:ramified-tower}} respectively. 
	\begin{enumerate}
		\renewcommand{\labelenumi}{{\rm(\theenumi)}}
		\item The completion $\widehat{\ca{F}_\infty}$ is the direct sum of $\widehat{\ca{F}}$ and $\ke(\widehat{\ca{T}})$ \eqref{eq:prop:tate-trace-1}. Moreover, $\ke(\widehat{\ca{T}})$ is a closed subspace of $\widehat{\ca{F}_\infty}$ identifying with the completion of $\ke(\ca{T})\subseteq \ca{F}_\infty$ \eqref{eq:para:tate-trace-2}.\label{item:cor:tate-trace-1}
		\item For any $\tau\in \gal(\ca{F}_\infty/\ca{F})$, the continuous $\widehat{\ca{F}}$-linear operator $\tau-1$ on $\widehat{\ca{F}_\infty}$ annihilates $\widehat{\ca{F}}$ and stabilizes $\ke(\widehat{\ca{T}})$.\label{item:cor:tate-trace-2}
		\item Let $\tau\in \gal(\ca{F}_\infty/\ca{F})$ be a topological generator (which exists by {\rm\ref{lem:tower}}). Then, the continuous $\widehat{\ca{F}}$-linear operator $\tau-1$ on $\widehat{\ca{F}_\infty}$ induces a homeomorphism $\ke(\widehat{\ca{T}})\iso \ke(\widehat{\ca{T}})$. \label{item:cor:tate-trace-3}
	\end{enumerate}
\end{mycor}
\begin{proof}
	(\ref{item:cor:tate-trace-1}) As $\widehat{\ca{T}}:\widehat{\ca{F}_\infty}\to\widehat{\ca{F}}$ is a retraction of the inclusion $\widehat{\ca{F}}\to \widehat{\ca{F}_\infty}$ by \ref{prop:tate-trace}, we see that $\widehat{\ca{F}_\infty}=\widehat{\ca{F}}\oplus \ke(\widehat{\ca{T}})$. As $\widehat{\ca{T}}$ is continuous and $\widehat{\ca{F}}$ is separated, $\ke(\widehat{\ca{T}})$ is closed in $\widehat{\ca{F}_\infty}$. Notice that $\ke(\ca{T})\subseteq \ke(\widehat{\ca{T}})$ via the inclusion $\ca{F}_\infty\subseteq \widehat{\ca{F}_\infty}$. Thus, we still have $\widehat{\ke(\ca{T})}\subseteq \ke(\widehat{\ca{T}})$. Taking completion of the canonical decomposition $\ca{F}_\infty=\ca{F}\oplus \ke(\ca{T})$ \eqref{eq:para:tate-trace-2}, we see that $\widehat{\ke(\ca{T})}= \ke(\widehat{\ca{T}})$.
	
	(\ref{item:cor:tate-trace-2}) It is clear that $(\tau-1)(\widehat{\ca{F}})=0$. As $\widehat{\ca{T}}$ is $G_{\ca{F}}$-equivariant by \ref{prop:tate-trace}, for any $x\in \widehat{\ca{F}_\infty}$, we have $\widehat{\ca{T}}(\tau(x))=\widehat{\ca{T}}(x)$. Thus, $(\tau-1)(\widehat{\ca{F}_\infty})\subseteq \ke(\widehat{\ca{T}})$.
	
	(\ref{item:cor:tate-trace-3}) Firstly, we claim that $\tau-1$ induces an $\ca{F}$-linear isomorphism $\ke(\ca{T})\iso \ke(\ca{T})$. Notice that for any $n\in\bb{N}$, the image of the $\ca{F}$-linear operator $(\tau-1)|_{\ca{F}_n}$ lies in $\ke(\ca{T}|_{\ca{F}_n})$ \eqref{eq:para:tate-trace-2} as above. Since its kernel is $\ca{F}$ (as $\gal(\ca{F}_\infty/\ca{F})=\bb{Z}_p\tau$), we obtain an exact sequence of finite free $\ca{F}$-modules 
	\begin{align}
		\xymatrix{
			0\ar[r]&\ca{F}\ar[r]&\ca{F}_n\ar[r]^-{\tau-1}&\ke (\ca{T}|_{\ca{F}_n})\ar[r]&0
		}
	\end{align}
	by the identity $\dim_{\ca{F}}\ca{F}_n=1+\dim_{\ca{F}}\ke (\ca{T}|_{\ca{F}_n})$ \eqref{eq:para:tate-trace-2}. The sequence remains exact for $n=\infty$ by taking filtered union, which proves the claim. 
	
	Then, we claim that the inverse $\ca{S}:\ke(\ca{T})\iso \ke(\ca{T})$ of $\tau-1$ is also continuous. Indeed, for any $x\in \ke(\ca{T})$ with $y=(\tau-1)(x)\in \ke(\ca{T})$ we have
	\begin{align}
		|\ca{S}(y)|=|x|=|\ca{T}(x)-x|\leq p^{n_1-n_0+1}\cdot |\tau(x)-x|=p^{n_1-n_0+1}\cdot|y|
	\end{align}
	by \ref{cor:trace-tau}, which proves the claim.
	
	Therefore, taking completion, we see that $\ca{S}:\ke(\ca{T})\iso \ke(\ca{T})$ induces a continuous map $\widehat{\ca{S}}:\ke(\widehat{\ca{T}})\to \ke(\widehat{\ca{T}})$ by (\ref{item:cor:tate-trace-1}) inverse to $\tau-1$. This proves that $\tau-1:\ke(\widehat{\ca{T}})\to \ke(\widehat{\ca{T}})$ is a homeomorphism.
\end{proof}

\begin{myrem}\label{rem:tate-trace}
	In fact, the injectivity of $\tau-1:\ke(\widehat{\ca{T}})\to \ke(\widehat{\ca{T}})$ also follows from Ax-Sen-Tate's theorem \cite[page 417]{ax1970ax} (where we used the fact that $\ca{F}$ is a Henselian valuation field). It implies that
	\begin{align}
		H^0(\gal(\ca{F}_{\infty}/\ca{F}),\widehat{\ca{F}_\infty})=H^0(\gal(\ca{F}_{\infty}/\ca{F}),\widehat{\ca{F}})=\widehat{\ca{F}}.
	\end{align}
	 On the other hand, as $\tau-1:\ke(\widehat{\ca{T}})\to \ke(\widehat{\ca{T}})$ is an injective continuous homomorphism of $\widehat{\ca{F}}$-Banach spaces, it is surjective if and only if it is a homeomorphism by the open mapping theorem (\cite[\Luoma{1}.3.3, Th\'eor\`eme 1]{bourbaki1981space1-5}, where we used the non-discreteness of the valuation on $\widehat{\ca{F}}$). Nevertheless, the surjectivity of $\tau-1$ (proved in \ref{cor:tate-trace}.(\ref{item:cor:tate-trace-3})) implies that
	 \begin{align}
	 	H^1(\gal(\ca{F}_{\infty}/\ca{F}),\widehat{\ca{F}_\infty})=H^1(\gal(\ca{F}_{\infty}/\ca{F}),\widehat{\ca{F}})\cong\widehat{\ca{F}},
	 \end{align}
	 where $H^1$ denotes taking the continuous group cohomology in degree $1$. See \cite[\textsection3.1, Proposition 8.(a)]{tate1967p} for a detailed argument for discrete valuation fields.
\end{myrem}

\begin{mycor}\label{cor:ax-sen-tate}
	Under the assumption in {\rm\ref{para:ramified-tower}} and with the same notation in {\rm\ref{para:tower}}, there is no element $x\in \widehat{\overline{\ca{F}}}$ such that $(\tau-1)(x)=\xi_t(\tau)$ for every $\tau\in G_{\ca{F}}$.
\end{mycor}
\begin{proof}
	If there exists such an element $x$, then firstly we see that it is fixed by $\gal(\overline{\ca{F}}/\ca{F}_\infty)$ so that it lies in $\widehat{\ca{F}_\infty}$ by Ax-Sen-Tate's theorem \cite[page 417]{ax1970ax} (note that $\ca{F}_\infty$ is a Henselian valuation field). Thus, for any $\tau\in G_{\ca{F}}$, we have $(\tau-1)(x)\in \ke(\widehat{\ca{T}})$ by \ref{cor:tate-trace}.(\ref{item:cor:tate-trace-2}). On the other hand, by assumption we see that $(\tau-1)(x)\in \widehat{\ca{F}}$. Hence, $(\tau-1)(x)\in\ke(\widehat{\ca{T}})\cap   \widehat{\ca{F}}=0$ by \ref{cor:tate-trace}.(\ref{item:cor:tate-trace-1}), which contradicts the fact that $\xi_t\neq 0$.
\end{proof}

\begin{mylem}\label{lem:perfd-val}
	The following statements are equivalent:
	\begin{enumerate}
		\renewcommand{\labelenumi}{{\rm(\theenumi)}}
		\item The subspace $\scr{E}_{\ca{O}_{\overline{\ca{F}}}}^{G_{\ca{F}}}$ of $G_{\ca{F}}$-invariant elements of $\scr{E}_{\ca{O}_{\overline{\ca{F}}}}$ \eqref{eq:thm:fal-ext-val-1} has dimension $1+d$ over $\widehat{\ca{F}}$.\label{item:lem:perfd-val-1}
		\item The morphism $\jmath:\scr{E}_{\ca{O}_{\overline{\ca{F}}}}^{G_{\ca{F}}}\to\widehat{\ca{F}}\otimes_{\ca{F}}\Omega^1_{\ca{F}/K}$ induced by \eqref{eq:thm:fal-ext-val-3} is surjective.\label{item:lem:perfd-val-2}
		\item The coboundary map $\delta:\widehat{\ca{F}}\otimes_{\ca{F}}\Omega^1_{\ca{F}/K}\to H^1(G_{\ca{F}},\widehat{\overline{\ca{F}}}(1))$ induced by \eqref{eq:thm:fal-ext-val-3} is zero.\label{item:lem:perfd-val-3}
	\end{enumerate}
\end{mylem}
\begin{proof}
	Recall that the Faltings extension of $\ca{O}_{\overline{\ca{F}}}$ over $\ca{O}_K$ \eqref{eq:thm:fal-ext-val-3},
	\begin{align}\label{eq:thm:perfd-val-1}
		\xymatrix{
			0\ar[r]&\widehat{\overline{\ca{F}}}(1)\ar[r]^-{\iota}&\scr{E}_{\ca{O}_{\overline{\ca{F}}}}\ar[r]^-{\jmath}&\widehat{\overline{\ca{F}}}\otimes_{\ca{F}}\Omega^1_{\ca{F}/K}\ar[r]&0
		}
	\end{align}
	is $G_{\ca{F}}$-equivariant. Moreover, the $G_{\ca{F}}$-action on each term is continuous by \ref{prop:fal-ext-val} (with respect to the canonical topology as finite free $\widehat{\overline{\ca{F}}}$-modules). Taking $G_{\ca{F}}$-invariants, we obtain an exact sequence
	\begin{align}\label{eq:thm:perfd-val-2}
		\xymatrix{
			0\ar[r]&\widehat{\ca{F}}(1)\ar[r]^-{\iota}&\scr{E}_{\ca{O}_{\overline{\ca{F}}}}^{G_\ca{F}}\ar[r]^-{\jmath}&\widehat{\ca{F}}\otimes_{\ca{F}}\Omega^1_{\ca{F}/K}\ar[r]^-{\delta}&H^1(G_{\ca{F}},\widehat{\overline{\ca{F}}}(1)),
		}
	\end{align}
	where the expressions for the first and third terms follow from Ax-Sen-Tate's theorem \cite[page 417]{ax1970ax} (note that $\ca{F}$ is a Henselian valuation field) and $H^1$ is the continuous group cohomology. The conclusion follows from the fact that $\dim_{\ca{F}}\Omega^1_{\ca{F}/K}=\mrm{trdeg}_K(\ca{F})=d$.
\end{proof}

\begin{mythm}\label{thm:perfd-val}
	Let $\scr{E}_{\ca{O}_{\overline{\ca{F}}}}$ be the canonical finite free $\widehat{\overline{\ca{F}}}$-representation of $G_{\ca{F}}$ defined in \eqref{eq:thm:fal-ext-val-3} (by taking the $\ca{F}$ in {\rm\ref{thm:fal-ext-val}} to be the $\overline{\ca{F}}$ in this section). Assume that the equivalent conditions in {\rm\ref{lem:perfd-val}} hold. Then, $\Omega^1_{\ca{O}_{\ca{F}}/\ca{O}_{\overline{K}}}=\Omega^1_{\ca{F}/\overline{K}}$. In particular, $\ca{F}$ is a pre-perfectoid field in the sense of {\rm\ref{para:notation-perfd}}.
\end{mythm}
\begin{proof}
	Note that $\Omega^1_{\ca{O}_{\ca{F}}/\ca{O}_{\overline{K}}}$ is $p$-torsion-free as $\overline{K}$ is algebraically closed (\cite[6.5.20.(\luoma{1})]{gabber2003almost}). Thus, there is a canonical inclusion $\Omega^1_{\ca{O}_{\ca{F}}/\ca{O}_{\overline{K}}}\subseteq \Omega^1_{\ca{F}/\overline{K}}$. We take a transcendental basis $t_1,\dots,t_d\in \ca{O}_{\ca{F}}$ of $\ca{F}$ over $\overline{K}$. Then, the images of $\df t_1,\dots,\df t_d \in \Omega^1_{\ca{O}_{\ca{F}}/\ca{O}_{\overline{K}}}$ generate $\Omega^1_{\ca{F}/\overline{K}}$ over $\ca{F}$. If each $\df t_i$ is $p^\infty$-divisible in $\Omega^1_{\ca{O}_{\ca{F}}/\ca{O}_{\overline{K}}}$, then we see that $\Omega^1_{\ca{O}_{\ca{F}}/\ca{O}_{\overline{K}}}=\Omega^1_{\ca{F}/\overline{K}}$. It remains to show that we are in this case.
	
	Assume that there exists $1\leq i\leq d$ such that $\df t_i$ is not $p^\infty$-divisible in $\Omega^1_{\ca{O}_{\ca{F}}/\ca{O}_{\overline{K}}}$. Then, we take $t=t_i$ so that $\df t\notin \bigcap_{n\in\bb{N}} p^n\Omega^1_{\ca{O}_{\ca{F}}/\ca{O}_{\overline{K}}}$ and thus we are in the situation of \ref{para:ramified-tower}. Taking again the notation in \ref{para:tower}, we see that $\jmath((\df\log t_{p^n})_{n\in\bb{N}})=\df \log(t)$ by \ref{thm:fal-ext-val}.(\ref{item:thm:fal-ext-val-2}). By the construction of coboundary map $\delta:\widehat{\ca{F}}\otimes_{\ca{F}}\Omega^1_{\ca{F}/K}\to H^1(G_{\ca{F}},\widehat{\overline{\ca{F}}}(1))$ induced by \eqref{eq:thm:fal-ext-val-3}, we see that $\delta(\df\log(t))$ is represented by the $1$-cocycle (see \eqref{eq:para:tower-1} and \ref{thm:fal-ext-val}.(\ref{item:thm:fal-ext-val-1}))
	\begin{align}\label{eq:thm:perfd-val-3}
		G_{\ca{F}}\longrightarrow \widehat{\overline{\ca{F}}}(1),\ \tau\mapsto (\tau-1)((\df\log t_{p^n})_{n\in\bb{N}})=\xi_t(\tau)(\zeta_{p^n})_{n\in\bb{N}},
	\end{align}
	where for the last equality we used the canonical embedding $\bb{Z}_p\subseteq \widehat{\overline{\ca{F}}}$ (fixed in \ref{para:notation-trace} as $\overline{\ca{F}}$ is an extension of $\bb{Q}_p$).	On the other hand, since $\delta=0$ by assumption \ref{lem:perfd-val}, the $1$-cocycle \eqref{eq:thm:perfd-val-3} is a $1$-coboundary, i.e., there exists an element $x\in \widehat{\overline{\ca{F}}}$ such that $(\tau-1)(x)=\xi_t(\tau)$ for every $\tau\in G_{\ca{F}}$, which is impossible by \ref{cor:ax-sen-tate}. 
	
	The ``in particular" part follows from \cite[10.17]{he2024purity}.
\end{proof}

\begin{myrem}\label{rem:thm:perfd-val-1}
	The arguments of \ref{thm:perfd-val} actually show that for any $t\in\ca{O}_{\ca{F}}$ such that the element $\df t\notin \bigcap_{n\in\bb{N}}p^n\Omega^1_{\ca{O}_{\ca{F}}/\ca{O}_{\overline{K}}}$, the coboundary map $\delta$ does not vanish on $\df t$ (or equivalently on $\df\log(t)$). In general, we can use the coboundary map $\delta$ to compute $H^1(G_{\ca{F}},\widehat{\overline{\ca{F}}}(1))$ (and then $H^q(G_{\ca{F}},\widehat{\overline{\ca{F}}}(1))$ for any $q\in\bb{N}$) in terms of differentials by \ref{cor:tate-trace}.(\ref{item:cor:tate-trace-3}) (see \ref{rem:tate-trace}). This computation (and neither of \ref{cor:tate-trace}.(\ref{item:cor:tate-trace-3}), \ref{cor:trace-tau}, \ref{prop:trace-tau}, \ref{lem:trace-tau}) will not be used in this article, but will be presented in the subsequent paper \cite{he2025galoiscoh} on the $p$-adic Galois cohomology of valuation fields.
\end{myrem}

\begin{myrem}\label{rem:thm:perfd-val-2}
	Theorem \ref{thm:perfd-val} also holds without the finiteness assumption on transcendental degree of $\overline{\ca{F}}$ after replacing the definition of $\scr{E}_{\ca{O}_{\overline{\ca{F}}}}$ by \eqref{eq:rem:fal-val-disc-1},  replacing the condition \ref{lem:perfd-val}.(\ref{item:lem:perfd-val-1}) by the surjectivity of $\widehat{\overline{\ca{F}}}\otimes_{\widehat{\ca{F}}}\scr{E}_{\ca{O}_{\overline{\ca{F}}}}^{G_{\ca{F}}}\to \scr{E}_{\ca{O}_{\overline{\ca{F}}}}$ and replacing $H^1(G_{\ca{F}},\widehat{\overline{\ca{F}}}(1))$ in \ref{lem:perfd-val}.(\ref{item:lem:perfd-val-3}) by the (discrete) group cohomology (see \cite[8.5, 9.7]{he2025galoiscoh}).
\end{myrem}

\begin{myrem}\label{rem:thm:perfd-val-3}
	We take $A=\ca{O}_K[t_1,\dots,t_d]\subseteq \ca{O}_{\ca{F}}$, which is a polynomial algebra. Note that there is a canonical map $V_p(\Omega^1_{\ca{O}_{\ca{F}}/A})\to V_p(\Omega^1_{\ca{O}_{\overline{\ca{F}}}/A})^{G_{\ca{F}}}$ but we don't know if it is an isomorphism or not (the surjectivity is the mysterious part). The injectivity would imply that if  $\dim_{\widehat{\ca{F}}}V_p(\Omega^1_{\ca{O}_{\ca{F}}/A})=1+d$ then $\dim_{\widehat{\ca{F}}}V_p(\Omega^1_{\ca{O}_{\overline{\ca{F}}}/A})^{G_{\ca{F}}}=1+d$. Under a similar condition on $V_p(\Omega^1_{\ca{O}_{\ca{F}}/A})$, Ohkubo deduces a perfectoidness criterion for algebraic extensions of a complete discrete valuation field extension of $\bb{Q}_p$ with imperfect residue field, see \cite[6.3]{ohkubo2010bdr}. Thus, \ref{thm:perfd-val} could be viewed as a generalization of Ohkubo's result in the geometric setting (i.e., over $\overline{K}$ already).
\end{myrem}

\section{Log Smooth Schemes in Admissibly \'Etale Topology after Temkin}\label{sec:aet}
Temkin \cite[2.5.2]{temkin2017etale} proved that any scheme flat and of finite type over a valuation ring of height $1$ with smooth generic fibre admits an admissibly \'etale covering by strictly semi-stable (thus log smooth) schemes. Restricting ourselves over complete discrete valuation rings and aiming at coverings by log smooth schemes, we show that his arguments still work when we also consider a strict normal crossings divisor on the smooth generic fibre (see \ref{thm:uniformization}). We refer to \cite{kato1989log,kato1994toric,gabber2004foundations,ogus2018log} for a systematic development of logarithmic geometry, and to \cite[\textsection8]{he2022sen} (or \cite[\textsection4]{he2024falmain}) for a brief summary of the theory.

\begin{mylem}\label{lem:ht-1-val-lift}
	Let $K$ be a valuation field of height $1$, $X$ an $\ca{O}_K$-scheme locally of finite type, $y\leadsto x$ a specialization of a closed point $y$ of $X_K$ to a closed point $x$ of $X$. Then, there exists a valuation field $L$ of height $1$ finite extension of $K$ with an $\ca{O}_K$-morphism $\spec(\ca{O}_L)\to X$ sending the generic point to $y$ and the closed point to $x$.
\end{mylem}
\begin{proof}
	We may assume that $X$ is affine and let $\spec(A)$ be the closure of $y$ in $X$ with reduced structure. Note that $A$ is a domain of finite type over $\ca{O}_K$ such that $K\otimes_{\ca{O}_K}A$ is the fraction field $\kappa(y)$ of $A$ (which is a finite extension of $K$ by \cite[\href{https://stacks.math.columbia.edu/tag/00FV}{00FV}]{stacks-project}) and that $x$ is identified with a maximal ideal $\ak{m}$ of $A$ lying over that of $\ca{O}_K$. Then, the special fibre of $\spec(A)\to\spec(\ca{O}_K)$ is also of dimension $0$ (\cite[\href{https://stacks.math.columbia.edu/tag/00QK}{00QK}]{stacks-project}). In particular, $\spec(A_{\ak{m}})$ consists of two points $y$ and $x$. Let $V$ be a valuation ring dominating $A_{\ak{m}}$ with the same fraction field (which exists by \cite[\href{https://stacks.math.columbia.edu/tag/00IA}{00IA}]{stacks-project}). In particular, the fraction field of $V$ is $\kappa(y)=A_{\ak{m}}[1/\pi]=V[1/\pi]$ for some nonzero element $\pi$ of $\ak{m}_K$. Since the family of ideals of $V$ is totally ordered by the inclusion relation (\cite[\Luoma{6}.\textsection1.2, Th\'eor\`eme 1]{bourbaki2006commalg5-7}), the radical ideal $\sqrt{(\pi)}$ of $V$ is the minimal prime ideal containing $\pi$. Thus, we see that it is of height $1$ as $V[1/\pi]$ is a field. Thus, the localization $V_{\sqrt{(\pi)}}$ is a valuation ring of height $1$ extension of $\ca{O}_K$ with fraction field finite over $K$ (\cite[\href{https://stacks.math.columbia.edu/tag/088Y}{088Y}]{stacks-project}). Taking $\ca{O}_L=V_{\sqrt{(\pi)}}$, then we see that the induced morphism $\spec(\ca{O}_L)\to X$ meets our requirements.
\end{proof}

\begin{myprop}\label{prop:adm-etale}
	Let $K$ be a valuation field of height $1$, $f:Y\to X$ a morphism of finitely presented $\ca{O}_K$-schemes with $f_K:Y_K\to X_K$ flat (resp. quasi-finite and flat). Then, the following conditions are equivalent:
	\begin{enumerate}
		\renewcommand{\labelenumi}{{\rm(\theenumi)}}
		\item For any valuation field $L$ of height $1$ finite extension of $K$ and any morphism $g:\spec(\ca{O}_L)\to X$ over $\ca{O}_K$, there exists a valuation field $L'$ of height $1$ (resp. finite) extension of $L$ and a morphism $g':\spec(\ca{O}_{L'})\to Y$ lifting $g$.\label{item:lem:adm-etale-1}
		\item There exists an $X_K$-admissible blowup $X'\to X$ such that the strict transform $f':Y'\to X'$ of $f:Y\to X$ is faithfully flat of finite presentation.\label{item:lem:adm-etale-2}
	\end{enumerate}
\end{myprop}
\begin{proof}
	(\ref{item:lem:adm-etale-2})$\Rightarrow$(\ref{item:lem:adm-etale-1}): Let $X_s$ be the special fibre of $X$ over $\ca{O}_K$. Then, $X'\coprod X_s\to X$ is proper and surjective. Thus, $Y\coprod X_s\to X$ is an arc-covering (\cite[3.4]{he2024coh}). Hence, any morphism $g:\spec(\ca{O}_L)\to X$ in (\ref{item:lem:adm-etale-1}) lifts to $g':\spec(\ca{O}_{L'})\to Y\coprod X_s$ for a valuation field extension $L'/L$ of height $1$ (\cite[3.4]{he2024coh}). It is clear that $g'$ factors through $Y$. If moreover $f_K$ is quasi-finite, then the residue field $\kappa(y)$ of the image $y$ of $g':\spec(L')\to Y_K$ is finite over the residue field $\kappa(x)$ of the image $x$ of $g:\spec(L)\to X_K$. Hence, we may replace $L'$ by the valuation subfield generated by $L$ and $\kappa(y)$ so that we may assume that $L'$ is finite over $L$.
	
	(\ref{item:lem:adm-etale-1})$\Rightarrow$(\ref{item:lem:adm-etale-2}) By Raynaud-Gruson's flattening theorem \cite[\Luoma{1}.5.2.2]{raynaudgruson1971plat}, there exists an $X_K$-admissible blowup $X'\to X$ such that the strict transform $f':Y'\to X'$ of $f:Y\to X$ is flat of finite presentation. It remains to check the surjectivity of $f'$. Since generalizations lifts along $f'$ (\cite[\href{https://stacks.math.columbia.edu/tag/03HV}{03HV}]{stacks-project}), it suffices to show that any locally closed point $x'$ of $X'$ lies in the image of $f'$. 
	
	Let $U'\subseteq X'$ be an open neighborhood of $x'$ such that $x'$ is closed in $U'$. Then, there exists a closed point $x'_0$ of $U'_K$ specializing to $x'$ (\cite[\href{https://stacks.math.columbia.edu/tag/053U}{053U}]{stacks-project}). Thus, there exists a valuation field $L$ of height $1$ finite extension over $K$ with an $\ca{O}_K$-morphism $\spec(\ca{O}_L)\to X'$ sending the generic point to $x'_0$ and the closed point to $x'$ by \ref{lem:ht-1-val-lift}. Thus, there exists a lifting $g':\spec(\ca{O}_{L'})\to Y$ of $g:\spec(\ca{O}_L)\to X$ by (\ref{item:lem:adm-etale-1}). As $f':Y'\to X'$ is the strict transform of $f:Y\to X$ by an $X_K$-admissible blowup $X'\to X$, the morphism $g'$ factors uniquely through $Y'$ (\cite[\href{https://stacks.math.columbia.edu/tag/080E}{080E}, \href{https://stacks.math.columbia.edu/tag/0806}{0806}, \href{https://stacks.math.columbia.edu/tag/090Q}{090Q}]{stacks-project}).
	\begin{align}
		\xymatrix{
			\spec(\ca{O}_{L'})\ar[d]\ar[r]&Y'\ar[d]^-{f'}\ar[r]&Y\ar[d]^-{f}\\
			\spec(\ca{O}_L)\ar[r]&X'\ar[r]&X.
		}
	\end{align}
	Therefore, we see that the image of the closed point of $\spec(\ca{O}_{L'})$ in $Y'$ maps to $x'$ via $f'$, which completes the proof.
\end{proof}

\begin{myrem}\label{rem:blowup}
	Note that an admissible blowup is not locally of finite presentation in general. But for any scheme $X$ flat and locally of finite type over a valuation ring $\ca{O}_K$, any $X_K$-admissible blowup $X'$ of $X$ is flat and locally of finite presentation over $\ca{O}_K$. Indeed, $X'$ is projective over $X$ (\cite[\href{https://stacks.math.columbia.edu/tag/02NS}{02NS}]{stacks-project}) and thus locally of finite type over $\ca{O}_K$. It is also flat over $\ca{O}_K$, as the complement of the exceptional divisor is flat over $\ca{O}_K$ and scheme theoretically dense in $X'$ (\cite[\href{https://stacks.math.columbia.edu/tag/07ZU}{07ZU}]{stacks-project}). Thus, the claim follows from the fact that any flat $\ca{O}_K$-scheme locally of finite type is locally of finite presentation (\cite[\Luoma{1}.3.4.7]{raynaudgruson1971plat}, see also \cite[\href{https://stacks.math.columbia.edu/tag/081P}{081P}]{stacks-project}).
\end{myrem}

\begin{mydefn}[{cf. \cite[\Luoma{3}.3.1]{faltings1988p}, \cite[2.2.3]{temkin2017etale}}]\label{defn:adm-etale}
	Let $K$ be a valuation field of height $1$. For any $\ca{O}_K$-scheme $X$ of finite presentation, we denote by $\sch^{\aetale}_{/X}$ the category of $X$-schemes $Y$ of finite presentation with $Y_K$ \'etale over $X_K$. 
	
	We endow it with the topology generated by families of morphisms $\{Y_i\to Y\}_{i\in I}$ with $I$ finite such that for any valuation field $L$ of height $1$ finite extension of $K$ and any morphism $g:\spec(\ca{O}_L)\to Y$ over $\ca{O}_K$, there exists a valuation field $L'$ of height $1$ extension of $L$ and a morphism $g':\spec(\ca{O}_{L'})\to \coprod_{i\in I}Y_i$ such that the following diagram is commutative 
	\begin{align}\label{eq:defn:adm-etale}
		\xymatrix{
			\spec(\ca{O}_{L'})\ar[d]\ar[r]^-{g'}& \coprod_{i\in I}Y_i\ar[d]\\
			\spec(\ca{O}_L)\ar[r]^-{g}& Y
		}
	\end{align}
	(note that if such diagram exists then we can require further that $L'$ is finite over $L$). A covering family in $\sch^{\aetale}_{/X}$ is called an \emph{admissibly \'etale covering}. We call $\sch^{\aetale}_{/X}$ the \emph{admissibly \'etale site} of $X$ over $\ca{O}_K$.
\end{mydefn}

We remark that \cite[\Luoma{3}.3.1]{faltings1988p} and \cite[2.2.3]{temkin2017etale} only consider the category/site of flat $\ca{O}_K$-schemes of finite type. Note that such schemes are actually of finite presentation: they are locally of finite presentation by \cite[\Luoma{1}.3.4.7]{raynaudgruson1971plat} (see also \cite[\href{https://stacks.math.columbia.edu/tag/081P}{081P}]{stacks-project}); and they are quasi-compact and quasi-separated because the underlying topological spaces are Noetherian by \cite[4.10]{he2024falmain} (see the proof of \cite[\href{https://stacks.math.columbia.edu/tag/01OY}{01OY}]{stacks-project}). Conversely, any $\ca{O}_K$-scheme $Y$ of finite presentation is covered by a flat $\ca{O}_K$-scheme $\overline{Y}$ of finite type in the admissibly \'etale site in our sense (as shown in the following Corollary \ref{cor:adm-etale}). Therefore, one can check that the associated admissibly \'etale topoi in our sense or in the sense of Faltings and Temkin are actually equivalent. This comparison will not be used in the following.

\begin{mypara}
	It is easy to check that $\sch^{\aetale}_{/X}$ is stable under taking finite limits of $X$-schemes and that the families of morphisms described in \ref{defn:adm-etale} form a pretopology on $\sch^{\aetale}_{/X}$. Moreover, for any object $Y$ of $\sch^{\aetale}_{/X}$, a family of morphisms of $\sch^{\aetale}_{/Y}$ lies in the pretopology of $\sch^{\aetale}_{/Y}$ if and only if its image lies in the pretopology of $\sch^{\aetale}_{/X}$. Thus, the localization $(\sch^{\aetale}_{/X})_{/Y}$ of $\sch^{\aetale}_{/X}$ at $Y$ is canonically equivalent to $\sch^{\aetale}_{/Y}$ (\cite[\Luoma{3}.3.3, \Luoma{3}.5.2]{sga4-1}).
\end{mypara}

\begin{mycor}\label{cor:adm-etale}
	Let $K$ be a valuation field of height $1$, $Y\to X$ a morphism of finitely presented $\ca{O}_K$-schemes with $Y_K$ \'etale over $X_K$. Then, a family of morphisms with target $Y$ in $\sch^{\aetale}_{/X}$ is a covering family if and only if it can be refined by a finite family of morphisms in $\sch^{\aetale}_{/X}$, 
	\begin{align}
		\{Y'_i\to Y'\to \overline{Y}\to Y\}_{i\in I},
	\end{align}
	where $\overline{Y}$ is the scheme theoretic closure of $Y_K$ in $Y$, $Y'\to \overline{Y}$ is a $Y_K$-admissible blowup and $\coprod_{i\in I}Y'_i\to Y'$ is faithfully flat.
\end{mycor}
\begin{proof}
	Notice that $\overline{Y}$ is a closed subscheme of $Y$ and thus quasi-separated of finite type over $\ca{O}_K$. As $\overline{Y}$ is also flat over $\ca{O}_K$, we see that it is of finite presentation by \cite[\Luoma{1}.3.4.7]{raynaudgruson1971plat} (or by the remark after \ref{defn:adm-etale}) and thus an object of $\sch^{\aetale}_{/X}$. It is clear that $\overline{Y}\to Y$ is an admissibly \'etale covering. This verifies the sufficiency of the statement as $\{Y'_i\to Y'\to \overline{Y}\}_{i\in I}$ is an admissibly \'etale covering by \ref{prop:adm-etale}.
	
	For the necessity, given a covering family $\{Y_i\to Y\}_{i\in I}$, we may assume that it satisfies the condition in \ref{defn:adm-etale} (\cite[\Luoma{2}.1.4]{sga4-1}). Then, there exists a $Y_K$-admissible blowup $Y'\to\overline{Y}$ such that the strict transform $Z\to Y'$ of $\overline{Y}\times_Y\coprod_{i\in I}Y_i\to \overline{Y}$ is faithfully flat of finite presentation by \ref{prop:adm-etale}.
	\begin{align}
		\xymatrix{
			Z\ar[r]\ar[d]&\overline{Y}\times_Y\coprod_{i\in I}Y_i\ar[r]\ar[d]&\coprod_{i\in I}Y_i\ar[d]\\
			Y'\ar[r]&\overline{Y}\ar[r]&Y.
		}
	\end{align}
	As $Z_K=\coprod_{i\in I}Y_{i,K}$ is \'etale over $X_K$, $Z$ is an object of $\sch^{\aetale}_{/X}$. Let $Y'_i\subseteq Z$ be the preimage of $Y_i$ of the morphism $Z\to \coprod_{i\in I}Y_i$, which is also an object of $\sch^{\aetale}_{/X}$. Notice that $Y'$ is flat of finite presentation over $\ca{O}_K$ by \ref{rem:blowup} and thus an object of $\sch^{\aetale}_{/X}$. In conclusion, $\{Y'_i\to Y'\to \overline{Y}\to Y\}_{i\in I}$ is a well-defined finite family of morphisms in $\sch^{\aetale}_{/X}$ which meets our requirements.
\end{proof}

\begin{mypara}\label{para:notation-regular-log}
	Let $(X,\scr{M}_X)$ be a regular fs log scheme, $X^{\triv}$ the maximal open subscheme of the underlying scheme $X$ on which the log structure is trivial. Recall that the log structure $\scr{M}_X\to \ca{O}_{X_\et}$ is the compactifying log structure associated to the open immersion $j:X^{\triv}\to X$, i.e., $\scr{M}_X\subseteq \ca{O}_{X_\et}$ is the preimage of $j_{\et*}\ca{O}_{X^{\triv}_\et}^\times$ under the canonical map $\ca{O}_{X_\et}\to j_{\et*}\ca{O}_{X^{\triv}_\et}$ (\cite[11.6]{kato1994toric}, \cite[2.6]{niziol2006toric}). 
\end{mypara}

\begin{mylem}[{cf. \cite[2.4.1]{temkin2017etale}}]\label{lem:log-smooth-1}
	 With the notation in {\rm\ref{para:notation-regular-log}}, assume that $X=\spec(A)$ is affine and let $a\in A$ be an element of $\scr{M}_X(X)\subseteq \ca{O}_{X_\et}(X)=A$. We put $X'=\spec(A[u,v]/(uv-a))$ endowed with the compactifying log structure $\scr{M}_{X'}$ associated to the open immersion $X'^{\triv}=X^{\triv}\times_XX'\to X'$. Then, the canonical morphism of log schemes $(X',\scr{M}_{X'})\to (X,\scr{M}_X)$ is a smooth and saturated morphism of regular fs log schemes.
\end{mylem}
\begin{proof}
	Consider the homomorphism of monoids $\Delta:\bb{N}\to \bb{N}^2$ sending $1$ to $(1,1)$. Since the homomorphism of the associated groups $\bb{Z}\to \bb{Z}^2$ is injective with cokernel isomorphic to $\bb{Z}$, we see that the morphism of fs log schemes $\bb{A}_{\bb{N}^2}\to \bb{A}_{\bb{N}}$ is smooth (\cite[\Luoma{4}.3.1.8]{ogus2018log}). Moreover, $\Delta$ is saturated by \cite[\Luoma{1}.4.1]{tsuji2019saturated} (see the arguments of \cite[12.3]{he2024purity}). Since $a$ is invertible over $X^{\triv}$, $u$ and $v$ are invertible over $X'^{\triv}$. Thus, there is a commutative diagram of monoids
	\begin{align}
		\xymatrix{
			\scr{M}_{X'}(X')&\bb{N}^2\ar[l]_-{\beta}\\
			\scr{M}_X(X)\ar[u]&\bb{N}\ar[l]_-{\alpha}\ar[u]_-{\Delta}
		}
	\end{align}
	where $\beta(n,m)=u^nv^m$ and $\alpha(n)=a^n$ for any $n,m\in\bb{N}$. It induces a commutative diagram of log schemes (\cite[\Luoma{3}.1.2.9]{ogus2018log})
	\begin{align}
		\xymatrix{
			(X',\scr{M}_{X'})\ar[r]\ar[d]&\bb{A}_{\bb{N}^2}\ar[d]\\
			(X,\scr{M}_X)\ar[r]&\bb{A}_{\bb{N}}.
		}
	\end{align}
	Consider the fibred product $(X'',\scr{M}_{X''})=(X,\scr{M}_X)\times_{\bb{A}_{\bb{N}}}\bb{A}_{\bb{N}^2}$ in the category of log schemes (see \cite[8.4]{he2022sen}). Note that the underlying scheme $X''$ is $X\times_{\bb{A}_{\bb{N}}}\bb{A}_{\bb{N}^2}=\spec(A[u,v]/(uv-a))=X'$ and we have $X''^{\triv}=X^{\triv}\times_{\bb{A}_{\bb{Z}}}\bb{A}_{\bb{Z}^2}=X^{\triv}\times_XX'=X'^{\triv}$. Since $\bb{A}_{\bb{N}^2}\to \bb{A}_{\bb{N}}$ is a saturated morphism of fs log schemes, $(X'',\scr{M}_{X''})$ is fs as $(X,\scr{M}_X)$ is so (\cite[\Luoma{3}.2.5.3.2]{ogus2018log}). Since $\bb{A}_{\bb{N}^2}\to \bb{A}_{\bb{N}}$ is also smooth, $(X'',\scr{M}_{X''})\to (X,\scr{M}_X)$ is smooth. In particular, $(X'',\scr{M}_{X''})$ is regular as $(X,\scr{M}_X)$ is so (\cite[\Luoma{4}.3.5.3]{ogus2018log}). Thus, the log structure on $X''$ is the compactifying log structure associated to the open immersion $X''^{\triv}\to X''$ (\cite[11.6]{kato1994toric}, \cite[2.6]{niziol2006toric}). Hence, the canonical morphism $(X',\scr{M}_{X'})\to (X'',\scr{M}_{X''})$ is an isomorphism of log schemes, which completes the proof.
\end{proof}

\begin{mylem}\label{lem:log-smooth-2}
	With the notation in {\rm\ref{para:notation-regular-log}}, assume that $X=\spec(A)$ is affine. We put $X'=\spec(A[T])$ and let $D'\subset X'$ be the closed subscheme defined by $T=0$. We endow $X'$ with the compactifying log structure $\scr{M}_{X'}$ associated to the open immersion $X'^{\triv}= (X^{\triv}\times_XX')\setminus D' \to X'$. Then, the canonical morphism of log schemes $(X',\scr{M}_{X'})\to (X,\scr{M}_X)$ is a smooth and saturated morphism of regular fs log schemes.
\end{mylem}
\begin{proof}
	Similar to the proof of \ref{lem:log-smooth-1}, it is easy to check that $(X',\scr{M}_{X'})=(X,\scr{M}_X)\times_{\spec(\bb{Z})}\bb{A}_{\bb{N}}$.
\end{proof}

\begin{myprop}[{cf. \cite[2.4.2]{temkin2017etale}}]\label{prop:log-smooth}
	With the notation in {\rm\ref{para:notation-regular-log}}, let $X'$ be a flat $X$-scheme, locally of finite presentation, smooth over $X^{\triv}\subseteq X$, whose non-empty fibres are equidimensional of dimension $1$ and have at-worst-nodal singularities {\rm(\cite[\href{https://stacks.math.columbia.edu/tag/0C47}{0C47}]{stacks-project})}, and let $D'\subseteq X'$ be a Cartier divisor \'etale over $X$ disjoint from the nodes of $X'$ (i.e., $(X',D')$ is a semi-stable multi-pointed curve over $X$ in the sense of {\rm\cite[1.2]{temkin2010stablecurve}}). We endow $X'$ with the compactifying log structure associated to the open immersion $X'^{\triv}=(X^{\triv}\times_XX')\setminus D'\to X'$. Then, the canonical morphism of log schemes $(X',\scr{M}_{X'})\to (X,\scr{M}_X)$ is a smooth and saturated morphism of regular fs log schemes.
\end{myprop}
\begin{proof}
	We follow closely the proof of \cite[12.6]{he2024purity}. The problem is \'etale local on $X'$ and $X$. For any geometric point $\overline{x}'$ of $X'$, after replacing $X'$ and $X$ by \'etale neighborhoods of $\overline{x}'$, we may assume that $X=\spec(A)$ and $X'=\spec(A')$ are affine. Note that $A$ is Noetherian as $(X,\scr{M}_X)$ is assumed to be regular.
	
	Assume firstly that the image $x'$ of $\overline{x}'\to X'$ does not lie in $D'$ (so that we may assume that $D'$ is empty). If $X'\to X$ is smooth at $x'$, then after shrinking $X$ and $X'$, we may assume that there is an \'etale homomorphism $A[T^{\pm1}]\to A'$. Thus, $(X',\scr{M}_{X'})\to (X,\scr{M}_X)$  is a smooth and saturated morphism of regular fs log schemes by \ref{lem:log-smooth-2}. If $X'\to X$ is not smooth at $x'$, then after shrinking $X$ and $X'$, we may assume that there exists an \'etale homomorphism $A[u,v]/(uv-a)\to A'$ for some $a\in A$ by \cite[\href{https://stacks.math.columbia.edu/tag/0CBY}{0CBY}]{stacks-project}. Since $X'\to X$ is smooth over $X^{\triv}$, $a$ is invertible over $X^{\triv}$. Thus, $(X',\scr{M}_{X'})\to (X,\scr{M}_X)$ is a smooth and saturated morphism of regular fs log schemes by \ref{lem:log-smooth-1}.
	
	If the image $x'$ of $\overline{x}'\to X'$ lies in $D'$, then we may assume further that $D'$ is defined by a non-zero divisor $t'\in A'$ and that the canonical morphism $D'\to X$ is an isomorphism. In particular, the canonical homomorphisms $A\to A'\to A'/tA'$ induce an isomorphism $A\iso A'/tA'$. Let $x\in X$ be the image of $x'$ and let $\ak{p}\subseteq A$ (resp. $\ak{p}'\subseteq A'$) be the prime ideal corresponding to $x$ (resp. $x'$). Note that $\ak{p}'$ is generated by $\ak{p}$ and $t$. Since $D'$ is disjoint from the nodes of $X'$, $A'$ is smooth over $A$ at $\ak{p}'$ of dimension $1$. Thus, the completion $\widehat{A'_{\ak{p}'}}$ of the Noetherian local ring $A'_{\ak{p}'}$ is isomorphic to the $\widehat{A_{\ak{p}}}$-algebra of formal power series $\widehat{A_{\ak{p}}}[[T]]$ by \cite[17.5.3]{ega4-4}. As $t$ generates the maximal ideal $(T)$ of $\widehat{A'_{\ak{p}'}}/\ak{p}\widehat{A'_{\ak{p}'}}=(A_{\ak{p}}/\ak{p}A_{\ak{p}})[[T]]$, we see that $t\in \widehat{A_{\ak{p}}}[[T]]$ is the multiple of $T$ with a unit. After replacing $T$ by multiplying by a unit, we may assume that $t=T\in \widehat{A_{\ak{p}}}[[T]]$. Hence, the $A$-homomorphism $A[T]\to A'$ sending $T$ to $t$ induces an isomorphism of the corresponding completed local rings at $(\ak{p},T)$ and $\ak{p}'$. Thus, it is \'etale at $\ak{p}'$ by \cite[17.5.3]{ega4-4}. After shrinking $X$ and $X'$, we may assume that $A[T]\to A'$ is \'etale. Thus, $(X',\scr{M}_{X'})\to (X,\scr{M}_X)$ is a smooth and saturated morphism of regular fs log schemes by \ref{lem:log-smooth-2}.
\end{proof}

\begin{mythm}[{cf. \cite[2.5.2]{temkin2017etale}}]\label{thm:uniformization}
	Let $K$ be a complete discrete valuation field, $\eta=\spec(K)$, $S=\spec(\ca{O}_K)$. Let $X$ be an $S$-scheme of finite presentation with smooth generic fibre $X_\eta$, $D_\eta$ a strict normal crossings divisor of $X_\eta$, $X^{\triv}=X_\eta\setminus D_\eta$. Then, there exists an admissibly \'etale covering in $\sch^{\aetale}_{/X}$,
	\begin{align}\label{eq:thm:uniformization}
		X'\longrightarrow S'\times_SX,
	\end{align}
	where $S'=\spec(\ca{O}_{K'})$, $K'$ is a finite field extension of $K$, and $X'$ is a flat $S'$-scheme of finite presentation such that the log scheme $(X',\scr{M}_{X'})$ endowed with the compactifying log structure associated to the open immersion $X'^{\triv}=X^{\triv}\times_XX'\to X'$ is a regular fs log scheme, smooth and saturated over the regular fs log scheme $(S',\scr{M}_{S'})$ endowed with the compactifying log structure defined by the closed point of $S'$.
\end{mythm}
\begin{proof}
	We follow the proof of \cite[2.5.2]{temkin2017etale} and we proceed by induction on the dimension $d$ of $X_\eta$. 
	
	Firstly, we note that if $X'\to S'=\spec(\ca{O}_{K'})$ satisfies the conditions for \eqref{eq:thm:uniformization}, then so does $X''=S''\times_{S'}X'\to S''=\spec(\ca{O}_{K''})$ for any finite field extension $K''$ of $K'$. Indeed, this follows from the fact that the base change of a morphism of regular fs log schemes by a smooth saturated morphism of regular fs log schemes in the category of log schemes is still a morphism of regular fs log schemes (see \cite[4.4]{he2024falmain}). 
	
	Then, we claim that the problem is admissibly \'etale local on $X$. Indeed, if $\{X_i\to X\}_{i\in I}$ is an admissibly \'etale covering such that for each $i\in I$ there exists an admissibly \'etale covering $X_i'\to S'_i\times_SX_i$ satisfying the (analogous) conditions for \eqref{eq:thm:uniformization}. As $X$ is quasi-compact in $\sch^{\aetale}_{/X}$, we may assume that $I$ is finite. Then, we take a finite field extension $K'$ of $K$ large enough such that $S'=\spec(\ca{O}_{K'})\to S$ factors through each $S_i'$. By the discussion above, we see that $X'=\coprod_{i\in I}S'\times_{S'_i}X_i'\to S'\times_SX$ is an admissibly \'etale covering satisfying the conditions for \eqref{eq:thm:uniformization}, which verifies the claim. 
	
	Now, we use the claim to make several simplifications. After replacing $X$ by the scheme theoretic closure of $X_\eta$ in $X$, we may assume that $X$ is flat over $\ca{O}_K$ (see \ref{cor:adm-etale}). As $D_\eta$ is a strict normal crossings divisor of $X_\eta$, after replacing $X$ by a Zariski covering of an $X_\eta$-admissible blowup (which can be chosen to refine any Zariski covering of $X_\eta$ by \cite[2.5.1]{temkin2017etale}), we may assume that $X=\spec(A)$ is an affine integral scheme and that there exists an \'etale morphism $X_\eta\to \spec(K[T_1,\dots,T_d])$ such that $D_\eta$ is defined by $T_1\cdots T_r=0$ for some integer $0\leq r\leq d$ by \cite[17.5.3]{ega4-4}. After replacing $T_1,\dots,T_d$ by their multiples by a power of $p$, we may extend $X_\eta\to \spec(K[T_1,\dots,T_d])$ to $X\to \spec(\ca{O}_K[T_1,\dots,T_d])$. 
	
	If $d=0$, after enlarging $K$, we see that $A$ is an $\ca{O}_K$-subalgebra of $K$ and thus equal to $K$ or $\ca{O}_K$. If $d>0$, let $f:X\to \bb{A}^{d-1}_S=\spec(\ca{O}_K[T_2,\dots,T_d])$ be the induced morphism. For $r>0$, we put $\bb{A}^{d-1,\triv}_S=\spec(K[T_2^{\pm1},\dots,T_r^{\pm1},T_{r+1},\dots,T_d])\subseteq \bb{A}^{d-1}_\eta$ and let $D^1_\eta\subseteq X_\eta$ be the closed subscheme defined by $T_1=0$; for $r=0$, we put $\bb{A}^{d-1,\triv}_S=\bb{A}^{d-1}_\eta$ and $D^1_\eta=\emptyset$. In particular, we have
	\begin{align}\label{eq:thm:uniformization-1}
		X^{\triv}=X_\eta\setminus D_\eta=(\bb{A}^{d-1,\triv}_S\times_{\bb{A}^{d-1}_S}X)\setminus D^1_\eta.
	\end{align}
	Note that $f_\eta:X_\eta\to \bb{A}^{d-1}_\eta$ is smooth of relative dimension $1$ and $D^1_\eta\to \bb{A}^{d-1}_\eta$ is \'etale. Let $D^1$ be the scheme theoretic closure of $D^1_\eta$ in $X$, which is thus flat of finite presentation over $S$ by \cite[\Luoma{1}.3.4.7]{raynaudgruson1971plat}. After replacing $X$ by a Zariski covering of an $X_\eta$-admissible blowup as above, we may assume that $X$ and $D^1$ are both affine and integral. 
	
	By Raynaud-Gruson's flattening theorem \cite[\Luoma{1}.5.2.2]{raynaudgruson1971plat} (and \cite[\href{https://stacks.math.columbia.edu/tag/080F}{080F}]{stacks-project}), there exists an $\bb{A}^{d-1}_\eta$-admissible blowup $Z\to \bb{A}^{d-1}_S$ such that the strict transform $f':X'\to Z$ of $f:X\to \bb{A}^{d-1}_S$ and the strict transform $D'^1\to Z$ of $D^1\to \bb{A}^{d-1}_S$ are both flat of finite presentation. Note that $D'^1\to X'$ is still a closed immersion as $D'\subseteq Z\times_{\bb{A}^{d-1}_S}D^1$ and $X'\subseteq Z\times_{\bb{A}^{d-1}_S}X$ are closed subschemes by definition (\cite[\href{https://stacks.math.columbia.edu/tag/080D}{080D}]{stacks-project}).
	\begin{align}
		\xymatrix{
			D^1_\eta\ar[r]\ar[d]&D^1\ar[d]&&D'^1_\eta\ar[r]\ar[d]&D'^1\ar[d]\\
			X_\eta\ar[r]\ar[d]_-{f_\eta}&X\ar[d]^-{f}&& X'_\eta\ar[r]\ar[d]_-{f'_\eta}&X'\ar[d]^-{f'}\\
			\bb{A}^{d-1}_\eta\ar[r]&\bb{A}^{d-1}_S&&Z_\eta\ar[r]&Z
		}
	\end{align}
	As $D'^1_\eta=D^1_\eta$, $X'_\eta=X_\eta$ and $Z_\eta=\bb{A}^{d-1}_\eta$ are integral and scheme theoretically dense in $D'^1$, $X'$ and $Z$ respectively (see \ref{rem:blowup}), we see that $D'^1$, $X'$ and $Z$ are also integral. Then, $f':X'\to Z$ is equidimensional of relative dimension $1$ and $D'^1\to Z$ is equidimensional of relative dimension $0$ (\cite[14.2.2]{ega4-3}). In particular, $(X',D'^1)$ is a multi-pointed $Z$-curve semi-stable over $Z_\eta$ in the sense of \cite[1.2]{temkin2010stablecurve}.
	
	By Temkin's stable modification theorem for relative curves \cite[2.3.3]{temkin2011rz}, after replacing $Z$ by an admissibly \'etale covering and $D'^1\subseteq X'$ by base change, there exists an $X'_\eta$-modification $(X'',D''^1)\to (X',D'^1)$ with $(X'',D''^1)$ semi-stable over $Z$. After replacing $(X',D'^1)$ by $(X'',D''^1)$, we may assume that $(X',D'^1)$ is semi-stable over $Z$.
	
	Since $Z_\eta$ is \'etale over $\bb{A}^{d-1}_\eta$, the equation $T_2\cdots T_r=0$ defines a strict normal crossings divisor on $Z_\eta$. Let $Z^{\triv}=\bb{A}^{d-1,\triv}_\eta\times_{\bb{A}^{d-1}_S}Z$ be its complement. By the induction hypothesis, after enlarging $K$, replacing $Z$ by an admissibly \'etale covering and replacing $D'^1\subseteq X'$ by base change, we may assume that $(Z,\scr{M}_{Z})$ (endowed with the compactifying log structure associated to the open immersion $Z^{\triv}\to Z$) is a regular fs log scheme smooth and saturated over $(S,\scr{M}_S)$. As $(X',D'^1)$ is still a semi-stable multi-pointed curve over $Z$, we conclude that $(X',\scr{M}_{X'})$ (endowed with the compactifying log structure associated to the open immersion $X'^{\triv}=(Z^{\triv}\times_ZX')\setminus D'^1\to X'$, see \eqref{eq:thm:uniformization-1}) is a regular fs log scheme smooth and saturated over $(S,\scr{M}_S)$ by \ref{prop:log-smooth}.
\end{proof}

\section{Review of Faltings extension and Sen Operators over Adequate Algebras}\label{sec:adequate}
We review the Faltings extension and Sen operators constructed in \cite{he2022sen} (see \ref{thm:A-fal-ext}, \ref{thm:sen-brinon-A} and \ref{thm:sen-lie-lift-A}). For simplicity, we only consider adequate algebras instead of quasi-adequate algebras, as they are sufficient for our subsequent applications. Especially, we emphasize the functoriality of these constructions (see \ref{rem:A-fal-ext}, \ref{rem:sen-brinon-A} and \ref{rem:sen-lie-lift-A}), which is a variant of the functoriality proved in \cite{he2022sen}. Finally, we discuss a relation between non-vanishing of the arithmetic Sen operator with infiniteness of inertia subgroups (see \ref{prop:geom-sen-nonzero}).

\begin{mypara}\label{para:zeta}
	In this section, we fix a complete discrete valuation field $K$ extension of $\bb{Q}_p$ with perfect residue field, an algebraic closure $\overline{K}$ of $K$, and a compatible system of primitive $n$-th roots of unity $(\zeta_n)_{n\in\bb{N}_{>0}}$ in $\overline{K}$. Sometimes we denote $\zeta_n$ by $t_{0,n}$.
\end{mypara}

%

\begin{mydefn}[{\cite[5.3]{he2024purity}, \cite[9.2]{he2022sen}}]\label{defn:triple}
	A \emph{$(K,\ca{O}_K,\ca{O}_{\overline{K}})$-triple} is a triple $(A_{\triv},A,\overline{A})$ where
	\begin{enumerate}
		\renewcommand{\labelenumi}{{\rm(\theenumi)}}
		\item $A$ is a Noetherian normal domain flat over $\ca{O}_K$ with $A/pA\neq 0$,
		\item $A_\triv$ is a $K$-algebra that is a localization of $A$ with respect to a nonzero element of $pA$,
		\item $\overline{A}$ is an $\ca{O}_{\overline{K}}$-algebra that is the integral closure of $A$ in a maximal unramified extension $\ca{K}_{\mrm{ur}}$ of the fraction field $\ca{K}$ of $A$ with respect to $(A_{\triv},A)$, i.e., $\ca{K}_{\mrm{ur}}$ is the union of all finite field extension $\ca{K}'$ of $\ca{K}$ contained in an algebraic closure $\overline{\ca{K}}$ such that the integral closure of $A_{\triv}$ in $\ca{K}'$ is \'etale over $A_{\triv}$,
	\end{enumerate}
	such that the diagram
	\begin{align}
		\xymatrix{
			A_{\triv}&A\ar[l]\ar[r]&\overline{A}\\
			K\ar[u]&\ca{O}_K\ar[l]\ar[r]\ar[u]&\ca{O}_{\overline{K}}\ar[u]
		}
	\end{align}
	formed by the structural morphisms is commutative. Note that $\overline{A}$ is stable under the action of $G_A=\gal(\ca{K}_{\mrm{ur}}/\ca{K})$ on $\ca{K}_{\mrm{ur}}$.
	
	A \emph{morphism of $(K,\ca{O}_K,\ca{O}_{\overline{K}})$-triples} $(A_{\triv},A,\overline{A})\to (A'_{\triv},A',\overline{A'})$ is a homomorphism of $\ca{O}_{\overline{K}}$-algebras $f:\overline{A}\to \overline{A'}$ such that $f(A)\subseteq A'$ and $f(A_{\triv})\subseteq A'_{\triv}$. If $f$ is injective, then it induces an extension of the fraction fields $\ca{K}_{\mrm{ur}}\to\ca{K}'_{\mrm{ur}}$ and thus a natural continuous homomorphism of Galois groups $G_{A'}\to G_A$.
\end{mydefn}

\begin{mydefn}[{\cite[5.4]{he2024purity}, \cite[9.5]{he2022sen}}]\label{defn:essential-adequate-alg}
	A $(K,\ca{O}_K,\ca{O}_{\overline{K}})$-triple $(A_{\triv},A,\overline{A})$ is called \emph{adequate} if there exists a commutative diagram of monoids
	\begin{align}
		\xymatrix{
			A& P\ar[l]_-{\beta}\\
			\ca{O}_K\ar[u]& \bb{N}\ar[l]_-{\alpha}\ar[u]_-{\gamma}
		}
	\end{align} 
	satisfying the following conditions:
	\begin{enumerate}
		\renewcommand{\labelenumi}{{\rm(\theenumi)}}
		\item The element $\alpha(1)$ is a uniformizer of $\ca{O}_K$.\label{item:defn:essential-adequate-alg-1}
		\item The monoid $P$ is fs (i.e., fine and saturated), and if we denote by $\gamma_\eta:\bb{Z}\to P_\eta=\bb{Z}\oplus_{\bb{N}}P$ the pushout of $\gamma$ by the inclusion $\bb{N}\to \bb{Z}$, then there exists an isomorphism for some $c, d\in \bb{N}$ with $c\leq d$,
		\begin{align}\label{eq:monoid-str}
			P_\eta\cong \bb{Z}\oplus \bb{Z}^c\oplus \bb{N}^{d-c},
		\end{align}
		identifying $\gamma_\eta$ with the inclusion of $\bb{Z}$ into the first component on the right hand side.\label{item:defn:essential-adequate-alg-2}
		\item The homomorphism $\beta$ induces an \'etale $\ca{O}_K$-algebra homomorphism $\ca{O}_K\otimes_{\bb{Z}[\bb{N}]}\bb{Z}[P]\to A$ such that $A\otimes_{\bb{Z}[P]}\bb{Z}[P^{\mrm{gp}}]=A_{\triv}$, where $P^{\mrm{gp}}$ is the associated group of $P$.\label{item:defn:essential-adequate-alg-3}
	\end{enumerate}
	We usually denote $(A_{\triv},A,\overline{A})$ by $A$, and call it an \emph{adequate $\ca{O}_K$-algebra} for simplicity. The triple $(\alpha:\bb{N}\to \ca{O}_K,\ \beta:P\to A,\ \gamma:\bb{N}\to P)$ is called an \emph{adequate chart} of $A$. If we fix an isomorphism \eqref{eq:monoid-str}, then we call the images $t_1,\dots,t_d\in A[1/p]$ of the standard basis of $\bb{Z}^c\oplus \bb{N}^{d-c}$ a \emph{system of coordinates} of the chart. We call $d$ the \emph{relative dimension} of $A$ over $\ca{O}_K$ (i.e., the Krull dimension of $A_{\triv}$ as $A_{\triv}$ is \'etale over $K[\bb{Z}^d]$).
\end{mydefn}

\begin{mypara}\label{para:adequate-setup}
	Let $A$ be an adequate $\ca{O}_K$-algebra with an adequate chart $(\alpha:\bb{N}\to \ca{O}_K,\ \beta:P\to A,\ \gamma:\bb{N}\to P)$ and an isomorphism
	$P_\eta\cong \bb{Z}\oplus \bb{Z}^c\oplus \bb{N}^{d-c}$ \eqref{eq:monoid-str}. Let $t_1,\dots,t_d\in A[1/p]$ be the associated system of coordinates of this chart.
	
	Following \cite[\textsection8]{tsuji2018localsimpson} (see also \cite[5.9]{he2024purity} and \cite[9.13]{he2022sen}), for any integer $1\leq i\leq d$, we fix a compatible system of $k$-th roots $(t_{i,k})_{k\in \bb{N}_{>0}}$ of $t_i$ in $\overline{A}[1/p]$. For any field extension $E'/E$, let $\scr{F}_{E'/E}$ (resp. $\scr{F}^{\mrm{fini}}_{E'/E}$) be the set of (resp. finite) field extensions of $E$ contained in $E'$, and we endow it with the partial order defined by the inclusion relation. For any $L\in \ff{K}$ and any $\underline{r}=(r_1,\dots,r_d)\in\bb{N}_{>0}^d$, we put
	\begin{align}\label{eq:para:adequate-setup}
		\ca{K}^L_{\underline{r}}=L\cdot \ca{K}(t_{1,r_1},\dots,t_{d,r_d})
	\end{align}
	where the composites of fields are taken in $\ca{K}_{\mrm{ur}}$. It is clear that $\ca{K}^L_{\underline{r}}$ forms an inductive system of fields over the directed partially ordered set $\ff{K}\times \bb{N}_{>0}^d$ (see \ref{para:product}). Let $A^L_{\underline{r}}$ be the integral closure of $A$ in $\ca{K}^L_{\underline{r}}$, $A^L_{\underline{r},\triv}=A_{\triv}\otimes_{A}A^L_{\underline{r}}$. We endow $X^L_{\underline{r}}=\spec(A^L_{\underline{r}})$ with the compactifying log structure $\scr{M}_{X^L_{\underline{r}}}$ associated to the open immersion $X^{L,\triv}_{\underline{r}}=\spec(A^L_{\underline{r},\triv})\to X^L_{\underline{r}}$.
	
	We extend the notation above to any $(L,\underline{r})\in\scr{F}_{\overline{K}/K}\times (\bb{N}_{>0}\cup\{\infty\})^d$ by taking filtered colimits, and we omit the index $L$ or $\underline{r}$ if $L=K$ or $\underline{r}=\underline{1}$ respectively. We endow $\overline{X}=\spec(\overline{A})$ with the log structure $\scr{M}'_{\overline{X}}$ inverse image of the log structure of $X^{\overline{K}}_{\underline{\infty}}$ via the canonical morphism $\spec(\overline{A})\to \spec(A^{\overline{K}}_{\underline{\infty}})$ as in \cite[9.21]{he2022sen}.
\end{mypara}

\begin{mythm}[{\cite[9.36]{he2022sen}}]\label{thm:A-fal-ext}
	Let $A$ be an adequate $\ca{O}_K$-algebra, $(X_K,\scr{M}_{X_K})$ the log scheme with underlying scheme $X_K=\spec(A[1/p])$ and with compactifying log structure $\scr{M}_{X_K}$ associated to the open immersion $\spec(A_{\triv})\to\spec(A[1/p])$. Then, there is a canonical exact sequence of finite free $\widehat{\overline{A}}[1/p]$-representations of $G_A$ {\rm(\ref{defn:repn})}, called the \emph{Faltings extension of $A$ over $\ca{O}_K$},
	\begin{align}\label{eq:thm:A-fal-ext-1}
		0\longrightarrow \widehat{\overline{A}}[\frac{1}{p}](1)\stackrel{\iota}{\longrightarrow}\scr{E}_A\stackrel{\jmath}{\longrightarrow} \widehat{\overline{A}}\otimes_{A}\Omega^1_{(X_K,\scr{M}_{X_K})/K}\longrightarrow 0,
	\end{align}
	together with a canonical $G_A$-equivariant group homomorphism
	\begin{align}\label{eq:thm:A-fal-ext-2}
		V_p(\overline{A}[1/p]\cap \overline{A}_{\triv}^\times)\longrightarrow \scr{E}_A,
	\end{align}
	for which we denote the image of an element $(s_{p^n})_{n\in\bb{N}}\in V_p(\overline{A}[1/p]\cap \overline{A}_{\triv}^\times)$ by $(\df\log(s_{p^n}))_{n\in\bb{N}}$, satisfying the following properties:
	\begin{enumerate}
		\renewcommand{\labelenumi}{{\rm(\theenumi)}}
		\item We have $\iota(1\otimes (\zeta_{p^n})_{n\in \bb{N}})=(\df\log(\zeta_{p^n}))_{n\in \bb{N}}$.\label{item:thm:A-fal-ext-1}
		\item For any element $s\in A[1/p]\cap A_{\triv}^\times$ and any compatible system of $p$-power roots $(s_{p^n})_{n\in \bb{N}}$ of $s$ in $\overline{A}[1/p]$, $\jmath((\df\log(s_{p^n}))_{n\in\bb{N}})=\df \log(s)$.\label{item:thm:A-fal-ext-2}
		\item With the notation in {\rm\ref{para:adequate-setup}}, there is a canonical isomorphism of finite free $\widehat{\overline{A}}[1/p]$-representations of $G_A$,
		\begin{align}
			V_p(\Omega^1_{(\overline{X},\scr{M}'_{\overline{X}})/(X,\scr{M}_X)})\iso \scr{E}_A,
		\end{align} 
		such that the induced map $V_p(\overline{A}[1/p]\cap \overline{A}_{\triv}^\times)\to V_p(\Omega^1_{(\overline{X},\scr{M}'_{\overline{X}})/(X,\scr{M}_X)})$ is the canonical map sending $(s_{p^n})_{n\in\bb{N}}$ to $(\df\log(s_{p^n}))_{n\in\bb{N}}$ (defined as in \eqref{eq:para:notation-fal-ext-2-4}, see {\rm\cite[9.31, 9.35]{he2022sen}}). In particular, the $\widehat{\overline{A}}[1/p]$-linear surjection $\jmath$ admits a section sending $\df\log(t_i)$ to $(\df \log(t_{i,p^n}))_{n\in\bb{N}}$ for any $1\leq i\leq d$. \label{item:thm:A-fal-ext-3}
	\end{enumerate} 
\end{mythm}
\begin{myrem}\label{rem:A-fal-ext}
	We keep the notation in \ref{thm:A-fal-ext}.
	\begin{enumerate}
		\renewcommand{\labelenumi}{{\rm(\theenumi)}}
		\item Let $K'$ be a finite field extension of $K$ contained in $\overline{K}$. Then, $A'=A^{K'}$ (\ref{para:adequate-setup}) is an adequate $\ca{O}_{K'}$-algebra by \cite[5.10.(1)]{he2024purity} (which holds also for ``adequate algebras" instead of ``essentially adequate algebras" by its proof). Then, there is a canonical isomorphism of Faltings extensions by \cite[9.38]{he2022sen},
		\begin{align}\label{eq:rem:A-fal-ext-1}
			\xymatrix{
				0\ar[r]& \widehat{\overline{A'}}[\frac{1}{p}](1)\ar[r]^-{\iota}&\scr{E}_{A'}\ar[r]^-{\jmath}&\widehat{\overline{A'}}\otimes_{A'}\Omega^1_{(X'_{K'},\scr{M}_{X'_{K'}})/K'}\ar[r]& 0\\
				0\ar[r]& \widehat{\overline{A}}[\frac{1}{p}](1)\ar[r]^-{\iota}\ar[u]^-{\wr}&\scr{E}_{A}\ar[r]^-{\jmath}\ar[u]^-{\wr}&\widehat{\overline{A}}\otimes_{A}\Omega^1_{(X_K,\scr{M}_{X_K})/K}\ar[r]\ar[u]^-{\wr}& 0
			}
		\end{align}
		together with a canonical commutative diagram
		\begin{align}\label{eq:rem:A-fal-ext-2}
			\xymatrix{
				V_p(\overline{A'}[1/p]\cap \overline{A'}_{\triv}^\times)\ar[r]&\scr{E}_{A'}\\
				V_p(\overline{A}[1/p]\cap \overline{A}_{\triv}^\times)\ar[u]^-{\wr}\ar[r]&\scr{E}_{A}\ar[u]^-{\wr}
			}
		\end{align}
		where we used the fact that $\overline{A'}=\overline{A}$ and $K\to K'$ is \'etale.
		\label{item:rem:A-fal-ext-1}
		\item Let $K'$ be a complete discrete valuation field extension of $K$ with perfect residue field and we fix an extension of their algebraic closures $\overline{K}\to \overline{K'}$. Let $A_1$ (resp. $A_2$) be an adequate $\ca{O}_{K_1}$-algebra (resp. adequate $\ca{O}_{K_2}$-algebra) with fraction field $\ca{K}_1$ (resp. $\ca{K}_2$) for some finite field extension $K_1$ (resp. $K_2$) of $K$ (resp. $K'$) contained in $\overline{K}$ (resp. $\overline{K'}$). Then, for any commutative diagram of $(K,\ca{O}_K,\ca{O}_{\overline{K}})$-triples
		\begin{align}\label{eq:rem:A-fal-ext-3}
			\xymatrix{
				(A_{1,\triv},A_1,\overline{A_1})\ar[r]& (A_{2,\triv},A_2,\overline{A_2})\\
				(K,\ca{O}_K,\ca{O}_{\overline{K}})\ar[u]\ar[r]&(K',\ca{O}_{K'},\ca{O}_{\overline{K'}})\ar[u]
			}
		\end{align}
		with $\overline{A_1}\to \overline{A_2}$ injective (which thus induces an extension of the fraction fields $\ca{K}_{1,\mrm{ur}}\to \ca{K}_{2,\mrm{ur}}$), we consider a finite field extension $K_1'$ of $K_1$ contained in $\overline{K}$ and a finite field extension $K_2'$ of $K_1'K_2$ contained in $\overline{K'}$, and we put $A_i'=A_i^{K_i'}$ (so that its fraction field is the composite $K_i'\ca{K}_i$ in $\ca{K}_{i,\mrm{ur}}$). Then, we obtain commutative diagrams of fraction fields:
		\begin{align}\label{eq:rem:A-fal-ext-4}
			\xymatrix{
				\overline{K'}&K_2'\ar[l]&K_2\ar[l]&K'\ar[l]&\ca{K}_{2,\mrm{ur}}&K_2'\ca{K}_2\ar[l]&\ca{K}_2\ar[l]&K'\ar[l]\\
				\overline{K}\ar[u]&K_1'\ar[l]\ar[u]&K_1\ar[l]&K\ar[l]\ar[u]&\ca{K}_{1,\mrm{ur}}\ar[u]&K_1'\ca{K}_1\ar[l]\ar[u]&\ca{K}_1\ar[l]\ar[u]&K\ar[l]\ar[u]
			}
		\end{align}
		from which we see that \eqref{eq:rem:A-fal-ext-3} induces a commutative diagram of $(K,\ca{O}_K,\ca{O}_{\overline{K}})$-triples
		\begin{align}\label{eq:rem:A-fal-ext-5}
			\xymatrix{
				(A_{1,\triv},A_1,\overline{A_1})\ar[r]&(A'_{1,\triv},A'_1,\overline{A'_1})\ar[r]& (A'_{2,\triv},A'_2,\overline{A'_2})& (A_{2,\triv},A_2,\overline{A_2})\ar[l]\\
				(K_1,\ca{O}_{K_1},\ca{O}_{\overline{K}})\ar[u]\ar[r]&(K_1',\ca{O}_{K_1'},\ca{O}_{\overline{K}})\ar[u]\ar[r]&(K_2',\ca{O}_{K_2'},\ca{O}_{\overline{K'}})\ar[u]&(K_2,\ca{O}_{K_2},\ca{O}_{\overline{K'}})\ar[u]\ar[l]
			}
		\end{align}
		where we used the fact that $\overline{A_i'}=\overline{A_i}$ for $i=1,2$. Thus, it induces a canonical morphism of Faltings extensions
		\begin{align}\label{eq:rem:A-fal-ext-6}
			\xymatrix{
				0\ar[r]& \widehat{\overline{A_2}}[\frac{1}{p}](1)\ar[r]^-{\iota}&\scr{E}_{A_2}\ar[r]^-{\jmath}&\widehat{\overline{A_2}}\otimes_{A_2}\Omega^1_{(X_{2,K'},\scr{M}_{X_{2,K'}})/K'}\ar[r]& 0\\
				0\ar[r]& \widehat{\overline{A_1}}[\frac{1}{p}](1)\ar[r]^-{\iota}\ar[u]&\scr{E}_{A_1}\ar[r]^-{\jmath}\ar[u]&\widehat{\overline{A_1}}\otimes_{A_1}\Omega^1_{(X_{1,K},\scr{M}_{X_{1,K}})/K}\ar[r]\ar[u]& 0
			}
		\end{align}
		together with a canonical commutative diagram by (\ref{item:rem:A-fal-ext-1}) and \cite[9.38]{he2022sen},
		\begin{align}\label{eq:rem:A-fal-ext-7}
			\xymatrix{
				V_p(\overline{A_2}[1/p]\cap \overline{A_2}_{\triv}^\times)\ar[r]&\scr{E}_{A_2}\\
				V_p(\overline{A_1}[1/p]\cap \overline{A_1}_{\triv}^\times)\ar[u]\ar[r]&\scr{E}_{A_1}\ar[u]
			}
		\end{align}
		both independent of the choices of $K_1'$ and $K_2'$, where we used the fact that $K\to K_1$ and $K'\to K_2$ are \'etale. In particular, the canonical morphism
		\begin{align}\label{eq:rem:A-fal-ext-8}
			\widehat{\overline{A_2}}\otimes_{\widehat{\overline{A_1}}}\scr{E}_{A_1}\longrightarrow \scr{E}_{A_2}
		\end{align}
		is an isomorphism if $K\to K'$ and $(X_{2,K},\scr{M}_{X_{2,K}})\to (X_{1,K},\scr{M}_{X_{1,K}})$ are both \'etale.
		\label{item:rem:A-fal-ext-2}
		\item As in \cite[9.39]{he2022sen}, taking a Tate twist of the dual of the Faltings extension \eqref{eq:thm:A-fal-ext-1} of $A$, we obtain a canonical exact sequence of finite free $\widehat{\overline{A}}[1/p]$-representations of $G_A$,
		\begin{align}\label{eq:A-fal-ext-dual}
			0\longrightarrow \ho_A(\Omega^1_{(X_K,\scr{M}_{X_K})/K}(-1),\widehat{\overline{A}})\stackrel{\jmath^*}{\longrightarrow} \scr{E}^*_A(1)\stackrel{\iota^*}{\longrightarrow}\widehat{\overline{A}}[\frac{1}{p}]\longrightarrow 0
		\end{align} 
		where $\scr{E}^*_A=\ho_{\widehat{\overline{A}}[1/p]}(\scr{E}_A,\widehat{\overline{A}}[1/p])$. There is a canonical $G_A$-equivariant $\widehat{\overline{A}}[1/p]$-linear Lie algebra structure on $\scr{E}^*_A(1)$ associated to the linear form $\iota^*$, defined by
		\begin{align}
			[f_1,f_2]=\iota^*(f_1)f_2-\iota^*(f_2)f_1,\ \forall f_1,f_2\in \scr{E}^*_A(1).
		\end{align}
		In particular, $\ho_A(\Omega^1_{(X_K,\scr{M}_{X_K})/K}(-1),\widehat{\overline{A}})$ is a Lie ideal of $\scr{E}^*_A(1)$, and $\widehat{\overline{A}}[1/p]$ is the quotient by this ideal. It is clear that the induced Lie algebra structures on them are both trivial. Any $\widehat{\overline{A}}[1/p]$-linear splitting of \eqref{eq:A-fal-ext-dual} identifies $\scr{E}^*_A(1)$ with the semi-direct product of Lie algebras of $\widehat{\overline{A}}[1/p]$ acting on $\ho_A(\Omega^1_{(X_K,\scr{M}_{X_K})/K}(-1),\widehat{\overline{A}})$ by multiplication. 
		
		With the notation in {\rm\ref{para:adequate-setup}}, let $\{T_i=(\df\log(t_{i,p^n}))_{n\in\bb{N}}\otimes\zeta^{-1}\}_{0\leq i\leq d}$ (where $t_{0,p^n}=\zeta_{p^n}$) denote the basis of $\scr{E}_A(-1)$ (\ref{thm:A-fal-ext}.(\ref{item:thm:A-fal-ext-1}, \ref{item:thm:A-fal-ext-3})), and let $\{T_i^*\}_{0\leq i\leq d}$ be the dual basis of $\scr{E}^*_A(1)$. Then, we see that the Lie bracket on $\scr{E}^*_A(1)$ is determined by
		\begin{align}
			[T_0^*,T_i^*]=T_i^*\quad\trm{ and }\quad [T_i^*,T_j^*]=0,
		\end{align}
		for any $1\leq i,j\leq d$. Indeed, this dual basis induces an isomorphism of $\widehat{\overline{A}}[1/p]$-linear Lie algebras
		\begin{align}
			\widehat{\overline{A}}[\frac{1}{p}]\otimes_{\bb{Q}_p}\lie(\bb{Z}_p\ltimes\bb{Z}_p^d)\iso \scr{E}^*_A(1),\ 1\otimes \partial_i\mapsto T_i^*,
		\end{align}
		where $\{\partial_i\}_{0\leq i\leq d}$ is the standard basis of $\lie(\bb{Z}_p\ltimes\bb{Z}_p^d)$ (see \cite[4.20]{he2022sen}).\label{item:rem:A-fal-ext-3}
	\end{enumerate}
\end{myrem}

\begin{mypara}\label{para:notation-A-galois}
	Let $A$ be an adequate $\ca{O}_K$-algebra with fraction field $\ca{K}$. Consider finitely many elements $s_1,\dots,s_e\in A[1/p]\cap A_{\triv}^\times$ with compatible systems of $k$-th roots $(s_{1,k})_{k\in\bb{N}_{>0}},\dots,(s_{e,k})_{k\in\bb{N}_{>0}}$ contained in $\overline{A}[1/p]$ such that $\df s_1,\dots,\df s_e\in \Omega^1_{\ca{K}/K}$ are linearly independent over $\ca{K}$. We put $J\subseteq \bb{N}_{>0}^e$ the subset consisting of elements $\underline{N}=(N_1,\dots,N_e)$ with $N_i$ prime to $p$ for any $1\leq i\leq e$. We endow $J$ with the partial order defined by the divisibility relation (see \ref{para:product}). 
	
	Following \cite[10.1]{he2022sen}, for any $\underline{N}\in J$, $n\in\bb{N}$ and $\underline{m}=(m_1,\dots,m_e)\in \bb{N}^e$, we define finite field extensions of $K$ and $\ca{K}$ in $\overline{K}$ and $\ca{K}_{\mrm{ur}}$ respectively by
	\begin{align}
		K^{(\underline{N})}_n=K(\zeta_{p^nN_1}, \dots, \zeta_{p^nN_e}),\qquad \ca{K}^{(\underline{N})}_{n,\underline{m}}=K^{(\underline{N})}_n\ca{K}(s_{1,p^{m_1}N_1},\dots,s_{e,p^{m_e}N_e}).
	\end{align}
	It is clear that these fields $\ca{K}^{(\underline{N})}_{n,\underline{m}}$ form an inductive system of fields over the directed partially ordered set $J\times\bb{N}\times \bb{N}^e$ (see \ref{para:product}). We extend this notation for one of the components of $\underline{N},n,\underline{m}$ being $\infty$ by taking the corresponding filtered union, and we omit the index $\underline{N}$ or $n$ or $\underline{m}$ if $\underline{N}=\underline{1}$ or $n=0$ or $\underline{m}=\underline{0}$ respectively. We remark that if we take again the notation in \ref{para:adequate-setup} with $(s_1,\dots,s_e)=(t_1,\dots,t_d)$, then $\ca{K}^{(\underline{N})}_{n,\underline{m}}=\ca{K}^{K^{(\underline{N})}_n}_{p^{\underline{m}}\underline{N}}$.
	Similarly, we denote by $A^{(\underline{N})}_{n,\underline{m}}$ the integral closure of $A$ in $\ca{K}^{(\underline{N})}_{n,\underline{m}}$.
	
	Following \cite[4.7]{he2022sen}, for any $\underline{N}\in J$ and $(n,\underline{m})\in(\bb{N}\cup\{\infty\})^{1+e}$, we denote by $\widehat{A}^{(\underline{N})}_{n,\underline{m}}$ the $p$-adic completion of $A^{(\underline{N})}_{n,\underline{m}}$, and we set
	\begin{align}\label{eq:tilde-A}
		\widetilde{A}^{(\underline{N})}_{n,\underline{m}}=\colim_{(n',\underline{m'})\in(\bb{N}^{1+e})_{\leq (n,\underline{m})}} \widehat{A}^{(\underline{N})}_{n',\underline{m'}}.
	\end{align} 
	We remark that the transition maps in the colimit of \eqref{eq:tilde-A} are closed embeddings with respect to the $p$-adic topology, and that $\widetilde{A}^{(\underline{N})}_{n,\underline{m}}$ is identified with a topological subring of $\widehat{A}^{(\underline{N})}_{n,\underline{m}}$ (both endowed with the $p$-adic topology) (\cite[4.7]{he2022sen}). We name some Galois groups as indicated in the following diagram:
	\begin{align}
		\xymatrix{
			\ca{K}_{\mrm{ur}}&&\\
			\ca{K}^{(\underline{N})}_{\infty,\underline{\infty}}\ar[u]^-{H^{(\underline{N})}_{\underline{\infty}}}&&\\
			\ca{K}^{(\underline{N})}_{\infty,\underline{m}}\ar[u]^-{\Delta^{(\underline{N})}_{\underline{m}}}\ar@/^3pc/[uu]^-{H^{(\underline{N})}_{\underline{m}}}&\ca{K}^{(\underline{N})}_{n,\underline{m}} \ar[l]^-{\Sigma^{(\underline{N})}_{n,\underline{m}}}\ar[lu]|-{\Gamma^{(\underline{N})}_{n,\underline{m}}}&\ca{K}\ar@/_1pc/[lluu]|{G_A}\ar[l]\ar@/_1pc/[llu]|(0.6){\Xi^{(\underline{N})}}
		}
	\end{align}
	Recall that for any object $V\in \repnpr(\Xi^{(\underline{N})},\widetilde{A}^{(\underline{N})}_{\infty,\underline{m}}[1/p])$ (\ref{defn:repn}), there is a canonical morphism of Lie algebras over $\bb{Q}_p$, called the \emph{infinitesimal Lie algebra action} (\cite[4.13]{he2022sen}), 
	\begin{align}\label{eq:para:notation-A-galois-4}
		\varphi:\lie(\Xi^{(\underline{N})})\longrightarrow \mrm{End}_{\widetilde{A}^{(\underline{N})}_{\infty,\underline{m}}[1/p]}(V).
	\end{align}
	Following \cite[4.14]{he2022sen}, we say that $V$ is \emph{$\Delta^{(\underline{N})}_{\underline{m}}$-analytic} if $\exp(\varphi_\tau)(v)=\sum_{k=0}^{\infty} \varphi_\tau^k(v)/k!$ converges to $\tau(v)$ for any $v\in V$ and $\tau\in \Delta^{(\underline{N})}_{\underline{m}}$, where $\varphi_\tau=\varphi(\log(\tau))$.
	
	We remark that the assumption on $(s_1,\dots,s_e)$ implies that we have natural identifications of Lie algebras for any $\underline{N}\in J$, $n\in \bb{N}$ and $\underline{m}\in\bb{N}^e$ (\cite[11.3]{he2022sen}),
	\begin{align}\label{eq:para:general-Anm-tower-lie}
		\lie(\Delta^{(\underline{N})}_{\underline{m}})=\lie(\Delta)\cong \bb{Q}_p^e,\quad \lie(\Xi^{(\underline{N})})=\lie(\Gamma),\quad \lie(\Sigma^{(\underline{N})}_{n,\underline{m}})=\lie(\Sigma)\cong \bb{Q}_p.
	\end{align}
	Let $\partial_0\in\lie(\Sigma_{0,\underline{\infty}})$ and $\partial_1,\dots,\partial_e\in\lie(\Delta)$ be the standard bases for the Kummer tower $(A_{n,\underline{m}})_{(n,\underline{m})\in \bb{N}^{1+e}}$ defined by $\zeta_{p^n},s_{1,p^n},\dots,s_{e,p^n}$ (\cite[4.20]{he2022sen}).
	We define $1+e$ elements of the finite free $\widehat{\overline{A}}[1/p]$-module $\scr{E}_A(-1)$ (\ref{thm:A-fal-ext}),  
	\begin{align}\label{eq:para:general-Anm-tower}
		T_0=(\df\log(\zeta_{p^n}))_{n\in\bb{N}}\otimes\zeta^{-1},\ T_1=(\df\log (s_{1,p^n}))_{n\in\bb{N}}\otimes\zeta^{-1},\  \dots,\ T_e=(\df\log (s_{e,p^n}))_{n\in\bb{N}}\otimes\zeta^{-1},
	\end{align}
	where $\zeta=(\zeta_{p^n})_{n\in\bb{N}}$.
\end{mypara}

\begin{mythm}[{\cite[14.2]{tsuji2018localsimpson}}]\label{thm:descent}
	With the notation in {\rm\ref{para:adequate-setup}} and {\rm\ref{para:notation-A-galois}}, if we put $(s_1,\dots,s_e)=(t_1,\dots,t_d)$, then for any $\underline{N}\in J$ and $\underline{m}\in \bb{N}^d$, the functor
	\begin{align}
		\repnan{\Delta^{(\underline{N})}_{\underline{m}}}(\Xi^{(\underline{N})},\widetilde{A}^{(\underline{N})}_{\infty,\underline{m}}[\frac{1}{p}])\longrightarrow \repnpr(G_A,\widehat{\overline{A}}[\frac{1}{p}]),\  V\mapsto \widehat{\overline{A}}\otimes_{\widetilde{A}^{(\underline{N})}_{\infty,\underline{m}}} V
	\end{align}
	is fully faithful, where $\repnan{\Delta^{(\underline{N})}_{\underline{m}}}(\Xi^{(\underline{N})},\widetilde{A}^{(\underline{N})}_{\infty,\underline{m}}[1/p])$ is the full subcategory of $\Delta^{(\underline{N})}_{\underline{m}}$-analytic objects of $\repnpr(\Xi^{(\underline{N})},\widetilde{A}^{(\underline{N})}_{\infty,\underline{m}}[1/p])$. Moreover, any object of $\repnpr(G_A,\widehat{\overline{A}}[1/p])$ lies in the essential image of the above functor for some $\underline{N}\in J$ and $\underline{m}\in \bb{N}^d$ such that $N_1=\cdots=N_c=1$ and $m_1=\cdots =m_c=0$.
\end{mythm}

\begin{mythm}[{\cite[11.4]{he2022sen}}]\label{thm:sen-brinon-A}
	Let $A$ be an adequate $\ca{O}_K$-algebra with fraction field $\ca{K}$, $G_A=\gal(\ca{K}_{\mrm{ur}}/\ca{K})$. Then, for any object $W$ of $\repnpr(G_A,\widehat{\overline{A}}[1/p])$ {\rm(\ref{defn:repn})}, there is a canonical homomorphism of $\widehat{\overline{A}}[1/p]$-linear Lie algebras, called the \emph{Sen action of the $\widehat{\overline{A}}[1/p]$-representation $W$ of $G_A$},
	\begin{align}\label{eq:sen-brinon-A}
		\varphi_{\sen}|_W:\scr{E}_A^*(1)\longrightarrow \mrm{End}_{\widehat{\overline{A}}[\frac{1}{p}]}(W),
	\end{align}
	which is functorial in $W$ and $G_A$-equivariant with respect to the canonical action on $\scr{E}_A^*(1)$ {\rm(\ref{rem:A-fal-ext}.(\ref{item:rem:A-fal-ext-3}))} and the adjoint action on $\mrm{End}_{\widehat{\overline{A}}[1/p]}(W)$ (i.e. $g\in G$ sends an endomorphism $\phi$ to $g\circ \phi \circ g^{-1}$).
	
	Moreover, under the assumption in {\rm\ref{para:notation-A-galois}} and with the same notation, assume that there exists an object $V$ of $\repnan{\Delta^{(\underline{N})}_{\underline{m}}}(\Xi^{(\underline{N})},\widetilde{A}^{(\underline{N})}_{\infty,\underline{m}}[1/p])$ for some $\underline{N}\in J$ and $\underline{m}\in\bb{N}^d$ with $W= \widehat{\overline{A}}\otimes_{\widetilde{A}^{(\underline{N})}_{\infty,\underline{m}}}V$. Then, for any $f\in \scr{E}_A^*(1)=\ho_{\widehat{\overline{A}}[1/p]}(\scr{E}_A(-1),\widehat{\overline{A}}[1/p])$,
	\begin{align}\label{eq:thm:sen-brinon-A}
		\varphi_{\sen}|_W(f)=\sum_{i=0}^e f(T_i)\otimes\varphi_{\partial_i}|_V,
	\end{align}
	where $\varphi_{\partial_i}|_V\in\mrm{End}_{\widetilde{A}^{(\underline{N})}_{\infty,\underline{m}}[1/p]}(V)$ is the infinitesimal action of $\partial_i\in\lie(\Xi^{(\underline{N})})$ on $V$ defined in \eqref{eq:para:notation-A-galois-4}.
\end{mythm}
\begin{myrem}\label{rem:sen-brinon-A}
	We keep the notation in \ref{thm:sen-brinon-A}.
	\begin{enumerate}
		\renewcommand{\labelenumi}{{\rm(\theenumi)}}
		\item Let $K'$ be a finite field extension of $K$ contained in $\overline{K}$, $A'=A^{K'}$, $W'=\widehat{\overline{A'}}\otimes_{\widehat{\overline{A}}}W$ the associated object of $\repnpr(G_{A'},\widehat{\overline{A'}}[1/p])$. Then, there is a canonical isomorphism of Sen actions of $W$ by \ref{rem:A-fal-ext}.(\ref{item:rem:A-fal-ext-1}) and \cite[11.7]{he2022sen},
		\begin{align}\label{eq:rem:sen-brinon-A-1}
			\xymatrix{
				\scr{E}^*_{A'}(1)\ar[rr]^-{\varphi_{\sen}|_{W'}}\ar[d]^-{\wr}&&\mrm{End}_{\widehat{\overline{A'}}[\frac{1}{p}]}(W')
				\\
				\scr{E}^*_A(1)
				\ar[rr]^-{\varphi_{\sen}|_W}&&\mrm{End}_{\widehat{\overline{A}}[\frac{1}{p}]}(W)\ar[u]^-{\wr}
			}
		\end{align}
		where we used the fact that $\overline{A'}=\overline{A}$ and $K\to K'$ is \'etale.
		\label{item:rem:sen-brinon-A-1}
		\item Let $K'$ be a complete discrete valuation field extension of $K$ with perfect residue field and we fix an extension of their algebraic closures $\overline{K}\to \overline{K'}$. Let $A_1$ (resp. $A_2$) be an adequate $\ca{O}_{K_1}$-algebra (resp. adequate $\ca{O}_{K_2}$-algebra) with fraction field $\ca{K}_1$ (resp. $\ca{K}_2$) for some finite field extension $K_1$ (resp. $K_2$) of $K$ (resp. $K'$) contained in $\overline{K}$ (resp. $\overline{K'}$). Then, for any commutative diagram of $(K,\ca{O}_K,\ca{O}_{\overline{K}})$-triples
		\begin{align}\label{eq:rem:sen-brinon-A-2}
			\xymatrix{
				(A_{1,\triv},A_1,\overline{A_1})\ar[r]& (A_{2,\triv},A_2,\overline{A_2})\\
				(K,\ca{O}_K,\ca{O}_{\overline{K}})\ar[u]\ar[r]&(K',\ca{O}_{K'},\ca{O}_{\overline{K'}})\ar[u]
			}
		\end{align}
		with $\overline{A_1}\to \overline{A_2}$ and $\ca{K}_2\otimes_{\ca{K}_1}\Omega^1_{\ca{K}_1/K}\to \Omega^1_{\ca{K}_2/K'}$ both injective, we consider a finite field extension $K_1'$ of $K_1$ contained in $\overline{K}$ and a finite field extension $K_2'$ of $K_1'K_2$ contained in $\overline{K'}$, and we put $A_i'=A_i^{K_i'}$. Then, \eqref{eq:rem:sen-brinon-A-2} induces a commutative diagram of $(K,\ca{O}_K,\ca{O}_{\overline{K}})$-triples (see \ref{rem:A-fal-ext}.(\ref{item:rem:A-fal-ext-2}))
		\begin{align}\label{eq:rem:sen-brinon-A-4}
			\xymatrix{
				(A_{1,\triv},A_1,\overline{A_1})\ar[r]&(A'_{1,\triv},A'_1,\overline{A'_1})\ar[r]& (A'_{2,\triv},A'_2,\overline{A'_2})& (A_{2,\triv},A_2,\overline{A_2})\ar[l]\\
				(K_1,\ca{O}_{K_1},\ca{O}_{\overline{K}})\ar[u]\ar[r]&(K_1',\ca{O}_{K_1'},\ca{O}_{\overline{K}})\ar[u]\ar[r]&(K_2',\ca{O}_{K_2'},\ca{O}_{\overline{K'}})\ar[u]&(K_2,\ca{O}_{K_2},\ca{O}_{\overline{K'}})\ar[u]\ar[l]
			}
		\end{align}
		where we used the fact that $\overline{A_i'}=\overline{A_i}$ for $i=1,2$. Thus, for any object $W_1$ of $\repnpr(G_{A_1},\widehat{\overline{A_1}}[1/p])$ with $W_2=\widehat{\overline{A_2}}\otimes_{\widehat{\overline{A_1}}}W_1$ the associated object of $\repnpr(G_{A_2},\widehat{\overline{A_2}}[1/p])$, there is a canonical commutative diagram by (\ref{item:rem:sen-brinon-A-1}), \ref{rem:A-fal-ext}.(\ref{item:rem:A-fal-ext-2}) and \cite[11.7]{he2022sen},
		\begin{align}\label{eq:rem:sen-brinon-A-5}
			\xymatrix{
				\scr{E}^*_{A_2}(1)\ar[rr]^-{\varphi_{\sen}|_{W_2}}\ar[d]&&\mrm{End}_{\widehat{\overline{A_2}}[\frac{1}{p}]}(W_2)\\
				\widehat{\overline{A_2}}\otimes_{\widehat{\overline{A_1}}}\scr{E}^*_{A_1}(1)
				\ar[rr]^-{\id_{\widehat{\overline{A_2}}}\otimes\varphi_{\sen}|_{W_1}}&&\widehat{\overline{A_2}}\otimes_{\widehat{\overline{A_1}}}\mrm{End}_{\widehat{\overline{A_1}}[\frac{1}{p}]}(W_1)
				\ar[u]_-{\wr}
			}
		\end{align}
		independent of the choices of $K_1'$ and $K_2'$.\label{item:rem:sen-brinon-A-2}
	\end{enumerate}
\end{myrem}

\begin{mythm}[{\cite[11.21]{he2022sen}}]\label{thm:sen-lie-lift-A}
	Let $A$ be an adequate $\ca{O}_K$-algebra with fraction field $\ca{K}$, $\ca{G}$ a quotient of $G_A=\gal(\ca{K}_{\mrm{ur}}/\ca{K})$ which is a (compact) $p$-adic analytic group (i.e., a closed subgroup of $\mrm{GL}_n(\bb{Z}_p)$ for some $n\in\bb{N}$, see \cite[Theorem 7.19]{ddms1999lie} and also {\rm\cite[3.9]{he2022sen}}), $\ca{L}$ the corresponding Galois extension of $\ca{K}$ (i.e., $\gal(\ca{L}/\ca{K})=\ca{G}$). Then, there is a canonical homomorphism of $\widehat{\overline{A}}[1/p]$-linear Lie algebras, called the \emph{universal Sen action of the $p$-adic analytic group quotient $\ca{G}$ of $G_A$} (or \emph{of the $p$-adic analytic Galois extension $\ca{L}$ of $\ca{K}$}),
	\begin{align}\label{eq:thm:sen-lie-lift-A-1}
		\varphi_{\sen}|_{\ca{G}}:\scr{E}^*_A(1)\longrightarrow \widehat{\overline{A}}[1/p]\otimes_{\bb{Q}_p}\lie(\ca{G}),
	\end{align}
	which is $G_A$-equivariant with respect to the canonical action on $\scr{E}_A^*(1)$ {\rm(\ref{rem:A-fal-ext}.(\ref{item:rem:A-fal-ext-3}))}, the canonical action on $\widehat{\overline{A}}[1/p]$ and the adjoint action on $\lie(\ca{G})$ {\rm(\cite[3.15]{he2022sen})}. Moreover, it makes the following diagram commutative for any object $V$ of $\repnpr(\ca{G},\bb{Q}_p)$,
	\begin{align}\label{eq:thm:sen-lie-lift-A-2}
		\xymatrix{
			\scr{E}^*_A(1)\ar[r]^-{\varphi_{\sen}|_{\ca{G}}}\ar[d]_-{\varphi_{\sen}|_W}&\widehat{\overline{A}}[\frac{1}{p}]\otimes_{\bb{Q}_p}\lie(\ca{G})\ar[d]^-{\id_{\widehat{\overline{A}}[\frac{1}{p}]}\otimes\varphi|_V}\\
			\mrm{End}_{\widehat{\overline{A}}[\frac{1}{p}]}(W)&\widehat{\overline{A}}[\frac{1}{p}]\otimes_{\bb{Q}_p}\mrm{End}_{\bb{Q}_p}(V)\ar[l]_-{\sim}
		}
	\end{align}
	where $W=\widehat{\overline{A}}[1/p]\otimes_{\bb{Q}_p}V$ is the associated object of $\repnpr(G_A,\widehat{\overline{A}}[1/p])$, $\varphi_{\sen}|_W$ is the Sen action of $W$ {\rm(\ref{thm:sen-brinon-A})}, and $\varphi|_V$ is the infinitesimal Lie algebra action of $\lie(\ca{G})$ on $V$ {\rm(\cite[4.13]{he2022sen})}. 
\end{mythm}

\begin{myrem}\label{rem:sen-lie-lift-A}
	We keep the notation in \ref{thm:sen-lie-lift-A}.
	\begin{enumerate}
		\renewcommand{\labelenumi}{{\rm(\theenumi)}}
		\item Let $K'$ be a finite field extension of $K$ contained in $\overline{K}$, $A'=A^{K'}$, $\ca{G}'$ the image of the composition of $G_{A'}\to G_A\to \ca{G}$ (i.e., $\ca{L}'=K'\ca{L}$). Then, there is a canonical isomorphism of universal Sen actions of $\ca{G}$ and $\ca{G}'$ by \ref{rem:A-fal-ext}.(\ref{item:rem:A-fal-ext-1}) and \cite[11.23.(4)]{he2022sen},
		\begin{align}\label{eq:rem:sen-lie-lift-A-1}
			\xymatrix{
				\scr{E}^*_{A'}(1)\ar[rr]^-{\varphi_{\sen}|_{\ca{G}'}}\ar[d]^-{\wr}&&\widehat{\overline{A'}}[\frac{1}{p}]\otimes_{\bb{Q}_p}\lie(\ca{G}')\ar[d]^-{\wr}\\
				\scr{E}^*_A(1)
				\ar[rr]^-{\varphi_{\sen}|_{\ca{G}}}&&\widehat{\overline{A}}[\frac{1}{p}]\otimes_{\bb{Q}_p}\lie(\ca{G})
			}
		\end{align}
		where we used the fact that $\overline{A'}=\overline{A}$ and $K\to K'$ is finite.
		\label{item:rem:sen-lie-lift-A-1}
		\item Let $K'$ be a complete discrete valuation field extension of $K$ with perfect residue field and we fix an extension of their algebraic closures $\overline{K}\to \overline{K'}$. Let $A_1$ (resp. $A_2$) be an adequate $\ca{O}_{K_1}$-algebra (resp. adequate $\ca{O}_{K_2}$-algebra) with fraction field $\ca{K}_1$ (resp. $\ca{K}_2$) for some finite field extension $K_1$ (resp. $K_2$) of $K$ (resp. $K'$) contained in $\overline{K}$ (resp. $\overline{K'}$). Then, for any commutative diagram of $(K,\ca{O}_K,\ca{O}_{\overline{K}})$-triples
		\begin{align}\label{eq:rem:sen-lie-lift-A-2}
			\xymatrix{
				(A_{1,\triv},A_1,\overline{A_1})\ar[r]& (A_{2,\triv},A_2,\overline{A_2})\\
				(K,\ca{O}_K,\ca{O}_{\overline{K}})\ar[u]\ar[r]&(K',\ca{O}_{K'},\ca{O}_{\overline{K'}})\ar[u]
			}
		\end{align}
		with $\overline{A_1}\to \overline{A_2}$ and $\ca{K}_2\otimes_{\ca{K}_1}\Omega^1_{\ca{K}_1/K}\to \Omega^1_{\ca{K}_2/K'}$ both injective, we consider a finite field extension $K_1'$ of $K_1$ contained in $\overline{K}$ and a finite field extension $K_2'$ of $K_1'K_2$ contained in $\overline{K'}$, and we put $A_i'=A_i^{K_i'}$. Then, \eqref{eq:rem:sen-lie-lift-A-2} induces a commutative diagram of $(K,\ca{O}_K,\ca{O}_{\overline{K}})$-triples (see \ref{rem:A-fal-ext}.(\ref{item:rem:A-fal-ext-2}))
		\begin{align}\label{eq:rem:sen-lie-lift-A-4}
			\xymatrix{
				(A_{1,\triv},A_1,\overline{A_1})\ar[r]&(A'_{1,\triv},A'_1,\overline{A'_1})\ar[r]& (A'_{2,\triv},A'_2,\overline{A'_2})& (A_{2,\triv},A_2,\overline{A_2})\ar[l]\\
				(K_1,\ca{O}_{K_1},\ca{O}_{\overline{K}})\ar[u]\ar[r]&(K_1',\ca{O}_{K_1'},\ca{O}_{\overline{K}})\ar[u]\ar[r]&(K_2',\ca{O}_{K_2'},\ca{O}_{\overline{K'}})\ar[u]&(K_2,\ca{O}_{K_2},\ca{O}_{\overline{K'}})\ar[u]\ar[l]
			}
		\end{align}
		where we used the fact that $\overline{A_i'}=\overline{A_i}$ for $i=1,2$. Thus, for any $p$-adic analytic Galois extension $\ca{L}_1$ of $\ca{K}_1$ contained in $\ca{K}_{1,\mrm{ur}}$ and any $p$-adic analytic Galois extension $\ca{L}_2$ of $\ca{K}_2$ containing $\ca{L}_1$  contained in $\ca{K}_{2,\mrm{ur}}$, 
		\begin{align}
			\xymatrix{
				\ca{L}_1\ar[r]&K_1'\ca{L}_1\ar[r]&K_2'\ca{L}_2&\ca{L}_2\ar[l]\\
				\ca{K}_1\ar[u]^-{\ca{G}_1}\ar[r]&K_1'\ca{K}_1\ar[r]\ar[u]^-{\ca{G}_1'}&K_2'\ca{K}_2\ar[u]_-{\ca{G}_2'}&\ca{K}_2\ar[l]\ar[u]_-{\ca{G}_2}\\
				K_1\ar[u]\ar[r]&K_1'\ar[r]\ar[u]&K_2'\ar[u]&K_2\ar[l]\ar[u]
			}
		\end{align}
		there is a canonical commutative diagram by (\ref{item:rem:sen-lie-lift-A-1}), \ref{rem:A-fal-ext}.(\ref{item:rem:A-fal-ext-2}) and \cite[11.23, (4), (6)]{he2022sen},
		\begin{align}\label{eq:rem:sen-lie-lift-A-5}
			\xymatrix{
				\scr{E}^*_{A_2}(1)\ar[rr]^-{\varphi_{\sen}|_{\ca{G}_2}}\ar[d]&&\widehat{\overline{A_2}}[\frac{1}{p}]\otimes_{\bb{Q}_p}\lie(\ca{G}_2)\ar[d]\\
				\widehat{\overline{A_2}}\otimes_{\widehat{\overline{A_1}}}\scr{E}^*_{A_1}(1)
				\ar[rr]^-{\id_{\widehat{\overline{A_2}}}\otimes\varphi_{\sen}|_{\ca{G}_1}}&&\widehat{\overline{A_2}}[\frac{1}{p}]\otimes_{\bb{Q}_p}\lie(\ca{G}_1)
			}
		\end{align}	
		independent of the choices of $K_1'$ and $K_2'$, where $\ca{G}_i=\gal(\ca{L}_i/\ca{K}_i)$ ($i=1,2$). \label{item:rem:sen-lie-lift-A-2}
	\end{enumerate}
\end{myrem}

\begin{myprop}\label{prop:geom-sen-nonzero}
	With the notation in {\rm\ref{thm:sen-lie-lift-A}}, let $L$ be a Galois extension of $K$ contained in $\ca{L}\cap\overline{K}$, $\ca{G}^{\mrm{geo}}=\gal(\ca{L}/L\ca{K})$ and $\ca{G}^{\mrm{ari}}=\gal(L\ca{K}/\ca{K})$. Then, $\varphi_{\sen}|_{\ca{G}}:\scr{E}^*_A(1)\to \widehat{\overline{A}}[1/p]\otimes_{\bb{Q}_p}\lie(\ca{G})$ induces a morphism of exact sequences of $\widehat{\overline{A}}[1/p]$-linear Lie algebras,
	\begin{align}\label{eq:prop:geom-sen-nonzero-1}
		\xymatrix{
			0\ar[r]& \ho_A(\Omega^1_{(X_K,\scr{M}_{X_K})/K}(-1),\widehat{\overline{A}})\ar[r]^-{\jmath^*}\ar[d]^-{\varphi^{\mrm{geo}}_{\sen}|_{\ca{G}}}& \scr{E}^*_A(1)\ar[r]^-{\iota^*}\ar[d]^-{\varphi_{\sen}|_{\ca{G}}}&\widehat{\overline{A}}[\frac{1}{p}]\ar[r]\ar[d]^-{\varphi^{\mrm{ari}}_{\sen}|_{\ca{G}}}& 0\\
			0\ar[r]& \widehat{\overline{A}}[1/p]\otimes_{\bb{Q}_p}\lie(\ca{G}^{\mrm{geo}})\ar[r]& \widehat{\overline{A}}[1/p]\otimes_{\bb{Q}_p}\lie(\ca{G})\ar[r]&\widehat{\overline{A}}[1/p]\otimes_{\bb{Q}_p}\lie(\ca{G}^{\mrm{ari}})\ar[r]& 0
		}
	\end{align}
	where the exact sequence on the top is \eqref{eq:A-fal-ext-dual}.
	Moreover, the induced homomorphism $\varphi^{\mrm{ari}}_{\sen}|_{\ca{G}}$ is not zero if and only if the inertia subgroup of $L/K$ is infinite, and in this case $\varphi^{\mrm{ari}}_{\sen}|_{\ca{G}}$ admits an $\widehat{\overline{A}}[1/p]$-linear retraction.
\end{myprop}
\begin{proof}
	Firstly, we note that $\ca{G}^{\mrm{geo}}$ and $\ca{G}^{\mrm{ari}}$ are still $p$-adic analytic groups (\cite[Theorems 9.6, 9.7]{ddms1999lie}) and we have a canonical exact sequence of $\bb{Q}_p$-linear Lie algebras $0\to \lie(\ca{G}^{\mrm{geo}})\to \lie(\ca{G})\to \lie(\ca{G}^{\mrm{ari}})\to 0$ (\cite[3.14]{he2022sen}). As the fraction field $\ca{K}$ of $A$ is a finitely generated field extension of $K$, $\ca{K}\cap L$ is finite over $K$  (\cite[\href{https://stacks.math.columbia.edu/tag/037J}{037J}]{stacks-project}). In particular, $\gal(L\ca{K}/\ca{K})$ is an open subgroup of $\gal(L/K)$. Consider the morphism of adequate $(K,\ca{O}_K,\ca{O}_{\overline{K}})$-triples $(K,\ca{O}_K,\ca{O}_{\overline{K}})\to (A_{\triv},A,\overline{A})$. By \ref{rem:sen-lie-lift-A}.(\ref{item:rem:sen-lie-lift-A-2}), there is a canonical commutative diagram 
	\begin{align}
		\xymatrix{
			\scr{E}^*_A(1)\ar[rrr]^-{\varphi_{\sen}|_{\gal(L\ca{K}/\ca{K})}}\ar[d]&&&\widehat{\overline{A}}[\frac{1}{p}]\otimes_{\bb{Q}_p}\lie(\gal(L\ca{K}/\ca{K}))\ar[d]^-{\wr}\\
			\widehat{\overline{A}}\otimes_{\ca{O}_{\widehat{\overline{K}}}}\scr{E}^*_{\ca{O}_K}(1)
			\ar[rrr]^-{\id_{\widehat{\overline{A}}}\otimes\varphi_{\sen}|_{\gal(L/K)}}&&&\widehat{\overline{A}}[\frac{1}{p}]\otimes_{\bb{Q}_p}\lie(\gal(L/K))
		}
	\end{align}
	where the right vertical arrow is an isomorphism.
	We note that the left vertical arrow is identified with $\iota^*:\scr{E}^*_A(1)\to \widehat{\overline{A}}[1/p]$ by the functoriality of Faltings extensions (\ref{rem:A-fal-ext}.(\ref{item:rem:A-fal-ext-2})) and by the fact that $\scr{E}_{\ca{O}_K}=\widehat{\overline{K}}(1)$ (as $\Omega^1_{K/K}=0$). 
	
	On the other hand, as $\ca{G}^{\mrm{ari}}=\gal(L\ca{K}/\ca{K})$ is a quotient of $\ca{G}=\gal(\ca{L}/\ca{K})$, there is a canonical commutative diagram
	\begin{align}
		\xymatrix{
			\scr{E}^*_A(1)\ar[rr]^-{\varphi_{\sen}|_{\ca{G}}}\ar@{=}[d]&&\widehat{\overline{A}}[\frac{1}{p}]\otimes_{\bb{Q}_p}\lie(\ca{G})\ar@{->>}[d]\\
			\scr{E}^*_A(1)\ar[rr]^-{\varphi_{\sen}|_{\ca{G}^{\mrm{ari}}}}&&\widehat{\overline{A}}[\frac{1}{p}]\otimes_{\bb{Q}_p}\lie(\ca{G}^{\mrm{ari}})
		}
	\end{align}
	which (together with the discussion above) implies that the right square of \eqref{eq:prop:geom-sen-nonzero-1} is commutative, where we put $\varphi^{\mrm{ari}}_{\sen}|_{\ca{G}}=\id_{\widehat{\overline{A}}}\otimes\varphi_{\sen}|_{\gal(L/K)}$. Thus, we obtain a homomorphism $\varphi^{\mrm{geo}}_{\sen}|_{\ca{G}}$ between the kernels (and thus a commutative diagram \eqref{eq:prop:geom-sen-nonzero-1}).
	
	Notice that $\varphi^{\mrm{ari}}_{\sen}|_{\ca{G}}$ is nonzero if and only if $\varphi_{\sen}|_{\gal(L/K)}$ is nonzero (as $\widehat{\overline{K}}\to \widehat{\overline{A}}[1/p]$ is faithfully flat). The latter is equivalent to the infiniteness of the inertia subgroup of $L/K$ by \cite[Theorem 11 and its Corollary]{sen1980sen} (see also \cite[11.23.(2)]{he2022sen}). In this case, $\varphi_{\sen}|_{\gal(L/K)}$ is injective and thus admits a $\widehat{\overline{K}}$-linear retraction. Then, $\varphi^{\mrm{ari}}_{\sen}|_{\ca{G}}$ admits an $\widehat{\overline{A}}[1/p]$-linear retraction.
\end{proof}

\begin{myrem}\label{rem:geom-sen-nonzero}
	In \ref{prop:geom-sen-nonzero}, if $K$ is a finite extension of $\bb{Q}_p$, then the inertia subgroup of $L/K$ is infinite if and only if $L$ is a pre-perfectoid field by Sen's theorem on ramifications of $p$-adic analytic Galois extension of local fields, see \cite[2.13]{coatesgreenberg1996kummer}.
\end{myrem}

\section{Induced Faltings Extension over Geometric Valuative Points}\label{sec:fal-ext}
Firstly, we deduce from our variant of Temkin's admissibly \'etale local uniformization theorem \ref{thm:uniformization} that adequate algebras form a neighborhood basis of any geometric valuative point of a smooth variety in the admissibly \'etale topology (see \ref{prop:aet-neighborhood}). Then, using the functorial construction \ref{thm:A-fal-ext} of Faltings extensions over adequate algebras, we define Faltings extensions at any geometric valuative point (see \ref{thm:fal-ext-val-pt}). Finally, we deduce a criterion for pointwise perfectoidness via the Galois action on these Faltings extensions from the criterion established in Section \ref{sec:trace} (see \ref{thm:perfd-val-log-1} and \ref{thm:perfd-val-log-2}).

\begin{mydefn}\label{defn:neighborhood}
	Let $K$ be a valuation field, $Y$ a $K$-scheme of finite presentation. A \emph{geometric valuative point} of $Y$ over $\ca{O}_K$ is a morphism of $K$-schemes $ \spec(\overline{\ca{F}})\to Y$, where $\overline{\ca{F}}$ is an algebraically closed valuation field extension of $K$. A \emph{neighborhood} of $Y$ over $\ca{O}_K$ at the geometric valuative point $\spec(\overline{\ca{F}})$ is a commutative diagram of schemes
	\begin{align}
		\xymatrix{
			\spec(\overline{\ca{F}})\ar[r]\ar[d]&\spec(\ca{O}_{\overline{\ca{F}}})\ar[d]\\
			X_K\ar[r]\ar[d]&X\ar[dd]\\
			Y\ar[d]&\\
			\spec(K)\ar[r]&\spec(\ca{O}_K).
		}
	\end{align}
	If moreover $X$ is of finite presentation over $\ca{O}_K$ (resp. and $X_K$ is \'etale over $Y$), then we call $X$ a \emph{finitely presented} (resp. \emph{admissibly \'etale}) neighborhood. We denote by $\nbd^{\mrm{fp}}_{\overline{\ca{F}}}(Y/\ca{O}_K)$ (resp. $\nbd^{\aetale}_{\overline{\ca{F}}}(Y/\ca{O}_K)$) the category of all finitely presented (resp. admissibly \'etale) neighborhoods of $Y$ over $\ca{O}_K$ at $\spec(\overline{\ca{F}})$. 
\end{mydefn}

\begin{mypara}
	One can check that $\nbd^{\mrm{fp}}_{\overline{\ca{F}}}(Y/\ca{O}_K)$ and $\nbd^{\aetale}_{\overline{\ca{F}}}(Y/\ca{O}_K)$ are both cofiltered. For a finitely presented (resp. admissibly \'etale) neighborhood $X$ of $Y/\ca{O}_K$ at $\spec(\overline{\ca{F}})$, it is clear that the forgetful functor from the localization at $X$ to the category of finitely presented (resp. admissibly \'etale, see \ref{defn:adm-etale}) $X$-schemes 
	\begin{align}
		\nbd^{\mrm{fp}}_{\overline{\ca{F}}}(Y/\ca{O}_K)_{/X}&\longrightarrow \sch^{\mrm{fp}}_{/X}\\
		(\trm{resp. }\nbd^{\aetale}_{\overline{\ca{F}}}(Y/\ca{O}_K)_{/X}&\longrightarrow \sch^{\aetale}_{/X})\nonumber
	\end{align}
	is faithful and stable under taking finite limits of $X$-schemes.
\end{mypara}

\begin{mylem}\label{lem:aet-et-neighborhood}
	Let $K$ be a valuation field, $Y$ a $K$-scheme of finite presentation, $ \spec(\overline{\ca{F}})\to Y$ a geometric valuative point of $Y$ over $\ca{O}_K$. Let $\nbd^{\et}_{\overline{\ca{F}}}(Y/K)$ be the category of \'etale neighborhoods of $Y$ at the geometric point $\spec(\overline{\ca{F}})$. Then, the canonical functor
	\begin{align}
		\nbd^{\aetale}_{\overline{\ca{F}}}(Y/\ca{O}_K)\longrightarrow \nbd^{\et}_{\overline{\ca{F}}}(Y/K),\ X\mapsto X_K
	\end{align}
	is cofinal.
\end{mylem}
\begin{proof}
	For any \'etale neighborhood $U$ of $Y$ at $\spec(\overline{\ca{F}})$, after shrinking, we may assume that $U=\spec(B)$ is affine. Let $\{x_1,\dots,x_n\}$ be a finite family of generators of the $K$-algebra $B$. After multiplying $x_1,\dots,x_n$ by a power of $p$, we may assume that their images in $\overline{\ca{F}}$ are contained in $\ca{O}_{\overline{\ca{F}}}$. Thus, the $\ca{O}_K$-subalgebra $A$ of $B$ generated by $x_1,\dots,x_n$ gives an admissibly \'etale neighborhood $X=\spec(A)$ of $Y$ at $\spec(\overline{\ca{F}})$ with $X_K=U$.
\end{proof}

\begin{myrem}\label{rem:aet-et-neighborhood}
	For an admissibly \'etale neighborhood $X$ of $Y$, we don't know whether the functor $\sch^{\aetale}_{/X}\to X_{K,\et}$ is cocontinuous or not (although it ``covers" all the geometric points by \ref{lem:aet-et-neighborhood}).
\end{myrem}

\begin{myprop}\label{prop:aet-local}
	Let $K$ be a valuation field, $Y$ a $K$-scheme of finite presentation, $ \spec(\overline{\ca{F}})\to Y$ a geometric valuative point of $Y$ over $\ca{O}_K$. Then,
	\begin{align}\label{eq:prop:aet-local-1}
		\lim_{\nbd^{\aetale}_{\overline{\ca{F}}}(Y/\ca{O}_K)}X=\spec(R)
	\end{align}
	is the spectrum of a ring $R$ satisfying the following properties: 
	\begin{enumerate}
		\renewcommand{\labelenumi}{{\rm(\theenumi)}}
		\item The $K$-algebra $K\otimes_{\ca{O}_K}R$ is canonically identified with the strict Henselization of $Y$ at $\spec(\overline{\ca{F}})$.\label{item:prop:aet-local-1}
		\item The $\ca{O}_K$-algebra $R$ is local and integrally closed in $K\otimes_{\ca{O}_K}R$.\label{item:prop:aet-local-2}
		\item Let $\overline{y}=\spec(\kappa(\overline{y}))$ be the image of $\spec(\overline{\ca{F}})\to \spec(K\otimes_{\ca{O}_K}R)$ and we put $\ca{O}_{\kappa(\overline{y})}=\kappa(\overline{y})\cap\ca{O}_{\overline{\ca{F}}}$. Then, the canonical commutative diagram of local rings
		\begin{align}\label{eq:prop:aet-local-2}
			\xymatrix{
				\kappa(\overline{y})&\ca{O}_{\kappa(\overline{y})}\ar[l]\\
				K\otimes_{\ca{O}_K}R\ar@{->>}[u]&R\ar[l]\ar@{->>}[u]
			}
		\end{align}
		is Cartesian and co-Cartesian, and it induces an isomorphism
		\begin{align}\label{eq:prop:aet-local-3}
			R/\pi R\iso \ca{O}_{\kappa(\overline{y})}/\pi\ca{O}_{\kappa(\overline{y})}
		\end{align}  
		for any non-zero element $\pi$ of $\ak{m}_K$.\label{item:prop:aet-local-3}
	\end{enumerate}
	We call $R$ the \emph{strict localization of $Y$ over $\ca{O}_K$ at the geometric valuative point $\spec(\overline{\ca{F}})$}.
\end{myprop}
\begin{proof}
	Notice that the canonical morphism $\spec(\ca{O}_{\overline{\ca{F}}})\to X$ factors through any $X_K$-modification $X'$ of $X$ (\cite[9.13]{he2024purity}). Note that $X'$ is quasi-separated, flat and of finite type over $\ca{O}_K$ and thus of finite presentation by \cite[\Luoma{1}.3.4.7]{raynaudgruson1971plat} (cf. \ref{rem:blowup}). Hence, $X'$ forms an admissibly \'etale neighborhood of $Y$ refining $X$. Consider the Riemann-Zariski space $\rz_{X_K}(X)=\lim X'$, where $X'$ runs through all the $X_K$-modifications of $X$ (see \cite[9.13]{he2024purity}). Let $x\in \rz_{X_K}(X)$ denote the image of the closed point under the canonical morphism of locally ringed spaces $\spec(\ca{O}_{\overline{\ca{F}}})\to \rz_{X_K}(X)$. Then, we have
	\begin{align}\label{eq:prop:aet-local-4}
		\lim_{\nbd^{\aetale}_{\overline{\ca{F}}}(Y/\ca{O}_K)}X=\lim_{\nbd^{\aetale}_{\overline{\ca{F}}}(Y/\ca{O}_K)}\spec(\ca{O}_{\rz_{X_K}(X),x})
	\end{align}
	and thus equal to the spectrum of 
	\begin{align}\label{eq:prop:aet-local-5}
		R=\colim_{\nbd^{\aetale}_{\overline{\ca{F}}}(Y/\ca{O}_K)}\ca{O}_{\rz_{X_K}(X),x},
	\end{align}
	which proves \eqref{eq:prop:aet-local-1}. 
	
	(\ref{item:prop:aet-local-1}) It follows from the following identifications by \ref{lem:aet-et-neighborhood}:
	\begin{align}\label{eq:prop:aet-local-6}
		\spec(K\otimes_{\ca{O}_K}R)=\lim_{X\in \nbd^{\aetale}_{\overline{\ca{F}}}(Y/\ca{O}_K)}X_K=\lim_{U\in\nbd^{\et}_{\overline{\ca{F}}}(Y/K)}U.
	\end{align}
	
	(\ref{item:prop:aet-local-2}) Firstly, $R$ is local since the transition homomorphisms in the directed system of \eqref{eq:prop:aet-local-5} are local. Moreover, since $\rz_{X_K}(X)$ is also the inverse limit of the integral closures of all the $X_K$-modifications $X'$ in $X'_K=X_K$, we see that $\ca{O}_{\rz_{X_K}(X),x}$ is integrally closed in $K\otimes_{\ca{O}_K}\ca{O}_{\rz_{X_K}(X),x}$ (\cite[3.18]{he2024coh}). Thus, $R$ shares the same properties by taking filtered colimits.

	(\ref{item:prop:aet-local-3}) Since for each object $X$ of $\nbd^{\aetale}_{\overline{\ca{F}}}(Y/\ca{O}_K)$ and $y_X\in X_K$ the image of $\spec(\overline{\ca{F}})$, the canonical commutative diagram of local rings
	\begin{align}
		\xymatrix{
			\kappa(y_X)&\ca{O}_{\kappa(y_X)}\ar[l]\\
			K\otimes_{\ca{O}_K}\ca{O}_{\rz_{X_K}(X),x}\ar@{->>}[u]&\ca{O}_{\rz_{X_K}(X),x}\ar[l]\ar@{->>}[u]
		}
	\end{align}
	is Cartesian and co-Cartesian, and it induces an isomorphism 
	\begin{align}
		\ca{O}_{\rz_{X_K}(X),x}/\pi\ca{O}_{\rz_{X_K}(X),x}\iso \ca{O}_{\kappa(y_X)}/\pi\ca{O}_{\kappa(y_X)},
	\end{align}
	where $\ca{O}_{\kappa(y_X)}=\kappa(y_X)\cap \ca{O}_{\overline{\ca{F}}}$ and $\pi\in\ak{m}_K$ (\cite[9.14, 9.15]{he2024purity}). Thus, $R$ shares the same properties by taking filtered colimits. 
\end{proof}

\begin{myprop}\label{prop:aet-neighborhood}
	Let $K$ be a complete discrete valuation field extension of $\bb{Q}_p$ with perfect residue field, $Y$ a smooth $K$-scheme of finite presentation, $D$ a normal crossings divisor on $Y$, $Y^{\triv}=Y\setminus D$, $\spec(\overline{\ca{F}})\to Y$ a geometric valuative point of $Y$ over $\ca{O}_K$. Then, the admissibly \'etale neighborhoods $X$ of $Y/\ca{O}_K$ at $\spec(\overline{\ca{F}})$ satisfying the following property form a cofinal full subcategory $\nbd^{Y^{\triv}\trm{-}\mrm{adq}}_{\overline{\ca{F}}}(Y/\ca{O}_K)$ of $\nbd^{\aetale}_{\overline{\ca{F}}}(Y/\ca{O}_K)$: 
	\begin{enumerate}
		\renewcommand{\labelenumi}{{\rm(\theenumi)}}
		\item[$(\star)$] There exists a finite field extension $K'$ of $K$ contained in the algebraic closure $\overline{K}$ of $K$ in $\overline{\ca{F}}$ and an adequate $(K',\ca{O}_{K'},\ca{O}_{\overline{K}})$-triple $(A_{\triv},A,\overline{A})$ {\rm(\ref{defn:essential-adequate-alg})} such that there is an isomorphism of open immersions of schemes over $\spec(K)\to \spec(\ca{O}_K)$,
		\begin{align}
			(\spec(A_{\triv})\to \spec(A))\iso (X^{\triv}\to X),
		\end{align}
		where $X^{\triv}=Y^{\triv}\times_YX_K$.
	\end{enumerate}
	
\end{myprop}
\begin{proof}
	After shrinking $X$ by \ref{lem:aet-et-neighborhood}, we may assume that $X^{\triv}\to X_K$ is the complement of a strict normal crossings divisor on $X_K$. Thus, there exists an admissibly \'etale covering in $\sch^{\aetale}_{/X}$ by \ref{thm:uniformization},
	\begin{align}
		X'\longrightarrow S'\times_SX,
	\end{align}
	where $S'=\spec(\ca{O}_{K'})$, $K'$ is a finite field extension of $K$, and $X'$ is a flat $S'$-scheme of finite presentation such that the log scheme $(X',\scr{M}_{X'})$ is a regular fs log scheme smooth and saturated over $(S',\scr{M}_{S'})$. By \eqref{eq:defn:adm-etale} (and the remark after it), the morphism $\spec(\ca{O}_{\overline{\ca{F}}})\to X$ factors through $X'$ so that $X'$ is an admissibly \'etale neighborhood of $Y$. By the structure theorem of log smooth morphisms due to Abbes-Gros and Tsuji (see \cite[8.11]{he2022sen}), after enlarging $K$ and replacing $X$ by an \'etale $X'$-scheme, we may assume that $(X,\scr{M}_X)$ admits an adequate chart over $(S,\scr{M}_S)$ in the sense of \cite[8.8]{he2022sen}. As the underlying scheme $X$ is normal of finite type over $\ca{O}_K$, after replacing $X$ by an affine open, we may assume that $(X^{\triv}\to X)=(\spec(A_{\triv})\to\spec(A))$ for an adequate $\ca{O}_K$-algebra $A$ (\ref{defn:essential-adequate-alg}).
\end{proof}

\begin{mypara}\label{para:notation-neighborhood-1}	
	In the rest of this section, we fix the following data: 
	\begin{enumerate}
		\renewcommand{\labelenumi}{{\rm(\theenumi)}}
		\item a complete discrete valuation field $K$ extension of $\bb{Q}_p$ with perfect residue field,
		\item an irreducible smooth $K$-scheme $Y$ of finite presentation with a normal crossings divisor $D$,
		\item an algebraic closure $\overline{\ca{K}}$ of the fraction field $\ca{K}$ of $Y$,
		\item a point $\overline{y}=\spec(\overline{\ca{F}})$ of the integral closure $Y^{\overline{\ca{K}}}$ of $Y$ in $\overline{\ca{K}}$,
		\item a valuation ring $\ca{O}_{\overline{\ca{F}}}$ of height $1$ extension of $\ca{O}_K$ with fraction field $\overline{\ca{F}}$.
	\end{enumerate}
	In particular, $\overline{\ca{F}}$ is algebraically closed (\cite[\href{https://stacks.math.columbia.edu/tag/0DCK}{0DCK}]{stacks-project}) and $\spec(\overline{\ca{F}})$ is a geometric valuative point of $Y$ over $\ca{O}_K$. Let $\overline{K}$ be the algebraic closure of $K$ in $\overline{\ca{F}}$ and let $y=\spec(\ca{F})\in Y$ be the image of $\overline{y}=\spec(\overline{\ca{F}})\in Y^{\overline{\ca{K}}}$. We see that $\overline{\ca{F}}$ is an algebraic closure of $\ca{F}$ and thus $\mrm{trdeg}_K(\overline{\ca{F}})=\mrm{trdeg}_K(\ca{F})<\infty$ (\cite[\href{https://stacks.math.columbia.edu/tag/00P0}{00P0}]{stacks-project}).
\end{mypara}

\begin{mypara}\label{para:notation-neighborhood-2}	
	Following \ref{para:notation-neighborhood-1}, for any algebraic field extension $\ca{L}$ of $\ca{K}$ contained in $\overline{\ca{K}}$, we denote by $Y^{\ca{L}}$ the integral closure of $Y$ in $\ca{L}$ and by $y_{\ca{L}}=\spec(\ca{F}_{\ca{L}})\in Y^{\ca{L}}$ the image of $\overline{y}=\spec(\overline{\ca{F}})\in Y^{\overline{\ca{K}}}$. Let $\ca{O}_{\ca{F}_{\ca{L}}^{\mrm{h}}}$ be the Henselization of $\ca{O}_{\ca{F}_{\ca{L}}}=\ca{F}_{\ca{L}}\cap \ca{O}_{\overline{\ca{F}}}$. It is still a valuation ring of height $1$ with fraction field $\ca{F}_{\ca{L}}^{\mrm{h}}$ algebraic over $\ca{F}_{\ca{L}}$ (\cite[6.1.12.(\luoma{6})]{gabber2003almost}). Then, there is a canonical extension $\ca{O}_{\ca{F}_{\ca{L}}^{\mrm{h}}}\to \ca{O}_{\overline{\ca{F}}}$ of valuation rings (\cite[\href{https://stacks.math.columbia.edu/tag/04GS}{04GS}]{stacks-project}). By \ref{para:notation-fal-ext-2}, we have
	\begin{align}\label{eq:para:notation-neighborhood-2-1}
		\mrm{Aut}_{\ca{O}_{\ca{F}_{\ca{L}}}}(\ca{O}_{\overline{\ca{F}}})=\{\sigma\in \gal(\overline{\ca{F}}/\ca{F}_{\ca{L}})\ |\  \sigma(\ca{O}_{\overline{\ca{F}}})=\ca{O}_{\overline{\ca{F}}}\}=\gal(\overline{\ca{F}}/\ca{F}_{\ca{L}}^{\mrm{h}}).
	\end{align}
	On the other hand, since the stalk $\ca{O}_{Y^{\overline{\ca{K}}},\overline{y}}$ of $Y^{\overline{\ca{K}}}$ at $\overline{y}$ is a normal local ring with algebraically closed fraction field $\overline{\ca{K}}$, it is strictly Henselian (\cite[\href{https://stacks.math.columbia.edu/tag/0BSQ}{0BSQ}]{stacks-project}). Thus, there are canonical local homomorphisms (\cite[\href{https://stacks.math.columbia.edu/tag/04GU}{04GU}]{stacks-project})
	\begin{align}\label{eq:para:notation-neighborhood-2-2}
		\ca{O}_{Y^{\overline{\ca{K}}},\overline{y}}\longleftarrow\ca{O}_{Y^{\ca{L}},y_{\ca{L}}}^{\mrm{sh}}\longleftarrow\ca{O}_{Y^{\ca{L}},y_{\ca{L}}}^{\mrm{h}}\longleftarrow \ca{O}_{Y^{\ca{L}},y_{\ca{L}}},
	\end{align}
	where $\ca{O}_{Y^{\ca{L}},y_{\ca{L}}}^{\mrm{h}}$ (resp. $\ca{O}_{Y^{\ca{L}},y_{\ca{L}}}^{\mrm{sh}}$) is the Henselization of $Y^{\ca{L}}$ at $y_{\ca{L}}$ (resp. the strict Henselization of $Y^{\ca{L}}$ at $\overline{y}$). Let $\overline{\ca{K}}/\ca{L}^{\mrm{sh}}/\ca{L}^{\mrm{h}}/\ca{L}$ be the extensions of their fraction fields. Recall that (\cite[\href{https://stacks.math.columbia.edu/tag/0BSW}{0BSW}]{stacks-project})
	\begin{align}
		\gal(\overline{\ca{K}}/\ca{L}^{\mrm{h}})&=\{\sigma\in\gal(\overline{\ca{K}}/\ca{L})\ |\ \sigma(\ca{O}_{Y^{\overline{\ca{K}}},\overline{y}})=\ca{O}_{Y^{\overline{\ca{K}}},\overline{y}}\},\label{eq:para:notation-neighborhood-2-3}\\
		\gal(\overline{\ca{K}}/\ca{L}^{\mrm{sh}})&=\{\sigma\in\gal(\overline{\ca{K}}/\ca{L}^{\mrm{h}})\ |\ \sigma\equiv\id_{\ca{\overline{F}}}\mod \ak{m}_{\overline{y}}\},\label{eq:para:notation-neighborhood-2-4}\\
		\gal(\ca{L}^{\mrm{sh}}/\ca{L}^{\mrm{h}})&=\mrm{Aut}_{\ca{O}_{Y^{\ca{L}},y_{\ca{L}}}^{\mrm{h}}}(\ca{O}_{Y^{\ca{L}},y_{\ca{L}}}^{\mrm{sh}})=\gal(\overline{\ca{F}}/\ca{F}_{\ca{L}}).\label{eq:para:notation-neighborhood-2-5}
	\end{align}
	Let $\ca{L}^{\mrm{vh}}$ be subfield of $\ca{L}^{\mrm{sh}}$ fixed by the closed subgroup $\gal(\overline{\ca{F}}/\ca{F}_{\ca{L}}^{\mrm{h}})\subseteq \gal(\overline{\ca{F}}/\ca{F}_{\ca{L}})=\gal(\ca{L}^{\mrm{sh}}/\ca{L}^{\mrm{h}})$ and we denote by $\ca{O}_{Y^{\ca{L}},y_{\ca{L}}}^{\mrm{vh}}$ the integral closure of $\ca{O}_{Y^{\ca{L}},y_{\ca{L}}}^{\mrm{h}}$ in $\ca{L}^{\mrm{vh}}$. Then, there is a canonical exact sequence of Galois groups
	\begin{align}\label{eq:para:notation-neighborhood-2-6}
		\xymatrix{
			1\ar[r]&\gal(\overline{\ca{K}}/\ca{L}^{\mrm{sh}})\ar[r]&\gal(\overline{\ca{K}}/\ca{L}^{\mrm{vh}})\ar[r]&\gal(\overline{\ca{F}}/\ca{F}_{\ca{L}}^{\mrm{h}})\ar[r]&1,
		}
	\end{align}
	and we have
	\begin{align}\label{eq:para:notation-neighborhood-2-7}
		\gal(\overline{\ca{K}}/\ca{L}^{\mrm{vh}})=\{\sigma\in\gal(\overline{\ca{K}}/\ca{L})\ |\ \sigma(\ca{O}_{Y^{\overline{\ca{K}}},\overline{y}})=\ca{O}_{Y^{\overline{\ca{K}}},\overline{y}}\trm{ and } \sigma|_{\overline{\ca{F}}}(\ca{O}_{\overline{\ca{F}}})=\ca{O}_{\overline{\ca{F}}}\}.
	\end{align}
	We call $\gal(\overline{\ca{K}}/\ca{L}^{\mrm{sh}}),\gal(\overline{\ca{K}}/\ca{L}^{\mrm{vh}}),\gal(\ca{L}^{\mrm{sh}}/\ca{L}^{\mrm{vh}})$ respectively the \emph{inertia, decomposition, residue Galois groups} of $Y^{\ca{L}}$ over $\ca{O}_K$ at the geometric valuative point $\spec(\overline{\ca{F}})$. In particular, the canonical action of $\gal(\overline{\ca{K}}/\ca{L}^{\mrm{vh}})$ on $\ca{O}_{\overline{\ca{F}}}$ (factoring through $\gal(\ca{L}^{\mrm{sh}}/\ca{L}^{\mrm{vh}})=\gal(\overline{\ca{F}}/\ca{F}_{\ca{L}}^{\mrm{h}})$) uniquely extends to a continuous action on its $p$-adic completion $\ca{O}_{\widehat{\overline{\ca{F}}}}$. We remark that the invariant subfield of $\widehat{\overline{\ca{F}}}$ by $\gal(\overline{\ca{F}}/\ca{F}_{\ca{L}}^{\mrm{h}})$ is the completion of $\ca{F}_{\ca{L}}^{\mrm{h}}$ by Ax-Sen-Tate's theorem \cite[page 417]{ax1970ax},
	\begin{align}\label{eq:para:notation-neighborhood-2-8}
		\widehat{\ca{F}_{\ca{L}}^{\mrm{h}}}=(\widehat{\overline{\ca{F}}})^{\gal(\overline{\ca{F}}/\ca{F}_{\ca{L}}^{\mrm{h}})}.
	\end{align}
	
	In conclusion, we obtain a canonical commutative diagram 
	\begin{align}\label{eq:para:notation-neighborhood-2-9}
		\xymatrix{
			\spec(\overline{\ca{K}})\ar[d]\ar@{~>}[r]&\overline{y}=\spec(\overline{\ca{F}})\ar@{=}[d]\ar[r]&\spec(\ca{O}_{Y^{\overline{\ca{K}}},\overline{y}})\ar[d]\\
			\spec(\ca{L}^{\mrm{sh}})\ar[d]\ar@{~>}[r]&y_{\ca{L}}^{\mrm{sh}}=\spec(\overline{\ca{F}})\ar[d]\ar[r]&\spec(\ca{O}_{Y^{\ca{L}},y_{\ca{L}}}^{\mrm{sh}})\ar[d]\\
			\spec(\ca{L}^{\mrm{vh}})\ar[d]\ar@{~>}[r]&y_{\ca{L}}^{\mrm{vh}}=\spec(\ca{F}_{\ca{L}}^{\mrm{h}})\ar[d]\ar[r]&\spec(\ca{O}_{Y^{\ca{L}},y_{\ca{L}}}^{\mrm{vh}})\ar[d]\\
			\spec(\ca{L}^{\mrm{h}})\ar[d]\ar@{~>}[r]&y_{\ca{L}}^{\mrm{h}}=\spec(\ca{F}_{\ca{L}})\ar@{=}[d]\ar[r]&\spec(\ca{O}_{Y^{\ca{L}},y_{\ca{L}}}^{\mrm{h}})\ar[d]\\
			\spec(\ca{L})\ar@{~>}[r]&y_{\ca{L}}=\spec(\ca{F}_{\ca{L}})\ar[r]&\spec(\ca{O}_{Y^{\ca{L}},y_{\ca{L}}})
		}
	\end{align}
	which is functorial in $\ca{L}$. When $\ca{L}=\ca{K}$, we omit the superscript/subscript $\ca{L}$.
\end{mypara}

\begin{mypara}\label{para:notation-neighborhood-3}	
	Following \ref{para:notation-neighborhood-1} and \ref{para:notation-neighborhood-2}, for any neighborhood $X$ of $Y/\ca{O}_K$ at the geometric valuative point $\spec(\overline{\ca{F}})$ such that $X_K\to Y$ is a pro-\'etale morphism of coherent schemes (\cite[7.13]{he2024coh}), we put $X^{\triv}=Y^{\triv}\times_YX_K$, where $Y^{\triv}=Y\setminus D$, and endow $X$ with the compactifying log structure $\scr{M}_X$ associated to the open immersion $X^{\triv}\to X$. 
	
	As the strict Henselization of $X_K$ at $\overline{y}$ coincides with that of $Y$ at $\overline{y}$, there is a canonical morphism over $Y$ induced by \eqref{eq:para:notation-neighborhood-2-2},
	\begin{align}\label{eq:para:notation-neighborhood-3-1}
		\spec(\ca{O}_{Y^{\overline{\ca{K}}},\overline{y}})\longrightarrow X_K.
	\end{align}
	We denote by $y_X=\spec(\ca{F}_X)\in X_K$ (resp. $\spec(\ca{K}_X)\in X_K$) the image of $\overline{y}=\spec(\overline{\ca{F}})\in Y^{\overline{\ca{K}}}$ (resp. $\spec(\overline{\ca{K}})\in Y^{\overline{\ca{K}}}$), and we put $\ca{O}_{\ca{F}_X}=\ca{F}_X\cap \ca{O}_{\overline{\ca{F}}}$ which is still a valuation ring of height $1$ extension of $\ca{O}_K$ (\cite[3.1]{he2024purity}). In particular, the given morphism $\spec(\ca{O}_{\overline{\ca{F}}})\to X$ factors through $\spec(\ca{O}_{\ca{F}_X})$. 
	
	Let $\ca{K}_{X,\mrm{ur}}$ be the maximal unramified extension of $\ca{K}_X$ contained in $\overline{\ca{K}}$ with respect to $(X^{\triv},X)$ (i.e., the filtered union of all finite field extension $\ca{K}'$ of $\ca{K}_X$ contained in $\overline{\ca{K}}$ such that the integral closure of $X$ in $\ca{K}'$ is \'etale over $X^{\triv}$). Then, the integral closure $\overline{X}$ of $X$ in $\ca{K}_{X,\mrm{ur}}$ is again a neighborhood of $Y/\ca{O}_K$ at the geometric valuative point $\spec(\overline{\ca{F}})$. Indeed, since $\ca{O}_{Y^{\overline{\ca{K}}},\overline{y}}$ is normal with fraction field $\overline{\ca{K}}$ containing $\ca{K}_{X,\mrm{ur}}$, the canonical morphism $\spec(\ca{O}_{Y^{\overline{\ca{K}}},\overline{y}})\to X_K$ \eqref{eq:para:notation-neighborhood-3-1} factors through $\overline{X}_K$. Similarly, we denote by $y_{\overline{X}}=\spec(\ca{F}_{\overline{X}})\in \overline{X}_K$ the image of $\overline{y}=\spec(\overline{\ca{F}})\in Y^{\overline{\ca{K}}}$, and we put $\ca{O}_{\ca{F}_{\overline{X}}}=\ca{F}_{\overline{X}}\cap \ca{O}_{\overline{\ca{F}}}$. In particular, the given morphism $\spec(\ca{O}_{\overline{\ca{F}}})\to \overline{X}$ factors through $\spec(\ca{O}_{\ca{F}_{\overline{X}}})$.
	
	In conclusion, we obtain a canonical commutative diagram
	\begin{align}\label{eq:para:notation-neighborhood-3-2}
		\xymatrix{
			\spec(\overline{\ca{K}})\ar[r]\ar[d]&\spec(\ca{O}_{Y^{\overline{\ca{K}}},\overline{y}})\ar[d]&\overline{y}=\spec(\overline{\ca{F}})\ar[l]\ar[r]&\spec(\ca{O}_{\overline{\ca{F}}})\ar[d]\\
			\spec(\ca{K}_{X,\mrm{ur}})\ar[r]\ar[d]&\overline{X}_K\ar[d]\ar[r]&\overline{X}\ar[d]&\spec(\ca{O}_{\ca{F}_{\overline{X}}})\ar[l]\ar[d]\\
			\spec(\ca{K}_X)\ar[r]\ar[d]&X_K\ar[d]\ar[r]&X&\spec(\ca{O}_{\ca{F}_X})\ar[l]\ar[d]\\
			\spec(\ca{K})\ar[r]&Y&y=\spec(\ca{F})\ar[l]\ar[r]&\spec(\ca{O}_{\ca{F}})
		}
	\end{align}
	which is functorial in $X$. When $X=\spec(A)$ is affine, we replace the subscript $X$ in the above diagram by $A$ and we put $\overline{X}=\spec(\overline{A})$. If moreover $X^{\triv}$ is affine, then we put $X^{\triv}=\spec(A_{\triv})$. This notation is consistent with \ref{defn:triple}. In particular, if $X=\spec(A)\in \nbd^{Y^{\triv}\trm{-}\mrm{adq}}_{\overline{\ca{F}}}(Y/\ca{O}_K)$ (\ref{prop:aet-neighborhood}), then $(A_{\triv},A,\overline{A})$ is an adequate $(K',\ca{O}_{K'},\ca{O}_{\overline{K}})$-triple for some finite field extension $K'$ of $K$ contained in $\overline{K}$.
\end{mypara}

\begin{mylem}\label{lem:log-str-val-pt}
	With the notation in {\rm\ref{para:notation-neighborhood-1}}, let $\{t_1,\dots, t_s\}$ be a regular system of parameters of the strict Henselization $\ca{O}_{Y,y}^{\mrm{sh}}$ such that $t_1\cdots t_r=0$ defines the normal crossings divisor $D$ at $\overline{y}$, and let $t_{s+1},\dots,t_d$ be units of $\ca{O}_{Y,y}^{\mrm{sh}}$ whose images in the residue field $\overline{\ca{F}}$ form a transcendental basis over $K$, where $0\leq r\leq s\leq d=\dim(Y)$ {\rm(\cite[\href{https://stacks.math.columbia.edu/tag/00P1}{00P1}]{stacks-project})}. We put 
	\begin{align}
		t_1'=t_1,\ \dots,\ t_r'=t_r,\ t_{r+1}'=1+t_{r+1},\ \dots,\ t_s'=1+t_s,\ t_{s+1}'=t_{s+1},\ \dots,\ t_d'=t_d.
	\end{align}
	Let $\scr{M}_Y\to \ca{O}_{Y_\et}$ be the compactifying log structure associated to the open immersion $Y^{\triv}\to Y$. Then, the homomorphism of monoids
	\begin{align}\label{eq:lem:log-str-val-pt-1}
		\bb{N}^r\oplus \bb{Z}^{d-r}\longrightarrow \scr{M}_{Y,\overline{y}},\ (a_1,\dots,a_d)\mapsto t_1'^{a_1}\cdots t_d'^{a_d},
	\end{align}
	induces an isomorphism $\bb{N}^r\iso \scr{M}_{Y,\overline{y}}/\ca{O}_{Y,y}^{\mrm{sh},\times}$, where $\scr{M}_{Y,\overline{y}}$ is the stalk of the \'etale sheaf of monoids $\scr{M}_Y$ at $\overline{y}$. Moreover, the $\ca{O}_{Y,y}^{\mrm{sh}}$-module of log differentials $\Omega^1_{(\ca{O}_{Y,y}^{\mrm{sh}},\scr{M}_{Y,\overline{y}})/K}$ is finite free with basis $\df\log(t_1'),\dots,\df\log(t_d')$.
%
\end{mylem}
\begin{proof}
	After replacing $Y$ by an \'etale neighborhood $\spec(R)$ of $\overline{y}$, we may assume that $t_1,\dots,t_d\in R$ with $t_{r+1}',\dots,t_d'\in R^\times$, and that $D$ is defined by $t_1\cdots t_r=0$. Let $\ak{q}$ be the prime ideal of $R$ corresponding to $y$. We claim that the $K$-algebra homomorphism $f:K[T_1,\dots,T_d]\to R$ sending $T_i$ to $t_i$ for any $1\leq i\leq d$ is \'etale at $\ak{q}$. Let $\ak{p}$ be inverse image of $\ak{q}\subseteq R$. Since $t_{s+1},\dots,t_d$ are algebraically independent over $K$ in the residue field $\ca{F}$ of $\ak{q}$, we see that $K[T_{s+1},\dots,T_d]\to R$ is injective and has zero intersection with $\ak{q}$. Therefore, $f$ induces a $K$-algebra homomorphism $(K[T_1,\dots,T_d])_{\ak{p}}=(K(T_{s+1},\dots,T_d)[T_1,\dots,T_s])_{\ak{m}}\to R_{\ak{q}}$ by localization, where $\ak{m}$ is the maximal ideal generated by $T_1,\dots,T_s$. Taking the completions with respect to $\ak{p}$, $\ak{m}$ and $\ak{q}$ respectively, we obtain a $K$-algebra homomorphism $\widehat{f}:(K[T_1,\dots,T_d])_{\ak{p}}^{\wedge}=K(T_{s+1},\dots,T_d)[[T_1,\dots,T_s]]\to \widehat{R_{\ak{q}}}$.
	
	As $\ca{F}$ is a finitely generated field extension of $K$ by Noether normalization theorem \cite[\href{https://stacks.math.columbia.edu/tag/00OY}{00OY}]{stacks-project}, we see that $\ca{F}$ is finite (\'etale) over $K(T_{s+1},\dots,T_d)$ via $\widehat{f}$ (\cite[\href{https://stacks.math.columbia.edu/tag/037J}{037J}]{stacks-project}). Thus, there exists a unique lifting of $\widehat{f}:K(T_{s+1},\dots,T_d)[[T_1,\dots,T_s]]\to \widehat{R_{\ak{q}}}$ to $g:\ca{F}[[T_1,\dots,T_s]]\to \widehat{R_{\ak{q}}}$ which induces the identity of the residue fields $\ca{F}=\ca{F}$. Since $\{t_1,\dots,t_s\}$ is a regular system of parameters of the complete regular local ring $\widehat{R_{\ak{q}}}$, we see that $g$ is an isomorphism (\cite[\href{https://stacks.math.columbia.edu/tag/00NO}{00NO}]{stacks-project}). On the other hand, since $\widehat{f}:K(T_{s+1},\dots,T_d)[[T_1,\dots,T_s]]\to \ca{F}[[T_1,\dots,T_s]]=\widehat{R_{\ak{q}}}$ is formally \'etale, we conclude that $f$ is \'etale at $\ak{q}$ by \cite[17.6.3]{ega4-4}, which verifies the claim.
	
	The claim shows that $K[T_1,\dots,T_r,T_{r+1}^{\pm 1},\dots,T_d^{\pm1}]\to R$ sending $T_i$ to $t_i'$ is \'etale at $\ak{q}$. Therefore, \eqref{eq:lem:log-str-val-pt-1} induces a strictly \'etale morphism of log schemes $(Y,\scr{M}_Y)\to \spec(K)\times_{\spec(\bb{Z})}\bb{A}_{\bb{N}^r\oplus\bb{Z}^{d-r}} $ and thus the conclusion follows.

\end{proof}

\begin{mylem}\label{lem:galois-action-adq-nbd}
	With the notation in {\rm\ref{para:notation-neighborhood-1}}, {\rm\ref{para:notation-neighborhood-2}} and {\rm\ref{para:notation-neighborhood-3}}, there is a canonical action of $\gal(\overline{\ca{K}}/\ca{K}^{\mrm{vh}})$ (the decomposition group of $Y/\ca{O}_K$ at the geometric valuative point $\spec(\overline{\ca{F}})$) on $\nbd^{Y^{\triv}\trm{-}\mrm{adq}}_{\overline{\ca{F}}}(Y/\ca{O}_K)$ {\rm(\ref{prop:aet-neighborhood})} such that $X^{\sigma}=\spec(\sigma(A))$ for any $\sigma\in \gal(\overline{\ca{K}}/\ca{K}^{\mrm{vh}})$ and any $X=\spec(A)\in \nbd^{Y^{\triv}\trm{-}\mrm{adq}}_{\overline{\ca{F}}}(Y/\ca{O}_K)$.
\end{mylem}
\begin{proof}
	Firstly, we note that $A\subseteq \ca{K}_X$ as $A$ is integral. In particular, $\sigma(A)$ is a subring of $\overline{\ca{K}}$. Since $(A^{\triv},A,\overline{A})$ is an adequate $(K',\ca{O}_{K'},\ca{O}_{\overline{K}})$-triple for some finite field extension $K'$ of $K$ contained in $\overline{K}$, we see that $(\sigma(A^{\triv}),\sigma(A),\sigma(\overline{A}))$ is an adequate $(\sigma(K'),\sigma(\ca{O}_{K'}),\sigma(\ca{O}_{\overline{K}}))$-triple. Moreover, the canonical morphism $\spec(\ca{O}_{\overline{\ca{F}}})\to X$ induces a morphism $\spec(\ca{O}_{\overline{\ca{F}}})=\spec(\sigma(\ca{O}_{\overline{\ca{F}}}))\to X^\sigma=\spec(\sigma(A))$, which shows that $X^{\sigma}$ is still naturally a neighborhood of $Y/\ca{O}_K$ at $\spec(\overline{\ca{F}})$ (see \ref{defn:neighborhood}). Therefore, we have $X^{\sigma}\in \nbd^{Y^{\triv}\trm{-}\mrm{adq}}_{\overline{\ca{F}}}(Y/\ca{O}_K)$. It is easy to see that this defines an action of $\gal(\overline{\ca{K}}/\ca{K}^{\mrm{vh}})$ on $\nbd^{Y^{\triv}\trm{-}\mrm{adq}}_{\overline{\ca{F}}}(Y/\ca{O}_K)$.
\end{proof}

\begin{mypara}\label{para:fal-ext-val-pt}
	With the notation in {\rm\ref{para:notation-neighborhood-1}}, {\rm\ref{para:notation-neighborhood-2}} and {\rm\ref{para:notation-neighborhood-3}}, for any $X=\spec(A)\in \nbd^{Y^{\triv}\trm{-}\mrm{adq}}_{\overline{\ca{F}}}(Y/\ca{O}_K)$ given by an adequate $(K',\ca{O}_{K'},\ca{O}_{\overline{K}})$-triple $(A^{\triv},A,\overline{A})$ for some finite field extension $K'$ of $K$ contained in $\overline{K}$, consider the Faltings extension of $A$ over $\ca{O}_{K'}$ \eqref{eq:thm:A-fal-ext-1},
	\begin{align}\label{eq:para:fal-ext-val-pt-1}
		0\longrightarrow \widehat{\overline{A}}[\frac{1}{p}](1)\stackrel{\iota}{\longrightarrow}\scr{E}_A\stackrel{\jmath}{\longrightarrow} \widehat{\overline{A}}\otimes_{A}\Omega^1_{(X_K,\scr{M}_{X_K})/K}\longrightarrow 0,
	\end{align}
	where $\scr{M}_{X_K}$ is the compactifying log structure associated to the open immersion $X^{\triv}\to X_K$ and where we used the fact that $\Omega^1_{(X_{K'},\scr{M}_{X_{K'}})/K'}=\Omega^1_{(X_K,\scr{M}_{X_K})/K}$ (as $K'$ is \'etale over $K$).
	
	Any morphism $X'=\spec(A')\to X=\spec(A)$ in $\nbd^{Y^{\triv}\trm{-}\mrm{adq}}_{\overline{\ca{F}}}(Y/\ca{O}_K)$ induces a morphism of $(K,\ca{O}_K,\ca{O}_{\overline{K}})$-triples $(A_{\triv},A,\overline{A})\to (A'_{\triv},A',\overline{A'})$ by \ref{para:notation-neighborhood-3}. Notice that $\overline{A}\to \overline{A'}$ is injective (since both of them are contained in $\overline{\ca{K}}$ by definition \ref{para:notation-neighborhood-3}). Thus, it induces a canonical isomorphism of Faltings extensions
	\begin{align}\label{eq:para:fal-ext-val-pt-2}
		\xymatrix{
			0\ar[r]& \widehat{\overline{\ca{F}}}(1)\ar[r]^-{\iota}&\widehat{\overline{\ca{F}}}\otimes_{\widehat{\overline{A'}}}\scr{E}_{A'}\ar[r]^-{\jmath}& \widehat{\overline{\ca{F}}}\otimes_{A'}\Omega^1_{(X'_K,\scr{M}_{X'_K})/K}\ar[r]& 0\\
			0\ar[r]& \widehat{\overline{\ca{F}}}(1)\ar[r]^-{\iota}\ar[u]^-{\wr}&\widehat{\overline{\ca{F}}}\otimes_{\widehat{\overline{A}}}\scr{E}_A\ar[r]^-{\jmath}\ar[u]^-{\wr}& \widehat{\overline{\ca{F}}}\otimes_{A}\Omega^1_{(X_K,\scr{M}_{X_K})/K}\ar[r]\ar[u]^-{\wr}& 0
		}
	\end{align}
	together with a canonical commutative diagram by \ref{rem:A-fal-ext}.(\ref{item:rem:A-fal-ext-2}),
	\begin{align}\label{eq:para:fal-ext-val-pt-3}
		\xymatrix{
			V_p(\overline{A'}[1/p]\cap \overline{A'}_{\triv}^\times)\ar[r]&\widehat{\overline{\ca{F}}}\otimes_{\widehat{\overline{A'}}}\scr{E}_{A'}\\
			V_p(\overline{A}[1/p]\cap \overline{A}_{\triv}^\times)\ar[u]\ar[r]&\widehat{\overline{\ca{F}}}\otimes_{\widehat{\overline{A}}}\scr{E}_A\ar[u]^-{\wr}
		}
	\end{align}
	where we used the fact that $(X'_K,\scr{M}_{X'_K})\to (X_K,\scr{M}_{X_K})$ is strictly \'etale.
	
	Taking filtered colimit over $\nbd^{Y^{\triv}\trm{-}\mrm{adq}}_{\overline{\ca{F}}}(Y/\ca{O}_K)$ (\ref{prop:aet-neighborhood}), we obtain a canonical exact sequence of $\widehat{\overline{\ca{F}}}$-modules
	\begin{align}\label{eq:para:fal-ext-val-pt-4}
		0\longrightarrow \widehat{\overline{\ca{F}}}(1)\stackrel{\iota}{\longrightarrow}\colim_{X=\spec(A)\in \nbd^{Y^{\triv}\trm{-}\mrm{adq}}_{\overline{\ca{F}}}(Y/\ca{O}_K)}\widehat{\overline{\ca{F}}}\otimes_{\widehat{\overline{A}}}\scr{E}_A\stackrel{\jmath}{\longrightarrow} \widehat{\overline{\ca{F}}}\otimes_{\ca{O}_Y}\Omega^1_{(Y,\scr{M}_Y)/K}\longrightarrow 0,
	\end{align}
	together with a canonical group homomorphism
	\begin{align}\label{eq:para:fal-ext-val-pt-5}
		\colim_{X=\spec(A)\in \nbd^{Y^{\triv}\trm{-}\mrm{adq}}_{\overline{\ca{F}}}(Y/\ca{O}_K)}V_p(\overline{A}[1/p]\cap \overline{A}_{\triv}^\times)\longrightarrow \colim_{X=\spec(A)\in \nbd^{Y^{\triv}\trm{-}\mrm{adq}}_{\overline{\ca{F}}}(Y/\ca{O}_K)}\widehat{\overline{\ca{F}}}\otimes_{\widehat{\overline{A}}}\scr{E}_A,
	\end{align}
	where we used the fact that $(X_K,\scr{M}_{X_K})$ is strictly \'etale over $(Y,\scr{M}_Y)$. We denote the image of an element $(s_{p^n})_{n\in\bb{N}}$ under \eqref{eq:para:fal-ext-val-pt-5} by $(\df\log(s_{p^n}))_{n\in\bb{N}}$.
\end{mypara}

\begin{mythm}\label{thm:fal-ext-val-pt}
	With the notation in {\rm\ref{para:notation-neighborhood-1}}, {\rm\ref{para:notation-neighborhood-2}} and {\rm\ref{para:notation-neighborhood-3}}, the finite free $\widehat{\overline{\ca{F}}}$-module (endowed with the canonical topology induced by the $p$-adic topology on $\ca{O}_{\widehat{\overline{\ca{F}}}}$) 
	\begin{align}\label{eq:thm:fal-ext-val-pt-1}
		\scr{E}_{Y^{\triv}\to Y,\overline{y}}=\colim_{X=\spec(A)\in \nbd^{Y^{\triv}\trm{-}\mrm{adq}}_{\overline{\ca{F}}}(Y/\ca{O}_K)}\widehat{\overline{\ca{F}}}\otimes_{\widehat{\overline{A}}}\scr{E}_A
	\end{align}
	is endowed with a canonical continuous action of $\gal(\overline{\ca{K}}/\ca{K}^{\mrm{vh}})$ (the decomposition group of $Y/\ca{O}_K$ at the geometric valuative point $\spec(\overline{\ca{F}})$). Moreover, the map \eqref{eq:para:fal-ext-val-pt-5}
	\begin{align}\label{eq:thm:fal-ext-val-pt-2}
		\colim_{X=\spec(A)\in \nbd^{Y^{\triv}\trm{-}\mrm{adq}}_{\overline{\ca{F}}}(Y/\ca{O}_K)}V_p(\overline{A}[1/p]\cap \overline{A}_{\triv}^\times)\longrightarrow\scr{E}_{Y^{\triv}\to Y,\overline{y}}
	\end{align}
	is $\gal(\overline{\ca{K}}/\ca{K}^{\mrm{vh}})$-equivariant, and the sequence \eqref{eq:para:fal-ext-val-pt-4},
	\begin{align}\label{eq:thm:fal-ext-val-pt-3}
		0\longrightarrow \widehat{\overline{\ca{F}}}(1)\stackrel{\iota}{\longrightarrow}\scr{E}_{Y^{\triv}\to Y,\overline{y}}\stackrel{\jmath}{\longrightarrow} \widehat{\overline{\ca{F}}}\otimes_{\ca{O}_Y}\Omega^1_{(Y,\scr{M}_Y)/K}\longrightarrow 0,
	\end{align}
	is an exact sequence of finite free $\widehat{\overline{\ca{F}}}$-representations of $\gal(\overline{\ca{K}}/\ca{K}^{\mrm{vh}})$	satisfying the following properties:
	\begin{enumerate}
		\renewcommand{\labelenumi}{{\rm(\theenumi)}}
		\item For any compatible system of $p$-power roots of unity $(\zeta_{p^n})_{n\in\bb{N}}$ contained in $\overline{\ca{F}}$, we have $\iota((\zeta_{p^n})_{n\in\bb{N}})=(\df\log(\zeta_{p^n}))_{n\in\bb{N}}$.\label{item:thm:fal-ext-val-pt-1}
		\item For any $s\in \scr{M}_{Y,\overline{y}}\subseteq \ca{K}^{\mrm{sh}}$ with compatible system of $p$-power roots $(s_{p^n})_{n\in\bb{N}}$ contained in $\overline{\ca{K}}$, we have $\jmath((\df\log(s_{p^n}))_{n\in\bb{N}})=\df\log(s)$ (note that $\widehat{\overline{\ca{F}}}\otimes_{\ca{O}_Y}\Omega^1_{(Y,\scr{M}_Y)/K}=\widehat{\overline{\ca{F}}}\otimes_{\ca{O}_{Y,y}^{\mrm{sh}}}\Omega^1_{(\ca{O}_{Y,y}^{\mrm{sh}},\scr{M}_{Y,\overline{y}})/K}$). \label{item:thm:fal-ext-val-pt-2}
		\item With the notation in {\rm\ref{lem:log-str-val-pt}}, let $(t_{1,p^n}')_{n\in\bb{N}}, \dots,(t_{d,p^n}')_{n\in\bb{N}}$ be compatible systems of $p$-power roots of $t_1',\dots,t_d'$ contained in $\overline{\ca{K}}$. Then, the $\widehat{\overline{\ca{F}}}$-linear surjection $\jmath$ admits a section sending $\df\log(t_i')$ to $(\df\log(t_{i,p^n}'))_{n\in\bb{N}}$.\label{item:thm:fal-ext-val-pt-3}
	\end{enumerate}
\end{mythm}
\begin{proof}
	For any $\sigma\in \gal(\overline{\ca{K}}/\ca{K}^{\mrm{vh}})$ and $X=\spec(A)\in \nbd^{Y^{\triv}\trm{-}\mrm{adq}}_{\overline{\ca{F}}}(Y/\ca{O}_K)$, there is a canonical morphism $X^\sigma\to X$ in $\nbd^{Y^{\triv}\trm{-}\mrm{adq}}_{\overline{\ca{F}}}(Y/\ca{O}_K)$ given by $\sigma:A\to\sigma(A)$ (see \ref{lem:galois-action-adq-nbd}). Thus, we obtain a canonical isomorphism $\widehat{\overline{\ca{F}}}\otimes_{\widehat{\overline{A}}}\scr{E}_A\iso \widehat{\overline{\ca{F}}}\otimes_{\widehat{\sigma(\overline{A})}}\scr{E}_{\sigma(A)}$ by \eqref{eq:para:fal-ext-val-pt-2}. Taking filtered colimits over $\nbd^{Y^{\triv}\trm{-}\mrm{adq}}_{\overline{\ca{F}}}(Y/\ca{O}_K)$, this defines a canonical $\widehat{\overline{\ca{F}}}$-semi-linear action of $\sigma$ (and thus of $\gal(\overline{\ca{K}}/\ca{K}^{\mrm{vh}})$) on $\scr{E}_{Y^{\triv}\to Y,\overline{y}}$. Except its continuity, the rest of the properties follow directly from \ref{thm:A-fal-ext} and the fact that $\scr{M}_{Y,\overline{y}}=\colim_{X=\spec(A)\in \nbd^{Y^{\triv}\trm{-}\mrm{adq}}_{\overline{\ca{F}}}(Y/\ca{O}_K)}A[1/p]\cap A_{\triv}^{\times}$ (\ref{lem:aet-et-neighborhood} and \ref{prop:aet-neighborhood}). 
	
	We fix $X=\spec(A)\in \nbd^{Y^{\triv}\trm{-}\mrm{adq}}_{\overline{\ca{F}}}(Y/\ca{O}_K)$. Note that its fraction field $\ca{K}_X$ is finite over the fraction field $\ca{K}$ of $Y$. Consider the finite field extension $\ca{K}_X\ca{K}^{\mrm{vh}}$ of $\ca{K}^{\mrm{vh}}$ generated by $\ca{K}_X$ in $\overline{\ca{K}}$. Then, we see that the canonical morphism $\overline{A}\to \ca{O}_{\overline{\ca{F}}}$ is compatible with the actions of $\gal(\overline{\ca{K}}/\ca{K}_X\ca{K}^{\mrm{vh}})$ (via the canonical group homomorphism $\gal(\overline{\ca{K}}/\ca{K}_X\ca{K}^{\mrm{vh}})\to \gal(\ca{K}_{X,\mrm{ur}}/\ca{K}_X)$), and thus so is $\widehat{\overline{A}}[1/p]\to \widehat{\overline{\ca{F}}}$. Recall that the $\gal(\ca{K}_{X,\mrm{ur}}/\ca{K}_X)$-action on the finite free $\widehat{\overline{A}}[1/p]$-module $\scr{E}_A$ (endowed with the canonical topology) is continuous (\ref{thm:A-fal-ext}). Thus, the $\gal(\overline{\ca{K}}/\ca{K}_X\ca{K}^{\mrm{vh}})$-action on the finite free $\widehat{\overline{\ca{F}}}$-module $\widehat{\overline{\ca{F}}}\otimes_{\widehat{\overline{A}}}\scr{E}_A$ is also continuous (see \cite[2.4]{he2022sen}). Notice that the vertical isomorphisms in the following canonical commutative diagram (induced by \eqref{eq:para:fal-ext-val-pt-2})
	\begin{align}
		\xymatrix{
			0\ar[r]& \widehat{\overline{\ca{F}}}(1)\ar[r]^-{\iota}&\scr{E}_{Y^{\triv}\to Y,\overline{y}}\ar[r]^-{\jmath}&\widehat{\overline{\ca{F}}}\otimes_{\ca{O}_Y}\Omega^1_{(Y,\scr{M}_Y)/K}\ar[r]& 0\\
			0\ar[r]& \widehat{\overline{\ca{F}}}(1)\ar[r]^-{\iota}\ar[u]^-{\wr}&\widehat{\overline{\ca{F}}}\otimes_{\widehat{\overline{A}}}\scr{E}_A\ar[r]^-{\jmath}\ar[u]^-{\wr}& \widehat{\overline{\ca{F}}}\otimes_{A}\Omega^1_{(X_K,\scr{M}_{X_K})/K}\ar[r]\ar[u]^-{\wr}& 0
		}
	\end{align}
	are $\gal(\overline{\ca{K}}/\ca{K}_X\ca{K}^{\mrm{vh}})$-equivariant and homeomorphic. Thus, the $\gal(\overline{\ca{K}}/\ca{K}_X\ca{K}^{\mrm{vh}})$-action on $\scr{E}_{Y^{\triv}\to Y,\overline{y}}$ is also continuous. As $\gal(\overline{\ca{K}}/\ca{K}_X\ca{K}^{\mrm{vh}})$ is an open subgroup of $\gal(\overline{\ca{K}}/\ca{K}^{\mrm{vh}})$, the $\gal(\overline{\ca{K}}/\ca{K}^{\mrm{vh}})$-action on $\scr{E}_{Y^{\triv}\to Y,\overline{y}}$ is also continuous.
\end{proof}

\begin{mydefn}\label{defn:fal-ext-val-pt}
	We call the exact sequence \eqref{eq:thm:fal-ext-val-pt-3} the \emph{Faltings extension of the open immersion $Y^{\triv}\to Y$ over $\ca{O}_K$ at the geometric valuative point $\overline{y}=\spec(\overline{\ca{F}})$}. If $Y^{\triv}$ is clearly fixed in the context, then we write $\scr{E}_{Y,\overline{y}}$ instead of $\scr{E}_{Y^{\triv}\to Y,\overline{y}}$ for simplicity.
\end{mydefn}

Note that the inertia subgroup $\gal(\overline{\ca{K}}/\ca{K}^{\mrm{sh}})\subseteq \gal(\overline{\ca{K}}/\ca{K}^{\mrm{vh}})$ acts trivially on both $\widehat{\overline{\ca{F}}}(1)$ and $\widehat{\overline{\ca{F}}}\otimes_{\ca{O}_Y}\Omega^1_{(Y,\scr{M}_Y)/K}$, but it may act non-trivially on $\scr{E}_{Y^{\triv}\to Y,\overline{y}}$ (see the proof of \ref{thm:perfd-val-log-2}).

\begin{myrem}\label{rem:fal-ext-val-pt}
	The construction of \eqref{eq:thm:fal-ext-val-pt-3} is functorial in the data in \ref{para:notation-neighborhood-1}. More precisely, consider another set of data in \ref{para:notation-neighborhood-1} by adding a prime superscript. Given a commutative diagram of $\bb{Z}_p$-schemes extending the structural morphisms given in \ref{para:notation-neighborhood-1},
	\begin{align}\label{eq:rem:fal-ext-val-pt-1}
		\xymatrix{
			\spec(\ca{O}_{\overline{\ca{F}'}})\ar[ddd]\ar@/^1pc/[rrr]&\spec(\overline{\ca{F}'})\ar[l]\ar[r]\ar[d]&\spec(\overline{\ca{F}})\ar[r]\ar[d]&\spec(\ca{O}_{\overline{\ca{F}}})\ar[ddd]\\
			&Y'^{\overline{\ca{K}'}}\ar[r]\ar[d]&Y^{\overline{\ca{K}}}\ar[d]&\\
			&Y'\ar[r]\ar[d]&Y\ar[d]&\\
			\spec(\ca{O}_{K'})\ar@/_1pc/[rrr]&\spec(K')\ar[l]\ar[r]&\spec(K)\ar[r]&\spec(\ca{O}_K)
		}
	\end{align}
	such that $Y'^{\triv}$ is over $Y^{\triv}$ and that $\spec(\overline{\ca{K}'})\in Y'^{\overline{\ca{K}'}}$ is over $\spec(\overline{\ca{K}})\in Y^{\overline{\ca{K}}}$, for any $X=\spec(A)\in \nbd^{Y^{\triv}\trm{-}\mrm{adq}}_{\overline{\ca{F}}}(Y/\ca{O}_K)$ and $X'=\spec(A')\in \nbd^{Y'^{\triv}\trm{-}\mrm{adq}}_{\overline{\ca{F'}}}(Y'/\ca{O}_{K'})$, we define $X\leq X'$ if and only if the inclusion $\overline{\ca{K}}\subseteq \overline{\ca{K}'}$ induces a morphism of $(K,\ca{O}_K,\ca{O}_{\overline{K}})$-triples $(A_{\triv},A,\overline{A})\to (A'_{\triv},A',\overline{A'})$. It is clear that the full subcategory $(\nbd^{Y'^{\triv}\trm{-}\mrm{adq}}_{\overline{\ca{F'}}}(Y'/\ca{O}_{K'}))_{\geq X}$ is cofinal in $\nbd^{Y'^{\triv}\trm{-}\mrm{adq}}_{\overline{\ca{F'}}}(Y'/\ca{O}_{K'})$ by \ref{prop:aet-neighborhood}. Hence, by \ref{rem:A-fal-ext}.(\ref{item:rem:A-fal-ext-2}), there is a natural morphism of exact sequences 
	\begin{align}
		\xymatrix{
			0\ar[r]& \widehat{\overline{\ca{F}'}}(1)\ar[r]^-{\iota}&\colim_{X'\in(\nbd^{Y'^{\triv}\trm{-}\mrm{adq}}_{\overline{\ca{F'}}}(Y'/\ca{O}_{K'}))_{\geq X}}\widehat{\overline{\ca{F}'}}\otimes_{\widehat{\overline{A'}}}\scr{E}_{A'}\ar[r]^-{\jmath}& \widehat{\overline{\ca{F}'}}\otimes_{A'}\Omega^1_{(X'_{K'},\scr{M}_{X'_{K'}})/K'}\ar[r]& 0\\
			0\ar[r]& \widehat{\overline{\ca{F}}}(1)\ar[r]^-{\iota}\ar[u]&\widehat{\overline{\ca{F}}}\otimes_{\widehat{\overline{A}}}\scr{E}_A\ar[r]^-{\jmath}\ar[u]& \widehat{\overline{\ca{F}}}\otimes_{A}\Omega^1_{(X_K,\scr{M}_{X_K})/K}\ar[r]\ar[u]& 0.
		}
	\end{align}
	Taking filtered colimit over $X\in \nbd^{Y^{\triv}\trm{-}\mrm{adq}}_{\overline{\ca{F}}}(Y/\ca{O}_K)$, we obtain a canonical morphism of exact sequences
	\begin{align}\label{eq:rem:fal-ext-val-pt-2}
		\xymatrix{
			0\ar[r]& \widehat{\overline{\ca{F}'}}(1)\ar[r]^-{\iota}&\scr{E}_{Y'^{\triv}\to Y',\overline{y'}}\ar[r]^-{\jmath}&\widehat{\overline{\ca{F}'}}\otimes_{\ca{O}_{Y'}}\Omega^1_{(Y',\scr{M}_{Y'})/K'}\ar[r]& 0\\
			0\ar[r]& \widehat{\overline{\ca{F}}}(1)\ar[r]^-{\iota}\ar[u]&\scr{E}_{Y^{\triv}\to Y,\overline{y}}\ar[r]^-{\jmath}\ar[u]& \widehat{\overline{\ca{F}}}\otimes_{\ca{O}_Y}\Omega^1_{(Y,\scr{M}_Y)/K}\ar[r]\ar[u]& 0.
		}
	\end{align}
	Similarly, we obtain a natural commutative diagram
	\begin{align}\label{eq:rem:fal-ext-val-pt-3}
		\xymatrix{
			\colim_{X'=\spec(A')\in \nbd^{Y'^{\triv}\trm{-}\mrm{adq}}_{\overline{\ca{F'}}}(Y'/\ca{O}_{K'})}V_p(\overline{A'}[1/p]\cap \overline{A'}_{\triv}^\times)\ar[r]&\scr{E}_{Y'^{\triv}\to Y',\overline{y'}}\\
			\colim_{X=\spec(A)\in \nbd^{Y^{\triv}\trm{-}\mrm{adq}}_{\overline{\ca{F}}}(Y/\ca{O}_K)}V_p(\overline{A}[1/p]\cap \overline{A}_{\triv}^\times)\ar[r]\ar[u]&\scr{E}_{Y^{\triv}\to Y,\overline{y}}.\ar[u]
		}
	\end{align}
	In particular, if $K'$ is finite over $K$, $(Y',\scr{M}_{Y'})$ is \'etale over $(Y,\scr{M}_Y)$ and $\ca{O}_{\overline{\ca{F}'}}=\ca{O}_{\overline{\ca{F}}}$, then the vertical homomorphisms in \eqref{eq:rem:fal-ext-val-pt-2} are isomorphisms.
\end{myrem}

\begin{mylem}\label{lem:fal-ext-val-pt-nolog-1}
	With the notation in {\rm\ref{para:notation-neighborhood-1}}, {\rm\ref{para:notation-neighborhood-2}} and {\rm\ref{para:notation-neighborhood-3}}, there is a canonical $\gal(\overline{\ca{K}}/\ca{K}^{\mrm{vh}})$-equivariant injection of Faltings extensions \eqref{eq:thm:fal-ext-val-pt-3},
	\begin{align}\label{eq:lem:fal-ext-val-pt-nolog-1-1}
		\xymatrix{
			0\ar[r]& \widehat{\overline{\ca{F}}}(1)\ar[r]^-{\iota}&\scr{E}_{Y^{\triv}\to Y,\overline{y}}\ar[r]^-{\jmath}&\widehat{\overline{\ca{F}}}\otimes_{\ca{O}_{Y}}\Omega^1_{(Y,\scr{M}_{Y})/K}\ar[r]& 0\\
			0\ar[r]& \widehat{\overline{\ca{F}}}(1)\ar[r]^-{\iota}\ar@{=}[u]&\scr{E}_{Y\to Y,\overline{y}}\ar[r]^-{\jmath}\ar@{^{(}->}[u]& \widehat{\overline{\ca{F}}}\otimes_{\ca{O}_Y}\Omega^1_{Y/K}\ar[r]\ar@{^{(}->}[u]& 0,
		}
	\end{align}
	where $Y\to Y$ is the identity morphism corresponding to the case where the associated divisor $D$ is empty and $Y^{\triv}=Y$.
\end{mylem}
\begin{proof}
	Firstly, note that replacing the divisor $D$ in \ref{para:notation-neighborhood-1} by the empty set (and thus replacing $Y^{\triv}$ by $Y$) again gives a set of data satisfying the conditions in \ref{para:notation-neighborhood-1}. Then, the canonical morphism is given by \ref{rem:fal-ext-val-pt} and the injectivity follows from the injectivity of $\Omega^1_{Y/K}\to \Omega^1_{(Y,\scr{M}_Y)/K}$ (as $D$ is a normal crossings divisor, see \ref{lem:log-str-val-pt}).
\end{proof}

\begin{mylem}\label{lem:fal-ext-val-pt-nolog-2}
	With the notation in {\rm\ref{para:notation-neighborhood-1}}, {\rm\ref{para:notation-neighborhood-2}} and {\rm\ref{para:notation-neighborhood-3}}, let $\scr{E}_{\ca{O}_{\overline{\ca{F}}}}$ be the canonical finite free $\widehat{\overline{\ca{F}}}$-representation of $\gal(\overline{\ca{F}}/\ca{F}^{\mrm{h}})$ defined in \eqref{eq:thm:fal-ext-val-1} (regarded as a finite free $\widehat{\overline{\ca{F}}}$-representation of $\gal(\overline{\ca{K}}/\ca{K}^{\mrm{vh}})$ via the canonical surjection $\gal(\overline{\ca{K}}/\ca{K}^{\mrm{vh}})\to \gal(\overline{\ca{F}}/\ca{F}^{\mrm{h}})$ \eqref{eq:para:notation-neighborhood-2-6}). Then, there is a canonical $\gal(\overline{\ca{K}}/\ca{K}^{\mrm{vh}})$-equivariant surjection between Faltings extensions \eqref{eq:thm:fal-ext-val-pt-3} and \eqref{eq:thm:fal-ext-val-3},
	\begin{align}\label{eq:prop:fal-ext-val-pt-nolog-2-1}
		\xymatrix{
			0\ar[r]& \widehat{\overline{\ca{F}}}(1)\ar[r]^-{\iota}&\scr{E}_{\ca{O}_{\overline{\ca{F}}}}\ar[r]^-{\jmath}&\widehat{\overline{\ca{F}}}\otimes_{\ca{F}}\Omega^1_{\ca{F}/K}\ar[r]& 0\\
			0\ar[r]& \widehat{\overline{\ca{F}}}(1)\ar[r]^-{\iota}\ar@{=}[u]&\scr{E}_{Y\to Y,\overline{y}}\ar[r]^-{\jmath}\ar@{->>}[u]& \widehat{\overline{\ca{F}}}\otimes_{\ca{O}_Y}\Omega^1_{Y/K}\ar[r]\ar@{->>}[u]& 0,
		}
	\end{align}
	where $Y\to Y$ is the identity morphism corresponding to the case where the associated divisor $D$ is empty and $Y^{\triv}=Y$.
\end{mylem}
\begin{proof}
	For any $X=\spec(A)\in \nbd^{Y\trm{-}\mrm{adq}}_{\overline{\ca{F}}}(Y/\ca{O}_K)$, recall that $\scr{E}_A=V_p(\Omega^1_{(\overline{X},\scr{M}'_{\overline{X}})/(X,\scr{M}_X)})$ with the notation in \ref{para:adequate-setup}. Let $\scr{M}_{\overline{X}}$ be the compactifying log structure associated to the open immersion $\overline{X}_K\to \overline{X}$. Then, there is a canonical morphism of log schemes $(\overline{X},\scr{M}_{\overline{X}})\to (\overline{X},\scr{M}'_{\overline{X}})$ by the definition of $\scr{M}'_{\overline{X}}$. Notice that $\Omega^1_{(X,\scr{M}_X)/X}$ is annihilated by a certain power of $p$, since it is of finite type over $A$ (\cite[\Luoma{4}.1.2.8]{ogus2018log}) and $\Omega^1_{(X,\scr{M}_X)/X}[1/p]=\Omega^1_{X_K/X_K}=0$. Thus, we deduce from the canonical exact sequence 
	\begin{align}
		\xymatrix{
			\overline{A}\otimes_A\Omega^1_{(X,\scr{M}_X)/X}\ar[r]&\Omega^1_{(\overline{X},\scr{M}'_{\overline{X}})/X}\ar[r]&\Omega^1_{(\overline{X},\scr{M}'_{\overline{X}})/(X,\scr{M}_X)}\ar[r]&0
		}
	\end{align}
	a canonical isomorphism (\cite[7.3.(2)]{he2022sen})
	\begin{align}\label{eq:prop:fal-ext-val-pt-nolog-2-2-0}
		V_p(\Omega^1_{(\overline{X},\scr{M}'_{\overline{X}})/X})\iso V_p(\Omega^1_{(\overline{X},\scr{M}'_{\overline{X}})/(X,\scr{M}_X)}).
	\end{align} 
	
	On the other hand, the log scheme defined by the log ring $\ca{O}_{\overline{\ca{F}}}\setminus 0\to \ca{O}_{\overline{\ca{F}}}$ is the scheme $\spec(\ca{O}_{\overline{\ca{F}}})$ endowed with the compactifying log structure defined by the closed point, since $\ca{O}_{\overline{\ca{F}}}$ is a strictly Henselian local ring. Thus, the morphism $\spec(\ca{O}_{\overline{\ca{F}}})\to \overline{X}$ \eqref{eq:para:notation-neighborhood-3-2} upgrades to a morphism of log schemes $\spec(\ca{O}_{\overline{\ca{F}}},\ca{O}_{\overline{\ca{F}}}\setminus 0)\to (\overline{X},\scr{M}_{\overline{X}})$. Thus, we obtain a canonical morphism 
	\begin{align}\label{eq:prop:fal-ext-val-pt-nolog-2-2}
		V_p(\Omega^1_{(\overline{X},\scr{M}'_{\overline{X}})/X})\longrightarrow V_p(\Omega^1_{(\ca{O}_{\overline{\ca{F}}},\ca{O}_{\overline{\ca{F}}}\setminus 0)/A_{\ca{F}}}),
	\end{align}
	where $A_{\ca{F}}$ denotes the image of $A\to\ca{O}_{\overline{\ca{F}}}$. Note that $\nbd^{Y\trm{-}\mrm{adq}}_{\overline{\ca{F}}}(Y/\ca{O}_K)\to\scr{C}_{\ca{O}_{\ca{F}}/\ca{O}_K}$ \eqref{para:fal-ext-val-func}
	sending $X=\spec(A)$ to the pair $(K,A_{\ca{F}})$ is a well-defined functor. For any morphism $X=\spec(A)\to X'=\spec(A')$ in $\nbd^{Y\trm{-}\mrm{adq}}_{\overline{\ca{F}}}(Y/\ca{O}_K)$, consider the following canonical diagram
	\begin{align}\label{eq:prop:fal-ext-val-pt-nolog-2-3}
		\xymatrix{
			V_p(\Omega^1_{(\ca{O}_{\overline{\ca{F}}},\ca{O}_{\overline{\ca{F}}}\setminus 0)/A'_{\ca{F}}})&V_p(\Omega^1_{(\overline{X'},\scr{M}'_{\overline{X'}})/(X',\scr{M}_{X'})})\ar[l]\ar[r]^-{\sim}&\scr{E}_{A'}\\
			V_p(\Omega^1_{(\ca{O}_{\overline{\ca{F}}},\ca{O}_{\overline{\ca{F}}}\setminus 0)/A_{\ca{F}}})\ar[u]&V_p(\Omega^1_{(\overline{X},\scr{M}'_{\overline{X}})/(X,\scr{M}_X)})\ar[l]\ar[r]^-{\sim}&\scr{E}_A\ar[u]
		}
	\end{align}
	where the left horizontal arrows are the compositions of \eqref{eq:prop:fal-ext-val-pt-nolog-2-2} with the inverses of \eqref{eq:prop:fal-ext-val-pt-nolog-2-2-0}.	We claim that it is commutative. Indeed, there is a canonical commutative diagram 
	\begin{align}\label{eq:prop:fal-ext-val-pt-nolog-2-4}
		\xymatrix{
			V_p(\Omega^1_{(\ca{O}_{\overline{\ca{F}}},\ca{O}_{\overline{\ca{F}}}\setminus 0)/A'_{\ca{F}}})&V_p(\Omega^1_{(\overline{X'},\scr{M}'_{\overline{X'}})/(X',\scr{M}_{X'})})\ar[l]&V_p(\overline{A'}[1/p]\cap \overline{A'}_{\triv}^\times)\ar[r]\ar[l]&\scr{E}_{A'}\\
			V_p(\Omega^1_{(\ca{O}_{\overline{\ca{F}}},\ca{O}_{\overline{\ca{F}}}\setminus 0)/A_{\ca{F}}})\ar[u]&V_p(\Omega^1_{(\overline{X},\scr{M}'_{\overline{X}})/(X,\scr{M}_X)})\ar[l]&V_p(\overline{A}[1/p]\cap \overline{A}_{\triv}^\times)\ar[r]\ar[l]\ar[u]&\scr{E}_A\ar[u]
		}
	\end{align}
	where the left rectangle is induced by \ref{thm:A-fal-ext}.(\ref{item:thm:A-fal-ext-3}) (see \eqref{eq:para:notation-fal-ext-2-4}) and the right square is induced by \eqref{eq:rem:A-fal-ext-7}. Since the image of $V_p(\overline{A}[1/p]\cap \overline{A}_{\triv}^\times)$ in $V_p(\Omega^1_{(\overline{X},\scr{M}'_{\overline{X}})/(X,\scr{M}_X)})$ generates the whole module by \ref{thm:A-fal-ext}, we see that \eqref{eq:prop:fal-ext-val-pt-nolog-2-3} is commutative.
	
	Recall that the canonical morphism 
	\begin{align}\label{eq:prop:fal-ext-val-pt-nolog-2-5}
		V_p(\Omega^1_{\ca{O}_{\overline{\ca{F}}}/A_{\ca{F}}})\longrightarrow V_p(\Omega^1_{(\ca{O}_{\overline{\ca{F}}},\ca{O}_{\overline{\ca{F}}}\setminus 0)/A_{\ca{F}}})
	\end{align}
	is an isomorphism by \ref{lem:val-log-diff}. Thus, we obtain a canonical morphism of Faltings extensions by \ref{prop:fal-ext-val} and \ref{thm:A-fal-ext},
	\begin{align}\label{eq:prop:fal-ext-val-pt-nolog-2-6}
		\xymatrix{
			0\ar[r]& \widehat{\overline{\ca{F}}}(1)\ar[r]^-{\iota}&V_p(\Omega^1_{\ca{O}_{\overline{\ca{F}}}/A_{\ca{F}}})\ar[r]^-{\jmath}&\widehat{\overline{\ca{F}}}\otimes_{\ca{F}}\Omega^1_{\ca{F}/K}\ar[r]& 0\\
			0\ar[r]& \widehat{\overline{A}}[\frac{1}{p}](1)\ar[r]^-{\iota}\ar[u]&\scr{E}_A\ar[r]^-{\jmath}\ar[u]& \widehat{\overline{A}}\otimes_{\ca{O}_Y}\Omega^1_{Y/K}\ar[r]\ar[u]& 0,
		}
	\end{align}
	which is functorial in $X=\spec(A)\in \nbd^{Y\trm{-}\mrm{adq}}_{\overline{\ca{F}}}(Y/\ca{O}_K)$ by the commutativity of \eqref{eq:prop:fal-ext-val-pt-nolog-2-3}. Taking filtered colimit over $\nbd^{Y\trm{-}\mrm{adq}}_{\overline{\ca{F}}}(Y/\ca{O}_K)$, we obtain a canonical morphism of Faltings extensions
	\begin{align}\label{eq:prop:fal-ext-val-pt-nolog-2-7}
		\xymatrix{
			0\ar[r]& \widehat{\overline{\ca{F}}}(1)\ar[r]^-{\iota}&\colim_{ \nbd^{Y\trm{-}\mrm{adq}}_{\overline{\ca{F}}}(Y/\ca{O}_K)}V_p(\Omega^1_{\ca{O}_{\overline{\ca{F}}}/A_{\ca{F}}})\ar[r]^-{\jmath}&\widehat{\overline{\ca{F}}}\otimes_{\ca{F}}\Omega^1_{\ca{F}/K}\ar[r]& 0\\
			0\ar[r]& \widehat{\overline{\ca{F}}}(1)\ar[r]^-{\iota}\ar@{=}[u]&\scr{E}_{Y\to Y,\overline{y}}\ar[r]^-{\jmath}\ar@{->>}[u]& \widehat{\overline{\ca{F}}}\otimes_{\ca{O}_Y}\Omega^1_{Y/K}\ar[r]\ar@{->>}[u]& 0,
		}
	\end{align} 
	where the surjectivity follows from that of $\ca{F}\otimes_{\ca{O}_Y}\Omega^1_{Y/K}\to \Omega^1_{\ca{F}/K}$. By construction, it is $\gal(\overline{\ca{K}}/\ca{K}^{\mrm{vh}})$-equivariant. The conclusion follows from the canonical isomorphism (see \ref{para:fal-ext-val-func})
	\begin{align}
		\colim_{ \nbd^{Y\trm{-}\mrm{adq}}_{\overline{\ca{F}}}(Y/\ca{O}_K)}V_p(\Omega^1_{\ca{O}_{\overline{\ca{F}}}/A_{\ca{F}}})\iso \scr{E}_{\ca{O}_{\overline{\ca{F}}}}.
	\end{align}
\end{proof}

\begin{mylem}\label{lem:perfd-val-log-1}
	With the notation in {\rm\ref{para:notation-neighborhood-1}} and {\rm\ref{para:notation-neighborhood-2}}, let $\ca{L}$ be an algebraic field extension of $\ca{K}$ contained in $\overline{\ca{K}}$ containing a compatible system of primitive $p$-power roots of unity. Then, the following statements are equivalent: 
	\begin{enumerate}
		\renewcommand{\labelenumi}{{\rm(\theenumi)}}
		\item The subspace $\scr{E}_{Y^{\triv}\to Y,\overline{y}}^{\gal(\overline{\ca{K}}/\ca{L}^{\mrm{vh}})}$ of $\gal(\overline{\ca{K}}/\ca{L}^{\mrm{vh}})$-invariant elements of $\scr{E}_{Y^{\triv}\to Y,\overline{y}}$ \eqref{eq:thm:fal-ext-val-pt-1} has dimension $1+\dim(Y)$ over $\widehat{\ca{F}_{\ca{L}}^{\mrm{h}}}$.\label{item:lem:perfd-val-log-1-1}
		\item The morphism $\jmath: \scr{E}_{Y^{\triv}\to Y,\overline{y}}^{\gal(\overline{\ca{K}}/\ca{L}^{\mrm{vh}})}\to \widehat{\ca{F}_{\ca{L}}^{\mrm{h}}}\otimes_{\ca{O}_Y}\Omega^1_{(Y,\scr{M}_{Y})/K}$ induced by \eqref{eq:thm:fal-ext-val-pt-3} is surjective.\label{item:lem:perfd-val-log-1-2}
		\item The coboundary map $\delta:\widehat{\ca{F}_{\ca{L}}^{\mrm{h}}}\otimes_{\ca{O}_Y}\Omega^1_{(Y,\scr{M}_{Y})/K}\to H^1(\gal(\overline{\ca{K}}/\ca{L}^{\mrm{vh}}),\widehat{\overline{\ca{F}}}(1))$ induced by \eqref{eq:thm:fal-ext-val-pt-3} is zero.\label{item:lem:perfd-val-log-1-3}
	\end{enumerate}
\end{mylem}
\begin{proof}
	As $\ca{L}$ contains a compatible system of primitive $p$-power roots of unity, the $\gal(\overline{\ca{K}}/\ca{L}^{\mrm{vh}})$-representation $\widehat{\overline{\ca{F}}}(1)$ is isomorphic to $\widehat{\overline{\ca{F}}}$. Taking $\gal(\overline{\ca{K}}/\ca{L}^{\mrm{vh}})$-invariants of the Faltings extension \eqref{eq:thm:fal-ext-val-pt-3}, we obtain an exact sequence by Ax-Sen-Tate's theorem \cite[page 417]{ax1970ax} \eqref{eq:para:notation-neighborhood-2-8},
	\begin{align}
		0\longrightarrow \widehat{\ca{F}_{\ca{L}}^{\mrm{h}}}(1)\stackrel{\iota}{\longrightarrow}\scr{E}_{Y^{\triv}\to Y,\overline{y}}^{\gal(\overline{\ca{K}}/\ca{L}^{\mrm{vh}})}\stackrel{\jmath}{\longrightarrow} \widehat{\ca{F}_{\ca{L}}^{\mrm{h}}}\otimes_{\ca{O}_Y}\Omega^1_{(Y,\scr{M}_Y)/K}\stackrel{\delta}{\longrightarrow} H^1(\gal(\overline{\ca{K}}/\ca{L}^{\mrm{vh}}),\widehat{\overline{\ca{F}}}(1)),
	\end{align}
	where $H^1$ is the continuous group cohomology. The conclusion follows immediately from the fact that $\Omega^1_{(Y,\scr{M}_Y)/K}$ is a finite locally free $\ca{O}_Y$-module of rank $\dim(Y)$ (see \ref{lem:log-str-val-pt}).
\end{proof}

\begin{mythm}\label{thm:perfd-val-log-1}
	With the notation in {\rm\ref{para:notation-neighborhood-1}} and {\rm\ref{para:notation-neighborhood-2}}, let $\ca{L}$ be an algebraic field extension of $\ca{K}$ contained in $\overline{\ca{K}}$ containing a compatible system of primitive $p$-power roots of unity. Assume that the equivalent conditions in {\rm\ref{lem:perfd-val-log-1}} hold. Then, $\ca{F}_{\ca{L}}$ is a pre-perfectoid field.
\end{mythm}
\begin{proof}	
	Firstly, we claim that we may assume that $\overline{K}\subseteq \ca{L}$. Note that $K_{\infty}=\bigcup_{n\in\bb{N}}K(\zeta_{p^n})$ is a pre-perfectoid field (see the proof of \cite[13.1]{he2022sen}). Hence, $\ca{O}_{\overline{K}}$ is a filtered colimit of almost finite \'etale $\ca{O}_{K_{\infty}}$-algebras by almost purity (\cite[6.6.2]{gabber2003almost}, see also \cite[7.12]{he2022sen}). Let $\ca{L}'$ be the composite of $\overline{K}$ and $\ca{L}$ in $\overline{\ca{K}}$. It is again an algebraic field extension of $\ca{K}$ in $\overline{\ca{K}}$ and we can apply the construction in $\ref{para:notation-neighborhood-2}$ to $\ca{L}'$. As $\overline{K}$ is ind-\'etale over $K_{\infty}$, $Y^{\ca{L}'}\to \spec(\overline{K})\times_{\spec(K_\infty)}Y^{\ca{L}}$ is a closed immersion (cf. \ref{lem:profet-component}). In particular, we see that $\ca{F}_{\ca{L}'}$ (the residue field of $Y^{\ca{L}'}$ at $y_{\ca{L}'}$) is the composite of $\overline{K}$ and $\ca{F}_{\ca{L}}$ (the residue field of $Y^{\ca{L}}$ at $y_{\ca{L}}$) in $\overline{\ca{F}}$. Therefore, $\ca{O}_{\ca{F}_{\ca{L}'}}$ is a filtered colimit of almost finite \'etale $\ca{O}_{\ca{F}_{\ca{L}}}$-algebras (\cite[\Luoma{5}.7.11]{abbes2016p}). In particular, $\ca{O}_{\ca{F}_{\ca{L}'}}$ is almost weakly \'etale and faithfully flat over $\ca{O}_{\ca{F}_{\ca{L}}}$. We see that it is pre-perfectoid if and only if $\ca{O}_{\ca{F}_{\ca{L}}}$ is so (cf. \cite[5.37.(2)]{he2024coh}). The assumption in the statement implies that the morphism $\jmath: \scr{E}_{Y^{\triv}\to Y,\overline{y}}^{\gal(\overline{\ca{K}}/\ca{L}^{\mrm{vh}})}\to \widehat{\ca{F}_{\ca{L}}^{\mrm{h}}}\otimes_{\ca{O}_Y}\Omega^1_{(Y,\scr{M}_{Y})/K}$ is surjective. Hence, $\jmath: \scr{E}_{Y^{\triv}\to Y,\overline{y}}^{\gal(\overline{\ca{K}}/\ca{L}'^{\mrm{vh}})}\to \widehat{\ca{F}_{\ca{L}'}^{\mrm{h}}}\otimes_{\ca{O}_Y}\Omega^1_{(Y,\scr{M}_{Y})/K}$ is also surjective. Therefore, we can replace $\ca{L}$ by $\ca{L}'$.
	
	Now assume that $\overline{K}\subseteq \ca{L}$. The assumption in the statement implies that the coboundary map $\delta:\widehat{\ca{F}_{\ca{L}}^{\mrm{h}}}\otimes_{\ca{O}_Y}\Omega^1_{(Y,\scr{M}_{Y})/K}\to H^1(\gal(\overline{\ca{K}}/\ca{L}^{\mrm{vh}}),\widehat{\overline{\ca{F}}}(1))$ induced by \eqref{eq:thm:fal-ext-val-pt-3} is zero. Hence, the coboundary map $\delta:\widehat{\ca{F}_{\ca{L}}^{\mrm{h}}}\otimes_{\ca{O}_Y}\Omega^1_{Y/K}\to H^1(\gal(\overline{\ca{K}}/\ca{L}^{\mrm{vh}}),\widehat{\overline{\ca{F}}}(1))$ induced by \eqref{eq:thm:fal-ext-val-pt-3} for the case $Y^{\triv}=Y$ is also zero by \ref{lem:fal-ext-val-pt-nolog-1}. Thus, the morphism $\jmath: \scr{E}_{Y\to Y,\overline{y}}^{\gal(\overline{\ca{K}}/\ca{L}^{\mrm{vh}})}\to \widehat{\ca{F}_{\ca{L}}^{\mrm{h}}}\otimes_{\ca{O}_Y}\Omega^1_{Y/K}$ induced by \eqref{eq:thm:fal-ext-val-pt-3} for the case $Y^{\triv}=Y$ is surjective by \ref{lem:perfd-val-log-1}. Hence, the morphism $\jmath: \scr{E}_{\ca{O}_{\overline{\ca{F}}}}^{\gal(\overline{\ca{K}}/\ca{L}^{\mrm{vh}})}\to \widehat{\ca{F}_{\ca{L}}^{\mrm{h}}}\otimes_{\ca{F}}\Omega^1_{\ca{F}/K}$ induced by \eqref{eq:thm:fal-ext-val-3} is surjective by \ref{lem:fal-ext-val-pt-nolog-2}. We see that $\dim_{\widehat{\ca{F}_{\ca{L}}^{\mrm{h}}}}\scr{E}_{\ca{O}_{\overline{\ca{F}}}}^{\gal(\overline{\ca{F}}/\ca{F}_{\ca{L}}^{\mrm{h}})}=1+\mrm{trdeg}_K(\overline{\ca{F}})$ by \ref{lem:perfd-val}. Notice that we are in the situation of \ref{para:notation-trace} for $\ca{O}_{\overline{\ca{F}}}\leftarrow\ca{O}_{\ca{F}_{\ca{L}}^{\mrm{h}}}\leftarrow \ca{O}_{\overline{K}}\leftarrow \ca{O}_K$. Hence, $\ca{F}_{\ca{L}}^{\mrm{h}}$ is a pre-perfectoid field by \ref{thm:perfd-val}, which implies that $\ca{F}_{\ca{L}}$ is also a pre-perfectoid field since $\ca{O}_{\ca{F}_{\ca{L}}}\to \ca{O}_{\ca{F}_{\ca{L}}^{\mrm{h}}}$ is faithfully flat and ind-\'etale (cf. \cite[5.37.(2)]{he2024coh}).
\end{proof}

\begin{mythm}\label{thm:perfd-val-log-2}
	With the notation in {\rm\ref{para:notation-neighborhood-1}} and {\rm\ref{para:notation-neighborhood-2}}, let $\ca{L}$ be an algebraic field extension of $\ca{K}$ contained in $\overline{\ca{K}}$ containing a compatible system of primitive $p$-power roots of unity, and let $\{t_1,\dots, t_s\}$ be a regular system of parameters of the strict Henselization $\ca{O}_{Y,y}^{\mrm{sh}}$ such that $t_1\cdots t_r=0$ defines the normal crossings divisor $D$ at $\overline{y}$, where $0\leq r\leq s\leq \dim(Y)$. Assume that the equivalent conditions in {\rm\ref{lem:perfd-val-log-1}} hold. Then, $t_1,\dots,t_r$ admit compatible systems of $p$-power roots in $\ca{L}^{\mrm{sh}}$.
\end{mythm}
\begin{proof}
	Assume that $t=t_i$ does not admit a compatible system of $p$-power roots in $\ca{L}^{\mrm{sh}}$ for some $1\leq i\leq r$. Let $(t_{p^n})_{n\in\bb{N}}$ be a compatible system of $p$-power roots of $t$ in $\overline{\ca{K}}$. Then, $\ca{L}^{\mrm{sh}}_\infty=\bigcup_{n\in\bb{N}}\ca{L}^{\mrm{sh}}(t_{p^n})$ is a non-trivial field extension of $\ca{L}^{\mrm{sh}}$. Consider the continuous group homomorphism 
	\begin{align}\label{eq:thm:perfd-val-log-2-1}
		\xi_t:\gal(\overline{\ca{K}}/\ca{L}^{\mrm{vh}})\longrightarrow \bb{Z}_p
	\end{align}
	characterized by $\tau(t_{p^n})=\zeta_{p^n}^{\xi_t(\tau)}t_{p^n}$ for any $\tau\in \gal(\overline{\ca{K}}/\ca{L}^{\mrm{vh}})$ and $n\in\bb{N}$. We remark that there exists an element $\tau_0$ of $\gal(\overline{\ca{K}}/\ca{L}^{\mrm{sh}})\subseteq \gal(\overline{\ca{K}}/\ca{L}^{\mrm{vh}})$ (the inertia group of $Y^{\ca{L}}/\ca{O}_K$ at the geometric valuative point $\spec(\overline{\ca{F}})$) such that $\xi_t(\tau_0)\neq 0$.
	
	By the assumption in the statement, the coboundary map $\delta:\widehat{\ca{F}_{\ca{L}}^{\mrm{h}}}\otimes_{\ca{O}_Y}\Omega^1_{(Y,\scr{M}_{Y})/K}\to H^1(\gal(\overline{\ca{K}}/\ca{L}^{\mrm{vh}}),\widehat{\overline{\ca{F}}}(1))$ induced by \eqref{eq:thm:fal-ext-val-pt-3} is zero. By definition, $\delta(\df\log(t))$ is represented by the $1$-cocycle (see \eqref{eq:thm:perfd-val-log-2-1} and \ref{thm:fal-ext-val-pt}.(\ref{item:thm:fal-ext-val-pt-2}))
	\begin{align}
		\gal(\overline{\ca{K}}/\ca{L}^{\mrm{vh}})\longrightarrow \widehat{\overline{\ca{F}}}(1),\ \tau\mapsto (\tau-1)((\df\log t_{p^n})_{n\in\bb{N}})=\xi_t(\tau)(\zeta_{p^n})_{n\in\bb{N}}.
	\end{align}
	The assumption $\delta=0$ implies that there exists an element $x\in \widehat{\overline{\ca{F}}}$ such that $(\tau-1)(x)=\xi_t(\tau)$ for any $\tau\in \gal(\overline{\ca{K}}/\ca{L}^{\mrm{vh}})$. However, $(\tau_0-1)(x)=0$ as $\tau_0$ fixes $\ca{L}^{\mrm{sh}}$ and thus $\widehat{\overline{\ca{F}}}$ (see \eqref{eq:para:notation-neighborhood-2-9}), but $\xi_t(\tau_0)\neq 0$ by construction. This is a contradiction.
\end{proof}

\section{Induced Sen Operators over Geometric Valuative Points}\label{sec:sen}
Similar to Section \ref{sec:fal-ext}, using the functorial construction \ref{thm:sen-lie-lift-A} of universal Sen actions over adequate algebras, we define universal Sen actions at any geometric valuative point of a smooth variety (see \ref{thm:sen-val-pt}). Finally, we deduce a criterion for pointwise perfectoidness via non-vanishing of the universal geometric Sen action (see \ref{thm:sen-val-perfd}).

\begin{mypara}\label{para:notation-sen-val}	
	In this section, we fix the following data: 
	\begin{enumerate}
		\renewcommand{\labelenumi}{{\rm(\theenumi)}}
		\item a complete discrete valuation field $K$ extension of $\bb{Q}_p$ with perfect residue field,
		\item an irreducible smooth $K$-scheme $Y$ of finite presentation with a normal crossings divisor $D$,
		\item an algebraic closure $\overline{\ca{K}}$ of the fraction field $\ca{K}$ of $Y$,
		\item a point $\overline{y}=\spec(\overline{\ca{F}})$ of the integral closure $Y^{\overline{\ca{K}}}$ of $Y$ in $\overline{\ca{K}}$,
		\item a valuation ring $\ca{O}_{\overline{\ca{F}}}$ of height $1$ extension of $\ca{O}_K$ with fraction field $\overline{\ca{F}}$,
		\item a Galois extension $\ca{L}$ of $\ca{K}$ contained in the maximal unramified extension $\ca{K}_{\mrm{ur}}\subseteq \overline{\ca{K}}$ with respect to $(Y^{\triv},Y)$ (where $Y^{\triv}=Y\setminus D$, see \ref{para:notation-neighborhood-3}) such that $\ca{G}=\gal(\ca{L}/\ca{K})$ is a $p$-adic analytic group (\cite[3.8]{he2022sen}).
	\end{enumerate}
	In particular, we are again in the situation of \ref{para:notation-neighborhood-1} and we adopt the same notation as in \ref{para:notation-neighborhood-2} and \ref{para:notation-neighborhood-3}.
\end{mypara}

\begin{mypara}\label{para:val-fal-ext-dual}
	Consider the Faltings extension \eqref{eq:thm:fal-ext-val-pt-3} of the open immersion $Y^{\triv}\to Y$ over $\ca{O}_K$ at the geometric valuative point $\overline{y}=\spec(\overline{\ca{F}})$,
	\begin{align}\label{eq:para:val-fal-ext-dual-1}
		0\longrightarrow \widehat{\overline{\ca{F}}}(1)\stackrel{\iota}{\longrightarrow}\scr{E}_{Y^{\triv}\to Y,\overline{y}}\stackrel{\jmath}{\longrightarrow} \widehat{\overline{\ca{F}}}\otimes_{\ca{O}_Y}\Omega^1_{(Y,\scr{M}_Y)/K}\longrightarrow 0.
	\end{align}
	As in \ref{rem:A-fal-ext}.(\ref{item:rem:A-fal-ext-3}), taking its dual and a Tate twist, we obtain a canonical exact sequence of finite free $\widehat{\overline{\ca{F}}}$-representations of $\gal(\overline{\ca{K}}/\ca{K}^{\mrm{vh}})$ (the decomposition group of $Y/\ca{O}_K$ at the geometric valuative point $\spec(\overline{\ca{F}})$, see \ref{para:notation-neighborhood-3}),
	\begin{align}\label{eq:para:val-fal-ext-dual-2}
		0\longrightarrow \ho_{\ca{O}_Y}(\Omega^1_{(Y,\scr{M}_Y)/K}(-1),\widehat{\overline{\ca{F}}})\stackrel{\jmath^*}{\longrightarrow} \scr{E}^*_{Y^{\triv}\to Y,\overline{y}}(1)\stackrel{\iota^*}{\longrightarrow}\widehat{\overline{\ca{F}}}\longrightarrow 0,
	\end{align} 
	where $\scr{E}^*_{Y^{\triv}\to Y,\overline{y}}=\ho_{\widehat{\overline{\ca{F}}}}(\scr{E}_{Y^{\triv}\to Y,\overline{y}},\widehat{\overline{\ca{F}}})$. There is a canonical $\gal(\overline{\ca{K}}/\ca{K}^{\mrm{vh}})$-equivariant $\widehat{\overline{\ca{F}}}$-linear Lie algebra structure on $\scr{E}^*_{Y^{\triv}\to Y,\overline{y}}(1)$ associated to the linear form $\iota^*$ given by
	\begin{align}\label{eq:para:val-fal-ext-dual-3}
		[f_1,f_2]=\iota^*(f_1)f_2-\iota^*(f_2)f_1,\ \forall f_1,f_2\in \scr{E}^*_{Y^{\triv}\to Y,\overline{y}}(1).
	\end{align}
	In particular, $\ho_{\ca{O}_Y}(\Omega^1_{(Y,\scr{M}_Y)/K}(-1),\widehat{\overline{\ca{F}}})$ is a Lie ideal of $\scr{E}^*_{Y^{\triv}\to Y,\overline{y}}(1)$, and $\widehat{\overline{\ca{F}}}$ is the quotient by this ideal. It is clear that the induced Lie algebra structures on them are both trivial. Any $\widehat{\overline{\ca{F}}}$-linear splitting of \eqref{eq:para:val-fal-ext-dual-1} identifies $\scr{E}^*_{Y^{\triv}\to Y,\overline{y}}(1)$ with the semi-direct product of Lie algebras of $\widehat{\overline{\ca{F}}}$ acting on $\ho_{\ca{O}_Y}(\Omega^1_{(Y,\scr{M}_Y)/K}(-1),\widehat{\overline{\ca{F}}})$ by multiplication. By the construction of $\scr{E}_{Y^{\triv}\to Y,\overline{y}}$ \eqref{eq:thm:fal-ext-val-pt-1}, we see that
	\begin{align}\label{eq:para:val-fal-ext-dual-4}
		\scr{E}^*_{Y^{\triv}\to Y,\overline{y}}(1)=\lim_{X=\spec(A)\in \nbd^{Y^{\triv}\trm{-}\mrm{adq}}_{\overline{\ca{F}}}(Y/\ca{O}_K)}\widehat{\overline{\ca{F}}}\otimes_{\widehat{\overline{A}}}\scr{E}^*_A(1)
	\end{align}
	is an equality of $\widehat{\overline{\ca{F}}}$-linear Lie algebras (see \ref{rem:A-fal-ext}.(\ref{item:rem:A-fal-ext-3})), where the transition morphisms in the limit are isomorphisms (see \ref{para:fal-ext-val-pt}).
\end{mypara}

\begin{mypara}\label{para:sen-val-pt}
	For any $X=\spec(A)\in \nbd^{Y^{\triv}\trm{-}\mrm{adq}}_{\overline{\ca{F}}}(Y/\ca{O}_K)$ (\ref{prop:aet-neighborhood}) given by an adequate $(K',\ca{O}_{K'},\ca{O}_{\overline{K}})$-triple $(A^{\triv},A,\overline{A})$ for some finite field extension $K'$ of $K$ contained in $\overline{K}$, let $\ca{G}_X$ be the image of the composition of $\gal(\ca{K}_{X,\mrm{ur}}/\ca{K}_X)\to \gal(\ca{K}_{\mrm{ur}}/\ca{K})\to \ca{G}=\gal(\ca{L}/\ca{K})$ and let $\ca{L}_X$ be the corresponding Galois extension of $\ca{K}_X$.
	\begin{align}
		\xymatrix{
			\overline{\ca{K}}&\ca{K}_{X,\mrm{ur}}\ar[l]&\ca{K}_{\mrm{ur}}\ar[l]\\
			&\ca{L}_X\ar[u]&\ca{L}\ar[l]\ar[u]\\
			&\ca{K}_X\ar[u]^-{\ca{G}_X}&\ca{K}\ar[u]_-{\ca{G}}\ar[l]
		}
	\end{align}
	We note that $\ca{L}_X=\ca{K}_X\ca{L}$ and $\ca{G}_X\subseteq \ca{G}$ is an open subgroup of finite index (as $\ca{K}_X$ is finite over $\ca{K}$). Consider the universal Sen action of the $p$-adic analytic Galois extension $\ca{L}_X$ of $\ca{K}_X$ \eqref{eq:thm:sen-lie-lift-A-1},
	\begin{align}\label{eq:para:sen-val-pt-2}
		\varphi_{\sen}|_{\ca{G}_X}:\scr{E}^*_A(1)\longrightarrow \widehat{\overline{A}}[1/p]\otimes_{\bb{Q}_p}\lie(\ca{G}_X).
	\end{align}
	
	Any morphism $X'=\spec(A')\to X=\spec(A)$ in $\nbd^{Y^{\triv}\trm{-}\mrm{adq}}_{\overline{\ca{F}}}(Y/\ca{O}_K)$ induces a morphism of $(K,\ca{O}_K,\ca{O}_{\overline{K}})$-triples $(A_{\triv},A,\overline{A})\to (A'_{\triv},A',\overline{A'})$ by \ref{para:notation-neighborhood-3}. Notice that $\overline{A}\to \overline{A'}$ is injective (since both of them are contained in $\overline{\ca{K}}$ by definition \ref{para:notation-neighborhood-3}) and that $\ca{K}_{X'}\otimes_{\ca{K}_X}\Omega^1_{\ca{K}_X/K}\to \Omega^1_{\ca{K}_{X'}/K}$ is an isomorphism (as $\ca{K}_{X'}$ is finite separable over $\ca{K}_X$). Thus, this morphism induces a canonical isomorphism of universal Sen actions by \ref{rem:sen-lie-lift-A}.(\ref{item:rem:sen-lie-lift-A-2}),
	\begin{align}\label{eq:para:sen-val-pt-3}
		\xymatrix{
			\widehat{\overline{\ca{F}}}\otimes_{\widehat{\overline{A'}}}\scr{E}^*_{A'}(1)\ar[rrr]^-{\id_{\widehat{\overline{\ca{F}}}}\otimes_{\widehat{\overline{A'}}}\varphi_{\sen}|_{\ca{G}_{X'}}}\ar[d]^-{\wr}&&&\widehat{\overline{\ca{F}}}\otimes_{\bb{Q}_p}\lie(\ca{G}_{X'})\ar[d]^-{\wr}\\
			\widehat{\overline{\ca{F}}}\otimes_{\widehat{\overline{A}}}\scr{E}^*_{A}(1)
			\ar[rrr]^-{\id_{\widehat{\overline{\ca{F}}}}\otimes_{\widehat{\overline{A}}}\varphi_{\sen}|_{\ca{G}_X}}&&&\widehat{\overline{\ca{F}}}\otimes_{\bb{Q}_p}\lie(\ca{G}_X).
		}
	\end{align}
	where we used the fact that $\ca{G}_{X'}\subseteq \ca{G}_X$ are open subgroups of $\ca{G}$.
	
	Taking cofiltered limit over $\nbd^{Y^{\triv}\trm{-}\mrm{adq}}_{\overline{\ca{F}}}(Y/\ca{O}_K)$ (\ref{prop:aet-neighborhood}), we obtain a canonical homomorphism of $\widehat{\overline{\ca{F}}}$-linear Lie algebras by \eqref{eq:para:val-fal-ext-dual-4},
	\begin{align}\label{eq:para:sen-val-pt-4}
		\varphi_{\sen}|_{\ca{G},\overline{y}}: \scr{E}^*_{Y^{\triv}\to Y,\overline{y}}(1)\longrightarrow \widehat{\overline{\ca{F}}}\otimes_{\bb{Q}_p}\lie(\ca{G}).
	\end{align}
\end{mypara}
\begin{mythm}\label{thm:sen-val-pt}
	The canonical homomorphism of $\widehat{\overline{\ca{F}}}$-linear Lie algebras defined in {\rm\ref{para:sen-val-pt}},
	\begin{align}\label{eq:thm:sen-val-pt-1}
		\varphi_{\sen}|_{\ca{G},\overline{y}}: \scr{E}^*_{Y^{\triv}\to Y,\overline{y}}(1)\longrightarrow \widehat{\overline{\ca{F}}}\otimes_{\bb{Q}_p}\lie(\ca{G})
	\end{align}
	is $\gal(\overline{\ca{K}}/\ca{K}^{\mrm{vh}})$-equivariant with respect to the canonical action on $\scr{E}^*_{Y^{\triv}\to Y,\overline{y}}(1)$ defined in {\rm\ref{para:val-fal-ext-dual}}, the canonical action on $\widehat{\overline{\ca{F}}}$ via the canonical homomorphism $\gal(\overline{\ca{K}}/\ca{K}^{\mrm{vh}})\to \gal(\overline{\ca{F}}/\ca{F}^{\mrm{h}})$ \eqref{eq:para:notation-neighborhood-2-6}, and the adjoint action on $\lie(\ca{G})$ {\rm(\cite[3.15]{he2022sen})} via the canonical homomorphism $\gal(\overline{\ca{K}}/\ca{K}^{\mrm{vh}})\to \gal(\ca{K}_{\mrm{ur}}/\ca{K})$.
\end{mythm}
\begin{proof}
	Recall that by \ref{lem:galois-action-adq-nbd}, there is a canonical action of $\gal(\overline{\ca{K}}/\ca{K}^{\mrm{vh}})$ on $\nbd^{Y^{\triv}\trm{-}\mrm{adq}}_{\overline{\ca{F}}}(Y/\ca{O}_K)$. Moreover, its construction implies that there is a canonical morphism $X^{\sigma}\to X$ in $\nbd^{Y^{\triv}\trm{-}\mrm{adq}}_{\overline{\ca{F}}}(Y/\ca{O}_K)$ for any $\sigma\in \gal(\overline{\ca{K}}/\ca{K}^{\mrm{vh}})$. Thus, the conclusion follows from again the functoriality \ref{rem:sen-lie-lift-A}.(\ref{item:rem:sen-lie-lift-A-2}) of the universal Sen actions with respect to $X^\sigma\to X$ (cf. \ref{thm:fal-ext-val-pt}).
\end{proof}

\begin{mydefn}\label{defn:sen-val-pt}
	We call the homomorphism \eqref{eq:thm:sen-val-pt-1} the \emph{universal Sen action of the $p$-adic analytic quotient $\ca{G}$ of $\gal(\ca{K}_{\mrm{ur}}/\ca{K})$} (or \emph{of the $p$-adic analytic Galois extension $\ca{L}$ of $\ca{K}$}) at the geometric valuative point $\overline{y}=\spec(\overline{\ca{F}})$ of the open immersion $Y^{\triv}\to Y$ over $\ca{O}_K$.
\end{mydefn}

\begin{myrem}\label{rem:sen-val-pt}
	The construction of \eqref{eq:thm:sen-val-pt-1} is functorial in the data in \ref{para:notation-sen-val}	. More precisely, consider another set of data in \ref{para:notation-sen-val}	 by adding a prime superscript. Then, any commutative diagram of $\bb{Z}_p$-schemes extending the structural morphisms
	\begin{align}\label{eq:rem:sen-val-pt-1}
		\xymatrix{
			\spec(\ca{O}_{\overline{\ca{F}'}})\ar[ddd]\ar@/^1pc/[rrr]&\spec(\overline{\ca{F}'})\ar[l]\ar[r]\ar[d]&\spec(\overline{\ca{F}})\ar[r]\ar[d]&\spec(\ca{O}_{\overline{\ca{F}}})\ar[ddd]\\
			&Y'^{\overline{\ca{K}'}}\ar[r]\ar[d]&Y^{\overline{\ca{K}}}\ar[d]&\\
			&Y'\ar[r]\ar[d]&Y\ar[d]&\\
			\spec(\ca{O}_{K'})\ar@/_1pc/[rrr]&\spec(K')\ar[l]\ar[r]&\spec(K)\ar[r]&\spec(\ca{O}_K)
		}
	\end{align}
	such that $Y'^{\triv}$ is over $Y^{\triv}$, that $\spec(\overline{\ca{K}'})\in Y'^{\overline{\ca{K}'}}$ is over $\spec(\overline{\ca{K}})\in Y^{\overline{\ca{K}}}$, that $\ca{L}\subseteq \ca{L}'$ via the inclusion $\overline{\ca{K}}\subseteq \overline{\ca{K}'}$, and that $\ca{K}'\otimes_{\ca{K}}\Omega^1_{\ca{K}/K}\to \Omega^1_{\ca{K}'/K'}$ is injective, induces a canonical commutative diagram by the same arguments of \ref{rem:fal-ext-val-pt} using \ref{rem:sen-lie-lift-A}.(\ref{item:rem:sen-lie-lift-A-2}),
	\begin{align}\label{eq:rem:sen-val-pt-2}
		\xymatrix{
			\scr{E}^*_{Y'^{\triv}\to Y',\overline{y'}}(1)\ar[rrr]^-{\varphi_{\sen}|_{\ca{G}',\overline{y'}}}\ar[d]&&& \widehat{\overline{\ca{F}'}}\otimes_{\bb{Q}_p}\lie(\ca{G}')\ar[d]\\
			\widehat{\overline{\ca{F}'}}\otimes_{\widehat{\overline{\ca{F}}}}\scr{E}^*_{Y^{\triv}\to Y,\overline{y}}(1)\ar[rrr]^-{\id_{\widehat{\overline{\ca{F}'}}}\otimes_{\widehat{\overline{\ca{F}}}}\varphi_{\sen}|_{\ca{G},\overline{y}}}&&& \widehat{\overline{\ca{F}'}}\otimes_{\bb{Q}_p}\lie(\ca{G}).
		}
	\end{align}	
	In particular, if $K'$ is finite over $K$, $(Y',\scr{M}_{Y'})$ is \'etale over $(Y,\scr{M}_Y)$, and $\ca{L}'$ is finite over $\ca{K}'\ca{L}$, then the vertical homomorphisms in \eqref{eq:rem:sen-val-pt-2} are isomorphisms (cf. \eqref{eq:rem:A-fal-ext-8}).
\end{myrem}

\begin{myrem}\label{rem:sen-val-pt-2}
	One can globalize the construction of Sen actions over adequate algebras \ref{thm:sen-brinon-A} to $Y$. More precisely, one can consider the Faltings site fibred over the admissibly \'etale site of $Y$. Then, the Faltings extension globalizes as an exact sequence of vector bundles over the Faltings site, and any vector bundle admits a canonical Lie algebra action of the twisted dual of the Faltings extension. Although this site theoretic perspective is adopted everywhere in this article, we actually do not need this formulation for any proof. We plan to develop it in the future.
\end{myrem}

\begin{myprop}\label{prop:val-geom-sen-nonzero}
	Let $L$ be a Galois extension of $K$ contained in $\ca{L}\cap\overline{K}$, $\ca{G}^{\mrm{geo}}=\gal(\ca{L}/L\ca{K})$ and $\ca{G}^{\mrm{ari}}=\gal(L\ca{K}/\ca{K})$. Then, $\varphi_{\sen}|_{\ca{G},\overline{y}}: \scr{E}^*_{Y^{\triv}\to Y,\overline{y}}(1)\to \widehat{\overline{\ca{F}}}\otimes_{\bb{Q}_p}\lie(\ca{G})$ induces a morphism of exact sequences of $\widehat{\overline{\ca{F}}}$-linear Lie algebras,
	\begin{align}\label{eq:prop:val-geom-sen-nonzero-1}
		\xymatrix{
			0\ar[r]& \ho_{\ca{O}_Y}(\Omega^1_{(Y,\scr{M}_Y)/K}(-1),\widehat{\overline{\ca{F}}})\ar[r]^-{\jmath^*}\ar[d]^-{\varphi^{\mrm{geo}}_{\sen}|_{\ca{G},\overline{y}}}& \scr{E}^*_{Y^{\triv}\to Y,\overline{y}}(1)\ar[r]^-{\iota^*}\ar[d]^-{\varphi_{\sen}|_{\ca{G},\overline{y}}}&\widehat{\overline{\ca{F}}}\ar[r]\ar[d]^-{\varphi^{\mrm{ari}}_{\sen}|_{\ca{G},\overline{y}}}& 0\\
			0\ar[r]& \widehat{\overline{\ca{F}}}\otimes_{\bb{Q}_p}\lie(\ca{G}^{\mrm{geo}})\ar[r]& \widehat{\overline{\ca{F}}}\otimes_{\bb{Q}_p}\lie(\ca{G})\ar[r]&\widehat{\overline{\ca{F}}}\otimes_{\bb{Q}_p}\lie(\ca{G}^{\mrm{ari}})\ar[r]& 0.
		}
	\end{align}
	Moreover, the induced homomorphism $\varphi^{\mrm{ari}}_{\sen}|_{\ca{G},\overline{y}}$ is not zero if and only if the inertia subgroup of $L/K$ is infinite, and in this case $\varphi^{\mrm{ari}}_{\sen}|_{\ca{G},\overline{y}}$ admits an $\widehat{\overline{\ca{F}}}$-linear retraction.
\end{myprop}
\begin{proof}
	It follows directly from \ref{para:sen-val-pt} and \ref{prop:geom-sen-nonzero}.
\end{proof}

\begin{mythm}\label{thm:sen-val-perfd}
	Let $\{t_1,\dots, t_s\}$ be a regular system of parameters of the strict Henselization $\ca{O}_{Y,y}^{\mrm{sh}}$ such that $t_1\cdots t_r=0$ defines the normal crossings divisor $D$ at $\overline{y}$, where $0\leq r\leq s\leq\dim(Y)$, and let $y_{\ca{L}}=\spec(\ca{F}_{\ca{L}})\in Y^{\ca{L}}$ be the image of $\overline{y}\in Y^{\overline{\ca{K}}}$ {\rm(\ref{para:notation-neighborhood-2})}. Assume that $\ca{L}$ contains a compatible system of primitive $p$-power roots of unity $(\zeta_{p^n})_{n\in\bb{N}}$ and that the restriction of the universal Sen action \eqref{eq:thm:sen-val-pt-1} (called the \emph{universal geometric Sen action}),
	\begin{align}\label{eq:thm:sen-val-perfd-1}
		\varphi^{\mrm{geo}}_{\sen}|_{\ca{G},\overline{y}}: \ho_{\ca{O}_Y}(\Omega^1_{(Y,\scr{M}_Y)/K}(-1),\widehat{\overline{\ca{F}}})\longrightarrow \widehat{\overline{\ca{F}}}\otimes_{\bb{Q}_p}\lie(\ca{G}),
	\end{align}
	is injective. Then, $\ca{F}_{\ca{L}}$ is a pre-perfectoid field with respect to the valuation ring $\ca{O}_{\ca{F}_{\ca{L}}}=\ca{F}_{\ca{L}}\cap\ca{O}_{\overline{\ca{F}}}$, and $t_1,\dots,t_r$ admit compatible systems of $p$-power roots in the strict Henselization $\ca{O}_{Y^{\ca{L}},y_{\ca{L}}}^{\mrm{sh}}$ of $Y^{\ca{L}}$ at $\overline{y}$.
\end{mythm}
\begin{proof}
	 Taking $L=\bigcup_{n\in\bb{N}}K(\zeta_{p^n})$ in \ref{prop:val-geom-sen-nonzero}, we deduce that $\varphi_{\sen}|_{\ca{G},\overline{y}}: \scr{E}^*_{Y^{\triv}\to Y,\overline{y}}(1)\to \widehat{\overline{\ca{F}}}\otimes_{\bb{Q}_p}\lie(\ca{G})$ is injective. Taking dual and Tate twist, we see that
	 \begin{align}\label{eq:thm:sen-val-perfd-2}
	 	\widehat{\overline{\ca{F}}}\otimes_{\bb{Q}_p}\lie(\ca{G})^*(1)\longrightarrow \scr{E}_{Y^{\triv}\to Y,\overline{y}}
	 \end{align}
	 is surjective, where $\lie(\ca{G})^*=\ho_{\bb{Q}_p}(\lie(\ca{G}),\bb{Q}_p)$. Note that it is $\gal(\overline{\ca{K}}/\ca{K}^{\mrm{vh}})$-equivariant by \ref{thm:sen-val-pt}. By the functoriality of \eqref{eq:para:notation-neighborhood-2-9}, there is a canonical commutative diagram of fields
	 \begin{align}\label{eq:thm:sen-val-perfd-3}
	 	\xymatrix{
	 		\overline{\ca{K}}&\ca{L}^{\mrm{vh}}\ar[l]&\ca{L}\ar[l]\\
	 		&\ca{K}^{\mrm{vh}}\ar[u]&\ca{K}.\ar[l]\ar[u]
	 	}
	 \end{align}
	 Since $\gal(\overline{\ca{K}}/\ca{L}^{\mrm{vh}})$ fixes $\ca{G}$ and $\widehat{\ca{F}_{\ca{L}}^{\mrm{h}}}=(\widehat{\overline{\ca{F}}})^{\gal(\overline{\ca{F}}/\ca{F}_{\ca{L}}^{\mrm{h}})}$ by Ax-Sen-Tate's theorem \eqref{eq:para:notation-neighborhood-2-8}, we see that 
	 \begin{align}\label{eq:thm:sen-val-perfd-4}
	 	\widehat{\overline{\ca{F}}}\otimes_{\widehat{\ca{F}_{\ca{L}}^{\mrm{h}}}}(\widehat{\overline{\ca{F}}}\otimes_{\bb{Q}_p}\lie(\ca{G})^*(1))^{\gal(\overline{\ca{K}}/\ca{L}^{\mrm{vh}})}\longrightarrow \widehat{\overline{\ca{F}}}\otimes_{\bb{Q}_p}\lie(\ca{G})^*(1)
	 \end{align}
	 is an isomorphism. In particular, the surjectivity of \eqref{eq:thm:sen-val-perfd-2} implies that
	 \begin{align}
	 	\widehat{\overline{\ca{F}}}\otimes_{\widehat{\ca{F}_{\ca{L}}^{\mrm{h}}}}\scr{E}_{Y^{\triv}\to Y,\overline{y}}^{\gal(\overline{\ca{K}}/\ca{L}^{\mrm{vh}})}\longrightarrow\scr{E}_{Y^{\triv}\to Y,\overline{y}}
	 \end{align}
	 is surjective. Hence, we see that the equivalent conditions of \ref{lem:perfd-val-log-1} are satisfied. Then, the conclusion follows directly from \ref{thm:perfd-val-log-1} and \ref{thm:perfd-val-log-2}.
\end{proof}

\section{Vanishing of \'Etale Cohomology in Higher Degrees}\label{sec:vanish}
We show that pointwise perfectoidness with $p$-infinite ramification at boundary points suffices for the vanishing of \'etale cohomology in higher degrees (see \ref{thm:vanish}). One key ingredient is to compare \'etale cohomology of a regular variety with its compactification by results on $K(\pi,1)$-schemes and Ahbyankar's lemma (see \ref{prop:open-et-coh}). The other key ingredient is Faltings' main $p$-adic comparison theorem for pro-schemes (see \ref{thm:fal-comp}) in order to relate the \'etale cohomology with the cohomology of Riemann-Zariski spaces (see \ref{cor:fal-comp-perfd}). We refer to \cite[\textsection7]{he2024coh} for basic definitions and properties of Faltings ringed sites.

\begin{mydefn}[{cf. \cite[3.3]{achinger2015kpi1}, \cite[2.2.2]{abbes2020suite}}]\label{defn:K-pi-1}
	Let $Y$ be a coherent scheme, $\rho:Y_\et\to Y_\fet$ the canonical morphism from the \'etale site to the finite \'etale site of $Y$. We say that $Y$ is $K(\pi,1)$ if for any finite locally constant abelian sheaf $\bb{L}$ on $Y_\et$, $\rr^q\rho_*\bb{L}=0$ for any integer $q>0$.
\end{mydefn}
We remark that the pullback functor $\rho^{-1}$ induces an equivalence $\mbf{LocSys}(Y_\fet)\to \mbf{LocSys}(Y_\et)$ between the categories of finite locally constant abelian sheaves with quasi-inverse $\rho_*$ (\cite[5.2]{he2024falmain}). In \cite{achinger2015kpi1} and \cite{abbes2020suite}, the authors impose an extra (unnecessary) condition that $Y$ has finitely many connected components to guarantee this result. Therefore, for such a coherent scheme $Y$, our definition of ``$K(\pi,1)$-scheme" coincides with the definition of ``$K(\pi,1)$-scheme for $\bb{P}$-torsion abelian coefficients" in \cite[3.3]{achinger2015kpi1} and \cite[2.2.2]{abbes2020suite}, where $\bb{P}$ is the set of prime numbers. In particular, for coherent $\bb{Q}$-schemes with finitely many connected components (which we will focus on later), our definition of ``$K(\pi,1)$-scheme" coincides with that in \cite[3.3]{achinger2015kpi1} and \cite[2.2.2]{abbes2020suite}. The following basic properties \ref{lem:K-pi-1} and \ref{lem:K-pi-1-limit} on $K(\pi,1)$-schemes extend naturally to our general setting.

\begin{mylem}[{cf. \cite[2.2.4]{abbes2020suite}}]\label{lem:K-pi-1}
	Let $f:Y'\to Y$ be a finite \'etale morphism of coherent schemes. If $Y$ is $K(\pi,1)$, then so is $Y'$. The converse is true if $f$ is surjective.
\end{mylem}
\begin{proof}
	For any finite locally constant abelian sheaf $\bb{L}'$ on $Y'_\et$, we have $f_{\et*}\bb{L}'=\rr f_{\et*}\bb{L}'$ (as $f$ is finite) and it is still a finite locally constant abelian sheaf on $Y_\et$ (\cite[\href{https://stacks.math.columbia.edu/tag/095B}{095B}]{stacks-project}). On the other hand, as $f$ is finite \'etale, there is a finite \'etale covering $V\to Y$ such that $V\times_YY'$ is a finite disjoint union of copies of $V$ (\cite[5.1]{he2024falmain}). In particular, for any abelian sheaf $\ca{M}'$ on $Y'_{\fet}$, we have $f_{\fet *}\ca{M}'=\rr f_{\fet *}\ca{M}'$ and $f_{\fet}^{-1}f_{\fet *}\ca{M}'\to \ca{M}'$ is surjective.
	\begin{align}
		\xymatrix{
			Y'_\et\ar[d]_-{f_\et}\ar[r]^-{\rho'}&Y'_\fet\ar[d]^-{f_\fet}\\
			Y_\et\ar[r]^-{\rho}&Y_\fet
		}
	\end{align}
	Hence, the condition that $Y$ is $K(\pi,1)$ implies that for any integer $q>0$,
	\begin{align}
		0=\rr^q\rho_*f_{\et *}\bb{L}'=H^q(\rr\rho_*\rr f_{\et *}\bb{L}')=H^q(\rr f_{\fet *}\rr\rho'_*\bb{L}')=f_{\fet *}\rr^q\rho'_*\bb{L}'.
	\end{align}
	This further implies that $\rr^q\rho'_*\bb{L}'=0$. Hence, $Y'$ is also $K(\pi,1)$.
	
	Conversely, if $Y'$ is $K(\pi,1)$ and $f$ is finite \'etale surjective, then we see that $\rr^q\rho_*\bb{L}=0$ for any finite locally constant abelian sheaf $\bb{L}$ on $Y_\et$ by restricting to the covering $Y'\to Y$ in $Y_\fet$. Hence, $Y$ is also $K(\pi,1)$.
\end{proof}

\begin{mylem}[{cf. \cite[2.2.5]{abbes2020suite}}]\label{lem:K-pi-1-limit}
	Let $(Y_\lambda)_{\lambda\in\Lambda}$ be a directed inverse system of coherent schemes with affine transition morphisms. If each $Y_\lambda$ is $K(\pi,1)$, then so is $Y=\lim_{\lambda\in\Lambda}Y_\lambda$.
\end{mylem}
\begin{proof}
	Any finite locally constant abelian sheaf $\bb{L}$ on $Y_\et$ is represented by a finite \'etale abelian group $Y$-scheme $V$ (\cite[\href{https://stacks.math.columbia.edu/tag/03RV}{03RV}]{stacks-project}). By \cite[8.8.2, 8.10.5]{ega4-3} and \cite[17.7.8]{ega4-4}, there exists an index $\lambda_0\in\Lambda$ and a finite \'etale abelian group $Y_{\lambda_0}$-scheme $V_{\lambda_0}$ such that $V=Y\times_{Y_{\lambda_0}}V_{\lambda_0}$. For any $\lambda\in\Lambda_{\geq \lambda_0}$, let $\bb{L}_\lambda$ be the finite locally constant abelian sheaf on $Y_{\lambda,\et}$ represented by $V_\lambda=Y_\lambda\times_{Y_{\lambda_0}}V_{\lambda_0}$. Since $Y_\et=\lim_{\lambda\in\Lambda}Y_{\lambda,\et}$ and $Y_\fet=\lim_{\lambda\in\Lambda}Y_{\lambda,\fet}$ (see \cite[\Luoma{7}.5.6]{sga4-2}), we have
	\begin{align}
		\rr^q\rho_*\bb{L}=\colim_{\lambda\in\Lambda_{\geq \lambda_0}} \varphi_\lambda^{-1}\rr^q\rho_{\lambda*}\bb{L}_{\lambda}=0
	\end{align}
	for any integer $q>0$ by \cite[\Luoma{6}.8.7.3]{sga4-2}, where $\varphi_\lambda:Y_\fet\to Y_{\lambda,\fet}$ is the canonical morphism of sites, which shows that $Y$ is $K(\pi,1)$.
\end{proof}

\begin{mylem}\label{lem:open-et-coh}
	Let $(R,\ak{m})$ be a strictly Henselian regular local ring over $\bb{Q}$ with a regular system of parameters $\{t_1,\dots,t_s\}$, $Y=\spec(R)$. Let $V=\spec(R')$ be an affine scheme pro-finite \'etale {\rm(\cite[7.13]{he2024coh})} over $Y^{\triv}=\spec(R[1/t_1\cdots t_r])$ for some integer $0\leq r\leq s$.
	\begin{enumerate}
		\renewcommand{\labelenumi}{{\rm(\theenumi)}}
		\item If $t_1,\dots,t_r$ admit compatible systems of roots of arbitrary order in $R'$, then for any finite locally constant abelian sheaf $\bb{L}$ on the \'etale site $V_\et$, we have $H^q(V_\et,\bb{L})=0$ for any integer $q>0$.\label{item:lem:open-et-coh-1}
		\item If $t_1,\dots,t_r$ admit compatible systems of $p$-power roots in $R'$, then for any finite locally constant abelian sheaf $\bb{L}$ killed by a power of $p$ on the \'etale site $V_\et$, we have $H^q(V_\et,\bb{L})=0$ for any integer $q>0$.\label{item:lem:open-et-coh-2}
	\end{enumerate}
\end{mylem}
\begin{proof}
	For any $n\in\bb{N}_{>0}$, we define a $Y$-scheme
	\begin{align}
		Y_n=\spec\nolimits_{\ca{O}_Y}(\ca{O}_Y[T_1,\dots,T_r]/(T_1^n-t_1,\dots,T_r^n-t_r)).
	\end{align}
	Notice that $(Y_n)_{n\in\bb{N}_{>0}}$ forms a directed inverse system over $\bb{N}_{>0}$ whose transition morphisms are finite free and \'etale over $Y^{\triv}$. We put $Y_\infty=\lim_{n\in\bb{N}_{>0}}Y_n$ and $Y_{p^\infty}=\lim_{n\in\bb{N}}Y_{p^n}$.
	
	Notice that since $(Y,\scr{M}_Y)$ (endowed with the compactifying log structure associated to the open immersion $Y^{\triv}\to Y$) is regular (\cite[\Luoma{3}.1.11.9]{ogus2018log}), $Y^{\triv}$ is $K(\pi,1)$ by \cite[8.1]{achinger2015kpi1} and \cite[2.2.7]{abbes2020suite}. Thus, the pro-finite \'etale $Y^{\triv}$-scheme $V$ is also $K(\pi,1)$ by \ref{lem:K-pi-1} and \ref{lem:K-pi-1-limit}.
	
	(\ref{item:lem:open-et-coh-1}) In this case, $V\to Y$ factors through $Y_\infty$. Let $W$ be the integral closure of $Y$ in $V$ (so that $V=Y^{\triv}\times_YW$). By Ahbyankar's lemma \cite[8.21]{he2024coh}, the canonical morphism of sites $j:V_\fet\to W_\fet$ is an equivalence (see \cite[8.22]{he2024coh}). Hence, we have $\rr\Gamma(V_\et,\bb{L})=\rr\Gamma(V_\fet,\rho_*\bb{L})=\rr\Gamma(W_\fet,j_*\rho_*\bb{L})$. Notice that $j_*\rho_*\bb{L}$ is a finite locally constant abelian sheaf on $W_\fet$ (\cite[5.2]{he2024falmain}). Since $W$ is integral over the strictly Henselian local scheme $Y$, it is the spectrum of a filtered colimit of finite products of strictly Henselian local rings local over $R$. We see that $H^q(W_\fet,j_*\rho_*\bb{L})=0$ for any integer $q>0$ by a limit argument as in \ref{lem:K-pi-1-limit}, which proves (\ref{item:lem:open-et-coh-1}).
	
	(\ref{item:lem:open-et-coh-2}) As $V$ is $K(\pi,1)$, it suffices to show that $H^q(V_\fet,\bb{L})=0$ for any integer $q>0$. We put $V_\infty=Y_\infty\times_{Y_{p^\infty}}V$. Then, we have $H^q(V_{\infty,\fet},\bb{L}|_{V_\infty})=0$ for any integer $q>0$ by (\ref{item:lem:open-et-coh-1}). Thus, $H^q(V_\fet,\bb{L})=H^q(V_\profet,\bb{L})$ (cf. \cite[7.32]{he2024coh}) is computed by the \v Cech cohomology $\check{H}^q(\{V_\infty\to V\},\bb{L})$.
	
	Note that $Y^{\triv}$ is an integral regular scheme. Fixing a geometric point $\overline{y}$ of $Y^{\triv}$, recall that the fibre functor induces an equivalence between the category of pro-finite \'etale $Y^{\triv}$-schemes $Y^{\triv}_\profet$ and the category of profinite $\pi_1(Y^{\triv},\overline{y})$-sets (\cite[3.5]{scholze2013hodge}). We put $G=\lim_{n\in(\bb{N}_{>0}\setminus p\bb{N})}\bb{Z}/n\bb{Z}$ endowed with trivial $\pi_1(Y^{\triv},\overline{y})$-action. Notice that $Y_\infty$ is a $(\prod_{i=1}^rG)$-torsor over $Y_{p^\infty}$ and there is a canonical isomorphism of profinite $\pi_1(Y^{\triv},\overline{y})$-sets
	\begin{align}
		(Y_\infty\times_{Y_{p^{\infty}}}Y_\infty)_{\overline{y}}\cong \left(\prod_{i=1}^rG\right)\times (Y_\infty)_{\overline{y}}.
	\end{align}
	One checks easily that the \v Cech cohomology $\check{H}^q(\{V_\infty\to V\},\bb{L})$ is computed by the continuous group cohomology $H^q(\prod_{i=1}^rG,H^0(V_{\infty,\fet},\bb{L}|_{V_\infty}))$. Since $\bb{L}$ is assumed to be killed by a power of $p$ and the $p$-cohomological dimension of $G$ is zero (\cite[Page 19, Corollary 1]{serre2002galois}), we see that $H^q(\prod_{i=1}^rG,H^0(V_{\infty,\fet},\bb{L}|_{V_\infty}))=0$ for any integer $q>0$.
\end{proof}

\begin{myprop}\label{prop:open-et-coh}
	Let $Y^{\triv}\to Y$ be an open immersion of Noetherian regular schemes over $\bb{Q}$ whose complement $D=Y\setminus Y^{\triv}$ is a normal crossings divisor on $Y$, $V$ a pro-finite \'etale $Y^{\triv}$-scheme, $W$ the integral closure of $Y$ in $V$ (so that $V=Y^{\triv}\times_YW$), $j:V\to W$ the open immersion. 
	\begin{enumerate}
		\renewcommand{\labelenumi}{{\rm(\theenumi)}}
		\item Assume that for any geometric point $\overline{y}$ of $W$, there exists a regular system of parameters $\{t_1,\dots,t_s\}$ of the strict Henselization of $Y$ at $\overline{y}$ such that $D$ is defined by $t_1\cdots t_r=0$ for some integer $0\leq r\leq s$ and that $t_1,\dots,t_r$ admit compatible systems of roots of arbitrary order in the strict Henselization of $W$ at $\overline{y}$. Then, for any finite locally constant abelian sheaf $\bb{L}$ on the \'etale site $V_\et$, we have $\rr^qj_{\et*}\bb{L}=0$ for any integer $q>0$.\label{item:prop:open-et-coh-1}
		\item Assume that for any geometric point $\overline{y}$ of $W$, there exists a regular system of parameters $\{t_1,\dots,t_s\}$ of the strict Henselization of $Y$ at $\overline{y}$ such that $D$ is defined by $t_1\cdots t_r=0$ for some integer $0\leq r\leq s$ and that $t_1,\dots,t_r$ admit compatible systems of $p$-power roots in the strict Henselization of $W$ at $\overline{y}$. Then, for any finite locally constant abelian sheaf $\bb{L}$ killed by a power of $p$ on the \'etale site $V_\et$, we have $\rr^qj_{\et*}\bb{L}=0$ for any integer $q>0$.\label{item:prop:open-et-coh-2}
	\end{enumerate}
\end{myprop}
\begin{proof}
	It suffices to show that for any geometric point $\overline{y}$ of $W$, we have $(\rr^qj_{\et*}\bb{L})_{\overline{y}}=0$. Thus, after replacing $Y$ by its strict Henselization at $\overline{y}$, we may assume that $Y$ is strictly Henselian local and that $\overline{y}\to Y$ is local. Since $W$ is integral over $Y$, it is the spectrum of a filtered colimit of finite products of strictly Henselian local rings finite and local over $Y$. Thus, the localization $W_{(\overline{y})}$ of $W$ at the image of $\overline{y}\to W$ is strictly Henselian local and integral over $Y$. Moreover, $V_{(\overline{y})}=V\times_WW_{(\overline{y})}=Y^{\triv}\times_YW_{(\overline{y})}$ is again pro-finite \'etale over $Y^{\triv}$. Thus, $(\rr^qj_{\et*}\bb{L})_{\overline{y}}=(\rr^qj_{\et*}\bb{L})(W_{(\overline{y})})=H^q(V_{(\overline{y}),\et},\bb{L})=0$ by \ref{lem:open-et-coh}, which implies that $\rr^qj_{\et*}\bb{L}=0$ in general.
\end{proof}

\begin{mypara}\label{para:notation-fal-comp}
	Let $L$ be an algebraically closed valuation field of height $1$ extension of $\bb{Q}_p$. We put
	\begin{align}\label{eq:para:notation-fal-comp-1}
		\eta=\spec(L), \quad S=\spec(\ca{O}_L),\quad s=\spec(\ca{O}_L/\ak{m}_L).
	\end{align}
	Similarly as in \cite[10.1]{he2024purity}, let $\schqcqs_{/S}$ be the category of coherent $S$-schemes and let $\pro(\schqcqs_{/S})$ the category of pro-objects of $\schqcqs_{/S}$ (\cite[\Luoma{1}.8.10]{sga4-1}). In other words, an object of $\pro(\schqcqs_{/S})$ is a directed inverse system $(X_\lambda)_{\lambda\in\Lambda}$ of coherent $S$-schemes, and the set of morphisms of such two objects is given by
	\begin{align}\label{eq:para:notation-fal-comp-2}
		\mo_{\pro(\schqcqs_{/S})}((X'_\xi)_{\xi\in\Xi},( X_\lambda)_{\lambda\in\Lambda})=\lim_{\lambda\in\Lambda}\colim_{\xi\in\Xi}\mo_{\schqcqs_{/S}}(X'_\xi,X_\lambda).
	\end{align}
	We regard $\schqcqs_{/S}$ as a full subcategory of $\pro(\schqcqs_{/S})$ (\cite[\Luoma{1}.8.10.6]{sga4-1}). 
	
	For any object $(X_\lambda)_{\lambda\in\Lambda}$ in $\pro(\schqcqs_{/S})$, we put
	\begin{align}\label{eq:fal-limit}
		X_{\eta,\et}=\lim_{\lambda\in\Lambda}X_{\lambda,\eta,\et},\quad(\fal^{\et}_{X_\eta\to X},\falb)=\lim_{\lambda\in\Lambda}(\fal^{\et}_{X_{\lambda,\eta}\to X_\lambda},\falb)
	\end{align}
	the cofiltered limits of (ringed) sites defined in \cite[8.2.3, 8.6.2]{sga4-2}, where $X_{\lambda,\eta,\et}$ is the \'etale site of $X_{\lambda,\eta}$ and $(\fal^{\et}_{X_{\lambda,\eta}\to X_\lambda},\falb)$ is the Faltings ringed site associated to the morphism of coherent schemes $X_{\lambda,\eta}\to X_\lambda$ (\cite[7.7]{he2024coh}). Note that $\falb$ is flat over $\ca{O}_L$. We remark that if the transition morphisms of $(X_\lambda)_{\lambda\in\Lambda}$ are affine, then $X_{\eta,\et}$ is canonically equivalent to the \'etale site of $\lim_{\lambda\in\Lambda} X_{\lambda,\eta}$ by \cite[\Luoma{7}.5.6]{sga4-2} and $(\fal^{\et}_{X_\eta\to X},\falb)$ is canonically equivalent to the Faltings ringed site associated to the morphism of coherent schemes $\lim_{\lambda\in\Lambda} X_{\lambda,\eta}\to \lim_{\lambda\in\Lambda} X_\lambda$ by \cite[7.12]{he2024coh}. 
	
	Similarly, we put
	\begin{align}\label{eq:rz-limit}
		 X_\eta=\lim_{\lambda\in\Lambda}X_{\lambda,\eta},\quad X=\lim_{\lambda\in\Lambda}X_\lambda
	\end{align}
	the cofiltered limits of locally ringed spectral spaces, which are also the cofiltered limits of ringed sites (\cite[9.12]{he2024purity}). The canonical morphisms of ringed sites 
	\begin{align}\label{eq:sigma_lambda}
		\sigma_\lambda:(\fal^{\et}_{X_{\lambda,\eta}\to X_\lambda},\falb)\longrightarrow (X_\lambda,\ca{O}_{X_\lambda}),
	\end{align}
	defined by the left exact continuous functor $\sigma^+_\lambda:X_{\lambda,\mrm{Zar}}\to \fal^{\et}_{X_{\lambda,\eta}\to X_\lambda}$ sending each quasi-compact open subset $U$ of $X_\lambda$ to $U_\eta\to U$ (\cite[7.8]{he2024coh}), define a canonical morphism of ringed sites by taking cofiltered limits,
	\begin{align}\label{eq:sigma}
		\sigma: (\fal^{\et}_{X_\eta\to X},\falb)\longrightarrow (X,\ca{O}_X).
	\end{align}
	We remark that any construction above is functorial in $\pro(\schqcqs_{/S})$, any site above is coherent (\cite[\Luoma{6}.2.3]{sga4-2}), and thus any morphism of sites above is coherent (\cite[\Luoma{6}.3.1]{sga4-2}).
\end{mypara}

\begin{mythm}[{Faltings' main $p$-adic comparison theorem, \cite[5.17]{he2024falmain}}]\label{thm:fal-comp}
	With the notation in {\rm\ref{para:notation-fal-comp}}, let $(X_\lambda)_{\lambda\in\Lambda}$ be a directed inverse system of proper $S$-schemes of finite presentation. Then, for any $n\in\bb{N}$, there exists a canonical morphism
	\begin{align}\label{eq:thm:fal-comp-1}
		\rr\Gamma(X_{\eta,\et},\bb{Z}/p^n\bb{Z})\otimes^{\dl}_{\bb{Z}}\ca{O}_L\longrightarrow \rr\Gamma(\fal^{\et}_{X_\eta\to X},\falb/p^n\falb)
	\end{align}
	which is an almost isomorphism {\rm(\cite[5.7]{he2024coh})}.
\end{mythm}
\begin{proof}
	By Faltings' main $p$-adic comparison theorem for non-smooth schemes \cite[5.17, 5.3]{he2024falmain}, there is a canonical morphism
	\begin{align}\label{eq:thm:fal-comp-2}
		\rr\Gamma(X_{\lambda,\eta,\et},\bb{Z}/p^n\bb{Z})\otimes^{\dl}_{\bb{Z}}\ca{O}_L\longrightarrow \rr\Gamma(\fal^{\et}_{X_{\lambda,\eta}\to X_\lambda},\falb/p^n\falb)
	\end{align}
	which is an almost isomorphism and functorial in $\lambda\in\Lambda$. Taking filtered colimit over $\lambda\in\Lambda$, the conclusion follows from \cite[\Luoma{6}.8.7.7]{sga4-2}.
\end{proof}

\begin{mylem}\label{lem:fal-comp-perfd}
	With the notation in {\rm\ref{para:notation-fal-comp}}, let $(X_\lambda)_{\lambda\in\Lambda}$ be a directed inverse system of coherent $S$-schemes. Assume that the stalk at any point of the locally ringed space $X=\lim_{\lambda\in\Lambda}X_\lambda$ is a pre-perfectoid $\ca{O}_L$-algebra. Then, for any $n\in\bb{N}$, the canonical morphism
	\begin{align}
		\ca{O}_X/p^n\ca{O}_X\longrightarrow \rr\sigma_*(\falb/p^n\falb)
	\end{align}
	is an almost isomorphism.
\end{mylem}
\begin{proof}
	For any point $x\in X$, we put $x_\lambda\in X_\lambda$ its image under the canonical morphism $X\to X_\lambda$. Then, it suffices to show that the canonical morphism of stalks 
	\begin{align}
		\ca{O}_{X,x}/p^n\ca{O}_{X,x}\longrightarrow (\rr\sigma_*(\falb/p^n\falb))_x
	\end{align}
	is an almost isomorphism. Note that $\ca{O}_{X,x}=\colim_{\lambda\in\Lambda}\ca{O}_{X_\lambda,x_{\lambda}}$ and that for any integer $q\geq 0$,
	\begin{align}
		(\rr^q\sigma_*(\falb/p^n\falb))_x=\colim_{\lambda\in\Lambda}\colim_{x_\lambda\in U_\lambda} H^q(\fal^{\et}_{U_{\lambda,\eta}\to U_\lambda},\falb/p^n\falb)
	\end{align}
	where $U_\lambda$ runs through all the quasi-compact open neighborhoods of $x_\lambda$ in $X_\lambda$ (\cite[\Luoma{5}.5.1, \Luoma{6}.8.7.7]{sga4-2}). We put $X_{(x)}=\spec(\ca{O}_{X,x})$ and $X_{\lambda,(x_\lambda)}=\spec(\ca{O}_{X_\lambda,x_\lambda})$. Since $\lim_{x_\lambda\in U_\lambda}U_\lambda=X_{\lambda,(x_\lambda)}$ and $X_{(x)}=\lim_{\lambda\in\Lambda}X_{\lambda,(x_\lambda)}$, we have
	\begin{align}
		(\rr^q\sigma_*(\falb/p^n\falb))_x=\colim_{\lambda\in\Lambda} H^q(\fal^{\et}_{X_{\lambda,(x_\lambda),\eta}\to X_{\lambda,(x_\lambda)}},\falb/p^n\falb)=H^q(\fal^{\et}_{X_{(x),\eta}\to X_{(x)}},\falb/p^n\falb)
	\end{align}
	by \cite[7.12]{he2024coh} and \cite[\Luoma{6}.8.7.7]{sga4-2}. Since $\ca{O}_{X,x}$ is assumed to be a pre-perfectoid $\ca{O}_L$-algebra, the canonical morphism
	\begin{align}
		\ca{O}_{X,x}/p^n\ca{O}_{X,x}\longrightarrow \rr\Gamma(\fal^{\et}_{X_{(x),\eta}\to X_{(x)}},\falb/p^n\falb)
	\end{align} 
	is an almost isomorphism by cohomological descent for Faltings ringed sites \cite[8.24]{he2024coh}.
\end{proof}

\begin{mycor}[{cf. \cite[12.8]{he2024purity}}]\label{cor:fal-comp-perfd}
	With the notation in {\rm\ref{para:notation-fal-comp}}, let $(X_\lambda)_{\lambda\in\Lambda}$ be a directed inverse system of flat proper $S$-schemes. Assume that the stalk at any point of the locally ringed space $X=\lim_{\lambda\in\Lambda}X_\lambda$ is a pre-perfectoid $\ca{O}_L$-algebra. Then, for any $n\in\bb{N}$, there exists a canonical isomorphism in the derived category of almost $\ca{O}_L$-modules {\rm(\cite[5.7]{he2024coh})},
	\begin{align}\label{eq:prop:fal-comp-1}
		\rr\Gamma(X_{\eta,\et},\bb{Z}/p^n\bb{Z})\otimes^{\dl}_{\bb{Z}}\ca{O}_L\iso \rr\Gamma(X,\ca{O}_X/p^n\ca{O}_X).
	\end{align}
	In particular, for any integer $q>\limsup_{\lambda\in\Lambda}\{\dim(X_{\lambda,\eta})\}$, we have
	\begin{align}
		H^q(X_{\eta,\et},\bb{Z}/p^n\bb{Z})=0.
	\end{align}
\end{mycor}
\begin{proof}
	Firstly, note that each $X_\lambda$ is of finite presentation over $S$ by \cite[\Luoma{1}.3.4.7]{raynaudgruson1971plat} (cf. \ref{rem:blowup}). Thus, \eqref{eq:prop:fal-comp-1} follows directly from \ref{thm:fal-comp} and \ref{lem:fal-comp-perfd}. For the ``in particular" part, we may assume that $q>\sup_{\lambda\in\Lambda}\{\dim(X_{\lambda,\eta})\}$ after replacing $\Lambda$ by a cofinal subsystem. By d\'evissage, it suffices to prove the case where $n=1$. 
	
	We claim that  $H^q(X_\lambda,\ca{O}_{X_\lambda}/p\ca{O}_{X_\lambda})=0$. Indeed, for any $\lambda\in\Lambda$, as $X_\lambda$ is flat, proper, of finite presentation over $S$, we have $\dim(X_\eta)=\dim(X_s)$ by \cite[\href{https://stacks.math.columbia.edu/tag/0D4J}{0D4J}]{stacks-project}. Let $f_\lambda:X_\lambda\to S$ denote the canonical morphism of schemes. As $f_\lambda$ is proper with special fibre of dimension strictly less than $q$, we have
	\begin{align}
		H^q(X_\lambda,\ca{O}_{X_\lambda}/p\ca{O}_{X_\lambda})=\rr^q f_{\lambda*}(\ca{O}_{X_\lambda}/p\ca{O}_{X_\lambda})=0
	\end{align}
	by Grothendieck's vanishing \cite[\href{https://stacks.math.columbia.edu/tag/0E7D}{0E7D}]{stacks-project}.
	
	The claim implies that $H^q(X_{\eta,\et},\bb{F}_p)\otimes_{\bb{F}_p}\ca{O}_L/p\ca{O}_L$ is almost zero by \eqref{eq:prop:fal-comp-1}. As $H^q(X_{\eta,\et},\bb{F}_p)$ is an (free) $\bb{F}_p$-module, its base change to $\ca{O}_L/p\ca{O}_L$ is almost zero if and only if itself is zero. This completes the proof.
\end{proof}

\begin{mylem}\label{lem:vanish}
	Let $L$ be a valuation field, $\eta=\spec(L)$, $S=\spec(\ca{O}_L)$, $Y$ a proper $L$-scheme, $\widetilde{Y}$ a $Y$-scheme. Consider the category $\scr{C}^{\widetilde{Y}}_{Y/S}$ formed by the following commutative diagrams of schemes
	\begin{align}
		\xymatrix{
			\widetilde{Y}\ar[r]\ar[dr]&X_\eta\ar[r]\ar[d]&X\ar[dd]\\
			&Y\ar[d]&\\
			&\eta\ar[r]&S
		}
	\end{align}
	where $X$ is a flat proper $S$-scheme. Then, we have the following properties:
	\begin{enumerate}
		\renewcommand{\labelenumi}{{\rm(\theenumi)}}
		\item The category $\scr{C}^{\widetilde{Y}}_{Y/S}$ is cofiltered.\label{item:lem:vanish-1}
		\item In the category of locally ringed spaces, we have $\lim_{X\in\scr{C}^{\widetilde{Y}}_{Y/S}}X=\lim_{X\in\scr{C}^{\widetilde{Y}}_{Y/S}}\rz_{X_\eta}(X)$, where $\rz_{X_\eta}(X)$ is the Riemann-Zariski space {\rm(\cite[9.13]{he2024purity})} of $X$ with respect to $X_\eta$.\label{item:lem:vanish-2}
		\item If $\widetilde{Y}\to Y$ is integral, then the full subcategory $\scr{C}^{\widetilde{Y},\mrm{fini}}_{Y/S}$ of $\scr{C}^{\widetilde{Y}}_{Y/S}$ consisting of those $X$ with $X_\eta$ finite over $Y$ is initial. Moreover, we have $\widetilde{Y}=\lim_{X\in\scr{C}^{\widetilde{Y},\mrm{fini}}_{Y/S}}X_\eta$ in the category of schemes.\label{item:lem:vanish-3}
	\end{enumerate}
\end{mylem}
\begin{proof}
	(\ref{item:lem:vanish-1}) Firstly, $\scr{C}^{\widetilde{Y}}_{Y/S}$ is non-empty. Indeed, by Nagata's compactification theorem \cite[\href{https://stacks.math.columbia.edu/tag/0F41}{0F41}]{stacks-project} there exists a proper $S$-scheme $X$ with an open immersion $Y\to X$. Replacing $X$ by the scheme theoretic closure of $Y$, we may assume that $X$ is flat and $X_\eta=Y$ (as $Y$ is proper over $\eta$) so that $X\in\scr{C}^{\widetilde{Y}}_{Y/S}$. 
	
	Secondly, for any two objects $X_1,X_2\in\scr{C}^{\widetilde{Y}}_{Y/S}$, let $X_3$ be the scheme theoretic image of $X_{1,\eta}\times_YX_{2,\eta}\to X_1\times_SX_2$. Then, $X_3$ is flat and proper over $S$ with $X_{3,\eta}=X_{1,\eta}\times_YX_{2,\eta}$ (as $X_{1,\eta}\times_YX_{2,\eta}\to X_{1,\eta}\times_\eta X_{2,\eta}$ is a closed immersion by \cite[\href{https://stacks.math.columbia.edu/tag/01KR}{01KR}]{stacks-project}). We see that $X_3\in\scr{C}^{\widetilde{Y}}_{Y/S}$. 
	
	Thirdly, for any two morphisms $X_1\rightrightarrows X_0$ in $\scr{C}^{\widetilde{Y}}_{Y/S}$, let $X_2$ be the scheme theoretic image of $\widetilde{Y}\to X_1$. Then, $X_2$ is flat and proper over $S$. Since the compositions of $\widetilde{Y}\to X_1\rightrightarrows X_0$ are equal to the given morphism  $\widetilde{Y}\to X_0$, we see that the compositions of $X_2\to X_1\rightrightarrows X_0$ are also equal to each other since $\widetilde{Y}\to X_2$ is scheme theoretically dominant and $X_0\to S$ is separated (\cite[5.4.1]{ega1-2}). This completes the proof of (\ref{item:lem:vanish-1}).
	
	(\ref{item:lem:vanish-2}) Recall that $\rz_{X_\eta}(X)=\lim_{X'\in\scr{MD}_{X_\eta}(X)}X'$, where $\scr{MD}_{X_\eta}(X)$ is the cofiltered category of $X_\eta$-modifications of $X$. In particular, $X_\eta=X'_\eta$ is scheme theoretically dense in $X'$ and $X'\to X$ is proper. This implies that each $X'$ is flat and proper over $S$ if $X$ is so, and thus $X'\in \scr{C}^{\widetilde{Y}}_{Y/S}$ if $X$ is so. Therefore, we have $\lim_{X\in\scr{C}^{\widetilde{Y}}_{Y/S}}X=\lim_{X\in\scr{C}^{\widetilde{Y}}_{Y/S}}\lim_{X'\in\scr{MD}_{X_\eta}(X)}X'=\lim_{X\in\scr{C}^{\widetilde{Y}}_{Y/S}}\rz_{X_\eta}(X)$.
	
	(\ref{item:lem:vanish-3}) If $\widetilde{Y}\to Y$ is finite, then there exists $X\in\scr{C}^{\widetilde{Y}}_{Y/S}$ with $X_\eta=\widetilde{Y}$ by Nagata's compactification theorem, see the arguments of (\ref{item:lem:vanish-1}). Then, for any morphism $X'\to X$ in $\scr{C}^{\widetilde{Y}}_{Y/S}$, after replacing $X'$ by the scheme theoretic image of $\widetilde{Y}$ in $X'$, we may assume that $X'_\eta=X_\eta=\widetilde{Y}$. Hence, $\scr{C}^{\widetilde{Y},\mrm{fini}}_{Y/S}\subseteq \scr{C}^{\widetilde{Y}}_{Y/S}$ is initial and we have $\widetilde{Y}=\lim_{X\in\scr{C}^{\widetilde{Y},\mrm{fini}}_{Y/S}}X_\eta$.
	
	In general, we write $\widetilde{Y}=\lim_{\lambda\in\Lambda}Y_\lambda$ as a cofiltered limit of finite $Y$-schemes. On the one hand, for any $\lambda\in\Lambda$, there is a canonical functor $\scr{C}^{Y_\lambda}_{Y/S}\to \scr{C}^{\widetilde{Y}}_{Y/S}$ sending $X$ to $X$. On the other hand, for any $X\in \scr{C}^{\widetilde{Y}}_{Y/S}$, the morphism $\widetilde{Y}=\lim_{\lambda\in\Lambda}Y_\lambda\to X$ is given by an $S$-morphism $Y_\lambda\to X$ for some $\lambda$ large enough by \cite[8.14.2]{ega4-3} as $X$ is of finite presentation over $S$, which implies that $X$ lies in the essential image of $\scr{C}^{Y_\lambda}_{Y/S}\to \scr{C}^{\widetilde{Y}}_{Y/S}$. Combining with the special case discussed in the beginning, we see that $\scr{C}^{\widetilde{Y},\mrm{fini}}_{Y/S}\subseteq \scr{C}^{\widetilde{Y}}_{Y/S}$ is initial and we have $\lim_{X\in\scr{C}^{\widetilde{Y},\mrm{fini}}_{Y/S}}X_\eta=\lim_{\lambda\in\Lambda}\lim_{X\in\scr{C}^{Y_\lambda,\mrm{fini}}_{Y/S}}X_\eta=\lim_{\lambda\in\Lambda}Y_{\lambda}=\widetilde{Y}$.
\end{proof}

\begin{mythm}\label{thm:vanish}
	Let $L$ be an algebraically closed valuation field of height $1$ extension of $\bb{Q}_p$, $Y$ a proper smooth $L$-scheme, $D$ a normal crossings divisor on $Y$, $\widetilde{Y}^{\triv}$ a coherent scheme pro-finite \'etale over $Y^{\triv}=Y\setminus D$, $\widetilde{Y}$ the integral closure of $Y$ in $\widetilde{Y}^{\triv}$ (so that $\widetilde{Y}^{\triv}=Y^{\triv}\times_Y\widetilde{Y}$). Assume that the following conditions hold for any point $\widetilde{y}\in \widetilde{Y}$:
	\begin{enumerate}
		\renewcommand{\labelenumi}{{\rm(\theenumi)}}
		\item Its residue field $\kappa(\widetilde{y})$ is a pre-perfectoid field with respect to any valuation ring $W$ of height $1$ extension of $\ca{O}_L$ with fraction field $W[1/p]=\kappa(\widetilde{y})$.\label{item:thm:vanish-1}
		\item There exists a regular system of parameters $\{t_1,\dots,t_s\}$ of the strict Henselization $\ca{O}_{Y,y}^{\mrm{sh}}$ of $\ca{O}_{Y,y}$ (where $y\in Y$ is the image of $\widetilde{y}\in\widetilde{Y}$) such that $D$ is defined by $t_1\cdots t_r=0$ over $\ca{O}_{Y,y}^{\mrm{sh}}$ for some integer $0\leq r\leq s$ and that $t_1,\dots,t_r$ admit compatible systems of $p$-power roots in the strict Henselization $\ca{O}_{\widetilde{Y},\widetilde{y}}^{\mrm{sh}}$ of $\ca{O}_{\widetilde{Y},\widetilde{y}}$.\label{item:thm:vanish-2}
	\end{enumerate}
	Then, for any integer $q>\dim(Y)$ and $n\in\bb{N}$, we have
	\begin{align}
		H^q(\widetilde{Y}^{\triv}_\et,\bb{Z}/p^n\bb{Z})=0.
	\end{align}
\end{mythm}
\begin{proof}
	We put $\eta=\spec(L)$ and $S=\spec(\ca{O}_L)$. Consider the cofiltered category $\scr{C}^{\widetilde{Y}}_{Y/S}$ defined in \ref{lem:vanish}. We put $\widetilde{X}=\lim_{X\in \scr{C}^{\widetilde{Y}}_{Y/S}}X$ in the category of locally ringed spaces (\cite[9.12]{he2024purity}). Note that $\scr{C}^{\widetilde{Y},\mrm{fini}}_{Y/S}\subseteq \scr{C}^{\widetilde{Y}}_{Y/S}$ is initial and $\widetilde{Y}=\lim_{X\in \scr{C}^{\widetilde{Y},\mrm{fini}}_{Y/S}}X_\eta$ by \ref{lem:vanish}.(\ref{item:lem:vanish-3}).
	
	We claim that the stalk of any point of $\widetilde{X}$ is a pre-perfectoid $\ca{O}_L$-algebra. Indeed, for any $\widetilde{x}\in \widetilde{X}$, as $\widetilde{X}=\lim_{X\in \scr{C}^{\widetilde{Y}}_{Y/S}}\rz_{X_\eta}(X)$ by \ref{lem:vanish}.(\ref{item:lem:vanish-2}), the image of $\ca{O}_{\widetilde{X},\widetilde{x}}\to \ca{O}_{\widetilde{Y},\widetilde{y}}\to \kappa(\widetilde{y})$ is a valuation ring $W$ extension of $\ca{O}_L$ with fraction field $W[1/p]=\kappa(\widetilde{y})$ by \cite[10.3.(5), (10.2.5)]{he2024purity}. Moreover, it induces an isomorphism of $p$-adic completions $\widehat{\ca{O}_{\widetilde{X},\widetilde{x}}}\iso \widehat{W}$. Note that the localization $W_{\sqrt{pW}}$ of $W$ at the radical ideal generated by $p$ is a valuation ring of height $1$ extension of $\ca{O}_L$ with fraction field $W_{\sqrt{pW}}[1/p]=W[1/p]=\kappa(\widetilde{y})$ and that the canonical morphism $W\to W_{\sqrt{pW}}$ is an almost isomorphism (\cite[10.4]{he2024purity}). By assumption (\ref{item:thm:vanish-1}), $\kappa(\widetilde{y})$ is a pre-perfectoid field with respect to $W_{\sqrt{pW}}$. Hence, $\ca{O}_{\widetilde{X},\widetilde{x}}$ is a pre-perfectoid $\ca{O}_L$-algebra by definition \ref{para:notation-perfd}.
	
	The claim implies that for any integer $q>\dim(Y)\geq\sup_{X\in \scr{C}^{\widetilde{Y},\mrm{fini}}_{Y/S}}\dim(X_\eta)$, we have
	\begin{align}
		H^q(\widetilde{Y}_\et,\bb{Z}/p^n\bb{Z})=\colim_{X\in \scr{C}^{\widetilde{Y},\mrm{fini}}_{Y/S}}H^q(X_{\eta,\et},\bb{Z}/p^n\bb{Z})=0
	\end{align}
	for any $n\in\bb{N}$ by \ref{cor:fal-comp-perfd} and \cite[\Luoma{7}.5.6, \Luoma{6}.8.7.7]{sga4-2}.
	
	On the other hand, since the open immersion $j:\widetilde{Y}^{\triv}\to \widetilde{Y}$ is a scheme theoretically dominant morphism of normal schemes, we see that $j_{\et*}(\bb{Z}/p^n\bb{Z})=\bb{Z}/p^n\bb{Z}$ by checking the stalks at each geometric point. Hence, the assumption (\ref{item:thm:vanish-2}) implies that 
	\begin{align}
		H^q(\widetilde{Y}^{\triv}_\et,\bb{Z}/p^n\bb{Z})=H^q(\widetilde{Y}_\et,\bb{Z}/p^n\bb{Z})
	\end{align}
	for any integers $q,n\in\bb{N}$ by \ref{prop:open-et-coh}. This completes the proof.
\end{proof}

\begin{mylem}\label{lem:profet-component}
	Let $Y$ be a coherent normal scheme with finitely many irreducible components, $\widetilde{Y}^{\triv}$ a coherent scheme pro-finite \'etale over an open subset $Y^{\triv}$ of $Y$, $\widetilde{Y}$ the integral closure of $Y$ in $\widetilde{Y}^{\triv}$. Then, any connected component of $\widetilde{Y}$ is an integral normal scheme.
\end{mylem}
\begin{proof}
	As $Y$ is a finite disjoint union of integral normal schemes, we may assume that $Y^{\triv}$ is dense in $Y$. Let $f:\widetilde{Y}\to Y$ denote the canonical morphism. Since $f$ is pro-finite \'etale over $Y^{\triv}$ and $\widetilde{Y}^{\triv}=Y^{\triv}\times_Y\widetilde{Y}$ is a dense open subset of $\widetilde{Y}$, we have $f^{-1}(\ak{G}(Y))=f^{-1}(\ak{G}(Y^{\triv}))=\ak{G}(\widetilde{Y}^{\triv})=\ak{G}(\widetilde{Y})$ (\cite[3.5.(2)]{he2024purity}), where $\ak{G}(X)$ denotes the set of generic points of irreducible components of a scheme $X$.
	
	We write $\widetilde{Y}^{\triv}=\lim_{\lambda\in\Lambda}Y^{\triv}_\lambda$ as a cofiltered limit of finite \'etale $Y^{\triv}$-schemes and we put $Y_\lambda$ the integral closure of $Y$ in $Y^{\triv}_\lambda$. As the transition morphisms are finite \'etale, we see that the sets of generic points form a directed inverse system $(\ak{G}(Y_\lambda))_{\lambda\in\Lambda}$ with limit $\ak{G}(\widetilde{Y})$ by \cite[3.7]{he2024purity}. For any generic point $\widetilde{\xi}\in \ak{G}(\widetilde{Y})$ of $\widetilde{Y}$, we put $\xi_\lambda\in \ak{G}(Y_\lambda)$ its image in $Y_\lambda$. In particular, we have $\widetilde{\xi}=\lim_{\lambda\in\Lambda}\xi_\lambda$. Thus, the integral closure $Y^{\widetilde{\xi}}$ of $Y$ in $\widetilde{\xi}$ is the cofiltered limit of the integral closures $Y^{\xi_\lambda}$ of $Y$ in $\xi_\lambda$ (\cite[3.18]{he2024coh}). As $Y_\lambda\to Y$ is finite \'etale over $Y^{\triv}$, each $Y^{\xi_\lambda}$ is the irreducible component of $Y_\lambda$ with generic point $\xi_\lambda$. As $Y_\lambda$ is normal with finitely many irreducible components, we see that $Y^{\xi_\lambda}$ is an open and closed subscheme of $Y_\lambda$. Taking cofiltered limit, we see that $Y^{\widetilde{\xi}}$ is the intersection of open and closed subschemes of $\widetilde{Y}$. In particular, $Y^{\widetilde{\xi}}$ is a union of connected components of $\widetilde{Y}$ (\cite[\href{https://stacks.math.columbia.edu/tag/04PL}{04PL}]{stacks-project}). But $\widetilde{\xi}\to Y^{\widetilde{\xi}}$ is dominant, we see that $Y^{\widetilde{\xi}}$ is the irreducible component of $\widetilde{Y}$ with generic point $\widetilde{\xi}$ which is also a connected component. This shows that any irreducible component of $\widetilde{Y}$ is also a connected component, which completes the proof.
\end{proof}

\begin{mycor}\label{cor:vanish}
	Let $K$ be a complete discrete valuation field extension of $\bb{Q}_p$ with perfect residue field, $Y$ a proper smooth $K$-scheme, $D$ a normal crossings divisor on $Y$, $\widetilde{Y}^{\triv}$ a coherent scheme pro-finite \'etale over $Y^{\triv}=Y\setminus D$. Assume that the following conditions hold:
	\begin{enumerate}
		\renewcommand{\theenumi}{\roman{enumi}}
		\renewcommand{\labelenumi}{{\rm(\theenumi)}}
		\item For any irreducible component of $\widetilde{Y}^{\triv}$, we denote its generic point by $\spec(\ca{L})\in \widetilde{Y}^{\triv}$. Let $\spec(\ca{K})\in Y^{\triv}$ be the image of $\spec(\ca{L})\in \widetilde{Y}^{\triv}$ (so that it is the generic point of its irreducible component of $Y^{\triv}$ by {\rm\cite[3.5.(2)]{he2024purity}}). Then, $\ca{L}$ is a Galois extension of $\ca{K}$ and $\ca{G}=\gal(\ca{L}/\ca{K})$ is a $p$-adic analytic group. \label{item:cor:vanish-cond-1}
		\item Let $\overline{\ca{K}}$ be an algebraic closure of $\ca{L}$, $\overline{y}=\spec(\overline{\ca{F}})$ a point of the integral closure $Y^{\overline{\ca{K}}}$ of $Y$ in $\overline{\ca{K}}$, $\ca{O}_{\overline{\ca{F}}}$ a valuation ring of height $1$ extension of $\ca{O}_K$ with fraction field $\overline{\ca{F}}$. Then, the universal geometric Sen action of the $p$-adic analytic Galois extension $\ca{L}$ of $\ca{K}$ \eqref{eq:thm:sen-val-perfd-1},
		\begin{align}
			\varphi^{\mrm{geo}}_{\sen}|_{\ca{G},\overline{y}}: \ho_{\ca{O}_Y}(\Omega^1_{(Y,\scr{M}_Y)/K}(-1),\widehat{\overline{\ca{F}}})\longrightarrow \widehat{\overline{\ca{F}}}\otimes_{\bb{Q}_p}\lie(\ca{G}),
		\end{align}
		is injective.\label{item:cor:vanish-cond-2}
	\end{enumerate}
	Then, the conditions in {\rm\ref{thm:vanish}} hold for any point $\widetilde{y}_{\overline{K}}\in\widetilde{Y}_{\overline{K}}$, where $\widetilde{Y}$ is the integral closure of $Y$ in $\widetilde{Y}^{\triv}$, $\widetilde{Y}_{\overline{K}}=\spec(\overline{K})\times_{\spec(K)}\widetilde{Y}$ and $\overline{K}$ is an algebraic closure of $K$. More precisely, we have:
	\begin{enumerate}
		\renewcommand{\labelenumi}{{\rm(\theenumi)}}
		\item The residue field $\kappa(\widetilde{y}_{\overline{K}})$ of $\widetilde{y}_{\overline{K}}\in\widetilde{Y}_{\overline{K}}$ is a pre-perfectoid field with respect to any valuation ring $W$ of height $1$ extension of $\ca{O}_{\overline{K}}$ with fraction field $W[1/p]=\kappa(\widetilde{y}_{\overline{K}})$.\label{item:cor:vanish-1}
		\item For any regular system of parameters $\{t_1,\dots,t_s\}$ of the strict Henselization $\ca{O}_{Y_{\overline{K}},y_{\overline{K}}}^{\mrm{sh}}$ of $\ca{O}_{Y_{\overline{K}},y_{\overline{K}}}$ (where $y_{\overline{K}}\in Y_{\overline{K}}$ is the image of $\widetilde{y}_{\overline{K}}\in\widetilde{Y}_{\overline{K}}$) such that $D_{\overline{K}}$ is defined by $t_1\cdots t_r=0$ over $\ca{O}_{Y_{\overline{K}},y_{\overline{K}}}^{\mrm{sh}}$ for some integer $0\leq r\leq s$, the elements $t_1,\dots,t_r$ admit compatible systems of $p$-power roots in the strict Henselization $\ca{O}_{\widetilde{Y}_{\overline{K}},\widetilde{y}_{\overline{K}}}^{\mrm{sh}}$ of $\ca{O}_{\widetilde{Y}_{\overline{K}},\widetilde{y}_{\overline{K}}}$.\label{item:cor:vanish-2}
	\end{enumerate}
	In particular, for any integer $q>\dim(Y)$ and $n\in\bb{N}$, we have
	\begin{align}
		H^q(\widetilde{Y}^{\triv}_{\overline{K},\et},\bb{Z}/p^n\bb{Z})=0,
	\end{align}
	where $\widetilde{Y}^{\triv}_{\overline{K}}=\spec(\overline{K})\times_{\spec(K)}\widetilde{Y}^{\triv}$.
\end{mycor}
\begin{proof}
	We claim that we may assume that $\ca{L}$ contains a compatible system of primitive $p$-power roots of unity $(\zeta_{p^n})_{n\in\bb{N}}$. Indeed, firstly we check that if the conditions (\ref{item:cor:vanish-cond-1}) and (\ref{item:cor:vanish-cond-2}) of \ref{cor:vanish} hold for $\widetilde{Y}^{\triv}$ then they also hold for $\widetilde{Y}'^{\triv}=\spec(K_\infty)\times_{\spec(K)}\widetilde{Y}^{\triv}$ (where $K_\infty=\bigcup_{n\in\bb{N}}K(\zeta_{p^n})$) in the following steps: 
	\begin{enumerate}
		\setcounter{enumi}{0}
		\renewcommand{\theenumi}{\alph{enumi}}
		\renewcommand{\labelenumi}{{\rm(\theenumi)}}
		\item Since $\widetilde{Y}'^{\triv}$ is pro-finite \'etale over $\widetilde{Y}^{\triv}$, it is also pro-finite \'etale over $Y^{\triv}$.
		\item For any irreducible component of $\widetilde{Y}'^{\triv}$, we denote its generic point by $\spec(\ca{L}')\in \widetilde{Y}'^{\triv}$. Let $\spec(\ca{L})\in \widetilde{Y}^{\triv}$ be the image of $\spec(\ca{L}')$, which is thus the generic point of its irreducible component of $\widetilde{Y}^{\triv}$. Notice that $\ca{L}'=K_\infty\ca{L}$. As $\ca{G}=\gal(\ca{L}/\ca{K})$ and $\gal(K_\infty/K)$ are $p$-adic analytic groups, so is $\ca{G}'=\gal(\ca{L}'/\ca{K})$ (\cite[3.11]{he2022sen}). 
		\item As $\overline{\ca{K}}$ is also an algebraic closure of $\ca{L}'$, applying the functoriality of the universal geometric Sen action \ref{rem:sen-val-pt} to the inclusion $\ca{L}\subseteq \ca{L}'$ (with all arrows in \eqref{eq:rem:sen-val-pt-1} being identity), we obtain a canonical commutative diagram
		\begin{align}
			\xymatrix{
				\ho_{\ca{O}_Y}(\Omega^1_{(Y,\scr{M}_Y)/K}(-1),\widehat{\overline{\ca{F}}})\ar[rrr]^-{\varphi^{\mrm{geo}}_{\sen}|_{\ca{G}',\overline{y}}}\ar@{=}[d]&&& \widehat{\overline{\ca{F}}}\otimes_{\bb{Q}_p}\lie(\ca{G}')\ar[d]\\
				\ho_{\ca{O}_Y}(\Omega^1_{(Y,\scr{M}_Y)/K}(-1),\widehat{\overline{\ca{F}}})\ar[rrr]^-{\varphi^{\mrm{geo}}_{\sen}|_{\ca{G},\overline{y}}}&&& \widehat{\overline{\ca{F}}}\otimes_{\bb{Q}_p}\lie(\ca{G}).
			}
		\end{align}	
		Thus, the injectivity of $\varphi^{\mrm{geo}}_{\sen}|_{\ca{G},\overline{y}}$ implies that of $\varphi^{\mrm{geo}}_{\sen}|_{\ca{G}',\overline{y}}$.
	\end{enumerate}
	Suppose that the conclusions (\ref{item:cor:vanish-1}) and (\ref{item:cor:vanish-2}) of \ref{cor:vanish} hold for $\widetilde{Y}'^{\triv}$. Then, we have
	\begin{align}
		\widetilde{Y}'^{\triv}_{\overline{K}}=\spec(K_\infty)\times_{\spec(K)}\widetilde{Y}^{\triv}_{\overline{K}}=\spec(K_\infty\otimes_K\overline{K})\times_{\spec(\overline{K})}\widetilde{Y}^{\triv}_{\overline{K}}.
	\end{align}
	Since $K_\infty\otimes_K\overline{K}=\colim_{n\in\bb{N}}\prod_{\gal(K_n/K)}\overline{K}$, we see that
	\begin{align}
		\widetilde{Y}'^{\triv}_{\overline{K}}=\lim_{n\in\bb{N}}\coprod_{\gal(K_n/K)}\widetilde{Y}^{\triv}_{\overline{K}},
	\end{align}
	which implies that $\widetilde{Y}^{\triv}_{\overline{K}}$ is isomorphic to an intersection of open and closed subschemes of $\widetilde{Y}'^{\triv}_{\overline{K}}$. In particular, the strict Henselization at a geometric point of $\widetilde{Y}^{\triv}_{\overline{K}}$ coincides with that of $\widetilde{Y}'^{\triv}_{\overline{K}}$. Hence, the conclusions (\ref{item:cor:vanish-1}) and (\ref{item:cor:vanish-2}) of \ref{cor:vanish} also hold for $\widetilde{Y}^{\triv}$. This proves the claim, and we shall assume that $\ca{L}$ contains a compatible system of primitive $p$-power roots of unity in the following.
	
	Back to the verification of (\ref{item:cor:vanish-1}) and (\ref{item:cor:vanish-2}), firstly note that the base change $\widetilde{Y}_{\overline{K}}$ is the integral closure of $Y_{\overline{K}}$ in $\widetilde{Y}^{\triv}_{\overline{K}}$, see \cite[3.17, 3.18]{he2024coh}. Let $\spec(\ca{L}_{\overline{K}})\in\widetilde{Y}_{\overline{K}}$ be the generic point of the irreducible component of $\widetilde{Y}_{\overline{K}}$ containing $\widetilde{y}_{\overline{K}}$. Note that the integral closure $Y^{\ca{L}_{\overline{K}}}$ of $Y$ in $\ca{L}_{\overline{K}}$ is a connected component of $\widetilde{Y}_{\overline{K}}$ by \ref{lem:profet-component}. We take an algebraic closure $\overline{\ca{K}}$ of $\ca{L}_{\overline{K}}$. As $Y^{\overline{\ca{K}}}\to Y^{\ca{L}_{\overline{K}}}$ is surjective, we take a point $\overline{y}=\spec(\overline{\ca{F}})\in Y^{\overline{\ca{K}}}$ lying over $\widetilde{y}_{\overline{K}}\in\widetilde{Y}_{\overline{K}}$. Then, we take a valuation ring $\ca{O}_{\overline{\ca{F}}}$ of height $1$ extension of $W\subseteq \kappa(\widetilde{y}_{\overline{K}})$ with fraction field $\overline{\ca{F}}$. We put $\spec(\ca{L})\in \widetilde{Y}^{\triv}$ (resp. $\spec(\ca{K})\in Y^{\triv}$) the image of $\spec(\ca{L}_{\overline{K}})\in \widetilde{Y}^{\triv}_{\overline{K}}$, and we put $y_{\ca{L}}=\spec(\ca{F}_{\ca{L}})\in\widetilde{Y}$ (resp. $y=\spec(\ca{F})$) the image of $\widetilde{y}_{\overline{K}}\in\widetilde{Y}_{\overline{K}}$. Note that the integral closure $Y^{\ca{L}}$ (resp. $Y^{\ca{K}}$) of $Y$ in $\ca{L}$ (resp. $\ca{K}$) is a connected component of $\widetilde{Y}$ (resp. $Y$) by \ref{lem:profet-component} and thus pro-finite \'etale over $Y^{\triv}$. Therefore, we are in the situation of \ref{para:notation-sen-val}.
	\begin{align}
		\xymatrix{
			\overline{y}=\spec(\overline{\ca{F}})\ar[d]\ar[r]&\widetilde{y}_{\overline{K}}=\spec(\kappa(\widetilde{y}_{\overline{K}}))\ar[d]\ar[r]&y_{\ca{L}}=\spec(\ca{F}_{\ca{L}})\ar[d]\ar[r]&y=\spec(\ca{F})\ar[d]\\
			Y^{\overline{\ca{K}}}\ar[r]&Y^{\ca{L}_{\overline{K}}}\ar[d]\ar[r]&Y^{\ca{L}}\ar[d]\ar[r]&Y^{\ca{K}}\ar[d]\\
			&\widetilde{Y}_{\overline{K}}\ar[r]&\widetilde{Y}\ar[r]&Y
		}
	\end{align}
	Then, we deduce from the assumptions (\ref{item:cor:vanish-cond-1}) and (\ref{item:cor:vanish-cond-2}) and \ref{thm:sen-val-perfd} that $\ca{F}_{\ca{L}}$ is a pre-perfectoid field with respect to the valuation ring $\ca{O}_{\ca{F}_{\ca{L}}}=\ca{F}_{\ca{L}}\cap\ca{O}_{\overline{\ca{F}}}=\ca{F}_{\ca{L}}\cap W$, and that $t_1,\dots,t_r\in \ca{O}_{Y_{\overline{K}},y_{\overline{K}}}^{\mrm{sh}}=\ca{O}_{Y,y}^{\mrm{sh}}$ admit compatible systems of $p$-power roots in the strict Henselization $\ca{O}_{Y^{\ca{L}},y_{\ca{L}}}^{\mrm{sh}}$ of $Y^{\ca{L}}$ at $\overline{y}$. As $\overline{K}$ is ind-finite \'etale over $K$, we see that $Y^{\ca{L}_{\overline{K}}}$ is a connected component of the base change $Y^{\ca{L}}_{\overline{K}}$ by \ref{lem:profet-component}. In particular, $\kappa(\widetilde{y}_{\overline{K}})$ is algebraic over $\ca{F}_{\ca{L}}$ and we have $\ca{O}_{Y^{\ca{L}_{\overline{K}}},\widetilde{y}_{\overline{K}}}^{\mrm{sh}}=\ca{O}_{Y^{\ca{L}},y_{\ca{L}}}^{\mrm{sh}}$. Therefore, $W$ is pre-perfectoid by almost purity (\cite[6.6.2]{gabber2003almost}, see also \cite[7.12]{he2022sen}) and $t_1,\dots,t_r$ admit compatible systems of $p$-power roots in $\ca{O}_{\widetilde{Y}_{\overline{K}},\widetilde{y}_{\overline{K}}}^{\mrm{sh}}=\ca{O}_{Y^{\ca{L}_{\overline{K}}},\widetilde{y}_{\overline{K}}}^{\mrm{sh}}$ (where the equality follows from the fact that $Y^{\ca{L}_{\overline{K}}}$ is an intersection of open and closed subscheme of $\widetilde{Y}_{\overline{K}}$ by \ref{lem:profet-component}). This verifies the conditions in \ref{thm:vanish} for any point $\widetilde{y}_{\overline{K}}$ of $\widetilde{Y}_{\overline{K}}$.
	
	Finally, the ``in particular" part follows directly from \ref{thm:vanish}.
\end{proof}

\section{Comparison with Rodr\'iguez Camargo's Construction of Geometric Sen Operators}\label{sec:camargo}
We compare our construction of the universal geometric Sen action with Rodr\'iguez Camargo's \cite{camargo2022sen} at the geometric valuative points (see \ref{prop:camargo-sen-comp}). We refer to \cite{dllz2023log} for a systematic development of the theory of log adic spaces and their pro-Kummer \'etale sites.

\begin{mypara}
	In this section, we fix a complete discrete valuation field $K$ extension of $\bb{Q}_p$ with perfect residue field, and an algebraic closure $\overline{K}$ of $K$ with $p$-adic completion $C=\widehat{\overline{K}}$. Then, for any $K$-scheme $X$ locally of finite type, we put $X^{\mrm{ad}}=\spa(K,\ca{O}_K)\times_{\spec(K)}X$ the analytification of $X$ as an adic space (\cite[3.8]{huber1994general}). In particular, there is a canonical morphism of locally ringed spaces $X^{\mrm{ad}}\to X$.
	
	For any integers $0\leq c\leq d$ and $n>0$, we put
	\begin{align}
		\bb{T}^{c,d-c}_n=\spec(\bb{Z}[T_1^{\pm\frac{1}{n}},\dots,T_c^{\pm\frac{1}{n}},T_{c+1}^{\frac{1}{n}},\dots,T_d^{\frac{1}{n}}]).
	\end{align}
	We obtain a directed inverse system of schemes $(\bb{T}^{c,d-c}_n)_{n\in\bb{N}_{>0}}$, and we put $\bb{T}^{c,d-c}_\infty=\lim_{n\in\bb{N}_{>0}}\bb{T}^{c,d-c}_n$ and $\bb{T}^{c,d-c}_{p^\infty}=\lim_{n\in\bb{N}}\bb{T}^{c,d-c}_{p^n}$. Consider the analytifications over $(K,\ca{O}_K)$ and over $(C,\ca{O}_C)$:
	\begin{align}
		(\bb{T}^{c,d-c}_{n,K})^{\mrm{ad}}&=\spa(K\langle T_1^{\pm\frac{1}{n}},\dots,T_c^{\pm\frac{1}{n}},T_{c+1}^{\frac{1}{n}},\dots,T_d^{\frac{1}{n}}\rangle, \ca{O}_K\langle T_1^{\pm\frac{1}{n}},\dots,T_c^{\pm\frac{1}{n}},T_{c+1}^{\frac{1}{n}},\dots,T_d^{\frac{1}{n}}\rangle),\\
		(\bb{T}^{c,d-c}_{n,C})^{\mrm{ad}}&=\spa(C\langle T_1^{\pm\frac{1}{n}},\dots,T_c^{\pm\frac{1}{n}},T_{c+1}^{\frac{1}{n}},\dots,T_d^{\frac{1}{n}}\rangle, \ca{O}_C\langle T_1^{\pm\frac{1}{n}},\dots,T_c^{\pm\frac{1}{n}},T_{c+1}^{\frac{1}{n}},\dots,T_d^{\frac{1}{n}}\rangle).
	\end{align}
	We still denote by $(\bb{T}^{c,d-c}_{\infty,C})^{\mrm{ad}}$ (resp. $(\bb{T}^{c,d-c}_{p^\infty,C})^{\mrm{ad}}$) the adic space of the perfectoid affinoid algebra (cf. \cite[5.20]{scholze2012perfectoid})
	\begin{align}
		&(C\langle T_1^{\pm\frac{1}{\infty}},\dots,T_c^{\pm\frac{1}{\infty}},T_{c+1}^{\frac{1}{\infty}},\dots,T_d^{\frac{1}{\infty}}\rangle, \ca{O}_C\langle T_1^{\pm\frac{1}{\infty}},\dots,T_c^{\pm\frac{1}{\infty}},T_{c+1}^{\frac{1}{\infty}},\dots,T_d^{\frac{1}{\infty}}\rangle)\\
		\trm{(resp. }&(C\langle T_1^{\pm\frac{1}{p^\infty}},\dots,T_c^{\pm\frac{1}{p^\infty}},T_{c+1}^{\frac{1}{p^\infty}},\dots,T_d^{\frac{1}{p^\infty}}\rangle, \ca{O}_C\langle T_1^{\pm\frac{1}{p^\infty}},\dots,T_c^{\pm\frac{1}{p^\infty}},T_{c+1}^{\frac{1}{p^\infty}},\dots,T_d^{\frac{1}{p^\infty}}\rangle)\trm{)}\nonumber
	\end{align}
	which is the $p$-adic completion of the colimits of the affinoid algebras of $(\bb{T}^{c,d-c}_{n,C})^{\mrm{ad}}$. We omit the subscript $n$ when $n=1$.
\end{mypara}

\begin{mypara}\label{para:camargo-sen}
	Let $(X,\scr{M}_X)$ be an fs log smooth adic space over $\spa(K,\ca{O}_K)$ with log structure defined by a normal crossings divisor $D$ (\cite[2.3.17]{dllz2023log}). We denote by $X_\proket$ the pro-Kummer \'etale site of $(X,\scr{M}_X)$ and by $\widehat{\ca{O}}_{X_\proket}$ the completed structural sheaf (\cite[5.1.2, 5.4.1]{dllz2023log}). Let $\widetilde{X}$ be a $\ca{G}$-torsor over $X_\proket$ under a compact $p$-adic analytic group $\ca{G}$. For any finite free $\bb{Q}_p$-representation $V$ of $\ca{G}$ (\ref{defn:repn}), we denote by $V_\ket$ the finite locally free $\bb{Q}_p$-module over $X_\proket$ defined by the descent along $\widetilde{X}\to X$ (which is characterized by the $\ca{G}$-equivariant identity $V_\ket(\widetilde{X})=V$). Then, Rodr\'iguez Camargo defines a Higgs field on the finite locally free $\widehat{\ca{O}}_{X_\proket}$-module $V_\ket\otimes_{\bb{Q}_p}\widehat{\ca{O}}_{X_\proket}$ \cite[Theorem 3.3.2]{camargo2022sen} (see also \cite[Theorem 5.2.1]{camargo2022completed}),
	\begin{align}\label{para:camargo-sen-1}
		\theta_V:V_\ket\otimes_{\bb{Q}_p}\widehat{\ca{O}}_{X_\proket}\longrightarrow V_\ket\otimes_{\bb{Q}_p}\widehat{\ca{O}}_{X_\proket}\otimes_{\ca{O}_X}\Omega^1_{(X,\scr{M}_X)/K}(-1),
	\end{align}
	or equivalently a homomorphism of finite locally free $\widehat{\ca{O}}_{X_\proket}$-modules with commuting images,
	\begin{align}\label{para:camargo-sen-2}
		\varphi_{\mrm{RC}}^{\geo}|_V:\ca{H}om_{\ca{O}_X}(\Omega^1_{(X,\scr{M}_X)/K}(-1),\widehat{\ca{O}}_{X_\proket})\longrightarrow \ca{E}nd_{\widehat{\ca{O}}_{X_\proket}}(V_\ket\otimes_{\bb{Q}_p}\widehat{\ca{O}}_{X_\proket}),
	\end{align}
	which is functorial in $V$ and Kummer \'etale local on $X$. 
	
	More concretely, assume that $X$ admits a toric chart, i.e., an \'etale morphism $X\to (\bb{T}^{c,d-c}_K)^{\mrm{ad}}$ that is the composition of a finite sequence of finite \'etale morphisms and rational localizations such that $D\subseteq X$ is defined by $T_{c+1}\cdots T_d=0$ (such $X$ forms a topological generating family of the \'etale site of $X$ by \cite[3.1.13]{dllz2023log}). For any $n\in\bb{N}_{>0}\cup\{p^\infty,\infty\}$, we denote by $X_{n,C}\to X_C$ the (pro-)Kummer \'etale covering given by the base change of $(\bb{T}^{c,d-c}_{n,C})^{\mrm{ad}}\to (\bb{T}^{c,d-c}_C)^{\mrm{ad}}$ along the toric chart $X_C\to (\bb{T}^{c,d-c}_C)^{\mrm{ad}}$ (\cite[\textsection3.2.1]{camargo2022sen}). Note that for any integer $N$ prime to $p$, $X_{p^\infty N,C}\to X_{N,C}$ is a $\bb{Z}_p^d$-torsor regarded as a sheaf over $X_{N,C,\proket}$ and its quotient by $(p^m\bb{Z}_p)^d$ is $X_{p^mN,C}$. We put
	\begin{align}\label{para:camargo-sen-3}
		B_n=\widehat{\ca{O}}_{X_\proket}(X_{n,C}).
	\end{align}
	Then, $(V_\ket\otimes_{\bb{Q}_p}\widehat{\ca{O}}_{X_\proket})(X_{p^\infty,C})$ is a finite free $B_{p^\infty}$-representation of $\bb{Z}_p^d$ and for $m\in\bb{N}$ large enough, its submodule $V'$ of $(p^m\bb{Z}_p)^d$-analytic vectors is a finite free $B_{p^m}$-representation of $\bb{Z}_p^d$ \cite[Proposition 2.2.14, Theorem 2.4.4]{camargo2022sen} such that
	\begin{align}\label{para:camargo-sen-4}
		(V_\ket\otimes_{\bb{Q}_p}\widehat{\ca{O}}_{X_\proket})(X_{p^\infty,C})=B_{p^\infty}\otimes_{B_{p^m}}V'
	\end{align}
	and that for any $v'\in V'$, we have
	\begin{align}\label{para:camargo-sen-5}
		\theta_V|_{X_{p^\infty,C}}:(V_\ket\otimes_{\bb{Q}_p}\widehat{\ca{O}}_{X_\proket})(X_{p^\infty,C})&\longrightarrow (V_\ket\otimes_{\bb{Q}_p}\widehat{\ca{O}}_{X_\proket})(X_{p^\infty,C})\otimes_{\ca{O}_X}\Omega^1_{(X,\scr{M}_X)/K}(-1),\\
		1\otimes v'&\longmapsto \sum_{i=1}^d 1\otimes \varphi_{\partial_i}|_{V'}(v') \otimes \df\log(t_i)\otimes\zeta^{-1},\nonumber
	\end{align}
	where $(\partial_1,\dots,\partial_d)$ is the standard basis of the Lie algebra $\lie(\bb{Z}_p^d)$, $\varphi|_{V'}:\lie(\bb{Z}_p^d)\to \mrm{End}_{B_{p^m}}(V')$ is the infinitesimal action, $\df\log(t_i)\in \Omega^1_{(X,\scr{M}_X)/K}$ denotes the pullback of $\df\log(T_i)$ via the toric chart $X\to (\bb{T}^{c,d-c}_K)^{\mrm{ad}}$, and $\zeta=(\zeta_{p^n})_{n\in\bb{N}}$ is a basis of $\bb{Z}_p(1)$.
	
	Furthermore, Rodr\'iguez Camargo defines a homomorphism of finite locally free $\widehat{\ca{O}}_{X_\proket}$-modules \cite[Theorem 3.3.4]{camargo2022sen} (see also \cite[Theorem 5.2.1]{camargo2022completed}),
	\begin{align}\label{para:camargo-sen-6}
		\varphi_{\mrm{RC}}^{\geo}|_{\widetilde{X}}: \ca{H}om_{\ca{O}_X}(\Omega^1_{(X,\scr{M}_X)/K}(-1),\widehat{\ca{O}}_{X_\proket})\longrightarrow\widehat{\ca{O}}_{X_\proket}\otimes_{\bb{Q}_p}\lie(\ca{G})_\ket,
	\end{align}
	which is functorial in $\ca{G}$, Kummer \'etale local on $X$, and makes the following diagram commutative for any finite free $\bb{Q}_p$-representation $V$ of $\ca{G}$,
	\begin{align}\label{para:camargo-sen-7}
		\xymatrix{
			\ca{H}om_{\ca{O}_X}(\Omega^1_{(X,\scr{M}_X)/K}(-1),\widehat{\ca{O}}_{X_\proket})\ar[r]^-{\varphi_{\mrm{RC}}^{\geo}|_{\widetilde{X}}}\ar[d]_-{\varphi_{\mrm{RC}}^{\geo}|_V}&\widehat{\ca{O}}_{X_\proket}\otimes_{\bb{Q}_p}\lie(\ca{G})_\ket \ar[d]^-{\id_{\widehat{\ca{O}}_{X_\proket}}\otimes\varphi|_V}\\
			\ca{E}nd_{\widehat{\ca{O}}_{X_\proket}}(V_\ket\otimes_{\bb{Q}_p}\widehat{\ca{O}}_{X_\proket})&\widehat{\ca{O}}_{X_\proket}\otimes_{\bb{Q}_p}\mrm{End}_{\bb{Q}_p}(V)_\ket\ar[l]_-{\sim}
		}
	\end{align}
	where $\varphi|_V:\lie(\ca{G})\to \mrm{End}_{\bb{Q}_p}(V)$ is the infinitesimal action.
\end{mypara}

\begin{mypara}\label{para:camargo-sen-comp}
	Let $Y$ be a quasi-compact smooth $K$-scheme, $D$ a normal crossings divisor on $Y$, $(Y^{\triv}_\lambda)_{\lambda\in\Lambda}$ a directed inverse system of coherent schemes finite \'etale over $Y^{\triv}=Y\setminus D$, $Y_\lambda$ the integral closure of $Y$ in $Y_\lambda^{\triv}$ (so that $Y^{\triv}_\lambda=Y^{\triv}\times_YY_\lambda$). We put
	\begin{align}\label{eq:para:camargo-sen-comp-1}
		\widetilde{Y}^{\triv}=\lim_{\lambda\in\Lambda}Y^{\triv}_\lambda,\quad \widetilde{Y}=\lim_{\lambda\in\Lambda}Y_\lambda,
	\end{align}
	and assume that $\widetilde{Y}^{\triv}$ is a $\ca{G}$-torsor regarded as a sheaf over $Y^{\triv}_\profet$ under a compact $p$-adic analytic group $\ca{G}$.
	
	Let $Y^{\mrm{ad}}$ be the analytification of $Y$ as an adic space over $\spa(K,\ca{O}_K)$ endowed with the log structure $\scr{M}_{Y^{\mrm{ad}}}$ defined by the normal crossings divisor $D^{\mrm{ad}}$ so that $(Y^{\mrm{ad}},\scr{M}_{Y^{\mrm{ad}}})$ is an fs log smooth adic space (\cite[2.3.17]{dllz2023log}). Note that $Y^{\triv,\mrm{ad}}=Y^{\mrm{ad}}\setminus D^{\mrm{ad}}$ (\cite[4.6.(\luoma{1})]{huber1994general}, \cite[9.1.9]{fujiwarakato2018rigid}). For any $\lambda\in\Lambda$, let $Y_{\lambda}^{\mrm{ad}}$ be the analytification of $Y_{\lambda}$ as an adic space endowed with the log structure $\scr{M}_{Y_{\lambda}^{\mrm{ad}}}$ defined by $(Y_{\lambda}\setminus Y_{\lambda}^{\triv})^{\mrm{ad}}=Y_{\lambda}^{\mrm{ad}}\setminus Y_{\lambda}^{\triv,\mrm{ad}}$ \cite[2.3.16]{dllz2023log}. Then, $(Y_{\lambda}^{\mrm{ad}},\scr{M}_{Y_{\lambda}^{\mrm{ad}}})$ is an fs log smooth adic space finite Kummer \'etale over $(Y^{\mrm{ad}},\scr{M}_{Y^{\mrm{ad}}})$ by \cite[4.2.1]{dllz2023log} and \cite[1.6]{hansen2020vanish}. Therefore, we put
	\begin{align}\label{eq:para:camargo-sen-comp-2}
		\widetilde{Y}^{\triv,\mrm{ad}}=\lim_{\lambda\in\Lambda}Y^{\triv,\mrm{ad}}_{\lambda},\quad \widetilde{Y}^{\mrm{ad}}=\lim_{\lambda\in\Lambda}Y_{\lambda}^{\mrm{ad}}
	\end{align}
	as objects of $Y^{\mrm{ad}}_\proket$ (\cite[5.1.2]{dllz2023log}). In particular, $\widetilde{Y}^{\mrm{ad}}$ is a $\ca{G}$-torsor regarded as a sheaf over $Y^{\mrm{ad}}_\proket$ (as $\widetilde{Y}^{\triv,\mrm{ad}}$ is a $\ca{G}$-torsor regarded as a sheaf over $\widetilde{Y}^{\triv,\mrm{ad}}_\profet$).
	
	As in \ref{cor:vanish}, we fix an irreducible component of $\widetilde{Y}^{\triv}$ and denote its generic point by $\spec(\ca{L})\in \widetilde{Y}^{\triv}$. Let $\overline{\ca{K}}$ be an algebraic closure of $\ca{L}$, $\overline{y}=\spec(\overline{\ca{F}})$ a point of the integral closure $Y^{\overline{\ca{K}}}$ of $Y$ in $\overline{\ca{K}}$, $\ca{O}_{\overline{\ca{F}}}$ a valuation ring of height $1$ extension of $\ca{O}_K$ with fraction field $\overline{\ca{F}}$. We identify $\overline{K}$ with the algebraic closure of $K$ in $\overline{\ca{K}}$. We denote by $\spec(\ca{K})\in Y^{\triv}$ the image of the generic point $\spec(\ca{L})\in \widetilde{Y}^{\triv}$ (which is thus the generic point of its irreducible component of $Y^{\triv}$ by {\rm\cite[3.5.(2)]{he2024purity}}). Note that
	\begin{align}\label{eq:para:camargo-sen-comp-3}
		\ca{G}^{\ca{L}}=\gal(\ca{L}/\ca{K})
	\end{align}
	is naturally identified with the closed subgroup of $\ca{G}$ which stabilizes the connected component $Y^{\ca{L}}$ (\ref{lem:profet-component}) and thus is still a $p$-adic analytic group. Therefore, there is a canonical commutative diagram of locally ringed spaces for any $\lambda\in\Lambda$,
	\begin{align}\label{eq:para:camargo-sen-comp-4}
		\xymatrix{
			\widehat{\overline{y}}=\spa(\widehat{\overline{\ca{F}}},\ca{O}_{\widehat{\overline{\ca{F}}}})\ar[r]\ar[d]&Y_{\lambda}^{\mrm{ad}}\ar[r]\ar[d]&Y^{\mrm{ad}}\ar[d]\\
			\overline{y}=\spec(\overline{\ca{F}})\ar[r]&Y_\lambda\ar[r]&Y.
		}
	\end{align}
	In particular, $\widehat{\overline{y}}\to Y^{\mrm{ad}}$ is a geometric point of adic spaces, and we endow $\widehat{\overline{y}}$ with a log structure $\scr{M}_{\widehat{\overline{y}}}$ such that $(\widehat{\overline{y}},\scr{M}_{\widehat{\overline{y}}})\to (Y^{\mrm{ad}},\scr{M}_{Y^{\mrm{ad}}})$ is a log geometric point (\cite[4.4.3]{dllz2023log}). Notice that for each $\lambda\in\Lambda$, the set of liftings of a $Y^{\mrm{ad}}$-morphism $\widehat{\overline{y}}\to Y^{\mrm{ad}}_\lambda$ to a $(Y^{\mrm{ad}},\scr{M}_{Y^{\mrm{ad}}})$-morphism $(\widehat{\overline{y}},\scr{M}_{\widehat{\overline{y}}})\to (Y^{\mrm{ad}}_\lambda,\scr{M}_{Y^{\mrm{ad}}_\lambda})$ is a non-empty finite set \cite[4.4.7]{dllz2023log}. Thus, we can fix a compatible system of morphisms of log adic spaces $((\widehat{\overline{y}},\scr{M}_{\widehat{\overline{y}}})\to (Y^{\mrm{ad}}_\lambda,\scr{M}_{Y^{\mrm{ad}}_\lambda}))_{\lambda\in\Lambda}$ lifting the given compatible system of morphisms of adic spaces $(\widehat{\overline{y}}\to Y^{\mrm{ad}}_\lambda)_{\lambda\in\Lambda}$ by the non-emptiness of cofiltered limits of non-empty finite sets.
	
	Recall that associated to the data above, there is a universal geometric Sen action of the $p$-adic analytic Galois extension $\ca{L}$ of $\ca{K}$ \eqref{eq:thm:sen-val-perfd-1},
	\begin{align}\label{eq:para:camargo-sen-comp-5}
		\varphi^{\mrm{geo}}_{\sen}|_{\ca{G}^{\ca{L}},\overline{y}}: \ho_{\ca{O}_Y}(\Omega^1_{(Y,\scr{M}_Y)/K}(-1),\widehat{\overline{\ca{F}}})\longrightarrow \widehat{\overline{\ca{F}}}\otimes_{\bb{Q}_p}\lie(\ca{G}^{\ca{L}}).
	\end{align}
	
	On the other hand, taking the pullback of the geometric Sen action constructed by Rodr\'iguez Camargo \eqref{para:camargo-sen-6} along the log geometric point $(\widehat{\overline{y}},\scr{M}_{\widehat{\overline{y}}})\to (Y^{\mrm{ad}},\scr{M}_{Y^{\mrm{ad}}})$ (cf. \cite[4.4.4]{dllz2023log}), we obtain an $\widehat{\overline{\ca{F}}}$-linear homomorphism
	\begin{align}\label{eq:para:camargo-sen-comp-6}
		\varphi_{\mrm{RC}}^{\geo}|_{\widetilde{Y}^{\mrm{ad}},\widehat{\overline{y}}}: \ho_{\ca{O}_{Y^{\mrm{ad}}}}(\Omega^1_{(Y^{\mrm{ad}},\scr{M}_{Y^{\mrm{ad}}})/K}(-1),\widehat{\overline{\ca{F}}})\longrightarrow\widehat{\overline{\ca{F}}}\otimes_{\bb{Q}_p}\lie(\ca{G})
	\end{align}
	where the identification of the stalk of $(\lie(\ca{G}))_\ket$ with $\lie(\ca{G})$ is defined by the chosen compatible system of morphisms $(\widehat{\overline{y}}\to Y^{\mrm{ad}}_\lambda)_{\lambda\in\Lambda}$.
\end{mypara}

\begin{myprop}\label{prop:camargo-sen-comp}
	Under the assumptions in {\rm\ref{para:camargo-sen-comp}} and with the same notation, the following canonical diagram is commutative,
	\begin{align}\label{eq:prop:camargo-sen-comp-1}
		\xymatrix{
			\ho_{\ca{O}_{Y^{\mrm{ad}}}}(\Omega^1_{(Y^{\mrm{ad}},\scr{M}_{Y^{\mrm{ad}}})/K}(-1),\widehat{\overline{\ca{F}}})\ar[rr]^-{\varphi_{\mrm{RC}}^{\geo}|_{\widetilde{Y}^{\mrm{ad}},\widehat{\overline{y}}}}\ar[d]^-{\wr}&&\widehat{\overline{\ca{F}}}\otimes_{\bb{Q}_p}\lie(\ca{G})\\
			\ho_{\ca{O}_Y}(\Omega^1_{(Y,\scr{M}_Y)/K}(-1),\widehat{\overline{\ca{F}}})\ar[rr]^-{\varphi^{\mrm{geo}}_{\sen}|_{\ca{G}^{\ca{L}},\overline{y}}}&& \widehat{\overline{\ca{F}}}\otimes_{\bb{Q}_p}\lie(\ca{G}^{\ca{L}})\ar@{^{(}->}[u]
		}
	\end{align}
	where the left vertical isomorphism is induced by the comparison morphism $Y^{\mrm{ad}}\to Y$.
\end{myprop}
\begin{proof}
	Firstly, by the functoriality of $\varphi_{\mrm{RC}}^{\geo}|_{\widetilde{Y}^{\mrm{ad}}}$, after replacing $\widetilde{Y}$ by the connected component $Y^{\ca{L}}$ (\ref{lem:profet-component}), we may assume that $\widetilde{Y}=Y^{\ca{L}}$ is irreducible and $\ca{G}=\ca{G}^{\ca{L}}$.
	
	Since $\varphi^{\mrm{geo}}_{\sen}|_{\ca{G},\overline{y}}$ (resp. $\varphi_{\mrm{RC}}^{\geo}|_{\widetilde{Y}^{\mrm{ad}},\widehat{\overline{y}}}$) is invariant after replacing $Y$ by an \'etale neighborhood of $\overline{y}$ and replacing $K$ by a finite field extension by \ref{rem:sen-val-pt} (resp. by \cite[\textsection3]{camargo2022sen}). Thus, we may assume that there exists an adequate $\ca{O}_K$-algebra $A$ with $(Y^{\triv}\to Y)=(\spec(A_{\triv})\to\spec(A[1/p]))$ with $A\subseteq \ca{O}_{\overline{\ca{F}}}$ by \ref{prop:aet-neighborhood}. Thus, we can use the Sen action over $A$ defined in \ref{thm:sen-lie-lift-A} to compute $\varphi^{\mrm{geo}}_{\sen}|_{\ca{G},\overline{y}}$. On the other hand, $X=\spa(\widehat{A}[1/p],\widehat{A})$ is canonically identified with an open subspace of $Y^{\mrm{ad}}$ containing the image of $\widehat{\overline{y}}$ (\cite[4.6.(\luoma{1})]{huber1994general}, \cite[7.4.16]{abbes2010rigid}). Thus, we can use the Sen action over $X$ defined in \ref{para:camargo-sen} to compute $\varphi_{\mrm{RC}}^{\geo}|_{\widetilde{Y}^{\mrm{ad}},\widehat{\overline{y}}}$. We remark that there is a canonical morphism of locally ringed spaces
	\begin{align}\label{eq:prop:camargo-sen-comp-2}
		X=\spa(\widehat{A}[1/p],\widehat{A})\longrightarrow Y=\spec(A[1/p]).
	\end{align}
	
	We take an adequate chart $(\alpha:\bb{N}\to \ca{O}_K,\ \beta:P\to A,\ \gamma:\bb{N}\to P)$ and an isomorphism
	$P_\eta\cong \bb{Z}\oplus \bb{Z}^c\oplus \bb{N}^{d-c}$ \eqref{eq:monoid-str}. Let $t_1,\dots,t_d\in A[1/p]$ be the associated system of coordinates of this chart and we take again the notation in \ref{para:notation-A-galois} for $(s_1,\dots,s_e)=(t_1,\dots,t_d)$. In particular, this chart induces an \'etale morphism
	\begin{align}\label{eq:prop:camargo-sen-comp-3}
		Y=\spec(A[1/p])\longrightarrow \spec(K\otimes_{\bb{Z}[\bb{Z}]}\bb{Z}[P_\eta])\cong \bb{T}^{c,d-c}_K,
	\end{align}
	which induces further an \'etale morphism $X\to (\bb{T}^{c,d-c}_K)^{\mrm{ad}}$ by analytification (\cite[1.7.3.(\luoma{1})]{huber1996etale}). After replacing $X$ by an \'etale neighborhood of $\widehat{\overline{y}}$, we may assume that $X\to (\bb{T}^{c,d-c}_K)^{\mrm{ad}}$ is the composition of a finite sequence of finite \'etale morphisms and rational localizations (\cite[2.2.8]{huber1996etale}) so that it is a toric chart.
	
	Recall that for any integers $n,m\in\bb{N}$ and any integer $N$ prime to $p$, the canonical commutative diagram of schemes
	\begin{align}\label{eq:prop:camargo-sen-comp-4}
		\xymatrix{
			Y^{(\underline{N})}_{n,\underline{m}}=\spec(A^{(\underline{N})}_{n,\underline{m}}[1/p])\ar[r]\ar[d]&\bb{T}^{c,d-c}_{p^m,K}\ar[d]\\
			Y^{(\underline{N})}_{n,\underline{0}}=\spec(A^{(\underline{N})}_{n,\underline{0}}[1/p])\ar[r]&\bb{T}^{c,d-c}_K
		}
	\end{align}
	is Cartesian (\cite[4.8]{tsuji2018localsimpson}, see \cite[5.10.(5), 5.12]{he2024purity}). Combining with the notation in \ref{para:camargo-sen}, we obtain a canonical commutative diagram of locally ringed spaces
	\begin{align}\label{eq:prop:camargo-sen-comp-5}
		\xymatrix{
			X_{p^mN,C}\ar[r]\ar[d]&Y^{(\underline{N})}_{n,\underline{m}}=\spec(A^{(\underline{N})}_{n,\underline{m}}[1/p])\ar[r]\ar[d]&\bb{T}^{c,d-c}_{p^m,K}\ar[d]\\
			X_{N,C}\ar[r]\ar[d]&Y^{(\underline{N})}_{n,\underline{0}}=\spec(A^{(\underline{N})}_{n,\underline{0}}[1/p])\ar[r]\ar[d]&\bb{T}^{c,d-c}_K\ar[d]\\
			\spa(C,\ca{O}_C)\ar[r]&\spec(K(\zeta_{p^nN}))\ar[r]&\spec(K)
		}
	\end{align}
	since $X_{p^mN,C}=X_{N,C}\times_{(\bb{T}^{c,d-c}_C)^{\mrm{ad}}}(\bb{T}^{c,d-c}_{p^m,C})^{\mrm{ad}}$ by definition. Recall that $B_{p^mN}=\widehat{\ca{O}}_{Y^{\mrm{ad}}_\proket}(X_{p^mN,C})$.
	
	We take a faithful finite free $\bb{Q}_p$-representation $V$ of $\ca{G}$ (\cite[3.9]{he2022sen}). Note that the infinitesimal action $\varphi|_V:\lie(\ca{G})\to\mrm{End}_{\bb{Q}_p}(V)$ is injective (\cite[4.10.(3)]{he2022sen}). By the universal properties of $\varphi^{\mrm{geo}}_{\sen}|_{\ca{G},\overline{y}}$ and $\varphi_{\mrm{RC}}^{\geo}|_{\widetilde{Y}^{\mrm{ad}},\widehat{\overline{y}}}$ (\ref{thm:sen-lie-lift-A} and \eqref{para:camargo-sen-7}), it suffices to check the commutativity of the following diagram
	\begin{align}\label{eq:prop:camargo-sen-comp-6}
	 	\xymatrix{
	 		\ho_{\ca{O}_{Y^{\mrm{ad}}}}(\Omega^1_{(Y^{\mrm{ad}},\scr{M}_{Y^{\mrm{ad}}})/K}(-1),\widehat{\overline{\ca{F}}})\ar[rr]^-{\id_{\widehat{\overline{\ca{F}}}}\otimes\varphi_{\mrm{RC}}^{\geo}|_V}\ar[d]^-{\wr}&&\widehat{\overline{\ca{F}}}\otimes_{\bb{Q}_p}\mrm{End}_{\bb{Q}_p}(V)\ar[d]^-{\wr}\\
	 		\ho_{\ca{O}_Y}(\Omega^1_{(Y,\scr{M}_Y)/K}(-1),\widehat{\overline{\ca{F}}})\ar[rr]^-{\id_{\widehat{\overline{\ca{F}}}}\otimes\varphi^{\mrm{geo}}_{\sen}|_W}&& \widehat{\overline{\ca{F}}}\otimes_{\widehat{\overline{A}}[1/p]}\mrm{End}_{\widehat{\overline{A}}[1/p]}(W)
	 	}
	\end{align}
	where $W=V\otimes_{\bb{Q}_p}\widehat{\overline{A}}[1/p]$ is the associated object of $\repnpr(G_A,\widehat{\overline{A}}[1/p])$. 
	\begin{align}\label{eq:prop:camargo-sen-comp-7}
		\xymatrix{
			\ca{K}_{\mrm{ur}}&&\\
			\ca{K}^{(\underline{N})}_{\infty,\underline{\infty}}\ar[u]^-{H^{(\underline{N})}_{\underline{\infty}}}&&\\
			\ca{K}^{(\underline{N})}_{\infty,\underline{m}}\ar[u]^-{\Delta^{(\underline{N})}_{\underline{m}}}\ar@/^3pc/[uu]^-{H^{(\underline{N})}_{\underline{m}}}&\ca{K}^{(\underline{N})}_{n,\underline{m}} \ar[l]^-{\Sigma^{(\underline{N})}_{n,\underline{m}}}\ar[lu]|-{\Gamma^{(\underline{N})}_{n,\underline{m}}}&\ca{K}\ar@/_1pc/[lluu]|{G_A}\ar[l]\ar@/_1pc/[llu]|(0.6){\Xi^{(\underline{N})}}
		}
	\end{align}
	By \ref{thm:descent}, there exists $\underline{N}\in J$, $\underline{m}\in \bb{N}^d$, and a $\Delta^{(\underline{N})}_{\underline{m}}=\gal(\ca{K}^{(\underline{N})}_{\infty,\underline{\infty}}/\ca{K}^{(\underline{N})}_{\infty,\underline{m}})$-analytic finite projective $\widetilde{A}^{(\underline{N})}_{\infty,\underline{m}}[1/p]$-representation $M$ of $\Xi^{(\underline{N})}=\gal(\ca{K}^{(\underline{N})}_{\infty,\underline{\infty}}/\ca{K})$ such that
	\begin{align}\label{eq:prop:camargo-sen-comp-8}
		V\otimes_{\bb{Q}_p}\widehat{\overline{A}}[1/p]=\widehat{\overline{A}}[1/p]\otimes_{\widetilde{A}^{(\underline{N})}_{\infty,\underline{m}}[1/p]}M.
	\end{align} 
	Taking $H^{(\underline{N})}_{\underline{\infty}}=\gal(\ca{K}_{\mrm{ur}}/\ca{K}^{(\underline{N})}_{\infty,\underline{\infty}})$-invariant of \eqref{eq:prop:camargo-sen-comp-8}, we get (\cite[6.5]{tsuji2018localsimpson}, cf. \cite[10.15]{he2024purity})
	\begin{align}\label{eq:prop:camargo-sen-comp-9}
		V\otimes_{\bb{Q}_p}\widehat{A}^{(\underline{N})}_{\infty,\underline{\infty}}[1/p]=\widehat{A}^{(\underline{N})}_{\infty,\underline{\infty}}[1/p]\otimes_{\widetilde{A}^{(\underline{N})}_{\infty,\underline{m}}[1/p]}M.
	\end{align}
	Notice that $\varphi^{\mrm{geo}}_{\sen}|_{\ca{G},\overline{y}}$ and $\varphi_{\mrm{RC}}^{\geo}|_{\widetilde{Y}^{\mrm{ad}},\widehat{\overline{y}}}$ are invariant after replacing $Y=\spec(A[1/p])$ by $Y^{(\underline{N})}=\spec(A^{(\underline{N})}[1/p])$ by \ref{rem:sen-val-pt} and \cite[\textsection3]{camargo2022sen}, since $(Y^{(\underline{N})},\scr{M}_{Y^{(\underline{N})}})$ is Kummer \'etale over $(Y,\scr{M}_Y)$ by \cite[\Luoma{9}.2.1]{gabber2014travaux}. Therefore, we may assume that $\underline{N}=\underline{1}$. Hence, tensoring \eqref{eq:prop:camargo-sen-comp-9} with $B_{p^\infty}$ over $\widehat{A}_{\infty,\underline{\infty}}[1/p]$ (cf. \eqref{eq:prop:camargo-sen-comp-5}), we obtain
	\begin{align}\label{eq:prop:camargo-sen-comp-10}
		V\otimes_{\bb{Q}_p}B_{p^\infty}=B_{p^\infty}\otimes_{B_{p^m}}(B_{p^m}\otimes_{\widetilde{A}_{\infty,\underline{m}}[1/p]}M).
	\end{align}
	By construction, $V'=B_{p^m}\otimes_{\widetilde{A}^{(\underline{N})}_{\infty,\underline{m}}[1/p]}M$ is a $\Delta_{\underline{m}}$-analytic finite projective $B_{p^m}$-representation of $\Delta=\gal(\ca{K}^{(\underline{N})}_{\infty,\underline{\infty}}/\ca{K}^{(\underline{N})}_{\infty,\underline{0}})\cong \bb{Z}_p^d$ (note that $\Delta_{\underline{m}}=p^m\Delta\cong (p^m\bb{Z}_p)^d$). After enlarging $m$, we may assume that this $V'$ coincides with that in \eqref{para:camargo-sen-4}. In conclusion, if $(\df\log(t_1)^*\otimes\zeta,\dots,\df\log(t_d)^*\otimes\zeta)$ denotes the dual basis of the basis $(\df\log(t_1)\otimes\zeta^{-1},\dots,\df\log(t_d)\otimes\zeta^{-1})$ of $\Omega^1_{(Y,\scr{M}_Y)/K}(-1)$, then for any integer $1\leq i\leq d$, we have
	\begin{align}
		\varphi^{\mrm{geo}}_{\sen}|_W(\df\log(t_i)^*\otimes\zeta)&=\id_{\widehat{\overline{A}}[1/p]}\otimes \varphi_{\partial_i}|_{M},\\
		\varphi_{\mrm{RC}}^{\geo}|_{V,X_{p^\infty,C}}(\df\log(t_i)^*\otimes\zeta)&=\id_{B_{p^\infty}}\otimes \varphi_{\partial_i}|_{V'},
	\end{align}
	by \eqref{eq:thm:sen-brinon-A} and \eqref{para:camargo-sen-5} respectively, where $(\partial_1,\dots,\partial_d)$ is the standard basis of $\lie(\Delta)\cong \lie(\bb{Z}_p^d)$ and $\varphi|_M:\lie(\Delta)\to \mrm{End}_{\widetilde{A}_{\infty,\underline{m}}[1/p]}(M)$ (resp. $\varphi|_{V'}:\lie(\bb{Z}_p^d)\to \mrm{End}_{B_{p^m}}(V')$) is the infinitesimal action. This verifies the commutativity of \eqref{eq:prop:camargo-sen-comp-6}.
\end{proof}

\begin{mycor}\label{cor:camargo-sen-vanish}
	Under the assumptions in {\rm\ref{para:camargo-sen-comp}} and with the same notation, assume moreover that $Y$ is proper over $K$ and that the universal geometric Sen action \eqref{para:camargo-sen-6},
	\begin{align}
		\varphi_{\mrm{RC}}^{\geo}|_{\widetilde{Y}^{\mrm{ad}}}: \ca{H}om_{\ca{O}_{Y^{\mrm{ad}}}}(\Omega^1_{(Y^{\mrm{ad}},\scr{M}_{Y^{\mrm{ad}}})/K}(-1),\widehat{\ca{O}}_{Y^{\mrm{ad}}_\proket})\longrightarrow\widehat{\ca{O}}_{Y^{\mrm{ad}}_\proket}\otimes_{\bb{Q}_p}\lie(\ca{G})_\ket,
	\end{align}
	is locally splitting injective. Then, for any integer $q>\dim(Y)$ and $n\in\bb{N}$, we have
	\begin{align}
		H^q(\widetilde{Y}^{\triv}_{\overline{K},\et},\bb{Z}/p^n\bb{Z})=0,
	\end{align}
	where $\widetilde{Y}^{\triv}_{\overline{K}}=\spec(\overline{K})\times_{\spec(K)}\widetilde{Y}^{\triv}$ and $\overline{K}$ is an algebraic closure of $K$.
\end{mycor}
\begin{proof}
	The assumption implies that the pullback $\varphi_{\mrm{RC}}^{\geo}|_{\widetilde{Y}^{\mrm{ad}},\widehat{\overline{y}}}$ \eqref{eq:para:camargo-sen-comp-6} of $\varphi_{\mrm{RC}}^{\geo}|_{\widetilde{Y}^{\mrm{ad}}}$ along the log geometric point $(\widehat{\overline{y}},\scr{M}_{\widehat{\overline{y}}})\to (Y^{\mrm{ad}},\scr{M}_{Y^{\mrm{ad}}})$ is still injective. Hence, the conclusion follows directly from \ref{prop:camargo-sen-comp} and \ref{cor:vanish}.
\end{proof}

\begin{myrem}\label{rem:camargo-sen-vanish}
	In fact, we can replace the ``splitting injectivity" assumption in \ref{cor:camargo-sen-vanish} by the ``surjectivity" of the dual of the universal geometric Sen action
	\begin{align}
		\varphi_{\mrm{RC}}^{\geo*}|_{\widetilde{Y}^{\mrm{ad}}}:\ca{H}om_{\bb{Q}_p}(\lie(\ca{G})_\ket,\widehat{\ca{O}}_{Y^{\mrm{ad}}_\proket})\longrightarrow \widehat{\ca{O}}_{Y^{\mrm{ad}}_\proket}\otimes_{\ca{O}_{Y^{\mrm{ad}}}}\Omega^1_{(Y^{\mrm{ad}},\scr{M}_{Y^{\mrm{ad}}})/K}(-1)
	\end{align}
	which still guarantees the ``pointwise injectivity" of the the universal geometric Sen action $\varphi_{\mrm{RC}}^{\geo}|_{\widetilde{Y}^{\mrm{ad}},\widehat{\overline{y}}}$.
\end{myrem}

\section{Applications to Shimura Varieties}\label{sec:shimura}
In this section, we prove Calegari-Emerton's conjecture on the vanishing of higher completed cohomology groups of Shimura varieties (see \ref{thm:vanish-shimura}). A key ingredient is the computation of universal geometric Sen action over Shimura varieties at infinite level due to Pan and Rodr\'iguez Camargo (see \ref{thm:camargo}).

\begin{mypara}\label{para:notation-shimura}
	Following \cite[\textsection4]{camargo2022completed}, we fix a Shimura datum $(G,X)$ (\cite[2.1.1]{deligne1979shimura}, see also \cite[5.5]{milne2005shimura}) and let $E\subseteq \bb{C}$ be its reflex field (which is a finite extension of $\bb{Q}$, \cite[2.2.1]{deligne1979shimura}, see also \cite[12.2]{milne2005shimura}). We denote by $\bb{A}_f$ (resp.  $\bb{A}_f^p$) the ring of (resp. prime-to-$p$) finite ad\`eles of $\bb{Q}$. For any neat compact open subgroup $K\subseteq G(\bb{A}_f)$ (\cite[0.6]{pink1990compact}), we denote by $\mrm{Sh}_K$ the canonical model of the Shimura variety associated to $(G,X)$ of level $K$ (see \cite[page 128]{milne2005shimura}). It is a quasi-projective smooth $E$-scheme, whose $\bb{C}$-points are canonically identified with
	\begin{align}
		\mrm{Sh}_K(\bb{C})=G(\bb{Q})\backslash (X\times G(\bb{A}_f))/K.
	\end{align}
	Moreover, these canonical models form a directed inverse system of $E$-schemes $(\mrm{Sh}_K)_{K\subseteq G(\bb{A}_f)}$ (note that open subgroups of $K$ are also neat) with finite \'etale transition morphisms (see \cite[2.1.2]{deligne1979shimura}). We denote by $d$ the common dimension of $\mrm{Sh}_K$.
	
	We fix a compact open subgroup $K^p\subseteq G(\bb{A}_f^p)$. Consider the directed inverse system of $E$-schemes $(\mrm{Sh}_{K^pK_p})_{K_p\subseteq G(\bb{Q}_p)}$, where $K_p$ runs through all the neat compact open subgroups of $G(\bb{Q}_p)$ (\cite[2.12]{hansenjohansson2023perfectoid}). Its limit
	\begin{align}
		\mrm{Sh}_{K^p}=\lim_{K_p\subseteq G(\bb{Q}_p)}\mrm{Sh}_{K^pK_p}
	\end{align}
	is pro-finite \'etale over each $\mrm{Sh}_{K^pK_p}$ and is a $\ca{G}_{K_p}$-torsor as a sheaf over $\mrm{Sh}_{K^pK_p,\profet}$ under a compact $p$-adic analytic group $\ca{G}_{K_p}$ (see \cite[\textsection4.1]{camargo2022completed}).
	
	Then, we also fix a neat compact open subgroup $K_p\subseteq G(\bb{Q}_p)$ and a toroidal compactification $\mrm{Sh}^{\mrm{tor}}_{K^pK_p}$ of $\mrm{Sh}_{K^pK_p}$ which is a projective smooth $E$-scheme with $D=\mrm{Sh}^{\mrm{tor}}_{K^pK_p}\setminus \mrm{Sh}_{K^pK_p}$ a normal crossings divisor (\cite[9.21, 12.4]{pink1990compact}). For any compact open subgroup $K'_p\subseteq K_p$, we denote by $\mrm{Sh}^{\mrm{tor}}_{K^pK_p'}$ the integral closure of $\mrm{Sh}^{\mrm{tor}}_{K^pK_p}$ in $\mrm{Sh}_{K^pK_p'}$. Then, the transition morphisms of the directed inverse system of $E$-schemes $(\mrm{Sh}^{\mrm{tor}}_{K^pK_p'})_{K_p'\subseteq K_p}$ are finite Kummer \'etale with respect to the divisor $D\subseteq \mrm{Sh}^{\mrm{tor}}_{K^pK_p}$ (\cite[\Luoma{9}.2.1]{gabber2014travaux}). We put
	\begin{align}
		\mrm{Sh}^{\mrm{tor}}_{K^p}=\lim_{K_p'\subseteq K_p}\mrm{Sh}^{\mrm{tor}}_{K^pK_p'}.
	\end{align}
	
	We fix a finite field extension $L$ of $\bb{Q}_p$ containing the reflex field $E$. Taking base change along $\spec(L)\to\spec(E)$, we obtain
	\begin{align}
		\mrm{Sh}_{K^p,L}=\lim_{K_p'\subseteq K_p}\mrm{Sh}_{K^pK_p',L},\quad \mrm{Sh}^{\mrm{tor}}_{K^p,L}=\lim_{K_p'\subseteq K_p}\mrm{Sh}^{\mrm{tor}}_{K^pK_p',L}.
	\end{align}
	Note that $\mrm{Sh}_{K^p,L}$ is a $\ca{G}_{K_p}$-torsor over $\mrm{Sh}_{K^pK_p,L,\profet}$. Thus, we are in the situation of \ref{para:camargo-sen-comp} and we denote by $\ca{S}h_{K^pK_p',L}$ (resp. $\ca{S}h^{\mrm{tor}}_{K^pK_p',L}$) the analytification of $\mrm{Sh}_{K^pK_p',L}$ (resp. $\mrm{Sh}^{\mrm{tor}}_{K^pK_p',L}$) as an adic space over $\spa(L,\ca{O}_L)$. Then, by \ref{para:camargo-sen-comp}, we obtain a directed inverse system of fs log adic spaces $(\ca{S}h^{\mrm{tor}}_{K^pK_p',L},\scr{M}_{K^pK'_p,L})_{K_p'\subseteq K_p}$ proper smooth over $\spa(L,\ca{O}_L)$ with finite Kummer \'etale transition morphisms and we put
	\begin{align}
		\ca{S}h_{K^p,L}=\lim_{K_p'\subseteq K_p}\ca{S}h_{K^pK_p',L},\quad \ca{S}h^{\mrm{tor}}_{K^p,L}=\lim_{K_p'\subseteq K_p}\ca{S}h^{\mrm{tor}}_{K^pK_p',L}
	\end{align}
	the corresponding objects of $\ca{S}h^{\mrm{tor}}_{K^pK_p,L,\proket}$. Note that $\ca{S}h^{\mrm{tor}}_{K^p,L}$ is a $\ca{G}_{K_p}$-torsor over $\ca{S}h^{\mrm{tor}}_{K^pK_p,L,\proket}$.
\end{mypara}

\begin{mythm}[{\cite[Theorem 5.2.5]{camargo2022completed}, cf. \cite[4.2.7]{pan2021locally}}]\label{thm:camargo}
	With the notation in {\rm\ref{para:notation-shimura}}, the universal geometric Sen action of the $\ca{G}_{K_p}$-torsor $\ca{S}h^{\mrm{tor}}_{K^p,L}$ over $\ca{S}h^{\mrm{tor}}_{K^pK_p,L,\proket}$ (in the sense of \cite{camargo2022sen}, see \eqref{para:camargo-sen-6}),
	\begin{align}
		\varphi_{\mrm{RC}}^{\geo}|_{\ca{S}h^{\mrm{tor}}_{K^p,L}}:\ca{H}om(\Omega^1_{(\ca{S}h^{\mrm{tor}}_{K^pK_p,L},\scr{M}_{K^pK_p,L})/L}(-1),\widehat{\ca{O}}_{\ca{S}h^{\mrm{tor}}_{K^pK_p,L,\proket}})\to\widehat{\ca{O}}_{\ca{S}h^{\mrm{tor}}_{K^pK_p,L,\proket}}\otimes_{\bb{Q}_p}\lie(\ca{G}_{K_p})_\ket,
	\end{align}
	 is locally splitting injective.
\end{mythm}

\begin{mythm}\label{thm:vanish-shimura}
	With the notation in {\rm\ref{para:notation-shimura}}, for any integers $q>d$ and $n\in\bb{N}$, we have
	\begin{align}
		H^q_\et(\mrm{Sh}_{K^p,\bb{C}},\bb{Z}/p^n\bb{Z})=\colim_{K_p\subseteq G(\bb{Q}_p)}H^q_\et(\mrm{Sh}_{K^pK_p,\bb{C}},\bb{Z}/p^n\bb{Z})=0,
	\end{align}
	where $\mrm{Sh}_{K^p,\bb{C}}=\spec(\bb{C})\times_{\spec(E)}\mrm{Sh}_{K^p}$ is the Shimura variety of $(G,X)$ over $\bb{C}$ at the infinite level $K^p$.
\end{mythm}
\begin{proof}
	Let $\overline{\bb{Q}}_p$ be an algebraic closure of $L$. By \ref{cor:camargo-sen-vanish} and \ref{thm:camargo}, we have  
	\begin{align}
		H^q_\et(\mrm{Sh}_{K^p,\overline{\bb{Q}}_p},\bb{Z}/p^n\bb{Z})=\colim_{K_p\subseteq G(\bb{Q}_p)}H^q_\et(\mrm{Sh}_{K^pK_p,\overline{\bb{Q}}_p},\bb{Z}/p^n\bb{Z})=0.
	\end{align}
	Then, the conclusion follows immediately as both $\overline{\bb{Q}}_p$ and $\bb{C}$ are algebraically closed fields of characteristic $0$ (\cite[\href{https://stacks.math.columbia.edu/tag/0F0B}{0F0B}]{stacks-project}).
\end{proof}

\begin{myrem}\label{rem:vanish-shimura-1}
	We have actually proven the pointwise perfectoidness and $p$-infinite ramification at boundary points of the compactified Shimura variety $\mrm{Sh}^{\mrm{tor}}_{K^p,\overline{\bb{Q}}_p}$. More precisely, it follows directly from \ref{cor:vanish} (whose assumptions are satisfied by \ref{prop:camargo-sen-comp} and \ref{thm:camargo}) that every point $\widetilde{y}\in \mrm{Sh}^{\mrm{tor}}_{K^p,\overline{\bb{Q}}_p}$ satisfies the following properties:
	\begin{enumerate}
		\renewcommand{\labelenumi}{{\rm(\theenumi)}}
		\item Its residue field $\kappa(\widetilde{y})$ is a pre-perfectoid field with respect to any valuation ring $W$ of height $1$ extension of $\overline{\bb{Z}}_p$ with fraction field $W[1/p]=\kappa(\widetilde{y})$.\label{item:rem:vanish-shimura-1-1}
		\item For any regular system of parameters $\{t_1,\dots,t_s\}$ of the strict Henselization $\ca{O}_{\mrm{Sh}^{\mrm{tor}}_{K^pK_p,\overline{\bb{Q}}_p},y}^{\mrm{sh}}$ of $\ca{O}_{\mrm{Sh}^{\mrm{tor}}_{K^pK_p,\overline{\bb{Q}}_p},y}$ (where $y\in \mrm{Sh}^{\mrm{tor}}_{K^pK_p,\overline{\bb{Q}}_p}$ is the image of $\widetilde{y}\in\mrm{Sh}^{\mrm{tor}}_{K^p,\overline{\bb{Q}}_p}$) such that $D_{\overline{\bb{Q}}_p}$ is defined by $t_1\cdots t_r=0$ over $\ca{O}_{\mrm{Sh}^{\mrm{tor}}_{K^pK_p,\overline{\bb{Q}}_p},y}^{\mrm{sh}}$ for some integer $0\leq r\leq s$, the elements $t_1,\dots,t_r$ admit compatible systems of $p$-power roots in the strict Henselization $\ca{O}_{\mrm{Sh}^{\mrm{tor}}_{K^p,\overline{\bb{Q}}_p},\widetilde{y}}^{\mrm{sh}}$ of $\ca{O}_{\mrm{Sh}^{\mrm{tor}}_{K^p,\overline{\bb{Q}}_p},\widetilde{y}}$.\label{item:rem:vanish-shimura-1-2} 
	\end{enumerate}
\end{myrem}

\begin{myrem}\label{rem:vanish-shimura-2}
	In fact, \ref{thm:vanish-shimura} can be also proved without using geometric Sen operators but only Faltings extension. Indeed, the key step in Rodr\'iguez Camargo's proof of \ref{thm:camargo} is to prove that the Faltings extension of $\ca{S}h^{\mrm{tor}}_{K^pK_p,L}$ (as an extension of finite locally free $\widehat{\ca{O}}_{\ca{S}h^{\mrm{tor}}_{K^pK_p,L,\proket}}$-modules) is trivialized by an \'etale covering of the $\ca{G}_{K_p}$-torsor $\ca{S}h^{\mrm{tor}}_{K^p,L}$ (\cite[Theorem 5.1.4]{camargo2022completed}, see \cite[4.2.2]{pan2021locally} for the curve case). This verifies the condition \ref{lem:perfd-val-log-1}.(\ref{item:lem:perfd-val-log-1-1}) for any geometric valuative point. Hence, \ref{thm:vanish-shimura} follows from \ref{thm:perfd-val-log-1}, \ref{thm:perfd-val-log-2} and \ref{thm:vanish}. Rodr\'iguez Camargo's proof for this fact essentially relies on the $p$-adic Riemann-Hilbert correspondence for Shimura varieties established in \cite{dllz2023rh} and its comparison with the complex counterpart. We wish to find a more elementary argument in the future.
	
	We also remark that assuming the poly-stable modification conjecture \cite[12.5]{he2024purity}, one can avoid the analysis \ref{thm:perfd-val-log-2} of the ramification at the boundary points in order to deduce \ref{thm:vanish-shimura}. Indeed, the perfectoidness of all the Riemann-Zariski stalks lying over the open subspace $\mrm{Sh}_{K^p,\overline{\bb{Q}}_p}\subseteq \mrm{Sh}^{\mrm{tor}}_{K^p,\overline{\bb{Q}}_p}$ is sufficient for the vanishing of \'etale cohomology in higher degrees by purity of perfectoidness \cite[12.15]{he2024purity}.
\end{myrem}

\begin{myrem}\label{rem:camargo-vanishing}
	Roughly speaking, geometric Sen theory is known for computing the \'etale cohomology with rational coefficients by the Lie algebra cohomology associated to the universal geometric Sen action. The latter vanishes due to the ``splitting injectivity" condition and a variant of Poincar\'e's lemma (see \cite[Proposition 6.2.8]{camargo2022completed}). This is roughly how Rodr\'iguez Camargo \cite[Corollary 6.2.12]{camargo2022completed} proves the vanishing for \'etale cohomology of Shimura varieties in higher degrees with rational coefficients, i.e., $(\lim_{n\to\infty}	H^q_\et(\mrm{Sh}_{K^p,\bb{C}},\bb{Z}/p^n\bb{Z}))[1/p]=0$ for any integer $q>d$. These arguments only work for rational coefficients. Nevertheless, our approach actually uncovers the integral $p$-adic geometric structure hidden within the geometric Sen theory.
\end{myrem}

\bibliographystyle{myalpha}
\bibliography{bibli}
\end{document}